\documentclass[a4paper,reqno,10pt]{amsart}

\usepackage{amssymb}
\usepackage{amsmath}
\usepackage{amsthm}
\usepackage{enumitem}
\usepackage{xcolor}
\usepackage{graphicx}
\usepackage{url}
\usepackage{textcomp}
\usepackage{ytableau}
\usepackage{tikz}
\usepackage{etoc}
\usepackage{fullpage}

\usepackage{todonotes}

\usepackage{texdraw,amstext,amsfonts,color,tabu}

\allowdisplaybreaks

\usepackage{mathtools,tabularx}

\usepackage{float}

\definecolor{foge}{rgb}{0.1, 0.6, 0.1}

\newcommand{\Z}{\mathbb{Z}}
\newcommand{\Zz}{\Z_{\geq 0}}
\newcommand{\Zu}{\Z_{\geq 1}}
\newcommand{\R}{\mathbb{R}}
\newcommand{\C}{\mathcal{C}}
\newcommand{\F}{\mathcal{F}}

\newcommand{\la}{\lambda}

\newcommand{\Lkl}{\mathcal{L}^{(k,l)}}
\newcommand{\Llk}{\mathcal{L}^{(l,k)}}
\newcommand{\Bkl}{\mathcal{B}^{(k,l)}}
\newcommand{\Blk}{\mathcal{B}^{(l,k)}}
\newcommand{\Lk}{\mathcal{L}^{(k,1)}}
\newcommand{\Ll}{\mathcal{L}^{(1,k)}}
\newcommand{\Bk}{\mathcal{B}^{(k,1)}}
\newcommand{\Bl}{\mathcal{B}^{(1,k)}}
\newcommand{\akl}[1]{a^{(k,l)}_{#1}}
\newcommand{\alk}[1]{a^{(l,k)}_{#1}}
\newcommand{\ak}[1]{a^{(k,1)}_{#1}}
\newcommand{\al}[1]{a^{(1,k)}_{#1}}
\newcommand{\bkl}[1]{b^{(k,l)}_{#1}}
\newcommand{\blk}[1]{b^{(l,k)}_{#1}}
\newcommand{\bk}[1]{b^{(k,1)}_{#1}}
\newcommand{\bl}[1]{b^{(1,k)}_{#1}}
\newcommand{\skl}[1]{s^{(k,l)}_{#1}}
\newcommand{\slk}[1]{s^{(l,k)}_{#1}}
\newcommand{\sk}[1]{s^{(k,1)}_{#1}}
\newcommand{\sll}[1]{s^{(1,k)}_{#1}}

\numberwithin{equation}{subsection}

\newtheorem{theo}{Theorem}[section]

\newtheorem{prop}[theo]{Proposition}
\newtheorem{lem}[theo]{Lemma}

\newtheorem{rem}[theo]{Remark}
\newtheorem{ex}[theo]{Example}
\newtheorem{exs}[theo]{Examples}
\theoremstyle{definition} 
\newtheorem{deff}[theo]{Definition}

\title{The $(k,l)$-Euler theorem and the combinatorics of $(k,l)$-sequences}
\author{Isaac KONAN} \address{IRIF \\ University
of Paris\\  Paris,  75013, France}
\email{konan@irif.fr}
\date{}
\begin{document}
\begin{abstract}
In 1997, Bousquet-M\'elou and Eriksson stated a broad generalization of Euler's distinct-odd partition theorem, namely the $(k,l)$-Euler theorem. Their identity involved the $(k,l)$-lecture-hall partitions, which, unlike usual difference conditions of partitions in Rogers-Ramanujan type identities, satisfy some ratio constraints. In a 2008 paper, in response to a question suggested by Richard Stanley, Savage and Yee provided a simple bijection for the $l$-lecture-hall partitions (the case $k=l$), whose specialization in $l=2$ corresponds to Sylvester's bijection. Subsequently, as an open question, a generalization of their bijection was suggested for the case $k,l\geq 2$.\\
In the spirit of Savage and Yee's work, we provide and prove in this paper slight variations of the suggested bijection, not only for the case $k,l\geq 2$, but also for the cases $(k,1)$ and $(1,k)$ with $k\geq 4$. Furthermore, we show that our bijections equal the recursive bijections given by Bousquet-M\'elou and Eriksson in their recursive proof of the $(k,l)$-lecture hall theorem and finally provide the analogous recursive bijection for the $(k,l)$-Euler theorem.
\end{abstract}
\maketitle
\tableofcontents
\section{Introduction}
Let $\la$ be a \textit{finite} sequence $(\la_1,\ldots,\la_t)$ of non-negative integers. The integers $\la_1,\ldots,\la_t$, $t$  and $\la_1+\cdots+\la_t=|\la|$ are respectively called the \textit{parts}, the \textit{length} and the \textit{weight} of $\la$. By convention, for $\la=\emptyset$, the length and weight both equal $0$.
Let $|\la|_e,|\la|_o$ be respectively the even and odd weights of $\la$, and such that 
$$|\la|_e = \sum_{i \text{ even}} \la_i\qquad , \qquad |\la|_o = \sum_{i \text{ odd}} \la_i\,\,\cdot$$ 
Thus, $|\la|_e+|\la|_o=|\la|$. Conventionally, a \textbf{partition} of an integer $n$ refers to such finite sequence $\la$ with positive parts, occurring in a \textit{non-increasing} order, and with weight $n$. The sequence $\emptyset$ is then the unique partition of $0$.
\begin{exs} The sequence $(0,0,1,2,5,5)$ has length $6$, even weight $7$ and odd weight $6$, while the sequence $(0,1,2,5,5)$ has length $5$, even weight $6$ and odd weight $7$.
\end{exs}
A major identity of the partition theory, Euler's theorem, states that for $n\in \Zz$, the sets of partitions of $n$ into respectively distinct parts and odd parts are equinumerous. There have been number of generalizations of Euler's theorem, mainly as identities dealing with partitions whose parts satisfied difference constraints \cite{AN66,G83,PAPO98,Z05}. Yet, in this paper, our discussion focuses on  an intricate identity given by Bousquet-M\'elou and Eriksson \cite{BME971}, with successive parts constrained by a certain ratio.
\begin{theo}[Lecture hall theorem]\label{theo:lecturehall}
Let $n$ be a positive integer, and denote by $\mathcal{L}_n$ the set of partitions $\la=(\la_1,\ldots,\la_t)$ such that  $t\leq n$ and $(\frac{\la_i}{n+1-i})_{i=1}^t$ is non-increasing. Then, for $m\in \Zz$, the number of partitions of $m$ in $\mathcal{L}_n$ is equal to the number of partitions of $m$ into odd parts less than $2n$. The corresponding identity is 
$$\sum_{\la\in \mathcal{L}_n} q^{|\la|} = \prod_{i=1}^n \frac{1}{1-q^{2i-1}}\,\cdot$$
Such partitions $\la$ are referred to as lecture hall partitions.
\end{theo}
The above theorem yields the Euler identity when $n$ tends to $\infty$. Subsequently, Bousquet-M\'elou and Eriksson discovered in \cite{BME972} a generalization of Theorem \ref{theo:lecturehall} involving a particular sequence of integers named the $(k,l)$-sequence.
For $k,l$ positive integers such that $kl\geq 4$, the $(k,l)$-sequence consists of the integers $(\akl{n})_{n\in \Z}$ recursively defined by
\begin{equation}\label{eq:klseqdef}
\begin{cases}
\akl{2n}=l\akl{2n-1}-\akl{2n-2}  \\
\akl{2n+1}=k\akl{2n}-\akl{2n-1}
\end{cases}
\end{equation}
for $n\in \Z$, with the initial conditions $\akl{i}=i$ for $i \in \{0,1\}$. When $k=l$, the sequence is called an $l$-sequence, and we note $a_n^{(l)}$ instead of $a_n^{(k,l)}$. By abuse of notation, in the remainder of the paper, when $k=l\geq 2$, we replace $(k,l)$ by $(l)$. This convention will be helpful in the study of the cases $(k,1)$ and $(1,k)$ for $k\geq 4$.
\begin{enumerate}
\item Let $n\in \Zu$. Denote by $\Lkl_n$ the set of non-negative integers sequences $\la$ with length $n$, i.e. $\la=(\la_1,\ldots,\la_n)$, and such that
\begin{equation}\label{eq:ratioklseq}
\frac{\la_1}{\akl{1}}\leq \frac{\la_2}{\akl{2}}\leq \cdots \leq \frac{\la_{n-1}}{\akl{n-1}}\leq \frac{\la_n}{\akl{n}}\,\cdot
\end{equation}
The sequences $\la$ are called the $(k,l)$\textbf{-lecture hall partitions}.
\item For $m\in \Z_{\geq 4}$, let $u_m$ be the greatest radical of the equation 
$0=x^2-\sqrt{m}\cdot x +1\,\cdot$
Denote by $\Lkl$ the set of sequences $\la$ with a positive even number of parts, i.e. $\la=(\la_1,\ldots,\la_{2t})$, and such that $0=\la_{2t}\leq \la_{2t-1}$,  and for $1\leq i\leq t-1$,
\begin{equation}\label{eq:klseq}
\sqrt{\frac{k}{l}} u_{kl} \cdot \la_{2i+1}<\la_{2i}< \sqrt{\frac{k}{l}} u^{-1}_{kl}\la_{2i-1}\,\cdot
\end{equation}
The set $\Lkl$ defines the set of $(k,l)$\textbf{-Euler partitions}.
\item Define $\bkl{i}=\akl{i}+\alk{i-1}$ for $i\in \Z$. Then, for $n\in \Zu$, let $\Bkl_n$ be the set of finite sequences $\la=(\bkl{i_1},\ldots,\bkl{i_t})$ such that $1\leq i_1\leq \cdots \leq i_t \leq n $, and set $$\Bkl= \bigcup_{n\geq 1} \Bkl_n\,\cdot$$ 
Note that the sequences in $\Bkl$ are determined by the number of occurrences of each $\bkl{i}$. To ease the notation, we denote by $$\displaystyle \prod_{i\geq 1} \left(\bkl{i}\right)^{m_i}$$ the sequence of $\Bkl$ with $m_i$ occurrences of $\bkl{i}$ for $i\geq 1$.
\end{enumerate}
\begin{rem} When $k=l=2$, $\akl{n}=n$ and $\bkl{n}=2n-1$  for $n\in \Z$. Hence, $\Bkl_n$ corresponds to the set of partitions into odd parts less than $2n$, and $\Bkl$ is to the set of partitions into odd parts. Furthermore, $\Lkl_n$ can be identified to $\mathcal{L}_n$.  In fact, for $\la=(\la_i)_{i=1}^n\in\Lkl_n$ with $t$ positive parts, i.e $\la_{n-t}=0<\la_{n+1-t}$, it suffices to consider the partition $\mu=(\mu_i)_{i=1}^t=(\la_{n+1-i})_{i=1}^t$ which, by definition, belongs to  $\Lkl_n$. Inversely, for $\mu=(\mu_i)_{i=1}^t\in \mathcal{L}_n$, set $\la_i=0$ for $1\leq i\leq n-t$, and $\la_{i}=\mu_{n+1-i}$ for $n+1-t\leq i\leq n$, and one can check that $\la=(\la_i)_{i=1}^n$ belongs to $\Lkl_n$. Finally, as $\sqrt{\frac{k}{l}} u_{kl} = \sqrt{\frac{k}{l}} u_{kl}^{-1}=1$ for $k=l=2$, $\Lkl$ then corresponds to the set of partitions into distinct parts. From any partition into distinct parts, the corresponding Euler partition is obtained by adding one or two parts equal to $0$ according to whether its length is odd or even. 
\end{rem}
The Bousquet-M\'elou--Eriksson generalization of Theorem \ref{theo:lecturehall}, denoted the $(k,l)$\textbf{-lecture hall theorem}, is the following.
\begin{theo}[The $(k,l)$-lecture hall theorem]\label{theo:klseqfin}
Let $k,l,n$ be positive integers such that $kl\geq 4$. Then, 
\begin{equation}\label{eq:klseqeven}
\sum_{\la \in \Lkl_{2n}} x^{|\la|_o}y^{|\la|_e}=\prod_{i=1}^{2n}\frac{1}{1-x^{\alk{i-1}}y^{\akl{i}}}\,,
\end{equation}
\begin{equation}\label{eq:klseqodd}
\sum_{\la \in \Lkl_{2n-1}} x^{|\la|_o}y^{|\la|_e}=\prod_{i=1}^{2n-1}\frac{1}{1-x^{\alk{i}}y^{\akl{i-1}}}\,\cdot
\end{equation}
This implies that, for fixed weight $m$ in $\Zz$, there are as many $(k,l)$-lecture hall partitions in $\Lkl_{2n}$ as sequences in $\Bkl_{2n}$,
and there are as many $(k,l)$-lecture hall partitions in $\Lkl_{2n-1}$ as sequences in $\Blk_{2n-1}$. 
\end{theo}
We recover the lecture hall theorem with the transformation $(k,l,x,y) \mapsto (2,2,q,q)$. As Euler's theorem is implied by Theorem \ref{theo:lecturehall} with $n\rightarrow \infty$, 
a suitable generalization can then be deduced from \eqref{eq:klseqeven} using the same approach. 
\begin{theo}[The $(k,l)$-Euler theorem]\label{theo:klseqinf}
Let $k,l$ be positive integers such that $kl\geq 4$. Then, 
\begin{equation}\label{eq:klseqeveninf}
\sum_{\la \in \Lkl} x^{|\la|_o}y^{|\la|_e}=\prod_{i=1}^{\infty}\frac{1}{1-x^{\akl{i}}y^{\alk{i-1}}}\,\cdot
\end{equation}
This implies that, for fixed weight $m$ in $\Zz$, there are as many $(k,l)$-Euler partitions in $\Lkl$ as sequences in $\Bkl$.
\end{theo}
Many variations of the lecture hall theorem can be found in the literature \cite{CLS05,CLS15,CS04,CSS12}, and  allow a deeper comprehension of sequences constrained by ratio conditions. Lecture hall partitions have raised interest not only in combinatorics, but also in number theory, algebra and geometry, and allow new interpretations and generalizations of classical theorems.\\
Our matter in this paper concerns the combinatorial aspects of these structures. One of the most effective way to uncover the combinatorial structures of mathematical objects consists in constructing bijections. In their proof of Theorem \ref{theo:klseqfin}, Bousquet-M\'elou and Eriksson used generating functions satisfying some recursive relations, recursion which indeed gives rise to a recursive bijection. However, the quest of finding a simple bijective proof remained unsolved until the work of Savage and Yee \cite{SY08}.\\
In the spirit of a bijective proof of Theorem \ref{theo:lecturehall} given by Yee \cite{Y01}, they provided simple bijections for the $l$-lecture hall and $l$-Euler theorems, which are the restrictions of respectively Theorems \ref{theo:klseqfin} and \ref{theo:klseqinf} for the case $k=l\geq 2$. They introduced for that purpose two novel tools, the Sylvester diagrams and $a^{(l)}$-interpretations of partitions. The method has a great connection to the Sylvester bijection for Euler's theorem \cite{Sy82}, which is recovered in the case $k=l=2$ of their bijection. They also suggested, without proof, a generalization of their method for the case $k,l\geq 2$.\\\\
The goal of this paper is to prove the well-definedness of the suggested bijection for $k,l\geq 2$, and also construct the suitable analogous bijection for the cases $(k,1)$ and $(1,k)$ with $k\geq 4$. In addition, in response to open question of Savage in \cite{S16}, we prove that the bijections for the $(k,l)$-lecture hall theorem are equivalent to the Bousquet-M\'elou--Eriksson recursive bijection, and provide the analogous recurvise bijection for the $(k,l)$-Euler theorem. Our argumentation dwells on a revisited $a^{(k,l)}$-interpretation of partitions. The case $k,l\geq 2$ is closely related the $a^{(l)}$-interpretation of Savage and Yee, and can be viewed as such but presenting in a more formalized way. The novelty of our work indeed consists in adapting the formal method to deal with the cases $(k,1)$ and $(1,k)$, and for that purpose, the main trick is to view certain parts, not as integers, but as pairs of integers satisfying simple arithmetic conditions. \\\\
The paper is organized as follows. In Section \ref{sec:prem}, we present some fundamental properties of the $(k,l)$-sequence and some tools and lemmas that will be useful to comprehend the mechanics of the bijections. 
Section \ref{sec:bijections} is dedicated to the description of our bijections for Theorem \ref{theo:klseqinf} and  Theorem \ref{theo:klseqfin}. Then, in Section \ref{sec:words}, we investigate the combinatorics  behind the structure of $(k,l)$-sequence through the $(k,l)$-admissible words, and provide in these terms interpretations of lecture hall and Euler partitions. In Sections \ref{part:interinf}, \ref{part:intereven} and \ref{part:interodd}, we prove the well-definedness of our bijections respectively for Theorem \ref{theo:klseqinf}, and Theorem \ref{theo:klseqfin} in the even and odd cases. 
Finally, we show in Section \ref{sec:equiv} the equivalence between our bijections and the recursive bijection, construct the analogous recursive bijection for the $(k,l)$-Euler theorem and prove its equivalence to our bijections.\\
For the convenience of the reader, we postpone the proof of technical lemmas and propositions to the appendices. 
\section{Preliminary} \label{sec:prem}
In this section, we outline key properties of the $(k,l)$-sequence, and provide simple tools to deal with the ratio constraints of our partitions.
\subsection{Properties of $(k,l)$-sequences}
Let $k,l$ be positive integers such that $kl\geq 4$. Recall that $u_{m}$ is the greatest radical of the equation $x^2+1=\sqrt{m}\cdot x$. Then, $u_{kl}=\frac{\sqrt{kl}+\sqrt{kl-4}}{2}$, $u_{kl}^{-1}=\frac{\sqrt{kl}-\sqrt{kl-4}}{2}$, and one can check that, for $n\in \Z$, 
\begin{equation}\label{eq:formuleklu}
\begin{cases}
\akl{2n} = \displaystyle \sqrt{\frac{l}{k}}\cdot \frac{u_{kl}^{2n}-u_{kl}^{-2n}}{u_{kl}-u_{kl}^{-1}}\\\\
\akl{2n+1} = \displaystyle \frac{u_{kl}^{2n+1}-u_{kl}^{-2n-1}}{u_{kl}-u_{kl}^{-1}}
\end{cases}\cdot
\end{equation}
Note that both expressions can be indeed written as polynomials in $u_{kl},u_{kl}^{-1}$, so that, in terms of limit, they \textit{still} hold for $u_{kl}=1$, i.e. , $kl=4$. We also observe that, for $n\in \Z$, $\akl{-n}=\akl{n}$. Moreover, since $u_m^2=u_{(m-2)^2}$, for $n\in \Z$,
\begin{equation}\label{eq:formulek1}
\begin{cases}
\akl{2n} = l a_{n}^{(kl-2)}\,,\\
\akl{2n-1} = a_{n}^{(kl-2)}+a_{n-1}^{(kl-2)}
\end{cases}\,,
\end{equation}
so that $\akl{2n-1}=\alk{2n-1}$ and $k\akl{2n}=l\alk{2n}$. Furthermore,
the sequences $(\akl{2n})_{n\geq 0}$ and $(\akl{2n+1})_{n\geq 0}$ are increasing. By \eqref{eq:formuleklu}, for $n\geq 1$
\begin{align*}
\akl{2n}-\akl{2n-2}&= \sqrt{\frac{l}{k}}(u^{2n-1}+u^{-2n+1})\,,\\
\akl{2n+1}-\akl{2n-1}&= u^{2n}+u^{-2n}
\end{align*}
so that both differences are positive. In particular, for $k,l\geq 2$, $(\akl{n})_{n\geq 0}$ is increasing, as
\begin{align*}
\akl{2n}&>\frac{1}{2}(\akl{2n}+\akl{2n-2})= \frac{l}{2}\akl{2n-1}\,,\\
\akl{2n+1}&>\frac{1}{2}(\akl{2n+1}+\akl{2n-1})= \frac{k}{2}\akl{2n}\,\cdot
\end{align*}
$$$$
Finally,for $n\in \Z$, we set $\skl{2n+1}=u_{kl}^{-2n}$ and $\skl{2n}=\sqrt{l/k}\cdot u^{-2n+1}_{kl}$. One can check that $(\skl{n})_{n\in\Z}$ satisfies the recursive relation \eqref{eq:klseqdef}, and the inequalities \eqref{eq:klseq} in the $(k,l)$-Euler partitions become
\begin{equation}\label{eq:klseqbis}
\begin{cases}
 \la_{2i-1}>\skl{0}\cdot\la_{2i} \quad \text{for}\quad 1\leq i\leq t-1 \,,\\
 \la_{2i}> \slk{0}\cdot \la_{2i+1} \quad \text{for}\quad 1\leq i\leq t-1 \,\cdot
\end{cases}
\end{equation}
Observe that $(\slk{0})^{-1}=\skl{2}$, and since $u_{kl}\geq 1 $, the sequences $(\skl{2n})_{n\geq 0}$ and $(\skl{2n-1})_{n\geq 1}$ are non-increasing. In particular, when $k,l\geq 2$, $(\skl{n})_{n\geq 0}$ is non-increasing. Furthermore, for $n\in\Z$, $\skl{2n}=\skl{2}s_{n}^{(kl-2)}$ and $\skl{2n-1}=s_{n}^{(kl-2)}$  
\begin{exs}$\quad$
 \begin{enumerate}
  \item For $(k,l)=(2,2)$, $a_{n}^{(2,2)}=a_{n}^{(2)}=n$ and $s_{n}^{(2,2)}=s_n^{(2)}=1$ for $n\in \Z$. 
  \item For $(k,l)=(4,1)$, $a_{2n}^{(1,4)}=4a_{2n}^{(4,1)}=4n$, and $a_{2n-1}^{(4,1)}=2n-1$, while $s_{2n}^{(1,4)}=4s_{2n}^{(4,1)}=2$ and $s_{2n-1}^{(4,1)}=1$ for $n\in \Z$.
  \item For $(k,l)=(5,1)$, $u_5$ equals the golden number $\phi=\frac{1+\sqrt{5}}{2}$, and $a_{2n}^{(1,5)}=5a_{2n}^{(5,1)}=5F_{2n}$, and $a_{2n-1}^{(5,1)}=F_{2n}+F_{2n-2}$, where $(F_n)_{n\in \Z}$ is the Fibonacci sequence, satisfying $(F_0,F_1)=(0,1)$ and $F_{n+2}=F_{n+1}+F_n$ for all $n\in \Z$. Moreover, $s_{2n}^{(1,5)}=5s_{2n}^{(5,1)}=\sqrt{5}\phi^{1-2n}$ and $s_{2n+1}^{(5,1)}=\phi^{-2n}$ for $n\in \Z$.
 \end{enumerate}
\end{exs}
We now state in the following proposition the main properties satisfied by the $(k,l)$-sequences, and which are fundamental for the study of the ratio constraints.
\begin{prop}\label{prop:propklgen}
The following properties hold.
\begin{enumerate}
\item For $n,m\in \Z$,
\begin{align}
\alk{2n-1}\akl{m}-\akl{2n}\alk{m-1} &= \akl{2n-m}\,,\label{eq:crosseven}\\
\akl{2n}\alk{m+1}-\alk{2n+1}\akl{m} &= \akl{2n-m}\,\cdot\label{eq:crossodd}
\end{align}
\item For $n\in\Z$,
\begin{align}
\akl{n}-\skl{0}\cdot\alk{n-1} &= \skl{n} \label{eq:ratio}
\end{align}
Hence, the sequence $(\akl{n+1}/\alk{n})_{n\geq 1}$ is decreasing and tends to $\skl{0}$.
\end{enumerate}
\end{prop}
\subsection{Tools and notations}
In the remainder of this paper, $\chi(prop)$ equals $1$ if $prop$ is true and $0$ if not, and for $i,n \in \Z$, $\delta_{i,n}=\chi(i=n)$. The notation $\R$ denotes the set of real numbers, and for $x\in \R$, $\lfloor x \rfloor$ denotes the greatest integer smaller or equal to $x$, i.e. $$\lfloor x \rfloor\leq x < \lfloor x \rfloor+1\,,$$
and $\lceil x \rceil$ refers to the smallest integer greater or equal to $x$, i.e.
$$\lceil x \rceil-1< x\leq \lceil x \rceil\,\cdot$$
\begin{rem}\label{rem:sumreal}
For $x,y\in \R$, if $x+y\in \Z$, then $x+y=\lceil x \rceil+\lfloor y \rfloor$, as 
$$x+y-1<\lceil x \rceil+\lfloor y \rfloor<x+y+1\,\cdot$$
\end{rem}
\begin{rem}\label{rem:rat}
For $r,s\in \Z$ and $a\in \R_{>0}$, 
\begin{align*}
 r> as &\Longleftrightarrow r\geq \lfloor as \rfloor+1 \Longleftrightarrow \lceil r/a \rceil-1\geq s\,,\\
 r\geq  as &\Longleftrightarrow r\geq \lceil as \rceil \Longleftrightarrow \lfloor r/a \rfloor\geq s\,\cdot
\end{align*}
\end{rem}
Remark \ref{rem:rat} yields a more accurate inequality between the parts involved in $(k,l)$-lecture hall and $(k,l)$-Euler partitions.
\subsubsection{Divide and conquer}
The key lemma playing a fundamental role in the combinatorics of $(k,1)$ and $(1,k)$ sequences is the following.
\begin{lem}\label{lem:divide}
 Let $\frac{1}{2}\leq x<1$ be a real number. 
 \begin{enumerate}
  \item Set $S_x^+ = \{(a,b)\in \Z:  0<(1-x)a-xb\leq 1\}$.
 Then, 
\begin{align}
 \label{eq:divide+}
 S_x^+ &= \left\{(a,b)\in \Z:  1+\left\lfloor \frac{xb}{1-x}\right\rfloor\leq a\leq 1+\left\lfloor \frac{x(b+1)}{1-x} \right\rfloor\right\}\cdot
\end{align}
In particular, 
\begin{equation}
  \label{eq:divide++}
  x<(1-x)a-xb\leq 1 \Longleftrightarrow a= 1+\left\lfloor \frac{x(b+1)}{1-x} \right\rfloor\,\cdot
\end{equation}
Moreover, for $(a,b) \in  S_x^+$, 
\begin{enumerate}
 \item $(a+1,b)\in  S_x^+$ if and only if $1+\left\lfloor \frac{xb}{1-x}\right\rfloor\leq a<1+\left\lfloor \frac{x(b+1)}{1-x} \right\rfloor$,
 \item and $(a,b+1)\in  S_x^+$ if and only if $ a=1+\left\lfloor \frac{x(b+1)}{1-x} \right\rfloor$.
\end{enumerate}
Hence, the function $(a,b)\mapsto a+b$ defines a bijection from $S_x^+$ to $\Z$, whose inverse is defined by $m\mapsto (\lfloor x m \rfloor +1,\lceil (1-x)m\rceil -1)$.\\
  \item Set $S_x^- = \{(a,b)\in \Z:  0\leq (1-x)a-xb<1\}$. Then,
  \begin{align}
   \label{eq:divide-}
 S_x^- &= \left\{(a,b)\in \Z:  \left\lceil \frac{xb}{1-x}\right\rceil\leq a\leq \left\lceil \frac{x(b+1)}{1-x} \right\rceil\right\}\,\cdot
  \end{align}
  In particular, 
\begin{equation}
  \label{eq:divide--}
  x\leq (1-x)a-xb<1 \Longleftrightarrow a= \left\lceil \frac{x(b+1)}{1-x} \right\rceil\,\cdot
\end{equation}
  Moreover, for $(a,b) \in  S_x^-$, 
  \begin{enumerate}
   \item $(a+1,b)\in  S_x^-$ if and only if $\left\lceil \frac{xb}{1-x}\right\rceil\leq a< \left\lceil \frac{x(b+1)}{1-x} \right\rceil$,
   \item and $(a,b+1)\in  S_x^-$ if and only if $a=\left\lceil \frac{x(b+1)}{1-x} \right\rceil$.
  \end{enumerate}
Hence, the function $(a,b)\mapsto a+b$ defines a bijection from
  from $S_x^-$ to $\Z$, whose inverse is defined by $m\mapsto (\lceil x m \rceil,\lfloor (1-x)m\rfloor)$.
\end{enumerate}
\end{lem}
For a pair of a element $s=(s_1,s_2)$, we set $p_i(s)=s_i$  for $i\in \{1,2\}$.
Note that the map $m\mapsto (\lfloor x m \rfloor +1,\lceil (1-x)m\rceil -1)$ induces a bijection from $\Zu$ to $\{s\in S_x^+: p_2(s)\in \Zz\}$. Splitting a positive integer into such a unique pair will be the main trick that eases the study of $(k,1)$ and $(1,k)$-Euler partitions. Similarly, the map $m\mapsto (\lceil x m \rceil,\lfloor (1-x)m\rfloor)$ induces a bijection from $\Zz$ to $\{s\in S_x^-: p_2(s)\in \Zz\}$, and this splitting will ease the study of the $(k,1)$ and $(1,k)$ lecture hall partitions.\\
\subsubsection{Succession in an ordered set}
Let $(S,\preceq)$ be a countable and total ordered set. Recall that $c\prec d$ means that $c \preceq d$ and $c\neq d$. We say that $c$ and $d$  are consecutive in $S$ if and only if $c\prec d$, but there is no $e\in S$ such that $c\prec e\prec d$. For such element $c$ and $d$, $c$ \textbf{precedes} $d$ and $d$ \textbf{follows} $c$ in $S$.
In general, for $m\in \Zz$, $c$ is the $m^{th}$ element that \textbf{precedes} $d$ in $S$ or $d$ is the $m^{th}$ element that \textbf{follows} $c$ in $S$, if and only if there exists $s_0,\ldots,s_m\in S$ such that $s_0 = c$, $s_m=d$ and for $1\leq i\leq m$, $s_{i-1}$ precedes $s_i$ in $S$. We note 
\begin{equation}\label{eq:deffollower}
d= \F(m,S,c)\,\cdot
\end{equation}
\begin{rem}
Observe that $c= \F(0,S,c)$, and $d= \F(1,S,c)$ is the element that follows $c$ in $S$.
\end{rem}
The last tool of this part is presented in the following lemma and will help to prove the well-definedness of our bijections.
\begin{lem}\label{lem:prechoice}
 Let $(S,\preceq)$ be an infinite, countable  and total ordered set with a minimal element. Let $m\in \Zz$ and $c_1,c_2,d_1,d_2\in S$ such that $d_i= \F(m,S,c_i)$ for $i\in\{1,2\}$.
 Then, $c_1\preceq c_2$ if and only if $d_1\preceq d_2$. Moreover, if $d_1\preceq d_2$ and $d_1$ admits an $m^{th}$ element $c_1$ that precedes it, then $d_2$ admits an $m^{th}$ element $c_2$ that precedes it and we have $c_1\preceq c_2$.
\end{lem}  
\section{Bijections}\label{sec:bijections}
This part is dedicated to the presentation of the bijection for Theorems \ref{theo:klseqfin} and \ref{theo:klseqinf}.
Denote by $\Phi^{(k,l)}$ the bijection from $\Bkl$ to $\Lkl$, and $n\in \Zu$, denote by $\Phi_{2n}^{(k,l)}$ the bijection from $\Bkl_{2n}$ to $\Lkl_{2n}$ and by $\Phi_{2n-1}^{(k,l)}$ the bijection from $\Blk_{2n-1}$ to $\Lkl_{2n-1}$.
In the remainder of the paper, we illustrate our bijections and key definitions with $kl=6$, i.e. with $(k,l)\in\{(2,3),(3,2),(6,1),(1,6)\}$.
\subsection{The map $\Phi^{(k,l)}$}
We first describe the map $\Phi^{(k,l)}$. Let $\nu$ be a partition in $\Bkl$, and set $\la=(\la_i)_{i\geq 1}$ an infinite sequence of terms all equal to $0$. We here proceed by \textit{inserting} the parts $\bkl{i}$ into the pairs $(\la_{2j-1},\la_{2j})$, starting from the smallest $j$ and the greatest $i$.
\subsubsection{The case $k,l\geq 2$}
\begin{enumerate}
\item To insert $\bkl{i}$ with $i>1$ into $(\la_{2j-1},\la_{2j})$, proceed as follows. 
If 
\begin{equation}\label{eq:condinsinf1}
\la_{2j-1}-\skl{0}\cdot \la_{2j}> \skl{i-1}-\skl{i}\,,
\end{equation}
then do
\begin{equation}\label{eq:insert1type1}
(\la_{2j-1},\la_{2j}) \mapsto (\la_{2j-1}+ \akl{i}-\akl{i-1}, \,\la_{2j}+\alk{i-1}-\alk{i-2})
\end{equation}
and store $\bkl{i-1}$ for the insertion into the pair $(\la_{2j+1},\la_{2j+2})$. This action is equivalent to inserting one part $\bkl{i}$ by adding $\bkl{i}-\bkl{i-1}$ to the pair $(\la_{2j-1},\la_{2j})$ and storing one part $\bkl{i-1}$ for the insertion into the pair $(\la_{2j+1},\la_{2j+2})$. Else, we do
\begin{equation}\label{eq:insert1type2}
(\la_{2j-1},\la_{2j}) \mapsto (\la_{2j-1}+ \akl{i}, \,\la_{2j}+\alk{i-1})\,\cdot
\end{equation}
This action means that we insert one part $\bkl{i}$ by adding it to the pair $(\la_{2j-1},\la_{2j})$.
\item To insert $\bkl{1}$, do \eqref{eq:insert1type2} for $i=1$. 
\end{enumerate}
\subsubsection{The case $(k,1)$}
\begin{enumerate}
 \item  To insert the parts $\bk{i}$ with $i>2$ into the pair $(\la_{2j-1},\la_{2j})$, proceed as follows. 
\begin{enumerate}
\item To insert $\bk{2i}$ for $i>1$, if 
\begin{equation}\label{eq:condinsinf121}
\la_{2j-1}-\sk{0}\cdot \la_{2j}> \sk{2i-2}-\sk{2i}\,,
\end{equation}
then do
\begin{equation}\label{eq:insert1type121}
(\la_{2j-1}, \la_{2j}) \mapsto (\la_{2j-1}+ \ak{2i}-\ak{2i-2}, \,\la_{2j}+\al{2i-1}-\al{2i-3})
\end{equation}
and store, for the insertion into the pair $(\la_{2j+1},\la_{2j+2})$, $\bk{2i-1}$ if there are at least two remaining parts $\bk{2i}$ to insert, and $\bkl{2i-2}$ if there is only one remaining part $\bk{2i}$. The first action consists in inserting two parts $\bk{2i}$ by adding 
$\bk{2i}-\bk{2i-2}= 2\bk{2i}-\bk{2i-1}$ to the pair  $(\la_{2j-1},\la_{2j})$, and storing $\bk{2i-1}$, while the second action, which can occur once, consists in inserting the last remaining part $\bk{2i}$ by adding 
to the pair $\bk{2i}-\bk{2i-2}$, and storing one part $\bk{2i-2}$.
If condition \eqref{eq:condinsinf121} is not satisfied, do
\begin{equation}\label{eq:insert1type221}
(\la_{2j-1}, \la_{2j}) \mapsto (\la_{2j-1}+ \ak{2i}, \,\la_{2j}+\al{2i-1})\,\cdot
\end{equation}
This implies that we insert one part $\bk{2i}$ by adding it to the pair $(\la_{2j-1},\la_{2j})$.
\item To insert $\bk{2i-1}$ for $i>1$, if 
\eqref{eq:condinsinf121} occurs, then do \eqref{eq:insert1type121}
and store for the the insertion into the pair $(\la_{2j+1},\la_{2j+2})$ two parts $\bk{2i-2}$. This action consists in inserting one part $\bk{2i-1}$ by adding $\bk{2i}-\bk{2i-2}=\bk{2i-1}-2\bk{2i-2}$ to the pair $(\la_{2j-1},\la_{2j})$, and storing two parts $\bk{2i-2}$ for the insertion into the pair $(\la_{2j+1},\la_{2j+2})$.  Else, do \eqref{eq:insert1type221}
and store for the insertion into the pair $(\la_{2j+1},\la_{2j+2})$ one part $\bk{2i-2}$.
This action consists in inserting one part $\bk{2i-1}$ by adding $\bk{2i}$ to the pair $(\la_{2j-1},\la_{2j})$, and storing one part $\bk{2i-2}=\bk{2i-1}-\bk{2i}$ for the insertion into the pair $(\la_{2j+1},\la_{2j+2})$. 
\end{enumerate}
\item To insert $\bk{2}$, if 
\begin{equation}\label{eq:condinsinf123}
\la_{2j-1}-\sk{0}\cdot \la_{2j}> \sk{0}\,,
\end{equation}
then do
\begin{equation}\label{eq:insert1type123}
(\la_{2j-1}, \la_{2j}) \mapsto (\la_{2j-1}+ \ak{2}-\ak{1}, \,\la_{2j}+\al{1}-\al{0})\,,
\end{equation}
and store for the the insertion into the pair $(\la_{2j+1},\la_{2j+2})$ one part $\bk{1}$. Else, do \eqref{eq:insert1type221} with $i=1$.
\item To insert $\bk{1}$, do 
\begin{equation}\label{eq:insert1type1234}
(\la_{2j-1},\la_{2j}) \mapsto (\la_{2j-1}+ \ak{1},\,\la_{2j}+\al{0})\,\cdot
\end{equation}
\end{enumerate}
\subsubsection{The case $(1,k)$}
\begin{enumerate}
\item To insert the parts $\bk{i}$ with $i>1$ into the pair $(\la_{2j-1},\la_{2j})$, proceed as follows. 
\begin{enumerate}
\item To insert $\bl{2i+1}$ for $i\geq 1$, if 
\begin{equation}\label{eq:condinsinf131}
\la_{2j-1}-\sll{0}\cdot \la_{2j}> \sll{2i-1}-\sll{2i+1}\,,
\end{equation}
then do
\begin{equation}\label{eq:insert1type131}
(\la_{2j-1},\la_{2j}) \mapsto (\la_{2j-1}+ \al{2i+1}-\al{2i-1},\,\la_{2j}+\ak{2i}-\ak{2i-2})
\end{equation}
and store for the the insertion into the pair $(\la_{2j+1},\la_{2j+2})$, $\bl{2i}$ if there are at least two remaining parts $\bl{2i+1}$ to insert, and $\bl{2i-1}$ if there is only one remaining part $\bl{2i+1}$. Else, do
\begin{equation}\label{eq:insert1type231}
(\la_{2j-1},\la_{2j}) \mapsto (\la_{2j-1}+ \al{2i+1},\,\la_{2j}+\ak{2i})\,\cdot
\end{equation}
\item To insert $\bl{2i}$ for $i\geq 1$, if \eqref{eq:condinsinf131} occurs
then do \eqref{eq:insert1type131}
and store for the the insertion into the pair $(\la_{2j+1},\la_{2j+2})$ two parts $\bl{2i-1}$. Else, do \eqref{eq:insert1type231}
and store for the the insertion into the pair $(\la_{2j+1},\la_{2j+2})$ one part $\bl{2i-1}$.
\end{enumerate}
\item To insert $\bl{1}$, do \eqref{eq:insert1type231} for $i=0$.
\end{enumerate}
\bigskip
Whatever the case, we stop at $(\la_{2t-1},\la_{2t})$ when there is no  part $\bkl{i}$ to insert for $i\geq 2$. Observe that the part size of $\la_{2j}$ increases after an insertion of $\bkl{i}$ if and only if $i>1$, and there is no store part after the insertion of $\bkl{1}$. Thus, $t$ is the smallest integer $j$ such that $\la_{2j}=0$ after all the insertions, and  $\la_{2t-1}$ is then equal to the number of inserted $\bkl{1}$ into the pair $(\la_{2t-1},\la_{2t})$. We finally set 
$\Phi^{(k,l)}(\nu)=(\la_1,\ldots,\la_{2t})$.\\\\
Observe that the map $\Phi^{(k,l)}$ is weight-preserving, and that the $\akl{i}$-component and $\alk{i-1}$-component of the parts $\bkl{i}$ only contribute respectively to the odd and the even weights of the image.  
\begin{rem}
 By \eqref{eq:ratio}, the following equivalences occur: 
 \begin{align*}
 \eqref{eq:condinsinf1} &\Longleftrightarrow  \la_{2j-1}+\akl{i}-\akl{i-1}>\skl{0}(\la_{2j}+ \alk{i-1}-\alk{i-2})\,,\\
 \eqref{eq:condinsinf121} &\Longleftrightarrow  \la_{2j-1}+\ak{2i}-\ak{2i-2}>\sk{0}(\la_{2j}+ \al{2i-1}-\al{2i-3})\,,\\
 \eqref{eq:condinsinf123} &\Longleftrightarrow  \la_{2j-1}+\ak{2}-\ak{1}>\sk{0}(\la_{2j}+ \al{1}-\al{0})\,,\\
 \eqref{eq:condinsinf131} &\Longleftrightarrow  \la_{2j-1}+\al{2i+1}-\al{2i-1}>\sll{0}(\la_{2j}+ \ak{2i}-\ak{2i-2})\,\cdot\\
 \end{align*}
 In practice, we conveniently apply these equivalences to compute the image by $\Phi^{(k,l)}$. Our formulations only aim at easing the calculations in the proof of the well-definedness. Moreover, the case $k=l$ of the first equivalence yields to the condition of the bijection given by Savage and Yee for the $l$-Euleur Theorem. Note that in the case of $k,l\geq 2$, with the order of the insertions, starting from the greatest available part $\bkl{i}$, we can equivalently insert $\bkl{i}$ in $(\la_1,\la_2)$, then insert the stored part $\bk{i-1}$ in $(\la_3,\la_4)$, and so on. For this reason, our map $\Phi^{(l,l)}$ matches  Savage and Yee's bijection. 
\end{rem}
\begin{exs}\label{exs:allpaper}
Let us apply the map $\Phi^{(k,l)}$ for $kl=6$, i.e. $(k,l)\in \{(1,6),(2,3),(3,2),(6,1)\}$.
Since $(a_0^{(4)},a_1^{(4)},a_2^{(4)},a_3^{(4)})=(0,1,4,15)$ and $u_{6}=\frac{\sqrt{6}+\sqrt{2}}{2}$, we then have by \eqref{eq:formulek1}
 $$\begin{cases}
  (a_0^{(k,l)},a_2^{(k,l)},a_4^{(k,l)},a_6^{(k,l)},\ldots)=(0,l,4l,15l,\ldots)\\
  (a_1^{(k,l)},a_3^{(k,l)},a_5^{(k,l)},a_7^{(k,l)},\ldots)=(1,5,19,71,\ldots)\\
  s_0^{(k,l)} = \frac{3+\sqrt{3}}{k}
\end{cases}\,\cdot
$$
\begin{enumerate}
 \item The case $(2,3)$ with $\nu=(b^{(2,3)}_{1})^5(b^{(2,3)}_{2})^4(b^{(2,3)}_{3})^2(b^{(2,3)}_{4})^3(b^{(2,3)}_{5})(b^{(2,3)}_{6})^3=(1+0)^5(3+1)^4(5+2)^2(12+5)^3(19+8)(45+19)^3$. We start by the pair $(\la_1,\la_2)$, initially equal to $(0,0)$. 
 \begin{enumerate}
  \item Insertions of $b_{6}^{(2,3)}$: Condition \eqref{eq:condinsinf1} does not occur, as $$(a_{6}^{(2,3)}-a_{5}^{(2,3)})=26\leq \frac{3+\sqrt{3}}{2}\cdot 11 =s_0^{(2,3)}(a_{5}^{(3,2)}-a_{4}^{(3,2)})\,,$$ and $(\la_1,\la_2)$ becomes $(45,19)$ after applying \eqref{eq:insert1type2}. For the second insertion, Condition \eqref{eq:condinsinf1} occur, as $$\la_1+(a_{6}^{(2,3)}-a_{5}^{(2,3)})=71>\frac{3+\sqrt{3}}{2}\cdot 30 =s_0^{(2,3)}(\la_2+ a_{5}^{(3,2)}-a_{4}^{(3,2)})\,,$$
  and $(\la_1,\la_2)$ then becomes $(71,30)$ after applying \eqref{eq:insert1type1}, and $b_{5}^{(2,3)}$ is stored for the insertion into the pair $(\la_3,\la_4)$. Finally, for the last insertion $b_{6}^{(2,3)}$,  Condition \eqref{eq:condinsinf1} does not occur, as $$\la_1+(a_{6}^{(2,3)}-a_{5}^{(2,3)})=97\leq \frac{3+\sqrt{3}}{2}\cdot 41 =s_0^{(2,3)}(\la_2+ a_{5}^{(3,2)}-a_{4}^{(3,2)})\,,$$ and $(\la_1,\la_2)$ becomes $(116,49)$. 
  \item Insertions of $b_{5}^{(2,3)}$:  Condition \eqref{eq:condinsinf1} does not occur, since $$\la_1+(a_{5}^{(2,3)}-a_{4}^{(2,3)})=(116+7)\leq \frac{3+\sqrt{3}}{2}\cdot(49+3)=s_0^{(2,3)}(\la_2+ a_{4}^{(3,2)}-a_{3}^{(3,2)})\,,$$ and $(\la_1,\la_2)$ becomes $(135,57)$ after applying \eqref{eq:insert1type2}. 
  \item Insertions of $b_{4}^{(2,3)}$: we successively apply \eqref{eq:insert1type1}, \eqref{eq:insert1type2} and \eqref{eq:insert1type1}, and $(\la_1,\la_2)$ respectively becomes $(142,60),(154,65)$ and $(161,68)$, and we store twice $b_{3}^{(2,3)}$ for the insertion into the pair $(\la_3,\la_4)$.
  \item Insertions of $b_{3}^{(2,3)}$: we apply \eqref{eq:insert1type2} then \eqref{eq:insert1type1}, and $(\la_1,\la_2)$ becomes $(166,70)$ then $(168,71)$, and we store once $b_{2}^{(2,3)}$ for the insertion into the pair $(\la_3,\la_4)$.
  \item Insertions of $b_{2}^{(2,3)}$: we successively apply \eqref{eq:insert1type2}, \eqref{eq:insert1type1},\eqref{eq:insert1type2} and \eqref{eq:insert1type1}, and $(\la_1,\la_2)$ respectively becomes $(171,72),(173,73),(176,74)$ and $(178,75)$, and we store twice $b_{1}^{(2,3)}$ for the insertion into the pair $(\la_3,\la_4)$.
  \item Insertions of $b_{1}^{(2,3)}$: we obtain $(\la_1,\la_2)=(183,75)$ by applying five times \eqref{eq:insert1type2}.
 \end{enumerate}
In the process we stored $1\times b_{5}^{(2,3)},0\times b_{4}^{(2,3)},2\times b_{3}^{(2,3)},1\times b_{2}^{(2,3)},2\times b_{1}^{(2,3)}$ for the insertion into the pair $(\la_3,\la_4)$. For these insertions, we do once \eqref{eq:insert1type2} for $i=5$, twice \eqref{eq:insert1type2} for $i=3$, once \eqref{eq:insert1type1} for $i=2$, twice \eqref{eq:insert1type2} for $i=1$, and store once $b_{1}^{(2,3)}$ for the insertion into the pair $(\la_5,\la_6)$. Thus, $(\la_3,\la_4)=(32,13)$, and $(\la_5,\la_6)=(1,0)$. Finally, set $\Phi^{(2,3)}(\nu)=(183,75,32,13,1,0)\in\mathcal{L}^{(2,3)}$.\\
 \item The case $(3,2)$ with $\nu=(b^{(3,2)}_{1})^5(b^{(3,2)}_{2})^4(b^{(3,2)}_{3})^2(b^{(3,2)}_{4})^3(b^{(3,2)}_{5})(b^{(3,2)}_{6})^3=(1+0)^5(2+1)^4(5+3)^2(8+5)^3(19+12)(30+19)^3$. For the insertion into the pair $(\la_1,\la_2)$, we have the following.
 \begin{enumerate} 
 \item Insertions of $b_{6}^{(3,2)}$: we successively apply \eqref{eq:insert1type2}, \eqref{eq:insert1type2} and \eqref{eq:insert1type1} to obtain $(\la_1,\la_2)=(71,45)$, and store once $b_{5}^{(3,2)}$ for the pair $(\la_3,\la_4)$.
 \item Insertions of $b_{5}^{(3,2)}$: we apply \eqref{eq:insert1type2} to obtain $(\la_1,\la_2)=(90,57)$.
 \item Insertions of $b_{4}^{(3,2)}$: we successively apply \eqref{eq:insert1type2}, \eqref{eq:insert1type1} and \eqref{eq:insert1type2} to obtain $(\la_1,\la_2)=(109,69)$, and store once $b_{3}^{(3,2)}$ for the pair $(\la_3,\la_4)$.
 \item Insertions of $b_{3}^{(3,2)}$: we successively apply \eqref{eq:insert1type1} and \eqref{eq:insert1type2} to obtain $(\la_1,\la_2)=(117,74)$, and store once $b_{2}^{(3,2)}$ for the pair $(\la_3,\la_4)$.
 \item Insertions of $b_{2}^{(3,2)}$: we successively apply \eqref{eq:insert1type2}, \eqref{eq:insert1type1} , \eqref{eq:insert1type2} and \eqref{eq:insert1type2} to obtain $(\la_1,\la_2)=(124,78)$, and store once $b_{1}^{(3,2)}$ for the pair $(\la_3,\la_4)$.
 \item Insertions of $b_{1}^{(3,2)}$: we apply five times \eqref{eq:insert1type2} to obtain $(\la_1,\la_2)=(129,78)$.
 \end{enumerate}
 Hence, we store once $b_{5}^{(3,2)},b_{3}^{(3,2)},b_{2}^{(3,2)},b_{1}^{(3,2)}$ for the insertion into the pair $(\la_3,\la_4)$. We then do \eqref{eq:insert1type2} for $i=5,3,2,1$ to obtain $(\la_3,\la_4)=(27,16)$. As there is no part stored for the insertion in $(\la_5,\la_6)$, we have $(\la_5,\la_6)=(0,0)$. Set $\Phi^{(3,2)}(\nu)=(129,78,27,16,0,0)\in\mathcal{L}^{(3,2)}$.\\
 \item The case $(6,1)$ with $\nu=(b_{1}^{(6,1)})^2(b_{2}^{(6,1)})^5(b_{3}^{(6,1)})^2(b_{4}^{(6,1)})^3(b_{6}^{(6,1)})^5=(1+0)^2(1+1)^5(5+6)^2(4+5)^3(15+19)^5$. For the insertion into the pair $(\la_1,\la_2)$, we have the following.
 \begin{enumerate}
  \item Insertions of $b_{6}^{(6,1)}$: 
 for $i=3$, we first apply thrice \eqref{eq:insert1type221}, and obtain $(\la_1,\la_2)=(45,57)$. Then, we apply  \eqref{eq:insert1type121} to obtain $(\la_1,\la_2)=(56,71)$ and store $b_{5}^{(6,1)}$ for the insertion into the pair $(\la_3,\la_4)$.
 \item Insertions of $b_{4}^{(6,1)}$: we apply thrice \eqref{eq:insert1type221} with $i=2$, and obtain $(\la_1,\la_2)=(68,86)$.
 \item Insertions of $b_{3}^{(6,1)}$: we apply successively \eqref{eq:insert1type121} and \eqref{eq:insert1type221} with $i=2$ to obtain $(\la_1,\la_2)=(75,95)$, and successively store two and one $b_{2}^{(6,1)}$ for the insertion into the pair $(\la_3,\la_4)$.
 \item Insertions of $b_{2}^{(6,1)}$: we apply four times \eqref{eq:insert1type221} for $i=1$ to obtain $(\la_1,\la_2)=(79,99)$, then apply \eqref{eq:insert1type123} to obtain $(\la_1,\la_2)=(79,100)$, and store $b_{1}^{(6,1)}$ for the insertion into the pair $(\la_3,\la_4)$.
 \item Insertions of $b_{1}^{(6,1)}$: we apply twice \eqref{eq:insert1type1234} and obtain $(\la_1,\la_2)=(81,100)$.
 \end{enumerate}
 We store $1\times b_{5}^{(6,1)},3\times b_{2}^{(6,1)},b_{1}^{(6,1)}$ for the insertion into the pair $(\la_3,\la_4)$, and only apply \eqref{eq:insert1type221} for $i=2,1$ to obtain $(\la_3,\la_4)=(19,22)$, and store $b_{4}^{(6,1)}$ for the insertion into the pair $(\la_5,\la_6)$. Then, for the pair $(\la_5,\la_6)$, we do once \eqref{eq:insert1type221} with $i=2$ and obtain $(\la_5,\la_6)=(4,5)$. We finally set $\Phi^{(6,1)}(\nu)=(81,100,19,22,4,5,0,0)\in \mathcal{L}^{(6,1)}$.\\
 \item The case $(1,6)$ with $\nu=(b_{1}^{(1,6)})^2(b_{2}^{(1,6)})^5(b_{3}^{(1,6)})^2(b_{4}^{(1,6)})^3(b_{6}^{(1,6)})^5=(1+0)^2(6+1)^5(5+1)^2(24+5)^3(90+19)^5$.
 We start by the pair $(\la_1,\la_2)$, initially equal to $(0,0)$. 
 \begin{enumerate}
  \item Insertions of $b_{6}^{(1,6)}$: 
 for $i=3$, we first apply thrice \eqref{eq:insert1type231} to obtain $(\la_1,\la_2)=(213,45)$, and store three $b_{5}^{(1,6)}$ for the insertion into the pair $(\la_3,\la_4)$. Then, we apply once  \eqref{eq:insert1type131} to obtain $(\la_1,\la_2)=(265,56)$ and store two $b_{5}^{(1,6)}$. Finally, we apply once \eqref{eq:insert1type231} and obtain $(\la_1,\la_2)=(336,71)$ and store $b_{5}^{(1,6)}$.
 \item Insertions of $b_{4}^{(1,6)}$: 
 for $i=2$, we apply thrice \eqref{eq:insert1type231} to obtain $(\la_1,\la_2)=(393,83)$, and store three $b_{3}^{(1,6)}$ for the insertion into the pair $(\la_3,\la_4)$.
 \item Insertions of $b_{3}^{(1,6)}$: 
 for $i=1$, we apply twice \eqref{eq:insert1type231} to obtain $(\la_1,\la_2)=(403,85)$.
 \item Insertions of $b_{2}^{(1,6)}$: 
 for $i=1$, we first apply once \eqref{eq:insert1type131} to obtain $(\la_1,\la_2)=(407,86)$, and store two $b_{1}^{(1,6)}$ for the insertion into the pair $(\la_3,\la_4)$. Then, we apply thrice \eqref{eq:insert1type231} to obtain $(\la_1,\la_2)=(422,89)$, and store three $b_{3}^{(1,6)}$. Finally, we apply once \eqref{eq:insert1type131} to obtain $(\la_1,\la_2)=(426,90)$, and store two $b_{1}^{(1,6)}$.
 \item Insertions of $b_{2}^{(1,6)}$: we apply twice \eqref{eq:insert1type231} for $i=0$ to obtain $(\la_1,\la_2)=(428,90)$.
 \end{enumerate}
 We then store $6\times b_{5}^{(1,6)},3\times b_{3}^{(1,6)},7\times b_{1}^{(1,6)}$ for the insertion into 
 into the pair $(\la_3,\la_4)$. We only apply \eqref{eq:insert1type231}, except once for $i=2$ while storing $b_{4}^{(1,6)}$ for the insertion into 
 into the pair $(\la_5,\la_6)$, and obtain $(\la_3,\la_4)=(112,22)$. For the pair $(\la_5,\la_6)$, we do 
 \eqref{eq:insert1type231} with $i=2$ to obtain $(\la_5,\la_6)=(19,4)$ and store $b_{3}^{(1,6)}$ for the insertion into 
 into the pair $(\la_7,\la_8)$. For the pair $(\la_7,\la_8)$, we do 
 \eqref{eq:insert1type231} with $i=1$ to obtain $(\la_7,\la_8)=(5,1)$. We finally set $\Phi^{(1,6)}(\nu)=(428,90,112,22,19,4,5,1,0,0)\in \mathcal{L}^{(1,6)}$.
 \end{enumerate}
\end{exs}
\subsection{The map $\Phi_{2n}^{(k,l)}$}
 We now describe the bijection $\Phi_{2n}^{(k,l)}$. Let $\nu$ be a partition in $\Bkl_{2n}$, and set $\la=(\la_1,\ldots,\la_{2n})$, all parts equal to $0$. As in the previous construction, we here proceed by step according to the insertion of the parts $\bkl{i}$ into the pairs $(\la_{2j},\la_{2j-1})$, starting from the greatest $i,j$.
 \subsubsection{The case $k,l\geq 2$}
\begin{enumerate}
\item To insert $\bkl{i}$ with $i>1$ in $(\la_{2j-1},\la_{2j})$, proceed as follows. 
If 
\begin{equation}\label{eq:condinsinf2}
\akl{2j-1}\la_{2j}-\akl{2j}\la_{2j-1}\geq \akl{2j-i+1}-\akl{2j-i}\,\,,
\end{equation}
then do 
\begin{equation}\label{eq:insert2type1}
(\la_{2j},\la_{2j-1}) \mapsto (\la_{2j}+ \akl{i}-\akl{i-1},\, \la_{2j-1}+\alk{i-1}-\alk{i-2})
\end{equation}
and store $\blk{i-1}$ for the insertion into the pair $(\la_{2j-2},\la_{2j-1})$. Else, do
\begin{equation}\label{eq:insert2type2}
(\la_{2j},\la_{2j-1}) \mapsto (\la_{2j}+ \akl{i},\, \la_{2j-1}+\alk{i-1})\,\cdot
\end{equation}
\item To insert $\bkl{1}$, do \eqref{eq:insert2type2} for $i=1$.
 \end{enumerate}
 \subsubsection{The case $(k,1)$}
 \begin{enumerate}
\item To insert the parts $\bk{i}$ with $i>2$ into the pair $(\la_{2j},\la_{2j-1})$, proceed as follows. 
\begin{enumerate}
\item To insert $\bk{2i}$ for $i>1$, if 
\begin{equation}\label{eq:condinsinf1211}
\ak{2j-1}\la_{2j}-\ak{2j}\la_{2j-1}\geq \ak{2j+2-2i}-\ak{2j-2i}\,,
\end{equation}
then do
\begin{equation}\label{eq:insert2type121}
(\la_{2j},\la_{2j-1}) \mapsto( \la_{2j}+ \ak{2i}-\ak{2i-2},\,\la_{2j-1}+\al{2i-1}-\al{2i-3})
\end{equation}
and store for the insertion in the pair $(\la_{2j-2},\la_{2j-3})$, $\bk{2i-1}$ if there are at least two remaining parts $\bk{2i}$ to insert, and $\bk{2i-2}$ if there is only one remaining part $\bk{2i}$. Else, do
\begin{equation}\label{eq:insert2type221}
(\la_{2j},\la_{2j-1}) \mapsto( \la_{2j}+ \ak{2i},\,\la_{2j-1}+\al{2i-1})\,\cdot
\end{equation}
\item To insert $\bk{2i-1}$ for $i>1$, if \eqref{eq:condinsinf1211} occurs, then do \eqref{eq:insert2type121}
and store for the insertion in the pair $(\la_{2j-2},\la_{2j-3})$ two parts $\bk{2i-2}$. Else, do \eqref{eq:insert2type221}
and store for the insertion in the pair $(\la_{2j-2},\la_{2j-3})$ one part $\bk{2i-2}$.
\end{enumerate}
\item To insert $\bk{2}$, if 
\begin{equation}\label{eq:condinsinf1231}
\ak{2j-1}\la_{2j}-\ak{2j}\la_{2j-1}\geq \ak{2j}
\end{equation}
then do
\begin{equation}\label{eq:insert2type123}
(\la_{2j},\la_{2j-1}) \mapsto(\la_{2j}+ \ak{2}-\ak{1},\,\la_{2j-1}+\al{1}-\al{0})
\end{equation}
and store for the insertion in the pair $(\la_{2j-2},\la_{2j-3})$ one part $\bk{1}$. Else, do \eqref{eq:insert2type221} for $i=1$.
\item To insert $\bk{1}$, do 
\begin{equation}\label{eq:insert2type1234}
(\la_{2j},\la_{2j-1}) \mapsto(\la_{2j}+\ak{1},\,\la_{2j-1}+\al{0})\,\cdot
\end{equation}
\end{enumerate}
\subsubsection{The case $(1,k)$}
\begin{enumerate}
\item To insert the parts $\bl{i}$ with $2j\neq i>1$ into the pair $(\la_{2j},\la_{2j-1})$, proceed as follows. 
\begin{enumerate}
\item To insert $\bl{2i+1}$ for $i\geq 1$, if 
\begin{equation}\label{eq:condinsinf1311}
\al{2j-1}\la_{2j}-\al{2j}\la_{2j-1}\geq \al{2j+1-2i}-\al{2j-1-2i}\,,
\end{equation}
then do
\begin{equation}\label{eq:insert2type131}
(\la_{2j},\la_{2j-1}) \mapsto (\la_{2j}+ \al{2i+1}-\al{2i-1},\, \la_{2j-1}+\ak{2i}-\ak{2i-2})
\end{equation}
and store for the insertion in the pair $(\la_{2j-2},\la_{2j-3})$, $\bl{2i}$ if there are at least two remaining parts $\bl{2i+1}$ to insert, and $\bl{2i-1}$ if there is only one remaining part $\bl{2i+1}$. Else, do
\begin{equation}\label{eq:insert2type231}
(\la_{2j},\la_{2j-1}) \mapsto (\la_{2j}+ \al{2i+1},\, \la_{2j-1}+\ak{2i})\,\cdot
\end{equation}
\item To insert $\bl{2i}$ for $j\neq i\geq 1$, if \eqref{eq:condinsinf1311} occurs,
then do \eqref{eq:insert2type131}
and store for the insertion in the pair $(\la_{2j-2},\la_{2j-3})$ two parts $\bl{2i-1}$. Else, do \eqref{eq:insert2type231}
and store for the insertion in the pair $(\la_{2j-2},\la_{2j-3})$ one part $\bl{2i-1}$.
\end{enumerate}
\item To insert $\bl{i}$ for $i\in \{1,2j\}$, do 
\begin{equation}\label{eq:insert2type2314}
(\la_{2j},\la_{2j-1}) \mapsto (\la_{2j}+ \al{i},\, \la_{2j-1}+\ak{i-1})\,\cdot
\end{equation}
\end{enumerate}

Similarly to $\Phi^{(k,l)}$, the map $\Phi^{(k,l)}$ is weight-preserving, and the $\akl{i}$-component and $\alk{i-1}$-component of the parts $\bkl{i}$ only contribute respectively to the even and the odd weights of the image.  
\begin{exs} We take the same examples as in the case $\Phi^{(k,l)}$, with $2n=6$.
 \begin{enumerate}
 \item The case $(2,3)$ with $\nu=(b^{(2,3)}_{1})^5(b^{(2,3)}_{2})^4(b^{(2,3)}_{3})^2(b^{(2,3)}_{4})^3(b^{(2,3)}_{5})(b^{(2,3)}_{6})^3=(1+0)^5(3+1)^4(5+2)^2(12+5)^3(19+8)(45+19)^3$. We start with the pair $(\la_6,\la_5)$, initially equal to $(0,0)$. 
 \begin{enumerate}
  \item Insertions of $b_{6}^{(2,3)}$: we apply thrice \eqref{eq:insert2type2}
  and $(\la_6,\la_5)$ becomes $(135,57)$. 
  \item Insertions of $b_{5}^{(2,3)}$:  we apply once \eqref{eq:insert2type2} and $(\la_6,\la_5)$ becomes $(154,65)$. 
  \item Insertions of $b_{4}^{(2,3)}$: we successively apply \eqref{eq:insert2type2}, \eqref{eq:insert2type1} and \eqref{eq:insert2type1}, and $(\la_6,\la_5)$ respectively becomes $(166,70),(173,73)$ and $(180,76)$, and we store twice $b_{3}^{(2,3)}$ for the insertion into the pair $(\la_4,\la_3)$.
  \item Insertions of $b_{3}^{(2,3)}$: we apply twice \eqref{eq:insert2type2} and $(\la_6,\la_5)$ becomes $(190,80)$.
  \item Insertions of $b_{2}^{(2,3)}$: we successively apply \eqref{eq:insert2type1}, \eqref{eq:insert2type2},\eqref{eq:insert2type1} and \eqref{eq:insert2type1}, and $(\la_6,\la_5)$ respectively becomes $(192,81),(195,82),(197,83)$ and $(199,84)$, and we store thrice $b_{1}^{(2,3)}$ for the insertion into the pair $(\la_4,\la_3)$.
  \item Insertions of $b_{1}^{(2,3)}$: we obtain $(\la_6,\la_5)=(204,84)$ by applying five times \eqref{eq:insert2type2}.
 \end{enumerate}
In the process we stored $2\times b_{3}^{(2,3)},3\times b_{1}^{(2,3)}$ for the insertion into the pair $(\la_4,\la_3)$. For these insertions, we do twice \eqref{eq:insert2type2} for $i=3$, thrice \eqref{eq:insert2type2} for $i=1$. Thus, $(\la_4,\la_3)=(13,4)$, and $(\la_5,\la_6)=(0,0)$. Finally, set $\Phi^{(2,3)}_6(\nu)=(0,0,4,13,84,204)\in\mathcal{L}^{(2,3)}_{6}$.\\
 \item The case $(3,2)$ with $\nu=(b^{(3,2)}_{1})^5(b^{(3,2)}_{2})^4(b^{(3,2)}_{3})^2(b^{(3,2)}_{4})^3(b^{(3,2)}_{5})(b^{(3,2)}_{6})^3=(1+0)^5(2+1)^4(5+3)^2(8+5)^3(19+12)(30+19)^3$. For the insertion into the pair $(\la_6,\la_5)$, we have the following.
 \begin{enumerate} 
 \item Insertions of $b_{6}^{(3,2)}$: we do thrice \eqref{eq:insert2type2} and obtain $(\la_6,\la_5)=(90,57)$.
 \item Insertions of $b_{5}^{(3,2)}$: we apply \eqref{eq:insert2type2} to obtain $(\la_6,\la_5)=(109,69)$.
 \item Insertions of $b_{4}^{(3,2)}$: we successively apply \eqref{eq:insert2type2}, \eqref{eq:insert2type1} and \eqref{eq:insert2type2} to obtain $(\la_6,\la_5)=(128,81)$, and store once $b_{3}^{(3,2)}$ for the pair $(\la_4,\la_3)$.
 \item Insertions of $b_{3}^{(3,2)}$: we successively apply \eqref{eq:insert2type2} and \eqref{eq:insert2type1} to obtain $(\la_6,\la_5)=(136,86)$, and store once $b_{2}^{(3,2)}$ for the pair $(\la_4,\la_3)$.
 \item Insertions of $b_{2}^{(3,2)}$: we successively apply \eqref{eq:insert2type2}, \eqref{eq:insert2type1} , \eqref{eq:insert2type2} and \eqref{eq:insert2type2} to obtain $(\la_6,\la_5)=(143,90)$, and store once $b_{1}^{(3,2)}$ for the pair $(\la_4,\la_3)$.
 \item Insertions of $b_{1}^{(3,2)}$: we apply five times \eqref{eq:insert1type2} to obtain $(\la_6,\la_5)=(148,90)$.
 \end{enumerate}
 Hence, we store once $b_{3}^{(3,2)},b_{2}^{(3,2)},b_{1}^{(3,2)}$ for the insertion into the pair $(\la_4,\la_3)$. We then do \eqref{eq:insert2type2} for $3,2,1$ to obtain $(\la_4,\la_3)=(8,4)$. As there is no part stored for the insertion in $(\la_2,\la_1)$, we have $(\la_2,\la_1)=(0,0)$. Set $\Phi^{(3,2)}_6(\nu)=(0,0,4,8,90,148)\in\mathcal{L}^{(3,2)}_6$.\\
 \item The case $(6,1)$ with $\nu=(b_{1}^{(6,1)})^2(b_{2}^{(6,1)})^5(b_{3}^{(6,1)})^2(b_{4}^{(6,1)})^3(b_{6}^{(6,1)})^5=(1+0)^2(1+1)^5(5+6)^2(4+5)^3(15+19)^5$. For the insertion into the pair $(\la_6,\la_5)$, we have the following.
 \begin{enumerate}
  \item Insertions of $b_{6}^{(6,1)}$: 
 for $i=3$, we apply five times  \eqref{eq:insert2type221}, and and obtain $(\la_6,\la_5)=(75,95)$.
 \item Insertions of $b_{4}^{(6,1)}$: we apply thrice \eqref{eq:insert2type221} with $i=2$, and obtain $(\la_6,\la_5)=(87,110)$.
 \item Insertions of $b_{3}^{(6,1)}$: we apply successively \eqref{eq:insert2type121} and \eqref{eq:insert2type221} with $i=2$ to obtain $(\la_6,\la_5)=(94,119)$, and successively store two and one $b_{2}^{(6,1)}$ for the insertion into the pair $(\la_4,\la_3)$.
 \item Insertions of $b_{2}^{(6,1)}$: we apply four times \eqref{eq:insert2type221} for $i=1$ to obtain $(\la_6,\la_5)=(98,123)$, then apply \eqref{eq:insert2type123} to obtain $(\la_6,\la_5)=(98,124)$, and store $b_{1}^{(6,1)}$ for the insertion into the pair $(\la_4,\la_3)$.
 \item Insertions of $b_{1}^{(6,1)}$: we apply twice \eqref{eq:insert2type1234} and obtain $(\la_6,\la_5)=(100,124)$.
 \end{enumerate}
 We store $3\times b_{2}^{(6,1)},b_{1}^{(6,1)}$ for the insertion into the pair $(\la_4,\la_3)$, and only thrice \eqref{eq:insert2type221} for $i=2$ and \eqref{eq:insert2type1234} to obtain $(\la_4,\la_3)=(4,3)$. Then, $(\la_5,\la_6)=(0,0)$ and we set $\Phi^{(6,1)}_6(\nu)=(0,0,3,4,124,100)\in \mathcal{L}^{(6,1)}$.\\
 \item The case $(1,6)$ with $\nu=(b_{1}^{(1,6)})^2(b_{2}^{(1,6)})^5(b_{3}^{(1,6)})^2(b_{4}^{(1,6)})^3(b_{6}^{(1,6)})^5=(1+0)^2(6+1)^5(5+1)^2(24+5)^3(90+19)^5$.
  For the insertion into the pair $(\la_6,\la_5)$, we have the following.
 \begin{enumerate}
  \item Insertions of $b_{6}^{(1,6)}$: 
we apply five times \eqref{eq:insert2type2314} with $i=6$ to obtain $(\la_6,\la_5)=(450,95)$.
 \item Insertions of $b_{4}^{(1,6)}$: 
 for $i=2$, we apply thrice \eqref{eq:insert2type231} to obtain $(\la_6,\la_5)=(507,107)$, and store three $b_{3}^{(1,6)}$ for the insertion into the pair $(\la_4,\la_3)$.
 \item Insertions of $b_{3}^{(1,6)}$: 
 for $i=1$, we apply twice \eqref{eq:insert2type231} to obtain $(\la_6,\la_5)=(517,109)$.
 \item Insertions of $b_{2}^{(1,6)}$: 
 for $i=1$, we first apply once \eqref{eq:insert2type231} to obtain $(\la_5,\la_5)=(522,110)$, and store one $b_{1}^{(1,6)}$ for the insertion into the pair $(\la_4,\la_3)$. Then, we apply once \eqref{eq:insert2type131} to obtain $(\la_6,\la_5)=(526,111)$, and store two $b_{1}^{(1,6)}$. After that, we apply twice \eqref{eq:insert2type231} to obtain $(\la_6,\la_2)=(536,113)$, and store two $b_{1}^{(1,6)}$. Finally, we apply once \eqref{eq:insert2type131} to obtain $(\la_6,\la_2)=(540,114)$, and store two $b_{1}^{(1,6)}$. 
 \item Insertions of $b_{1}^{(1,6)}$: we apply twice \eqref{eq:insert2type2314} for $i=1$ to obtain $(\la_6,\la_5)=(542,114)$.
 \end{enumerate}
 We then store $3\times b_{3}^{(1,6)},7\times b_{1}^{(1,6)}$ for the insertion into 
 into the pair $(\la_4,\la_3)$. We thrice apply \eqref{eq:insert2type231} for $i=1$ and seven times \eqref{eq:insert2type2314} for $i=1$ to obtain $(\la_3,\la_4)=(22,3)$. We finally set $\Phi^{(1,6)}_6(\nu)=(0,0,3,22,114,542)\in \mathcal{L}^{(1,6)}_6$.
 \end{enumerate}
\end{exs}
\subsection{The map $\Phi_{2n-1}^{(k,l)}$}
We finally describe the bijection $\Phi_{2n-1}^{(k,l)}$. Let $\nu$ be a partition in $\Blk_{2n-1}$, and set $\la=(\la_1,\ldots,\la_{2n-1})$, all parts equal to $0$. We again proceed by inserting the parts $\blk{i}$ into the pairs $(\la_{2j+1},\la_{2j})$ with $0\leq j<n$, starting from the greatest $i,j$. We here consider a fictitious part $\la_0=0$. 
\subsubsection{The case $k,l\geq 2$}
\begin{enumerate}
\item To insert $\blk{i}$ with $i>1$ in $(\la_{2j+1},\la_{2j})$, proceed as follows. 
If 
\begin{equation}\label{eq:condinsinf3}
\akl{2j}\la_{2j+1}-\akl{2j+1}\la_{2j}\geq \akl{2j+2-i}-\akl{2j+1-i}\,\,,
\end{equation}
then do
\begin{equation}\label{eq:insert3type1}
(\la_{2j+1},\la_{2j})\mapsto (\la_{2j+1}+\alk{i}-\alk{i-1},\,\la_{2j}+ \akl{i-1}-\akl{i-2})
\end{equation}
and store $\blk{i-1}$ for the insertion into the pair $(\la_{2j-1},\la_{2j-2})$. Else, do
\begin{equation}\label{eq:insert3type2}
(\la_{2j+1},\la_{2j})\mapsto (\la_{2j+1}+\alk{i},\,\la_{2j}+ \akl{i-1})\,\cdot
\end{equation}
\item To insert $\bkl{1}$, do \eqref{eq:insert3type2} for $i=1$.
\end{enumerate}
\subsubsection{The case $(k,1)$}
\begin{enumerate}
\item To insert the parts $\bl{i}$ with $i>1$ into the pair $(\la_{2j+1},\la_{2j})$, proceed as follows. 
\begin{enumerate}
\item To insert $\bl{2i+1}$ for $i\geq 1$, if 
\begin{equation}\label{eq:condinsinf1212}
\ak{2j}\la_{2j+1}-\ak{2j+1}\la_{2j}\geq \ak{2j+2-2i}-\ak{2j-2i}\,,
\end{equation}
then do
\begin{equation}\label{eq:insert3type121}
(\la_{2j+1},\la_{2j}) \mapsto (\la_{2j+1}+ \al{2i+1}-\al{2i-1},\, \la_{2j}+\ak{2i}-\ak{2i-2})
\end{equation}
and store for the insertion in the pair $(\la_{2j-1},\la_{2j-2})$, $\bl{2i}$ if there are at least two remaining parts $\bl{2i+1}$ to insert, and $\bl{2i-1}$ if there is only one remaining part $\bl{2i+1}$. Else, do
\begin{equation}\label{eq:insert3type221}
(\la_{2j+1},\la_{2j}) \mapsto (\la_{2j+1}+ \al{2i+1},\, \la_{2j}+\ak{2i})\,\cdot
\end{equation}
\item To insert $\bl{2i}$ for $i\geq 1$, if 
\eqref{eq:condinsinf1212} occurs,
then do
\eqref{eq:insert3type121}
and store for the insertion in the pair $(\la_{2j-1},\la_{2j-2})$ two parts $\bl{2i-1}$. Else, do
\eqref{eq:insert3type221}
and store for the insertion in the pair $(\la_{2j-1},\la_{2j-2})$ one part $\bl{2i-1}$.
\end{enumerate}
\item To insert $\bl{1}$, do \eqref{eq:insert3type221} for $i=0$.
\end{enumerate}
\subsubsection{The case $(1,k)$}
\begin{enumerate}
\item 
To insert the parts $\bk{i}$ with $2j+1\neq i>2$ into the pair $(\la_{2j+1},\la_{2j})$, proceed as follows. 
\begin{enumerate}
\item To insert $\bk{2i}$ for $i>1$, if 
\begin{equation}\label{eq:condinsinf1312}
\al{2j}\la_{2j+1}-\al{2j+1}\la_{2j}\geq \al{2j+3-2i}-\al{2j+1-2i}\,,
\end{equation}
then do
\begin{equation}\label{eq:insert3type131}
(\la_{2j+1}, \la_{2j}) \mapsto (\la_{2j+1}+ \ak{2i}-\ak{2i-2},\, \la_{2j}+\al{2i-1}-\al{2i-3})
\end{equation}
and store for the insertion in the pair $(\la_{2j-1},\la_{2j-2})$, $\bk{2i-1}$ if there are at least two remaining parts $\bk{2i}$ to insert, and $\bk{2i-2}$ if there is only one remaining part $\bk{2i}$. Else, do
\begin{equation}\label{eq:insert3type231}
(\la_{2j+1}, \la_{2j}) \mapsto (\la_{2j+1}+ \ak{2i},\, \la_{2j}+\al{2i-1})\,\cdot
\end{equation}
\item To insert $\bk{2i-1}$ for $j+1\neq i>1$, if 
\eqref{eq:condinsinf1312}, 
then do \eqref{eq:insert3type131}
and store for the insertion in the pair $(\la_{2j-1},\la_{2j-2})$ two parts $\bk{2i-2}$. Else, do
\eqref{eq:insert3type231}
and store for the insertion in the pair $(\la_{2j-1},\la_{2j-2})$ one part $\bk{2i-2}$.
\end{enumerate}
\item To insert $\bk{2}$, if 
\begin{equation}\label{eq:condinsinf1331}
\ak{2j}\la_{2j+1}-\ak{2j+1}\la_{2j}\geq \ak{2j+1}
\end{equation}
then do
\begin{equation}\label{eq:insert3type133}
(\la_{2j+1}, \la_{2j}) \mapsto (\la_{2j+1}+ \ak{2}-\ak{1},\, \la_{2j}+\al{1}-\al{0})\,,
\end{equation}
and store for the insertion in the pair $(\la_{2j-1},\la_{2j-2})$ one part $\bk{1}$. Else, do \eqref{eq:insert3type231} for $i=1$.
\item To insert $\bk{i}$ for $i\in\{2j+1,1\}$, do  \begin{equation}\label{eq:insert3type1334}
(\la_{2j+1}, \la_{2j}) \mapsto (\la_{2j+1}+ \ak{i},\, \la_{2j}+\al{i-1})\,\cdot
\end{equation}
\end{enumerate}
The map is weight-preserving, and the $\alk{i}$-component and $\akl{i-1}$-component of the parts $\blk{i}$ only contribute respectively to the odd and the even weights of the image.  
\begin{exs} We take the same examples as in the case $\Phi^{(k,l)}$, with $2n-1=7$.
 \begin{enumerate}
 \item The case $(2,3)$ with $\nu=(b^{(3,2)}_{1})^5(b^{(3,2)}_{2})^4(b^{(3,2)}_{3})^2(b^{(3,2)}_{4})^3(b^{(3,2)}_{5})(b^{(3,2)}_{6})^3=(1+0)^5(2+1)^4(5+3)^2(8+5)^3(19+12)(30+19)^3$. We start with the pair $(\la_7,\la_6)$, initially equal to $(0,0)$. 
 \begin{enumerate}
  \item Insertions of $b_{6}^{(3,2)}$: we apply successively \eqref{eq:insert3type2}, \eqref{eq:insert3type2} and \eqref{eq:insert3type1}
  to obtain $(\la_7,\la_6)=(71,45)$, and store once $b_{5}^{(3,2)}$ for the insertion into the pair $(\la_5,\la_4)$.  
  \item Insertions of $b_{5}^{(3,2)}$:  we apply once \eqref{eq:insert3type2} to obtain $(\la_7,\la_6)=(90,57)$. 
  \item Insertions of $b_{4}^{(3,2)}$: we successively apply \eqref{eq:insert3type2}, \eqref{eq:insert3type1} and \eqref{eq:insert3type2}, and $(\la_7,\la_6)$ to obtain $(109,69)$, and we store once $b_{3}^{(3,2)}$ for the insertion into the pair $(\la_5,\la_4)$.
  \item Insertions of $b_{3}^{(3,2)}$: we apply successively \eqref{eq:insert3type2} and \eqref{eq:insert3type1} to obtain $(\la_7,\la_6)=(117,74)$, and store  once $b_{2}^{(3,2)}$ for the insertion into the pair $(\la_5,\la_4)$.
  \item Insertions of $b_{2}^{(3,2)}$: we successively apply \eqref{eq:insert3type2}, \eqref{eq:insert3type1},\eqref{eq:insert3type2} and \eqref{eq:insert3type2}, and $(\la_6,\la_5)=(124,78)$, and we store once $b_{1}^{(3,2)}$ for the insertion into the pair $(\la_5,\la_4)$.
  \item Insertions of $b_{1}^{(3,2)}$: we obtain $(\la_7,\la_6)=(129,78)$ by applying five times \eqref{eq:insert3type2}.
 \end{enumerate}
In the process we stored $1\times b_{5}^{(3,2)},1\times b_{3}^{(3,2)},1\times b_{2}^{(3,2)},1\times b_{1}^{(3,2)}$ for the insertion into the pair $(\la_4,\la_3)$. For these insertions, we only do \eqref{eq:insert3type2} and obtain $(\la_5,\la_4)=(27,16)$. Finally, set $\Phi^{(2,3)}_7(\nu)=(0,0,0,16,27,78,129)\in\mathcal{L}^{(2,3)}_{7}$.\\
 \item The case $(3,2)$ with $\nu=(b^{(2,3)}_{1})^5(b^{(2,3)}_{2})^4(b^{(2,3)}_{3})^2(b^{(2,3)}_{4})^3(b^{(2,3)}_{5})(b^{(2,3)}_{6})^3=(1+0)^5(3+1)^4(5+2)^2(12+5)^3(19+8)(45+19)^3$. For the insertion into the pair $(\la_6,\la_5)$, we have the following.
 \begin{enumerate} 
 \item Insertions of $b_{6}^{(2,3)}$: we apply successively \eqref{eq:insert3type2}, \eqref{eq:insert3type1} and \eqref{eq:insert3type2}
  to obtain $(\la_7,\la_6)=(116,49)$, and store once $b_{5}^{(2,3)}$ for the insertion into the pair $(\la_5,\la_4)$.
 \item Insertions of $b_{5}^{(2,3)}$: we apply \eqref{eq:insert3type2} to obtain $(\la_7,\la_6)=(135,57)$.
 \item Insertions of $b_{4}^{(2,3)}$: we successively apply \eqref{eq:insert3type1}, \eqref{eq:insert3type2} and \eqref{eq:insert3type1} to obtain $(\la_7,\la_6)=(161,68)$, and store twice $b_{3}^{(3,2)}$ for the pair $(\la_5,\la_4)$.
 \item Insertions of $b_{3}^{(2,3)}$: we successively apply \eqref{eq:insert3type2} and \eqref{eq:insert3type2} to obtain $(\la_7,\la_6)=(171,72)$. 
 \item Insertions of $b_{2}^{(2,3)}$: we successively apply \eqref{eq:insert3type1}, \eqref{eq:insert3type2} , \eqref{eq:insert3type1} and \eqref{eq:insert3type1} to obtain $(\la_6,\la_5)=(180,76)$, and store thrice $b_{1}^{(2,3)}$ for the pair $(\la_5,\la_4)$.
 \item Insertions of $b_{1}^{(2,3)}$: we apply five times \eqref{eq:insert3type2} to obtain $(\la_7,\la_6)=(185,76)$.
 \end{enumerate}
 Hence, we store $1\times b_{5}^{(2,3)},2\times b_{3}^{(2,3)},3\times b_{1}^{(2,3)}$ for the insertion into the pair $(\la_5,\la_4)$. We then do \eqref{eq:insert3type2} to obtain $(\la_5,\la_4)=(32,12)$.  Set $\Phi^{(3,2)}_7(\nu)=(0,0,0,12,32,76,185)\in\mathcal{L}^{(3,2)}_7$.\\
 \item The case $(6,1)$ with $\nu=(b_{1}^{(1,6)})^2(b_{2}^{(1,6)})^5(b_{3}^{(1,6)})^2(b_{4}^{(1,6)})^3(b_{6}^{(1,6)})^5=(1+0)^2(6+1)^5(5+1)^2(24+5)^3(90+19)^5$. For the insertion into the pair $(\la_7,\la_6)$, we have the following.
 \begin{enumerate}
  \item Insertions of $b_{6}^{(1,6)}$: 
 for $i=3$, we apply five times  \eqref{eq:insert3type221} to obtain $(\la_7,\la_6)=(355,75)$, and store five $b_{5}^{(1,6)}$ for the insertion into into the pair $(\la_5,\la_4)$.
 \item Insertions of $b_{4}^{(1,6)}$: we apply thrice \eqref{eq:insert3type221} with $i=2$ to obtain $(\la_7,\la_6)=(412,87)$, and store three $b_{3}^{(1,6)}$.
 \item Insertions of $b_{3}^{(1,6)}$: we apply twice \eqref{eq:insert3type221} with $i=1$ to obtain $(\la_7,\la_6)=(422,89)$.
 \item Insertions of $b_{2}^{(1,6)}$: we apply successively once \eqref{eq:insert3type121}, thrice  \eqref{eq:insert3type221} and once  \eqref{eq:insert3type121} for $i=1$ to obtain $(\la_7,\la_6)=(446,95)$, and store successively two, three an two  $b_{1}^{(6,1)}$ for the insertion into the pair $(\la_5,\la_4)$.
 \item Insertions of $b_{1}^{(1,6)}$: we apply twice \eqref{eq:insert3type221} with $i=0$ to obtain $(\la_7,\la_6)=(443,93)$.
 \end{enumerate}
 We store $5\times b_{5}^{(1,6)},3\times b_{3}^{(1,6)}, 7\times b_{1}^{(6,1)}$ for the insertion into the pair $(\la_5,\la_4)$. Thus, apply five times and thrice \eqref{eq:insert3type221} with respectively $i=2,1$, Finally, apply seven times \eqref{eq:insert3type221} with $i=0$. Hence, we obtain $(\la_5,\la_4)=(117,23)$.  Therefore, $\Phi^{(6,1)}_7(\nu)=(0,0,0,23,117,93,443)\in \mathcal{L}^{(6,1)}_7$.\\
 \item The case $(1,6)$ with $\nu=(b_{1}^{(6,1)})^2(b_{2}^{(6,1)})^5(b_{3}^{(6,1)})^2(b_{4}^{(6,1)})^3(b_{6}^{(6,1)})^5=(1+0)^2(1+1)^5(5+6)^2(4+5)^3(15+19)^5$.
  For the insertion into the pair $(\la_7,\la_6)$, we have the following.
 \begin{enumerate}
  \item Insertions of $b_{6}^{(6,1)}$: 
we apply four times \eqref{eq:insert3type231} and once \eqref{eq:insert3type131} with $i=3$ to obtain $(\la_7,\la_6)=(71,90)$, and store one $b_{4}^{(6,1)}$ for the insertion into the pair $(\la_5,\la_4)$.
 \item Insertions of $b_{4}^{(6,1)}$: 
 for $i=2$, we apply thrice \eqref{eq:insert3type231} to obtain $(\la_7,\la_6)=(83,105)$.
 \item Insertions of $b_{3}^{(6,1)}$: 
 for $i=1$, we apply successively \eqref{eq:insert3type131} and \eqref{eq:insert3type231} to obtain $(\la_7,\la_6)=(90,114)$, and store successively two and one $b_{2}^{(6,1)}$ for the insertion into the pair $(\la_5,\la_4)$. 
 \item Insertions of $b_{2}^{(6,1)}$: 
 for $i=1$, we successively apply thrice \eqref{eq:insert3type231}, \eqref{eq:insert3type131} and \eqref{eq:insert3type231} to obtain $(\la_5,\la_5)=(94,119)$, and store one $b_{1}^{(6,1)}$. 
 \item Insertions of $b_{1}^{(6,1)}$: we apply twice \eqref{eq:insert3type1334} for $i=1$ to obtain $(\la_6,\la_5)=(96,119)$.
 \end{enumerate}
 We then store $1\times b_{4}^{(6,1)},3\times b_{2}^{(6,1)},1\times b_{1}^{(6,1)}$ for the insertion into 
 into the pair $(\la_5,\la_4)$. We only apply \eqref{eq:insert3type231} for $i=2,1$ and then \eqref{eq:insert3type1334} for $i=1$ to obtain $(\la_5,\la_4)=(8,8)$. We finally set $\Phi^{(1,6)}_7(\nu)=(0,0,0,8,8,119,96)\in \mathcal{L}^{(1,6)}_7$.
 \end{enumerate}
\end{exs}
\section{Combinatorics of $(k,l)$-admissible words}
\label{sec:words}
In this section, we investigate a system of numeration related to the $(k,l)$-sequence, namely the $(k,l)$-admissible words. This is a generalization of the $l$-admissible words used by Savage and Yee in \cite{SY08}, which is a numeration system of non-negative integers presented by Frankeal in \cite{FRA85}. 
\subsection{Definition of $(k,l)$-admissible words}
\subsubsection{The case $k,l\geq 2$} Let $k,l\in \Z_{\geq 2}$. Then, $(\akl{n})_{n\geq 1}$ is an increasing sequence of positive integers. For conveniences, we define the sequence $(o^{(k,l)}_i)_{i\geq 1}$ such that $o^{(k,l)}_{2i-1}=l-2$ and $o^{(k,l)}_{2i}=k-2$ for $i\geq 1$.
\begin{deff}\label{def:kladm}
An \textit{infinite} $(k,l)$-admissible word is a sequence $(c_i)_{i\geq 1}$ of non-negative integers such that :
\begin{enumerate}
\item there are finitely many positive terms, 
\item $c_i \in \{0,\ldots,o^{(k,l)}_i+1\}$,
\item there is no pair $1\leq i<j$ such that
\begin{equation}\label{eq:forbpat}
c_h = o^{(k,l)}_h+\chi(h\in \{i,j\})\qquad\text{for}\quad h\in\{i,i+1,\ldots,j\}\,\cdot
\end{equation}
\end{enumerate}
Denote $\C^{(k,l)}$ by the set of all $(k,l)$-admissible words.
For $n\geq 1$, denote by $^n\C^{(k,l)}$ the set of all infinite $(k,l)$-admissible words such that $c_i=0$ for $i\geq n$, and $_n\C^{(k,l)}$ the set of all infinite $(k,l)$-admissible words such that $c_i=0$ for $i<n$. Note that $(\,^n\C^{(k,l)})_{n\geq 1}$ is an increasing sequence of sets, while $(\,_n\C^{(k,l)})_{n\geq 1}$ is decreasing. Set $^n_i\C^{(k,l)}=\,^n\C^{(k,l)}\cap \,_i\C^{(k,l)}$. Finally, for $s\in \C^{(k,l)} $, set the sequence $_ns$ to be the sequence $s$ where the $(n-1)$ first terms  are replaced by $0$.
\end{deff}
\begin{ex} In the case $(k,l)=(2,3)$, $(o^{(k,l)}_{2i-1},o^{(k,l)}_{2i})=(1,0)$ for $i\geq 1$. Then, 
 $(2,0,1,0,0,\ldots),(1,1,1,0,0,\ldots)\in \,^4\C^{(2,3)}$, but $(\underbrace{2,0,1,1},0,0,\ldots),(1,\underbrace{1,1,1},0,0,\ldots)\notin \C^{(2,3)}$, as the under-braced sub-sequences are forbidden. In the same way, $(\underbrace{1,1,1},0,0,\ldots)\notin\C^{(3,2)}$. We also have $_1(2,0,1,0,0,\ldots)=\,_2(2,0,1,0,0,\ldots)=(0,0,1,0,0,\ldots)$.
\end{ex}

\begin{deff}\label{def:lexicorder}
Let $\prec$ be the lexicographic strict order on the set of integer sequences, defined by the following: 
$$(c_i)\prec (d_i)\text{ if and only if there exists }n>0\text{ such that }c_n<d_n\text{ and }c_i=d_i\text{ for }\textcolor{red}{i>n}\,\cdot$$
We also say that $(c_i)\preceq (d_i)$ if and only if $(c_i)=(d_i)$ or $(c_i)\prec (d_i)$,  $(d_i)\succ (c_i)$ if and only if $(c_i)\prec (d_i)$, and $(d_i)\succeq (c_i)$ if and only if $(c_i)\preceq (d_i)$.\\
Let $\ll$ be the lexicographic strict order on the set of integer sequences, defined by the following: 
$$(c_i)\ll (d_i)\text{ if and only if there exists }n>0\text{ such that }c_n<d_n\text{ and }c_i=d_i\text{ for }\textcolor{red}{i<n}\,\cdot$$
We say that $(d_i)\gg (c_i)$ if and only if $(c_i)\ll (d_i)$.\\
As both $\prec$ and $\ll$ are both total orders, we can sort any set of integer sequences according either $\prec$ or $\ll$. \\
\end{deff}
\begin{ex} In  $^4\C^{(2,3)}$,
 $(2,0,1,0,0,\ldots)\prec (1,1,1,0,0,\ldots)$, but $(2,0,1,0,0,\ldots)\gg (1,1,1,0,0,\ldots)$.
\end{ex}
\begin{rem}
 Observe that, for $m\in \Zz$ and $n\in \Zz$, and $(d_i)_{i\geq 1} \in \C^{(k,l)}$, 
\begin{equation}\label{eq:leq}
_n(d_i)_{i\geq 1}=(0,\ldots,0,d_n,d_{n+1},\ldots)=\max\{(c_i)_{i\geq 1}\in\, _n\C^{(k,l)}:  (c_i)_{i\geq 1} \preceq (d_i)_{i\geq 1}\}\,\cdot
\end{equation}
\end{rem}
The next lemma unveils a particular property of the forbidden sequences in $(k,l)$-admissible words that does not fit the lexicographic orders $\prec$ and $\ll$.
\begin{lem}\label{lem:klsommax}
Let $(w_n^{(k,l)})_{n\in\Z}$ be a sequence satisfying the same recursive relation as $(\akl{n})_{n\in \Z}$, i.e.
\begin{equation}\label{eq:klseqdefbis}
w_{n-1}^{(k,l)}+w_{n+1}^{(k,l)}=(o^{(k,l)}_n+2) \cdot w_{n}^{(k,l)}
\end{equation}
Then, for $ i\leq j \in \Z$,
\begin{equation}\label{eq:klsommaxbis}
\sum_{h=i}^j o^{(k,l)}_h w_h^{(k,l)} = w_{i-1}^{(k,l)}+w_{j+1}^{(k,l)}-w_{i}^{(k,l)}-w_{j}^{(k,l)}\,\cdot
\end{equation}
Hence, for $1\leq i<j$ and $(c_i,\ldots,c_j)$ a finite sequence satisfying \eqref{eq:forbpat},
\begin{equation}\label{eq:klsommax}
\sum_{h=i}^j c_hw_h^{(k,l)} = w_{i-1}^{(k,l)}+w_{j+1}^{(k,l)}\,\cdot
\end{equation}
\end{lem}
In the following proposition, we present the interaction between the $(k,l)$-admissible words and the $(k,l)$-sequences through the lexicographic order, and show how that relation enables a numeration system of the non-negative integers in terms of the $(k,l)$-sequences.
\begin{prop}\label{prop:kladmseq}
Let $(w_{i}^{(k,l)})_{i\geq 0}$ be a sequence of non-negative real numbers satisfying \eqref{eq:klseqdefbis}, and
define the function 
\begin{align*}
\Gamma_{(k,l)}  \colon \C^{(k,l)}&\to \mathbb{R}_{\geq 0}\\\
(c_i)_{i\geq 1}&\mapsto \sum_{i\geq 1} c_i w_{i}^{(k,l)}\,\cdot
\end{align*}
\begin{enumerate}
\item
When $(w_{i}^{(k,l)})_{i\geq 0}$ is increasing, 
\begin{equation}\label{eq:preckladm}
(c_i)\prec (d_i) \Longleftrightarrow \Gamma_{(k,l)}((c_i))<\Gamma_{(k,l)}((d_i))\,,
\end{equation}
and $\Gamma_{(k,l)}$ is injective. Furthermore, if $(w_{i}^{(k,l)})_{i\geq 0}=(\akl{i})_{i\geq 0}$, $\Gamma_{(k,l)}$ then describes a bijection from $\C^{(k,l)}$ to $\Zz$ which implies for $n\geq 1$ a bijection from $^n\C^{(k,l)}$ to $\{0,\ldots,\akl{n}-1\}$.
Moreover, for $m\in \Zz$, the inverse function, denoted by $\left[\cdot\right]^{(k,l)}$, consists in constructing the sequence $(c_i)_{i\geq 1}$ as follows:
\begin{enumerate}
\item for the unique $n\geq 1$ such that $\akl{n-1}\leq m<\akl{n}$, set $c_i=0$ for $i\geq n$,
\item if $c_{i+1},\cdots,c_{n}$ are set for any $1\leq i<n$, then set
$$c_{i} = \left\lfloor \frac{m-\sum_{j=i+1}^{n-1}c_j\akl{j}}{\akl{i}}\right\rfloor\,\cdot$$
\end{enumerate}
\item When $(w_{i}^{(k,l)})_{i\geq 0}$ is non-increasing, we have
\begin{equation}\label{eq:ggkladm}
(c_i)\ll(d_i) \Longrightarrow \Gamma_{(k,l)}((c_i))\leq \Gamma_{(k,l)}((d_i))\,,
\end{equation}
and $\Gamma_{(k,l)}(_{n+1}\C^{(k,l)}) \subset (0;w_{n}^{(k,l)})$.
\end{enumerate}
\end{prop}
Another interpretation of the map $[\cdot]^{(k,l)}$ is that the terms of the sequences are obtained by the greedy algorithm, as we \textit{take} at each step $i$ the maximal number of $\akl{i}$ from the remaining amount.\\\\
The infinite $(k,l)$-admissible words as defined before will find their use in the comprehension of $(k,l)$-Euler partitions. To deal with the $(k,l)$-lecture hall partitions, there is a need of defining analogous objects, viewed as the \textit{finite} $(k,l)$-admissible words.    
\begin{deff}\label{def:kladmfinite}
For positive integer $n$, and all $m\in\Zz$, define $\left[m\right]_n^{(k,l)}=(c_1,\ldots,c_n)$ such that $c_n = \lfloor m/\akl{n}\rfloor$, and 
\begin{equation}\label{eq:klwordrest}
(c_1,\ldots,c_{n-1},0,\ldots) = \left[m-c_n\akl{n}\right]^{(k,l)}\,\cdot
\end{equation}
Thus, $m=\sum_{i=1}^n c_i\akl{i}$. Set $\C^{(k,l)}_n=\left[\Zz\right]^{(k,l)}_n$. One can check  that $\C^{(k,l)}_n$ is in bijection with $^n\C^{(k,l)}\times\Zz$. We  then extend $\prec$ to the set $\C^{(k,l)}_n$ with the following relation: 
$$(c_1,\ldots,c_n)\prec (d_1,\ldots,d_n)\text{ if and only if there exists }1\leq i\leq n\text{ such that }c_i<d_i\text{ and }c_j=d_j\text{ for } n\geq j>i\,\cdot$$
By writing  
$$(c_1,\cdots,c_n) \equiv ((c_1,\ldots,c_{n-1},0,\ldots), c_n)\,,$$
we have 
$(c_1,\cdots,c_n) \prec (d_1,\cdots,d_n)$ if and only $c_n<d_n$ or $c_n=d_n$ and $(c_1,\ldots,d_{n-1},0,\ldots)\prec (d_1,\ldots,d_{n-1},0,\ldots)$.
Hence, for $r,s\in \Zz$, the relation \eqref{eq:preckladm} then yields 
\begin{equation}\label{eq:ordkladmfin}
r < s \Longleftrightarrow \left[r\right]^{(k,l)}_n \prec \left[s\right]^{(k,l)}_n \,\cdot
\end{equation} 
We also set by convention 
$$\left[\Zz\right]^{(k,l)}_0=\C^{(k,l)}_0=\{\emptyset\}\cdot$$ 
For $1\leq i\leq n+1$, denote by $_i\C^{(k,l)}_n$ the subset of $\C^{(k,l)}_n$ consisting of sequences with the $(i-1)$ first terms equal to $0$, and note that $$\C^{(k,l)}_n=\,_1\C^{(k,l)}_n \supset \,_2\C^{(k,l)}_n\supset \cdots \supset \,_n\C^{(k,l)}_n\supset \,_{n+1}\C^{(k,l)}_n = \{(0)_{j=1}^{n}\} \,\cdot$$
Remark that for $1\leq i\leq n$, $_i\C^{(k,l)}_n$ is in bijection with $^n_i\C^{(k,l)}\times\Zz$.
For $s\in \C^{(k,l)}_n$, we also set  the sequence $_is$ to be the sequence $s$ where the $(i-1)$ first terms are replaced by $0$, and observe that, analogously to \eqref{eq:leq}, 
$$_is=\max\{(c_j)_{j=1}^n\in\, _i\C^{(k,l)}_n:  (c_j)_{j=1}^n\preceq s\}\,\cdot$$
\end{deff}
\begin{ex}
 For example, we enumerate in the following table the twenty first non-negative integers in terms of $(2,3)$ and $(3,2)$-admissible words:
 \begin{footnotesize}
 $$
 \begin{array}{|c|c|c|c|c||c|c|c|c|c|}
 \hline
 m& [m]^{(2,3)}& [m]^{(3,2)}& [m]^{(2,3)}_3& [m]^{(3,2)}_3&m& [m]^{(2,3)}& [m]^{(3,2)}& [m]^{(2,3)}_3& [m]^{(3,2)}_3\\
 \hline
0&(0,0,\ldots)& (0,0,\ldots)&(0,0,0)&(0,0,0)& 10&(0,0,2,0,\ldots)& (0,1,0,1,0,\ldots)&(0,0,2)&(0,0,2)\\
1&(1,0,\ldots)& (1,0,\ldots)&(1,0,0)&(1,0,0)& 11&(1,0,2,0,\ldots)& (1,1,0,1,0,\ldots)&(1,0,2)&(1,0,2)\\
2&(2,0,\ldots)& (0,1,0,\ldots)&(2,0,0)&(0,1,0)& 12&(0,0,0,1,0,\ldots)& (0,2,0,1,0,\ldots)&(2,0,2)&(0,1,2)\\
3&(0,1,0,\ldots)& (1,1,0\ldots)&(0,1,0)&(1,1,0)& 13&(1,0,0,1,0,\ldots)& (0,0,1,1,0,\ldots)&(0,1,2)&(1,1,2)\\
4&(1,1,0,\ldots)& (0,2,0,\ldots)&(1,1,0)&(0,2,0)& 14&(2,0,0,1,0,\ldots)& (1,0,1,1,0,\ldots)&(1,1,2)&(0,2,2)\\
5&(0,0,1,0,\ldots)& (0,0,1,0\ldots)&(0,0,1)&(0,0,1)& 15&(0,1,0,1,0,\ldots)& (0,1,1,1,0,\ldots)&(0,0,3)&(0,0,3)\\
6&(1,0,1,0,\ldots)& (1,0,1,0,\ldots)&(1,0,1)&(1,0,1)& 16&(1,1,0,1,0,\ldots)& (0,0,0,2,0,\ldots)&(1,0,3)&(1,0,3)\\
7&(2,0,1,0,\ldots)& (0,1,1,0,\ldots)&(2,0,1)&(0,1,1)& 17&(0,0,1,1,0,\ldots)& (1,0,0,2,0,\ldots)&(2,0,3)&(0,1,3)\\
8&(0,1,1,0,\ldots)& (0,0,0,1,0,\ldots)&(0,1,1)&(1,1,1)& 18&(1,0,1,1,0,\ldots)& (0,1,0,2,0\ldots)&(0,1,3)&(1,1,3)\\
9&(1,1,1,0,\ldots)& (1,0,0,1,0,\ldots)&(1,1,1)&(0,2,1)& 19&(0,0,0,0,1,0,\ldots)& (0,0,0,0,1,0,\ldots)&(1,1,3)&(0,2,3)\\
 \hline
 \end{array}\,\cdot
$$
\end{footnotesize}
\end{ex}
$$$$
\subsubsection{The case $(k,1)$}
Recall \eqref{eq:formulek1}. For $n\in \Z$, 
\begin{equation}\label{eq:formulek2}
\begin{cases}
\al{2n}=k\ak{2n} = a_{n}^{(k-2)}\\
\al{2n-1}=\ak{2n-1} = a_{n}^{(k-2)}+a_{n-1}^{(k-2)}
\end{cases}\cdot
\end{equation}
Our definition of $(k,1)$-admissible words is deeply linked to the $(k-2)$-admissible words.
\begin{deff}
Let $m$ be a non negative integer.  Set $\left[m\right]^{(k,1)}=\left[m\right]^{(k-2)}$. For $n \in \Zu$, set $\left[m\right]^{(k,1)}_{2n}=\left[m\right]^{(k-2)}_n$, and set $\left[m\right]^{(k,1)}_{2n-1}=(c_1,\ldots,c_{n})$ such that $c_{n}=\lfloor m/\ak{2n-1}\rfloor$, and
 \begin{enumerate}
  \item if $m-c_{n}\ak{2n-1}<a_{n-1}^{(k-2)}$, then $$(c_1,\ldots,c_{n-1})=\left[m-c_{2n-1}\ak{2n-1}\right]^{(k-2)}_{n-1}\,,$$
  \item otherwise, $0\leq m-c_{n}\ak{2n-1}-a_{n-1}^{(k-2)}<a_{n}^{(k-2)}$, and then $$(c_1,\ldots,c_{n-2},c_{n-1}-1,0)=\left[m-c_{2n-1}\ak{2n-1}-a_{n-1}^{(k-2)}\right]^{(k-2)}_{n}\,\cdot$$
 \end{enumerate}
 In particular, $\left[m\right]^{(k,1)}_{1}=\left[m\right]^{(k,1)}_{2}=(m)$, and we set by convention $\left[m\right]^{(k,1)}_{0}=\emptyset$.
 These definitions imply that, for $m\in \Zz$ and $n\in\Zu$,
 $$\left[m\right]^{(k,1)}_n=(c_i)_{i=1}^{\lceil n/2\rceil} \Longrightarrow m = c_{\lceil n/2\rceil}\ak{n}+\sum_{i=1}^{\lceil n/2\rceil-1} c_i a_i^{(k-2)}\,\cdot$$
 Denote by $\C^{(k,1)}$ the set $\left[\Zz\right]^{(k,1)}$ and denote by $\C^{(k,1)}_n$ the set $\left[\Zz\right]^{(k,1)}_n$ for all positive integer $n$. The sequences in these sets are called the $(k,1)$-admissible words. In particular, $\C^{(k,1)}=\C^{(k-2)}$ and $\C^{(k,1)}_{2n}=\C^{(k-2)}_n$. Moreover, set $\C^{(k,1)}_{2n-1} =\,^{1}\C^{(k,1)}_{2n-1} \sqcup \,^{0}\C^{(k,1)}_{2n-1}$ with 
 $$^{0}\C^{(k,1)}_{2n-1} = \{(c_i)_{i=1}^n \in \C^{(k,1)}_{2n-1}: c_{n-1}=0\}$$
 and 
 $$^{1}\C^{(k,1)}_{2n-1} = \{(c_i)_{i=1}^n \in \C^{(k,1)}_{2n-1}: c_{n-1}>0\}\,\cdot$$
 Then, $^{0}\C^{(k,1)}_{2n-1}$ is in bijection with $^{n-1}\C^{(k-2)}\times \Zz$ with 
 $$(c_1,\ldots,c_{n-2},0,c_n) \mapsto ((c_1,\ldots,c_{n-2},0,\ldots),c_n)\,,$$
 and 
 $^{1}\C^{(k,1)}_{2n-1}$ is in bijection with $^{n}\C^{(k-2)}\times \Zz$ with
 $$(c_1,\ldots,c_{n-2},c_{n-1}+1,c_n) \mapsto ((c_1,\ldots,c_{n-1},0,\ldots),c_n)\,\cdot$$
 Note that $\prec$ define a total strict order on $\C^{(k,1)}$ and on $\C^{(k,1)}_{n}$, and by \eqref{eq:preckladm}, we have the equivalence 
 $$r<s \Longleftrightarrow \left[r\right]^{(k,1)}_{2n-1}\prec \left[s\right]^{(k,1)}_{2n-1}\,\cdot$$
 Finally, define $_n\C^{(k,1)}=\,_n\C^{(k-2)}$ and for $1\leq i\leq \lceil n/2\rceil +1$, $_i\C^{(k,1)}_n$ the set of sequences of $\C^{(k,1)}_n$ with the $(i-1)^{th}$ first terms equal to $0$.
\end{deff}
\begin{ex}
 For example, we enumerate in the following table the twenty first non-negative integers in terms of $(6,1)$-admissible words:
 $$
 \begin{array}{|c|c|c|c||c|c|c|c|}
 \hline
 m& [m]^{(6,1)}& [m]^{(6,1)}_6& [m]^{(6,1)}_5&m& [m]^{(6,1)}& [m]^{(6,1)}_6& [m]^{(6,1)}_5\\
 \hline
0&(0,0,\ldots)&(0,0,0)&(0,0,0)& 10&(2,2,0,\ldots)& (2,2,0)&(2,2,0)\\
1&(1,0,\ldots)&(1,0,0)&(1,0,0)& 11&(3,2,0,\ldots)& (3,2,0)&(3,2,0)\\
2&(2,0,\ldots)&(2,0,0)&(2,0,0)& 12&(0,3,0,\ldots)& (0,3,0)&(0,3,0)\\
3&(3,0,\ldots)&(3,0,0)&(3,0,0)& 13&(1,3,0,\ldots)& (1,3,0)&(1,3,0)\\
4&(0,1,0,\ldots)&(0,1,0)&(0,1,0)& 14&(2,3,0,\ldots)& (2,3,0)&(2,3,0)\\
5&(1,1,0,\ldots)&(1,1,0)&(1,1,0)& \textcolor{red}{15}&\textcolor{red}{(0,0,1,0,\ldots)}&\textcolor{red}{ (0,0,1)}&\textcolor{red}{(3,3,0)}\\
6&(2,1,0,\ldots)&(2,1,0)&(2,1,0)& 16&(1,0,1,0,\ldots)& (1,0,1)&(0,4,0)\\
7&(3,1,0,\ldots)&(3,1,0)&(3,1,0)& 17&(2,0,1,0,\ldots)& (2,0,1)&(1,4,0)\\
8&(0,2,0,\ldots)&(0,2,0)&(0,2,0)& 18&(3,0,1,0,\ldots)& (3,0,1)&(2,4,0)\\
9&(1,2,0,\ldots)&(1,2,0)&(1,2,0)& 19&(0,1,1,0,\ldots)& (0,1,1)&(0,0,1)\\
 \hline
 \end{array}
$$
\end{ex}
$$$$
\subsubsection{The case $(1,k)$}
The definition of the $(1,k)$-admissible words dwells on the definition of $(k,1)$-admissible words and the properties of Lemma \ref{lem:divide}, with $x=\sk{0}$ in $(1)$ and $x=\frac{\ak{n+1}}{\al{n}}$ in $(2)$.
\begin{deff}\label{def:1k}
Set $\left[0\right]^{(1,k)}= (\left[0\right]^{(k,1)},\left[0\right]^{(k,1)})$,  for $m\in \Zu$  
$$  \left[m\right]^{(1,k)}=\left(\left[\left\lfloor\sk{0}\cdot m\right\rfloor +1\right]^{(k,1)},\left[\left\lceil\sk{2}\cdot m\right\rceil-1\right]^{(k,1)}\right)\,,$$
and for $m\in \Zz$ and $n\in\Zu$,
$$
  \left[m\right]^{(1,k)}_{n}=\left(\left[\left\lceil\frac{\ak{n+1}\cdot m}{\al{n}}\right\rceil \right]^{(k,1)}_{n+1},\left[\left\lfloor\frac{\ak{n-1}\cdot m}{\al{n}}\right\rfloor\right]^{(k,1)}_{n-1}\right)\,\cdot
$$
In particular, $ \left[m\right]^{(1,k)}_{1} = ((m),\emptyset)$.
Denote by $\C^{(1,k)}$ the set $\left[\Zz\right]^{(1,k)}$ and denote by $\C^{(1,k)}_n$ the set $\left[\Zz\right]^{(1,k)}_n$ for all positive integer $n$. The pair of sequences in these sets are called the $(1,k)$-admissible words.
\end{deff}
\begin{ex}
 We have
 $$
 \begin{array}{|c|c|c|c|}
 \hline
 m& [m]^{(1,6)}& [m]^{(1,6)}_5& [m]^{(1,1)}_4\\
 \hline
17&(2,3,0,0\ldots),(3,0,0,\ldots)&(2,3,0),(3,0)&(2,3,0),(3,0)\\
\textcolor{red}{18}&\textcolor{red}{(0,0,1,0\ldots),(3,0,0,\ldots)}&\textcolor{red}{(0,0,1),(3,0)}&\textcolor{red}{(3,3,0),(3,0)}\\
19&(0,0,1,0\ldots),(0,1,0,\ldots)&(0,0,1),(0,1)&(0,4,0),(3,0)\\
20&(1,0,1,0\ldots),(0,1,0,\ldots)&(1,0,1),(0,1)&(0,4,0),(4,0)\\
21&(2,0,1,0\ldots),(0,1,0,\ldots)&(2,0,1),(0,1)&(1,4,0),(4,0)\\
22&(3,0,1,0\ldots),(0,1,0,\ldots)&(3,0,1),(0,1)&(2,4,0),(4,0)\\
23&(0,1,1,0\ldots),(0,1,0,\ldots)&(0,1,1),(0,1)&(0,0,1),(4,0)\\
24&(0,1,1,0\ldots),(1,1,0,\ldots)&(0,1,1),(1,1)&(0,0,1),(0,1)\\
 \hline
 \end{array}\,\cdot
$$
\end{ex}
$$$$
\subsection{The transformation $0\cdot$}
In this part, we define a shifting transformation. This will allows us to investigate the ratio constraint in our partition in terms of shifting of the $(k,l)$-admissible words.
\begin{deff}\label{def:decalage}
 For $t\in \Zz\cup \{\infty\}$ and all integer sequence $c=(c_i)_{i=1}^t$, $0\cdot c$ denotes the sequence $d=(d_i)_{i=1}^{t+1}$ satisfying $ d_1=0$ and $d_{i+1}=c_i$ for $1\leq i\leq t$, and denote by $0\cdot$ the function induced by the corresponding transformation.
 \end{deff}
 Observe that for $c=\emptyset$,  $0\cdot c = (0)$. 
 \begin{lem}\label{lem:decalage}Let $k,l\geq 2$.
For positive integers $n$, $i$ and $n+1\geq j\geq 1$, $0\cdot$ induces a bijection from $_i^n\C^{(l,k)}$ to $_{i+1}^{n+1}\C^{(k,l)}$ and a bijection from $_j\C^{(l,k)}_n$ to $_{j+1}\C^{(k,l)}_{n+1}$ which preserve the order $\prec$.\\
Let $k\geq 4$.
For positive integers $n,i$ and $\lceil n/2\rceil +1\geq j\geq 1$, the transformation $0\cdot$ induces a bijection from $_i^n\C^{(k,1)}$ to $_{i+1}^{n+1}\C^{(k,1)}$ and a bijection from $_j\C^{(k,1)}_n$ to $_{j+1}\C^{(k,1)}_{n+2}$ which preserve the order $\prec$.
\end{lem}
 Extend this transformation to the set of sequence pairs with $p_i(0\cdot s) = 0\cdot p_i(s)$ for $i\in \{1,2\}$, and for all $d\in \Zz$, set $_{+d}\C^{(1,k)}= \underbrace{0\cdots 0}_{d \text{ times}}\cdot\C^{(1,k)}$, and $_{+d}\C^{(1,k)}_n= \underbrace{0\cdots 0}_{d \text{ times}}\cdot\C^{(1,k)}_n$.
We conventionally extend $\preceq$ on $ _{+d}\C^{(1,k)}$ such that for $u,v\in\, _{+d}\C^{(1,k)}$,
\begin{equation}\label{eq:ordcd}
u\preceq  v \Longleftrightarrow p_i(u)\preceq p_i(v)\text{ for  }i\in\{1,2\}\,\cdot 
\end{equation}
In the same way, we extend $\preceq$ on $\, _{+d}\C^{(1,k)}_n$. In particular, for $d=0$, $\preceq$ defines an order on $\C^{(1,k)}$ and $\C^{(1,k)}_n$.
\subsection{Followers and admissible words}
This part is dedicated to the study of the relations between two $(k,l)$-admissible words consecutive in terms of $\preceq$. 
\subsubsection{The case $k,l\geq 2$}
The next lemma relates consecutive $(k,l)$-admissible words for the infinite  version. 
\begin{lem}\label{lem:follow}
Let $n$ be a positive integer and $(c_i)_{i\geq 1}\in \, _n\C^{(k,l)}$. Let $(d_i)=\F(1,\,_n\C^{(k,l)},(c_i))$. The following then occurs.
\begin{enumerate}
\item If there exists $j\geq n$ such that $c_i= o^{(k,l)}_i+\delta_{i,j}$ for $ n\leq i\leq j$, then 
\begin{equation}\label{eq:follow1}
0=d_n=\cdots=d_j\quad, \quad d_{j+1}=c_{j+1}+1\quad\text{and}\quad d_{i}=c_{i} \text{ for }i\geq j+2\,\cdot
\end{equation}
\item Otherwise, 
\begin{equation}\label{eq:follow2}
d_{n}=c_{n}+1\quad\text{and}\quad d_{i}=c_{i} \text{ for }i\geq n+1\,\cdot
\end{equation}
\end{enumerate}
\end{lem}
We now state the analogous finite version of Lemma \ref{lem:follow}.
\begin{lem}\label{lem:followfin}
Let $n\geq j\geq 1$ be positive integers and $(c_i)_{i=1}^n\in \, _j\C^{(k,l)}_n$. Let $(d_i)=\F(1,\,_j\C^{(k,l)}_n,(c_i))$. The following then occurs.
\begin{enumerate}
\item If there exists $j\leq h\leq n-1$ such that $c_i= o^{(k,l)}_i+\delta_{i,h}$ for $ j\leq i\leq h$, then 
\begin{equation}\label{eq:follow1fin}
0=d_j=\cdots=d_h\quad, \quad d_{h+1}=c_{h+1}+1\quad\text{and}\quad d_{i}=c_{i} \text{ for } h+2\leq i \leq n\,\cdot
\end{equation}
\item Otherwise, 
\begin{equation}\label{eq:follow2fin}
d_{j}=c_{j}+1\quad\text{and}\quad d_{i}=c_{i} \text{ for } j+1\leq i\leq n \,\cdot
\end{equation}
\end{enumerate}
\end{lem}
\begin{exs}
 We have the following.
 \begin{align*}
  \F(1,\,_2\C^{(2,3)},(0,0,1,1,0,\ldots))&=\F(1,\,_3\C^{(2,3)},(0,0,1,1,0,\ldots))=(0,0,0,0,1,0,\ldots)\,,\\
  \F(1,\,_2\C^{(2,3)}_4,(0,0,2,1))&=(0,0,0,2)\,,\\
  \F(1,\,_1\C^{(2,3)}_4,(0,0,2,1))&=(1,0,2,1)\,\cdot
 \end{align*}
\end{exs}
These lemmas allows us to relate the integers corresponding to consecutive $(k,l)$-admissible words and to present the connection to the ratio constraint.
\begin{prop}\label{prop:klinserted} For positive integer $n$, we have the following.
\begin{enumerate}
\item For $n\geq 1$, if $(d_i)_{i\geq 1}=\F(1,\,_n\C^{(l,k)},(c_i)_{i\geq 1})$,   
one of following then holds:
\begin{align}
\label{eq:klratinfdif1} (d_i)_{i\geq n} \in\, _n\C^{(l,k)}\setminus\, _{n+1}\C^{(l,k)}&\Longleftrightarrow \sum_{i= n}^\infty (d_i-c_i)\akl{i+1} = \akl{n+1} \nonumber\\
&\Longleftrightarrow 0\leq \sum_{i\geq n}c_i\akl{i+1}-s_{0}^{(k,l)}\sum_{i\geq n}c_i\alk{i}\leq s_n^{(k,l)}-s_{n+1}^{(k,l)} \,,\\\nonumber\\
\label{eq:klratinfdif2}   (d_i)_{i\geq n} \in \,_{n+1}\C^{(l,k)}&\Longleftrightarrow \sum_{i=n}^\infty (d_i-c_i)\akl{i+1} = \akl{n+1}-\alk{n}\nonumber\\
&\Longleftrightarrow s_n^{(k,l)}-s_{n+1}^{(k,l)}<\sum_{i\geq n}c_i\akl{i+1}-s_{0}^{(k,l)}\sum_{i\geq n}c_i\alk{i}\leq s_n^{(k,l)}\,\cdot
\end{align}
 \item
For $2n-1\geq j\geq 1$, if $(d_i)_{i=1}^{2n-1}=\F(1,\,_j\C^{(l,k)}_{2n-1},(c_i)_{i\geq 1})$,  
one of following holds: 
\begin{small}
\begin{align}
\label{eq:klratevendif1} (d_i)_{i=1}^{2n-1} \in\, _j\C^{(l,k)}_{2n-1}\setminus\, _{j+1}\C^{(l,k)}_{2n-1}&\Longleftrightarrow \sum_{i=j}^{2n-1} (d_i-c_i)\akl{i+1} = \akl{j+1}\nonumber\\
&\Longleftrightarrow 0\leq \alk{2n-1}\sum_{i=j}^{2n-1}c_i\akl{i+1}-\akl{2n}\sum_{i=j}^{2n-1}c_i\alk{i} < \akl{2n-j}-\akl{2n-1-j}\,,\\\nonumber\\
\label{eq:klratevendif2}  (d_i)_{i=1}^{2n-1} \in\, _{j+1}\C^{(l,k)}_{2n-1} &\Longleftrightarrow \sum_{i=j}^{2n-1} (d_i-c_i)\akl{i+1} = \akl{j+1}-\akl{j}\nonumber\\
&\Longleftrightarrow \akl{2n-j}-\akl{2n-1-j}\leq \alk{2n-1}\sum_{i=j}^{2n-1}c_i\akl{i+1}-\akl{2n}\sum_{i=j}^{2n-1}c_i\alk{i} < \akl{2n-j}\,\cdot
\end{align}
\end{small}
\item For $2n\geq j\geq 1$, if $(d_i)_{i=1}^{2n}=\F(1,\,_j\C^{(k,l)}_{2n},(c_i)_{i=1}^{2n})$,  
one of following holds: 
\begin{small}
\begin{align}
\label{eq:klratodddif1} (d_i)_{i}^{2n} \in\, _j\C^{(k,l)}_{2n}\setminus\, _{j+1}\C^{(k,l)}_{2n}&\Longleftrightarrow \sum_{i=j}^{2n} (d_i-c_i)\alk{i+1} = \alk{j+1} \nonumber\\
&\Longleftrightarrow 0\leq \akl{2n}\sum_{i=j}^{2n}c_i\alk{i+1}-\alk{2n+1}\sum_{i=j}^{2n}c_i\akl{i} < \akl{2n+1-j}-\akl{2n-j}\,,\\\nonumber\\
\label{eq:klratodddif2} (d_i)_{i}^{2n} \in\, _{j+1}\C^{(k,l)}_{2n} &\Longleftrightarrow  \sum_{i=j}^{2n} (d_i-c_i)\alk{i+1} = \alk{j+1}-\alk{j}\nonumber\\
&\Longleftrightarrow \akl{2n+1-j}-\akl{2n-j}\leq \akl{2n}\sum_{i=j}^{2n}c_i\alk{i+1}-\alk{2n+1}\sum_{i=j}^{2n}c_i\akl{i} < \akl{2n+1-j}\,\cdot
\end{align}
\end{small}
\end{enumerate}
\end{prop}
\subsubsection{The cases $(k,1)$}
We proceed similarly to the case $(k,l)$ with $k,l\geq 2$. As $\C^{(k,1)}=\C^{(k-2)}$ and $\C^{(k,1)}_{2n}=\C^{(k-2)}_n$, we only discuss in the next lemma the notion on succession in $\C^{(k,1)}_{2n-1}$.
\begin{lem}\label{lem:followfinbis}
Let $n\geq j\geq 1$ be positive integers and $(c_i)_{i=1}^n\in \, _j\C^{(k,1)}_{2n-1}$. Let $(d_i)_{i=1}^{2n-1}=\F(1,\,_j\C^{(k,l)}_{2n-1},(c_i)_{i=1}^{2n-1})$. The following then occurs.
\begin{enumerate}
\item If there exists $j\leq h\leq n-1$ such that $c_i= o^{(k-2)}_i+\delta_{i,h}+\delta_{i,n-1}$ for $ j\leq i\leq h$, then 
\begin{equation}\label{eq:follow1finbis}
0=d_j=\cdots=d_h\quad, \quad d_{h+1}=c_{h+1}+1\quad\text{and}\quad d_{i}=c_{i} \text{ for } h+2\leq i \leq n\,\cdot
\end{equation}
\item Otherwise, 
\begin{equation}\label{eq:follow2finbis}
d_{j}=c_{j}+1\quad\text{and}\quad d_{i}=c_{i} \text{ for } j+1\leq i\leq n \,\cdot
\end{equation}
\end{enumerate}
\end{lem}
The proposition analogous to Proposition \ref{prop:klinserted} is the following.
\begin{prop}\label{prop:k1inserted} For positive integer $n$, and
for $\lfloor n/2\rfloor\geq j\geq 1$, $(c_i)_{i=1}^{\lceil n/2\rceil},(d_i)_{i=1}^{\lceil n/2\rceil} \in\,_j\C^{(k,1)}_{n}$
such that $(d_i)_{i=1}^{\lceil n/2\rceil}=\F(1,\,_j\C^{(k,1)}_{n},(c_i)_{i= 1}^{\lceil n/2\rceil})$,  
one of following holds: 
\begin{small}
\begin{align}
\label{eq:k1ratdif1} (d_i)_{i=1}^{\lceil n/2\rceil} \in\, _j\C^{(k,1)}_{n}\setminus\, _{j+1}\C^{(k,1)}_{n}&\Longleftrightarrow 0\leq \ak{n+2}\sum_{i=j}^{\lceil n/2\rceil-1}c_i\ak{2i+2}-\ak{n}\sum_{i=j}^{\lceil n/2\rceil-1}c_i\ak{2i} < \ak{n+2-2j}-\ak{n-2j}\nonumber\\
&\Longleftrightarrow (d_{\lceil n/2\rceil}-c_{\lceil n/2\rceil})\ak{n+2}+\sum_{i=j}^{\lceil n/2\rceil-1} (d_i-c_i)\ak{2i+2} = \ak{2j+2}\,,\\\nonumber\\
\label{eq:k1ratdif2} (d_i)_{i=1}^{\lceil n/2\rceil} \in\, _{j+1}\C^{(k,1)}_{n} &\Longleftrightarrow  \ak{n+2-2j}-\ak{n-2j}\leq \ak{n+2}\sum_{i=j}^{\lceil n/2\rceil-1}c_i\ak{2i+2}-\ak{n}\sum_{i=j}^{\lceil n/2\rceil-1}c_i\ak{2i} < \ak{n+2-2j}\nonumber\\
&\Longleftrightarrow (d_{\lceil n/2\rceil}-c_{\lceil n/2\rceil})\ak{n+2}+\sum_{i=j}^{\lceil n/2\rceil-1} (d_i-c_i)\ak{2i+2} = \ak{2j+2}-\ak{2j}\,\cdot
\end{align}
\end{small}
\end{prop}
\subsection{Relations between admissible words, lecture-hall partitions and Euler partitions}
\subsubsection{The case $k,l\geq 2$}
We now state the proposition that relates the $(k,l)$-admissible words and the shifting operation to the $(k,l)$-Euler partitions and $(k,l)$-lecture hall partitions.
\begin{prop}\label{prop:klratioinfini}$\quad$
\begin{enumerate}
 \item For $r,s\in \Zu$,
\begin{equation}
\label{eq:klratinf0}
r>s_0^{(k,l)} s \Longleftrightarrow \,_2\left[r\right]^{(k,l)}=0\cdot\left[\lceil s_2^{(l,k)}r\rceil-1\right]^{(l,k)}\succeq 0 \cdot\left[s\right]^{(l,k)}=\left[1+\lfloor s_0^{(k,l)} s \rfloor\right]^{(k,l)}\cdot
\end{equation}
Hence, for a sequence $\la=(\la_1,\ldots,\la_{2t})$ such that $t\geq 1$,  $0=\la_{2t}\leq \la_{2t-1}$ and $\la_i>0$ for $1\leq i\leq 2t-2$, 
\begin{equation}\label{eq:inlkl}
 \la \in \Lkl \Longleftrightarrow 
 \begin{cases}
  \left[\la_{2i-1}\right]^{(k,l)}\succeq 0\cdot\left[\la_{2i}\right]^{(l,k)}\quad \text{for} \quad 1\leq i\leq t-1\,,\\
  \left[\la_{2i}\right]^{(l,k)}\succeq 0\cdot\left[\la_{2i+1}\right]^{(k,l)}\quad \text{for} \quad 1\leq i\leq t-1\,\cdot
  \end{cases}
\end{equation}
\item Let $n\in \Zu$. For $r,s\in \Zz$,
\begin{equation}
\label{eq:klrateven0}
r\geq \frac{\akl{n+1}}{\alk{n}} s \Longleftrightarrow \, _2\left[r\right]^{(k,l)}_{n+1}=0\cdot\left[\left\lfloor\frac{\alk{n}}{\akl{n+1}} r\right\rfloor\right]^{(l,k)}_{n}\succeq 0\cdot\left[s\right]^{(l,k)}_{n}=\left[\left\lceil \frac{\akl{n+1}}{\alk{n}} s\right\rceil\right]^{(k,l)}_{n+1}\,\cdot
\end{equation}
Hence, for a sequence $\la=(\la_1,\ldots,\la_{2n-1})$ of non-negative integers,
\begin{equation}\label{eq:inlklodd}
 \la \in \Lkl_{2n-1} \Longleftrightarrow 
 \begin{cases}
  \left[\la_{2i}\right]^{(k,l)}_{2i}\succeq 0\cdot\left[\la_{2i-1}\right]^{(l,k)}_{2i-1}\quad \text{for} \quad 1\leq i\leq n-1\,\,\\
 \left[\la_{2i-1}\right]^{(l,k)}_{2i-1} \succeq 0\cdot\left[\la_{2i-2}\right]^{(k,l)}_{2i-2}\quad \text{for} \quad 1\leq i\leq n-1\,,
 \end{cases}
\end{equation}
and for a sequence $\la=(\la_1,\ldots,\la_{2n})$ of non-negative integers,
\begin{equation}\label{eq:inlkleven}
 \la \in \Lkl_{2n} \Longleftrightarrow 
 \begin{cases}
  \left[\la_{2i}\right]^{(k,l)}_{2i}\succeq 0\cdot\left[\la_{2i-1}\right]^{(l,k)}_{2i-1}\quad \text{for} \quad 1\leq i\leq n\,\,\\
 \left[\la_{2i-1}\right]^{(l,k)}_{2i-1} \succeq 0\cdot\left[\la_{2i-2}\right]^{(k,l)}_{2i-2}\quad \text{for} \quad 1\leq i\leq n-1\,\cdot
 \end{cases}
\end{equation}
\end{enumerate}
\end{prop}
\subsubsection{The cases $(k,1)$ and $(1,k)$}
We first present the relation between the shift operation and the ratio constraint for $(k,1)$-admissible words.  
\begin{prop}\label{prop:k1ratiofin}
For  $k\geq 4$, $n\in \Zu$ and $r,s\in \Zz$,
\begin{equation}
\label{eq:k1rat0}
r\geq \frac{\ak{n+2}}{\ak{n}} s \Longleftrightarrow \, _2\left[r\right]^{(k,1)}_{n+2}=0\cdot\left[\left\lfloor\frac{\ak{n}}{\ak{n+2}} r\right\rfloor\right]^{(k,1)}_{n}\succeq 0\cdot\left[s\right]^{(k,1)}_{n}=\left[\left\lceil \frac{\ak{n+2}}{\ak{n}} s\right\rceil\right]^{(k,1)}_{n+2}\,\cdot
\end{equation}
\end{prop}
In the following lemma, we rewrite Lemma \ref{lem:divide} in terms of $(1,k)$-admissible words, and present the corresponding properties related to the succession.
\begin{lem}\label{lem:follow1k}
Let $d,m\in \Zz$ and $n\in \Zu$. 
 \begin{enumerate}
  \item We have 
  \begin{equation}\label{eq:c1kinf}
  _{+d}\C^{(1,k)} = \left\{(s_1,s_2)\in (\,_{d+1}\C^{(k,1)})^2: 0\cdot s_2\preceq s_1\preceq 0\cdot \F(1,\,_{d+1}\C^{(k,1)},s_2)\right\}\,\cdot
  \end{equation}
  Therefore, the sequence $(t_1,t_2)=\F(1,\,_{+d}\C^{(1,k)},(s_1,s_2))$ satisfies the following:
  \begin{enumerate}
   \item if $s_1\neq 0\cdot \F(1,\,_{d+1}\C^{(k,1)},s_2)$, then $t_2=s_2$ and $t_1=\F(1,\,_{d+1}\C^{(k,1)},s_1)$,
   \item if $s_1=0\cdot \F(1,\,_{d+1}\C^{(k,1)},s_2)$, then $t_2=\F(1,\,_{d+1}\C^{(k,1)},s_2)$ and $t_1=s_1$.
  \end{enumerate}
  Furthermore, if $(t_1,t_2)=\F(m,\,_{+d}\C^{(1,k)},(s_1,s_2))$, then 
  \begin{equation}\label{eq:1ksominf}
  m=m_1+2m_2-\chi(t_1=0\cdot \F(1,\,_{d+1}\C^{(k,1)},t_2))+ \chi(s_1=0\cdot \F(1,\,_{d+1}\C^{(k,1)},s_2))\,,
  \end{equation}
  where 
  \begin{align*}
  m_1&= \sharp \left\{r\in\, _{d+1}\C^{(k,1)}\setminus\, _{d+2}\C^{(k,1)}: s_1\preceq r\prec t_1\right\}\\
  m_2&= \sharp \left\{r\in\, _{d+2}\C^{(k,1)}: s_1\preceq r\prec t_1\right\}\,\cdot
  \end{align*}
   \item When $(d,n)\neq (0,1)$, we have 
  \begin{equation}\label{eq:c1kfin}
  _{+d}\C^{(1,k)}_n = \left\{(s_1,s_2)\in \,_{d+1}\C^{(k,1)}_{n+1+2d}\times \,_{d+1}\C^{(k,1)}_{n-1+2d}: 0\cdot s_2\preceq s_1\preceq 0\cdot \F(1,\,_{d+1}\C^{(k,1)}_{n-1+2d},s_2)\right\}\,\cdot
  \end{equation}
  Therefore, the sequence $(t_1,t_2)=\F(1,\,_{+d}\C^{(1,k)}_{n},(s_1,s_2))$ satisfies the following:
  \begin{enumerate}
   \item if $s_1\neq 0\cdot\F(1,\,_{d+1}\C^{(k,1)}_{n-1+2d},s_2)$, then $t_2=s_2$ and $t_1=\F(1,\,_{d+1}\C^{(k,1)}_{n+1+2d},s_1)$,
   \item if $s_1=0\cdot\F(1,\,_{d+1}\C^{(k,1)}_{n-1+2d},s_2)$, then $t_2=\F(1,\,_{d+1}\C^{(k,1)}_{n-1+2d},s_2)$ and $t_1=s_1$.
  \end{enumerate}
  Furthermore, if $(t_1,t_2)=\F(m,\,_{+d}\C^{(1,k)}_n,(s_1,s_2))$, then 
  \begin{equation}\label{eq:1ksomfin}
  m=m_1+2m_2-\chi(t_1=0\cdot\F(1,\,_{d+1}\C^{(k,1)}_{n-1+2d},t_2))+ \chi(s_1=0\cdot\F(1,\,_{d+1}\C^{(k,1)}_{n-1+2d},s_2))\,,
  \end{equation}
  where
  \begin{align*}
  m_1&= \sharp \left\{r\in\,_{d+1}\C^{(k,1)}_{n+1+2d}\setminus\, _{d+2}\C^{(k,1)}_{n+1+2d}: s_1\prec r\preceq t_1\right\}\\
  m_2&= \sharp \left\{r\in\, _{d+2}\C^{(k,1)}_{n+1+2d}: s_1\prec r\preceq t_1\right\}\,\cdot
  \end{align*}
  \item When $(d,n)=(0,1)$,
  \begin{equation}
  _{+0}\C^{(1,k)}_1 = \C^{(1,k)}_1 = \{(s_1,s_2): s_1\in \Zz, s_2=\emptyset\}\,\cdot
  \end{equation}
  Therefore, the sequence $(t_1,t_2)=\F(1,\,_{+0}\C^{(1,k)}_{1},(s_1,s_2))$ is such that $t_1=s_1+1$ and $t_2=s_2=\emptyset$.
 \end{enumerate}
\end{lem}
The characterization of $(k,1)$ and $(1,k)$-Euler and lecture hall partitions in terms of $(k,1)$ and $(1,k)$-admissible words is the following.
\begin{prop}\label{prop:1kratiofin}
For $k\geq 4$, we have the following.
\begin{enumerate}
 \item For $r,s\in \Zu$, 
\begin{align}
\label{eq:1kratinf0}
r>\sk{0}\cdot s &\Longleftrightarrow \, \left[r\right]^{(k,1)}\succeq p_1(\left[s\right]^{(1,k)})\,,\\
\label{eq:1kratinf1} r> \sll{0} \cdot s &\Longleftrightarrow \, p_2(\left[r\right]^{(1,k)}) \succeq \left[s\right]^{(k,1)}\,\cdot
\end{align}
Hence, for a sequence $\la=(\la_1,\ldots,\la_{2t})$ such that $t\geq 1$,  $0=\la_{2t}\leq \la_{2t-1}$ and $\la_i>0$ for $1\leq i\leq 2t-2$, 
\begin{align}
\label{eq:inlk1} \la \in \Lk &\Longleftrightarrow 
\begin{cases}
 \left[\la_{2i-1}\right]^{(k,1)}\succeq p_1\left(\left[\la_{2i}\right]^{(1,k)}\right)\ \ \text{for} \ \ 1\leq i\leq t\,,\\
 p_2\left(\left[\la_{2i}\right]^{(1,k)}\right)\succeq \left[\la_{2i+1}\right]^{(k,1)} \ \ \text{for} \ \ 1\leq i\leq t-1\,,
\end{cases}
\\
\label{eq:inl1k} \la \in \Ll &\Longleftrightarrow
\begin{cases}
 \left[\la_{2i}\right]^{(k,1)}\succeq p_1\left(\left[\la_{2i+1}\right]^{(1,k)}\right)\ \ \text{for} \ \ 1\leq i\leq t-1\,,\\
 p_2\left(\left[\la_{2i-1}\right]^{(1,k)}\right)\succeq \left[\la_{2i}\right]^{(k,1)} \ \ \text{for} \ \ 1\leq i\leq t\,\cdot
\end{cases}
\end{align}
\item For $r,s\in \Zz$ and for $n\in\Zu$,
\begin{align}
\label{eq:1krat0}
r\geq \frac{\ak{n+1}}{\al{n}} s &\Longleftrightarrow \, \left[r\right]^{(k,1)}_{n+1}\succeq p_1(\left[s\right]^{(1,k)}_{n})\,,\\
\label{eq:1krat1} r\geq \frac{\al{n+1}}{\ak{n}} s &\Longleftrightarrow \, p_2(\left[r\right]^{(1,k)}_{n+1})\succeq \left[s\right]^{(k,1)}_{n}\,,
\end{align}
Hence, for a sequence $\la=(\la_1,\ldots,\la_{2n})$ of non-negative integers, 
\begin{align}
\label{eq:inlk12n} \la \in \Lk_{2n}&\Longleftrightarrow
\begin{cases}
 \left[\la_{2i}\right]^{(k,1)}_{2i}\succeq p_1\left(\left[\la_{2i-1}\right]^{(1,k)}_{2i-1}\right)\ \ \text{for} \ \ 1\leq i\leq n\,,\\
 p_2\left(\left[\la_{2i+1}\right]^{(1,k)}_{2i+1}\right)\succeq \left[\la_{2i}\right]^{(k,1)}_{2i} \ \ \text{for} \ \ 1\leq i\leq n-1\,,
\end{cases}\\
\label{eq:inl1k2n} \la \in \Ll_{2n} &\Longleftrightarrow \begin{cases}
 \left[\la_{2i+1}\right]^{(k,1)}_{2i+1}\succeq p_1\left(\left[\la_{2i}\right]^{(1,k)}_{2i}\right)\ \ \text{for} \ \ 1\leq i\leq n-1\,,\\
 p_2\left(\left[\la_{2i}\right]^{(1,k)}_{2i}\right)\succeq \left[\la_{2i-1}\right]^{(k,1)}_{2i-1}\ \ \text{for} \ \ 1\leq i\leq n\,,
\end{cases}
\end{align}
and for a sequence $\la=(\la_1,\ldots,\la_{2n-1})$ of non-negative integers,
\begin{align}
\label{eq:inlk12n1} \la \in \Lk_{2n-1}&\Longleftrightarrow
\begin{cases}
 \left[\la_{2i}\right]^{(k,1)}_{2i}\succeq p_1\left(\left[\la_{2i-1}\right]^{(1,k)}_{2i-1}\right)\ \ \text{for} \ \ 1\leq i\leq n-1\,,\\
 p_2\left(\left[\la_{2i+1}\right]^{(1,k)}_{2i+1}\right)\succeq \left[\la_{2i}\right]^{(k,1)}_{2i}\ \ \text{for} \ \ 1\leq i\leq n-1\,,
\end{cases}\\
\label{eq:inl1k2n1} \la \in \Ll_{2n-1} &\Longleftrightarrow \begin{cases}
 \left[\la_{2i+1}\right]^{(k,1)}_{2i+1}\succeq p_1\left(\left[\la_{2i}\right]^{(1,k)}_{2i}\right)\ \ \text{for} \ \ 1\leq i\leq n-1\,,\\
 p_2\left(\left[\la_{2i}\right]^{(1,k)}_{2i}\right)\succeq\left[\la_{2i-1}\right]^{(k,1)}_{2i-1}\ \ \text{for} \ \ 1\leq i\leq n-1\,\cdot
\end{cases}
\end{align}
\end{enumerate}
\end{prop}
\section{Well-definedness of $\Phi^{(k,l)}$}\label{part:interinf}
We first give an interpretation of the insertion of the parts $\bkl{i}$ into the pairs $(\la_{2j-1},\la_{2j})$ in terms of the admissible words. Denote by $(\la_{2j-1}^{(i)},\la_{2j}^{(i)})$ the pairs $(\la_{2j-1},\la_{2j})$ after the insertion of all the parts $\bkl{i}$, and let $m_i^{(j)}$ be the number of parts $\bkl{i}$ inserted into the pair $(\la_{2j-1},\la_{2j})$. Note that $m_i^{(1)}$ equals the number of occurrences of $\bkl{i}$ in $\nu$, and the image by $\Phi^{(k,l)}$ consists of $(\la_j^{(1)})_{j=1}^{2t}$, where $t$ is the smallest $j$ such that $\la_{2j}^{(1)}=0$.
\subsection{The case $k,l\geq 2$}\label{part:interinfkl}
We now state the key proposition that induces the well-definedness of $\Phi^{(k,l)}$.
\begin{prop}\label{prop:welldefinf}
For  $t\geq j\geq 1$,
\begin{enumerate}
\item for $i\geq 2$, we have 
\begin{equation}\label{eq:klbelongsji}
\left[\la_{2j-1}^{(i)}\right]^{(k,l)} = 0\cdot \left[\la_{2j}^{(i)}\right]^{(l,k)}\in \,_i\C^{(k,l)}\,\cdot
\end{equation}
\begin{enumerate}
\item Furthermore,
$
 \left[\la_{2j-1}^{(i)}\right]^{(k,l)}= \F\left(m_i^{(j)},\,_{i}\C^{(k,l)},\left[\la_{2j-1}^{(i+1)}\right]^{(k,l)}\right)\,\cdot
$
\item Reciprocally, 
$
 _{i+1}\left[\la_{2j-1}^{(i)}\right]^{(k,l)}= \F\left(m_{i-1}^{(j+1)},\,_{i+1}\C^{(k,l)},\left[\la_{2j-1}^{(i+1)}\right]^{(k,l)}\right)\,\cdot
$
\end{enumerate}
\item Finally, $\la_{2j}^{(1)}=\la_{2j}^{(2)}$,  
\begin{enumerate}
 \item and $\left[\la_{2j-1}^{(1)}\right]^{(k,l)}= \F\left(m_1^{(j)},\C^{(k,l)},\left[\la_{2j-1}^{(2)}\right]^{(k,l)}\right)\,\cdot$
 \item Reciprocally, 
 $\left[\la_{2j-1}^{(1)}\right]^{(k,l)}= \F\left(m_1^{(j)},\C^{(k,l)},0\cdot\left[\la_{2j}^{(1)}\right]^{(l,k)}\right)\,\cdot$
\end{enumerate}
\end{enumerate}
\end{prop}
\medskip
\begin{exs}
We illustrate the previous proposition with Examples \ref{exs:allpaper}.
For $(k,l)=(2,3)$ with \\
$\nu=(b^{(2,3)}_{1})^5(b^{(2,3)}_{2})^4(b^{(2,3)}_{3})^2(b^{(2,3)}_{4})^3(b^{(2,3)}_{5})(b^{(2,3)}_{6})^3=(1+0)^5(3+1)^4(5+2)^2(12+5)^3(19+8)(45+19)^3$, the corresponding table is the following,
$$
\begin{array}{|c||c|c||c|c||c|c|}
 \hline 
 i&m_i^{(1)}&[\la_1^{(i)}]^{(2,3)}&m_i^{(2)}&[\la_3^{(i)}]^{(2,3)}&m_i^{(3)}&[\la_5^{(i)}]^{(2,3)}\\
 \hline\hline
  7&0&(0,0,0,0,0,0,\textcolor{red}{0},0,0,\ldots)&0&(0,0,0,0,0,0,\textcolor{red}{0},\ldots)&0&(0,0,0,0,0,0,\textcolor{red}{0},\ldots)\\
  6&3&(0,0,0,0,0,\textcolor{red}{1},1,0,0,\ldots)&0&(0,0,0,0,0,\textcolor{red}{0},0,\ldots)&0&(0,0,0,0,0,\textcolor{red}{0},0,\ldots)\\
  5&1&(0,0,0,0,\textcolor{red}{1},1,1,0,0,\ldots)&1&(0,0,0,0,\textcolor{red}{1},0,0,\ldots)&0&(0,0,0,0,\textcolor{red}{0},0,0,\ldots)\\
  4&3&(0,0,0,\textcolor{red}{0},1,0,2,0,0,\ldots)&0&(0,0,0,\textcolor{red}{0},1,0,0,\ldots)&0&(0,0,0,\textcolor{red}{0},0,0,0,\ldots)\\
  3&2&(0,0,\textcolor{red}{0},0,0,0,0,1,0,\ldots)&2&(0,0,\textcolor{red}{2},0,1,0,0,\ldots)&0&(0,0,\textcolor{red}{0},0,0,0,0,\ldots)\\
  2&4&(0,\textcolor{red}{0},2,0,0,0,0,1,0,\ldots)&1&(0,\textcolor{red}{0},0,1,1,0,0,\ldots)&0&(0,\textcolor{red}{0},0,0,0,0,0,\ldots)\\
  1&5&(\textcolor{red}{0},1,0,1,0,0,0,1,0,\ldots)&2&(\textcolor{red}{2},0,0,1,1,0,0,\ldots)&1&(\textcolor{red}{1},0,0,0,0,0,0,\ldots)\\
  \hline
\end{array}\,\,,
$$
and for $(k,l)=(3,2)$ with $\nu=(b^{(3,2)}_{1})^5(b^{(3,2)}_{2})^4(b^{(3,2)}_{3})^2(b^{(3,2)}_{4})^3(b^{(3,2)}_{5})(b^{(3,2)}_{6})^3=(1+0)^5(2+1)^4(5+3)^2(8+5)^3(19+12)(30+19)^3$, 
$$
\begin{array}{|c||c|c||c|c||c|c|}
 \hline 
 i&m_i^{(1)}&[\la_1^{(i)}]^{(3,2)}&m_i^{(2)}&[\la_3^{(i)}]^{(3,2)}&m_i^{(3)}&[\la_5^{(i)}]^{(3,2)}\\
 \hline\hline
  7&0&(0,0,0,0,0,0,\textcolor{red}{0},0,0,\ldots)&0&(0,0,0,0,0,0,\textcolor{red}{0},\ldots)&0&(0,0,0,0,0,0,\textcolor{red}{0},\ldots)\\
  6&3&(0,0,0,0,0,\textcolor{red}{0},1,0,0,\ldots)&0&(0,0,0,0,0,\textcolor{red}{0},0,\ldots)&0&(0,0,0,0,0,\textcolor{red}{0},0,\ldots)\\
  5&1&(0,0,0,0,\textcolor{red}{1},0,1,0,0,\ldots)&1&(0,0,0,0,\textcolor{red}{1},0,0,\ldots)&0&(0,0,0,0,\textcolor{red}{0},0,0,\ldots)\\
  4&3&(0,0,0,\textcolor{red}{1},0,1,1,0,0,\ldots)&0&(0,0,0,\textcolor{red}{0},1,0,0,\ldots)&0&(0,0,0,\textcolor{red}{0},0,0,0,\ldots)\\
  3&2&(0,0,\textcolor{red}{1},0,0,0,0,1,0,\ldots)&1&(0,0,\textcolor{red}{1},0,1,0,0,\ldots)&0&(0,0,\textcolor{red}{0},0,0,0,0,\ldots)\\
  2&4&(0,\textcolor{red}{2},0,1,0,0,0,1,0,\ldots)&1&(0,\textcolor{red}{1},1,0,1,0,0,\ldots)&0&(0,\textcolor{red}{0},0,0,0,0,0,\ldots)\\
  1&5&(\textcolor{red}{1},0,0,2,0,0,0,1,0,\ldots)&1&(\textcolor{red}{0},0,0,1,1,0,0,\ldots)&0&(\textcolor{red}{0},0,0,0,0,0,0,\ldots)\\
  \hline
\end{array}\,\,\cdot
$$
\end{exs}
$$$$
Note that for $i\geq 1$, $m_i^{(1)}$ recursively derives the numbers $m_{i+1-j}^{(j)}$ and the pairs $(\la_{2j-1}^{(i+1-j)},\la_{2j}^{(i+1-j)})$ for $1\leq j\leq i$.
The next figure gives an interpretation of the relation between the parts at intermediate phases. 
\begin{figure}[H]
\label{fig:proofinf}
\begin{tikzpicture}[scale=1.2, every node/.style={scale=0.8}]

\draw (-1+2.5+0.8,1.5) node[left] {$_{i+1}\left[\la_{2j-1}^{(i)}\right]^{(k,l)}$};
\draw (-2+2.5+0.8,3) node[left] {$\left[\la_{2j-1}^{(i+1)}\right]^{(k,l)}$};

\draw (-1+0.8,1.5) node[left] {$\left[\la_{2j-1}^{(i)}\right]^{(k,l)}$};
\draw (6+0.8,3) node[left] {$\left[\la_{2j+1}^{(i)}\right]^{(k,l)}$};

\draw (6+0.8,1.5) node[left] {$\left[\la_{2j+1}^{(i-1)}\right]^{(k,l)}$};

\draw[dotted,<->] (-1.5+2.5,3)--(5.4,3); \draw[dotted,<->] (-0.5+2.5,1.5)--(5.4,1.5);
\draw(0,2.75)--(-0.25,2.5); \draw[->](-0.75,2)--(-1,1.75);
\draw[dashed] (0.7,2.75)--(0.9,2.5); \draw[dashed,->] (1.3,2)--(1.5,1.75);
\draw[dashed,->] (0.6,1.5)--(-0.6,1.5);
\draw(1.5+4.5,2.65)--(1.5+4.5,2.5); \draw[->] (1.5+4.5,2)--(1.5+4.5,1.75);
 
\draw (-0.6,2.25) node {$\F(m_i^{(j)},\,_{i}\C^{(k,l)})$};
\draw (-1.45,2.5)--(0.25,2.5)--(0.25,2)--(-1.45,2)--cycle;

\draw (1.5,2.25) node {$\F(m_{i-1}^{(j+1)},\,_{i+1}\C^{(k,l)})$};
\draw (0.4,2.5)--(2.65,2.5)--(2.65,2)--(0.4,2)--cycle;

\draw (1.5+4.5,2.25) node {$\F(m_{i-1}^{(j+1)},\,_{i-1}\C^{(k,l)})$};
\draw (0.4+4.5,2.5)--(2.65+4.5,2.5)--(2.65+4.5,2)--(0.4+4.5,2)--cycle;

\end{tikzpicture}
\caption{Relations between the insertions into the $j^{th}$ and $(j+1)^{th}$ pairs}
\end{figure}
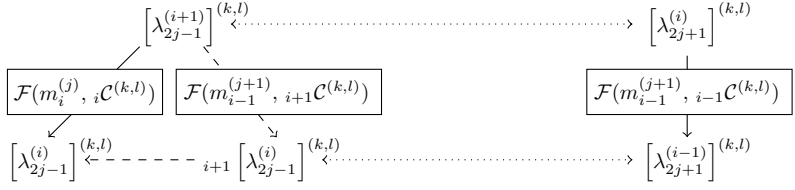
By Lemmas \ref{lem:prechoice} and \ref{lem:decalage}, in $_{i+1}\C^{(k,l)}$, 
$$
00\cdot\left[\la_{2j+1}^{(i)}\right]^{(k,l)} \preceq  \left[\la_{2j-1}^{(i+1)}\right]^{(k,l)} \Longleftrightarrow 00\cdot\left[\la_{2j+1}^{(i-1)}\right]^{(k,l)}\preceq\,  _{i+1}\left[\la_{2j-1}^{(i)}\right]^{(k,l)}\,\cdot
$$
By \eqref{eq:leq}, this yields
\begin{equation}\label{eq:proofinfkl}
00\cdot\left[\la_{2j+1}^{(i)}\right]^{(k,l)} \preceq  \left[\la_{2j-1}^{(i+1)}\right]^{(k,l)} \Longleftrightarrow 00\cdot\left[\la_{2j+1}^{(i-1)}\right]^{(k,l)}\preceq\,  \left[\la_{2j-1}^{(i)}\right]^{(k,l)}\,\cdot
\end{equation}
In the remainder of this part, the number fact refer to Proposition \ref{prop:welldefinf}.
\subsubsection{Proof that $\Phi^{(k,l)}(\Bkl)\subset \Lkl$}
By \eqref{eq:klbelongsji} and fact $(2)$, for $j\geq 1$,
$$\left[\la_{2j-1}^{(1)}\right]^{(k,l)}\succeq \left[\la_{2j-1}^{(2)}\right]^{(k,l)}=0\cdot\left[\la_{2j}^{(2)}\right]^{(l,k)}=0\cdot\left[\la_{2j}^{(1)}\right]^{(l,k)}\,\cdot$$
By \eqref{eq:inlkl}, to prove that $\la \in \Lkl$, it suffices to show that to show that
$$\left[\la_{2j-1}^{(2)}\right]^{(k,l)}=0\cdot\left[\la_{2j}^{(2)}\right]^{(l,k)}=0\cdot\left[\la_{2j}^{(1)}\right]^{(l,k)}\succeq 00\cdot \left[\la_{2j+1}^{(1)}\right]^{(k,l)}\,\cdot$$ 
For $n$ the greatest index $i$ such that $\bkl{i}$ occurs in $\nu$, we have for $1\leq j\leq n+1$ that
$\la_{2j-1}^{(n+2-j)}=\la_{2j}^{(n+2-j)}=0$, and
$$\left[\la_{2j-1}^{(n+2-j)}\right]^{(k,l)}=00\cdot\left[\la_{2j+1}^{(n+1-j)}\right]^{(k,l)}\,\cdot$$
Thus, by \eqref{eq:proofinfkl}, by induction  on $i$ for $n\geq i>j\geq 1$,
$$\left[\la_{2j-1}^{(i+1-j)}\right]^{(k,l)}\succeq 00\cdot\left[\la_{2j+1}^{(i-j)}\right]^{(k,l)}\,\cdot$$
In particular, for $i=j+1$, we obtain the result. Note that $\la_{2n}^{(1)}= 0$, and by the choice of $t$, we have $t$ is at most equal to $n$.
\subsubsection{Existence and the uniqueness of the reverse image}
Let $\la$ be a sequence in $\Lkl$ with $2t$ parts, and suppose that the first part $\la_{1}$ belongs to $\C^{(k,l)}_{2n+1}$. Since by  \eqref{eq:inlkl},
$$(0)_{i\geq 1}=\underbrace{0\cdots0}_{2t-1}\left[\la_{2t}\right]^{(l,k)}\preceq\underbrace{0\cdots0}_{2t-2}\left[\la_{2t-1}\right]^{(k,l)}\preceq \cdots\preceq 0\cdot\left[\la_{2}\right]^{(l,k)}\preceq\left[\la_{1}\right]^{(k,l)}\prec (\delta_{i,2n+1})_{i\geq 1}\,$$
we necessarily have by Lemma \ref{lem:decalage} that $\left[\la_{2j-1}\right]^{(k,l)}\prec (\delta_{i,2n+3-2j})_{i\geq 1}$.
Let us set $\la_{j}^{(1)}=\la_j$ for all $1\leq j\leq 2t$.
By considering the pair $(\la_{2t-1}^{(1)},\la_{2t}^{(1)})$, we have $m_{1}^{(t)}=\la_{2t-1}^{(1)}$ and $m_{i}^{(t)}=0$ for $1<i$. Suppose now that, for some $1\leq j\leq t-1$, we retrieve $\la_{2j+1}^{(i)}$ for $i\geq 1$ (and then $m_i^{(j+1)}$ by fact $(1.a)$ and $(2)$).
\begin{enumerate}
 \item  We retrieve $\left[\la_{2j-1}^{(2)}\right]^{(k,l)}$ and $m_1^{(j)}$ from fact $(2)$, and we have by \eqref{eq:klbelongsji} that 
 $$\left[\la_{2j-1}^{(2)}\right]^{(k,l)}=0\left[\la_{2j}^{(2)}\right]^{(l,k)}=0\left[\la_{2j}^{(1)}\right]^{(l,k)}\succeq 00 \cdot\left[\la_{2j+1}^{(2)}\right]^{(k,l)}\,\cdot$$
 \item Suppose that we retrieve for $\la_{2j-1}^{(i)}$ for $i\geq 2$, and that in $_{i+1}\C^{(k,l)}$, $$\left[\la_{2j-1}^{(i)}\right]^{(k,l)}\succeq 00 \cdot\left[\la_{2j+1}^{(i-1)}\right]^{(k,l)}\,\cdot$$
 Since $00 \cdot\left[\la_{2j+1}^{(i-1)}\right]^{(k,l)}$ admits a $(m_{i-1}^{(j+1)})^{th}$ predecessor $00 \cdot\left[\la_{2j+1}^{(i)}\right]^{(k,l)}$ in $_{i+1}\C^{(k,l)}$, by Lemma \ref{lem:prechoice},  $\left[\la_{2j-1}^{(i)}\right]^{(k,l)}$ also admits a $(m_{i-1}^{(j+1)})^{th}$ predecessor in $_{i+1}\C^{(k,l)}$, namely $\left[\la_{2j-1}^{(i+1)}\right]^{(k,l)}$ (by fact $(1.b)$), which satisfies
 $$\left[\la_{2j-1}^{(i+1)}\right]^{(k,l)}\succeq 00 \cdot\left[\la_{2j+1}^{(i)}\right]^{(k,l)}\,\cdot$$
 \item Finally, by \eqref{eq:klbelongsji}, $\left[\la_{2j-1}^{(2n+3-2j)}\right]^{(k,l)}\in\,_{2n+3-2j} \C^{(k,l)}$, and since $(\left[\la_{2j-1}^{(i)}\right]^{(k,l)})_{i\geq 1}$ is non-increasing in terms of $\preceq$, we necessarily have that
 $\left[\la_{2j-1}^{(2n+3-2j)}\right]^{(k,l)}\prec (\delta_{i,2n+3-2j})_{i\geq 1}$, so that $\la_{2j-1}^{(i)}=0$ for $i\geq 2n+3-2j$.
\end{enumerate}
Recursively on $i$, we obtain the datum $(\la_{2j-1}^{(i)})_{i\geq 1}$. Remark that, by \eqref{eq:klbelongsji} and fact $(2)$, the datum $(\la_{2j-1}^{(i)})_{i\geq 2}$ allows us to retrieve $(\la_{2j}^{(i)})_{i\geq 2}$, and fact $(1)$, $(\la_{2j-1}^{(i)})_{i\geq 1}$ allows us to recover $(m_i^{(j)})_{i\geq 1}$.
Finally, we recover the datum $(\la_{1}^{(i)})_{i\geq 1}$ and then $(m_i^{(1)})_{i\geq 1}$ and we conclude.
\subsection{The case $(k,1)$}\label{part:interinfk1}
Analogously to Proposition \ref{prop:welldefinf}, we have the following.
\begin{prop}\label{prop:welldefinf1}
For  $j\geq 1$,
\begin{enumerate}
\item for $i\geq 3$, we have 
\begin{equation}\label{eq:k1belongsji}
\left[\la_{2j-1}^{(i)}\right]^{(k,1)} =  p_1\left(\left[\la_{2j}^{(i)}\right]^{(1,k)}\right)= 0\cdot p_2\left(\left[\la_{2j}^{(i)}\right]^{(1,k)}\right) \in \,_{\lceil i/2\rceil}\C^{(k,1)}\,\cdot
\end{equation}
\begin{enumerate}
\item For $i\geq 2$, let $\left[\mu_{2j}^{(2i)}\right]^{(1,k)}=\F\left(m_{2i}^{(j)},\,_{+(i-1)}\C^{(1,k)},\left[\la_{2j}^{(2i+1)}\right]^{(1,k)}\right)$. 
\begin{enumerate}
 \item Then, $\left[\la_{2j-1}^{(2i)}\right]^{(k,1)}=p_1\left(\left[\mu_{2j}^{(2i)}\right]^{(1,k)}\right)$.
 \item Furthermore, $\left[\la_{2j-1}^{(2i-1)}\right]^{(k,1)}=\F\left(m_{2i-1}^{(j)},\,_{i}\C^{(k,1)},\left[\la_{2j-1}^{(2i)}\right]^{(k,1)}\right)$.
\end{enumerate}
\item Reciprocally, set 
$\left[\mu_{2j}^{(2i-1)}\right]^{(1,k)}=\left(p_1\left(\left[\la_{2j}^{(2i-1)}\right]^{(1,k)}\right), \,_{i}p_2\left(\left[\la_{2j}^{(2i-1)}\right]^{(1,k)}\right)\right)\,\cdot$
\begin{enumerate}
 \item Then, $\left[\mu_{2j}^{(2i-1)}\right]^{(1,k)}= \F\left(m_{2i-2}^{(j+1)},\,_{+(i-1)}\C^{(1,k)},\left[\mu_{2j}^{(2i)}\right]^{(1,k)}\right)$.
\item Moreover, $p_2\left(\left[\la_{2j}^{(2i+1)}\right]^{(1,k)}\right)=\F\left(m_{2i-1}^{(j+1)},\,_{i}\C^{(k,1)},p_2\left(\left[\mu_{2j}^{(2i)}\right]^{(1,k)}\right)\right)$.
\end{enumerate}
\end{enumerate}
\item Let $\left[\mu_{2j}^{(2)}\right]^{(1,k)}=\left[\la_{2j}^{(2)}\right]^{(1,k)}=\F\left(m_{2}^{(j)},\C^{(1,k)},\left[\la_{2j}^{(3)}\right]^{(1,k)}\right)$.
\begin{enumerate}
 \item Then, $\left[\la_{2j-1}^{(2)}\right]^{(k,1)}=p_1\left(\left[\mu_{2j}^{(2)}\right]^{(1,k)}\right)$.
 \item Reciprocally, $p_2\left(\left[\la_{2j}^{(2)}\right]^{(1,k)}\right)=\F\left(m_{1}^{(j+1)},\C^{(k,1)},p_2\left(\left[\mu_{2j}^{(3)}\right]^{(1,k)}\right)\right)$.
\end{enumerate}
\item Finally, $\la_{2j}^{(1)}=\la_{2j}^{(2)}$,
\begin{enumerate}
 \item and, $\left[\la_{2j-1}^{(1)}\right]^{(k,1)}=\F\left(m_{1}^{(j)},\C^{(k,1)}, \left[\la_{2j-1}^{(2)}\right]^{(k,1)}\right)$.
 \item Reciprocally, $\left[\la_{2j-1}^{(1)}\right]^{(1,k)}=\F\left(m_{1}^{(j)},\C^{(k,1)}, p_1\left(\left[\la_{2j}^{(2)}\right]^{(1,k)}\right)\right)$.
\end{enumerate}
\end{enumerate}
\end{prop}
\begin{ex}
 For $k=6$ with $\nu=(b_{1}^{(6,1)})^2(b_{2}^{(6,1)})^5(b_{3}^{(6,1)})^2(b_{4}^{(6,1)})^3(b_{6}^{(6,1)})^5=(1+0)^2(1+1)^5(5+6)^2(4+5)^3(15+19)^5$, we have
 \begin{footnotesize}
 $$
\begin{array}{|c||c|c||c|c||c|c|}
 \hline 
 i&m_i^{(1)}&[\mu_2]^{(1,6)}&m_i^{(2)}&[\mu_4]^{(1,6)}&m_i^{(3)}&[\mu_6]^{(1,6)}\\
 \hline\hline
  7&0&(0,0,0,0,\textcolor{red}{0},\ldots),(0,0,0,\textcolor{red}{0},\ldots)&0&(0,0,0,\textcolor{red}{0},\ldots),(0,0,0,\textcolor{red}{0},\ldots)&0&(0,0,0,\textcolor{red}{0},\ldots),(0,0,0,\textcolor{red}{0},\ldots)\\
  6&5&(0,0,\textcolor{red}{0},1,0,\ldots),(0,0,\textcolor{red}{1},0,\ldots)&0&(0,0,\textcolor{red}{0},0,\ldots),(0,0,\textcolor{red}{0},0,\ldots)&0&(0,0,\textcolor{red}{0},0,\ldots),(0,0,\textcolor{red}{0},0,\ldots)\\
  5&0&(0,0,\textcolor{red}{0},1,0,\ldots),(0,0,\textcolor{red}{1},0,\ldots)&1&(0,0,\textcolor{red}{1},0,\ldots),(0,0,\textcolor{red}{0},0,\ldots)&0&(0,0,\textcolor{red}{0},0,\ldots),(0,0,\textcolor{red}{0},0,\ldots)\\
  4&3&(0,\textcolor{red}{3},0,1,0,\ldots),(0,\textcolor{red}{0},1,0,\ldots)&0&(0,\textcolor{red}{0},1,0,\ldots),(0,\textcolor{red}{0},0,0,\ldots)&1&(0,\textcolor{red}{1},0,0,\ldots),(0,\textcolor{red}{0},0,0,\ldots)\\
  3&2&(0,\textcolor{red}{1},1,1,0,\ldots),(0,\textcolor{red}{1},1,0,\ldots)&0&(0,\textcolor{red}{0},1,0,\ldots),(0,\textcolor{red}{0},0,0,\ldots)&0&(0,\textcolor{red}{1},0,0,\ldots),(0,\textcolor{red}{0},0,0,\ldots)\\
  2&5&(\textcolor{red}{0},2,1,1,0,\ldots),(\textcolor{red}{2},1,1,0,\ldots)&3&(\textcolor{red}{3},0,1,0,\ldots),(\textcolor{red}{0},1,0,0,\ldots)&0&(\textcolor{red}{0},1,0,0,\ldots),(\textcolor{red}{1},0,0,0,\ldots)\\
  \hline\hline
  &m_1^{(1)}&[\la_1]^{(6,1)}_1&m_1^{(2)}&[\la_3]^{(6,1)}_1&m_1^{(3)}&[\la_5]^{(6,1)}_1\\
  \hline
  &2&(\textcolor{red}{2},2,1,1,0,\ldots)&1&(\textcolor{red}{0},1,1,0,\ldots)&0&(\textcolor{red}{0},1,0,0,\ldots)\\
  \hline
\end{array}\,\,\cdot
$$
\end{footnotesize}
\end{ex}

$$$$
The relations between the insertions in $(\la_{2j-1},\la_{2j})$ and $(\la_{2j+1},\la_{2j+2})$ are summarized in the following diagram.
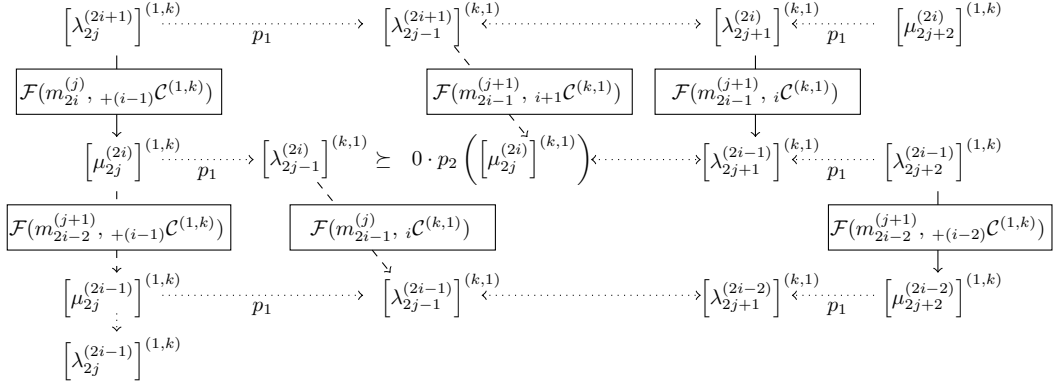
\begin{figure}[H]
\label{fig:proofinf1}
\begin{tikzpicture}[scale=1.2, every node/.style={scale=0.8}]

\draw (-1+2.5+0.8,1.5) node[left] {$\left[\la_{2j-1}^{(2i)}\right]^{(k,1)}\succeq \,\,\,\,0\cdot p_2\left(\left[\mu_{2j}^{(2i)}\right]^{(k,1)}\right)$};
\draw (-2+2.5+0.8,3) node[left] {$\left[\la_{2j-1}^{(2i+1)}\right]^{(k,1)}$};
\draw (-2+2.5+0.8,0) node[left] {$\left[\la_{2j-1}^{(2i-1)}\right]^{(k,1)}$};

\draw[dotted,->] (-2.5,3)--(-.3,3); \draw[dotted,->] (-2.5,1.5)--(-1.5,1.5); \draw[dotted,->] (-2.5,0)--(-.3,0);
\draw (-1.4,2.85) node {$p_1$};\draw (-2,1.35) node {$p_1$}; \draw (-1.4,-.15) node {$p_1$};

\draw[dotted,<->] (-1.5+2.5,3)--(5.4-2,3); \draw[dotted,<->] (2.2,1.5)--(5.4-2,1.5);\draw[dotted,<->] (-1.5+2.5,0)--(5.4-2,0);

\draw[dashed] (0.7,2.75)--(0.9,2.5); \draw[dashed,->] (1.3,2)--(1.5,1.75);
\draw (1.5,2.25) node {$\F(m_{2i-1}^{(j+1)},\,_{i+1}\C^{(k,1)})$};
\draw (0.4,2.5)--(2.65,2.5)--(2.65,2)--(0.4,2)--cycle;

\draw[dashed] (0.7-1.5,2.75-1.5)--(0.9-1.5,2.5-1.5); \draw[dashed,->] (1.3-1.5,2-1.5)--(1.5-1.5,1.75-1.5);
\draw (1.5-1.5,2.25-1.5) node {$\F(m_{2i-1}^{(j)},\,_{i}\C^{(k,1)})$};
\draw (0.4-1.5,2.5-1.5)--(2.65-1.5,2.5-1.5)--(2.65-1.5,2-1.5)--(0.4-1.5,2-1.5)--cycle;

\draw (6+0.8-2,3) node[left] {$\left[\la_{2j+1}^{(2i)}\right]^{(k,1)}$};
\draw (6+0.8-2,1.5) node[left] {$\left[\la_{2j+1}^{(2i-1)}\right]^{(k,1)}$};
\draw (6+0.8-2,0) node[left] {$\left[\la_{2j+1}^{(2i-2)}\right]^{(k,1)}$};
\draw (1.5+4.5-2,2.25) node {$\F(m_{2i-1}^{(j+1)},\,_{i}\C^{(k,1)})$};
\draw (0.4+4.5-2,2.5)--(2.65+4.5-2,2.5)--(2.65+4.5-2,2)--(0.4+4.5-2,2)--cycle;
\draw(1.5+4.5-2,2.65)--(1.5+4.5-2,2.5); \draw[->] (1.5+4.5-2,2)--(1.5+4.5-2,1.75);

\draw (6+0.8-9,3) node[left] {$\left[\la_{2j}^{(2i+1)}\right]^{(1,k)}$};
\draw (6+0.8-9,1.5) node[left] {$\left[\mu_{2j}^{(2i)}\right]^{(1,k)}$};
\draw (1.5+4.5-9,2.25) node {$\F(m_{2i}^{(j)},\,_{+(i-1)}\C^{(1,k)})$};
\draw (0.4+4.5-9,2.5)--(2.65+4.5-9,2.5)--(2.65+4.5-9,2)--(0.4+4.5-9,2)--cycle;
\draw(1.5+4.5-9,2.65)--(1.5+4.5-9,2.5); \draw[->] (1.5+4.5-9,2)--(1.5+4.5-9,1.75);
\draw (6+0.8-9,1.5-1.5) node[left] {$\left[\mu_{2j}^{(2i-1)}\right]^{(1,k)}$};
\draw (1.5+4.5-9,2.25-1.5) node {$\F(m_{2i-2}^{(j+1)},\,_{+(i-1)}\C^{(1,k)})$};
\draw (0.3+4.5-9,2.5-1.5)--(2.75+4.5-9,2.5-1.5)--(2.75+4.5-9,2-1.5)--(0.3+4.5-9,2-1.5)--cycle;
\draw[dashed](1.5+4.5-9,2.65-1.5)--(1.5+4.5-9,2.5-1.5); \draw[dashed,->] (1.5+4.5-9,2-1.5)--(1.5+4.5-9,1.75-1.5);
\draw (6+0.8-9,-0.7) node[left] {$\left[\la_{2j}^{(2i-1)}\right]^{(1,k)}$};
\draw[dotted,->] (1.5+4.5-9,1.3-1.5)--(1.5+4.5-9,1.1-1.5);

\draw (6+0.8,3) node[left] {$\left[\mu_{2j+2}^{(2i)}\right]^{(1,k)}$};
\draw (6+0.8,1.5) node[left] {$\left[\la_{2j+2}^{(2i-1)}\right]^{(1,k)}$};
\draw (6+0.8,1.5-1.5) node[left] {$\left[\mu_{2j+2}^{(2i-2)}\right]^{(1,k)}$};
\draw (1.5+4.5,2.25-1.5) node {$\F(m_{2i-2}^{(j+1)},\,_{+(i-2)}\C^{(1,k)})$};
\draw (0.3+4.5,2.5-1.5)--(2.75+4.5,2.5-1.5)--(2.75+4.5,2-1.5)--(0.3+4.5,2-1.5)--cycle;
\draw(1.5+4.5,2.65-1.5)--(1.5+4.5,2.5-1.5); \draw[->] (1.5+4.5,2-1.5)--(1.5+4.5,1.75-1.5);

\draw[dotted,->] (5.3,3)--(4.4,3);\draw[dotted,->] (5.3,1.5)--(4.4,1.5); \draw[dotted,->] (5.3,0)--(4.4,0);
\draw (4.9,2.85) node {$p_1$};\draw (4.9,1.35) node {$p_1$}; \draw (4.9,-.15) node {$p_1$};

\end{tikzpicture}
\caption{Relations between the insertions into the $j^{th}$ and $(j+1)^{th}$ pairs}
\end{figure}
For $i\geq 1$, by Lemma \ref{lem:prechoice}, facts $(1.a),(1.b)$ and $2$, we have in $_{i}\C^{(k,1)}$, 
\begin{equation}\label{eq:k1proofinf1}
 p_2\left(\left[\la_{2j}^{(2i+1)}\right]^{(1,k)}\right)\succeq p_1\left(\left[\mu_{2j+2}^{(2i)}\right]^{(1,k)}\right)\Longleftrightarrow p_2\left(\left[\mu_{2j}^{(2i)}\right]^{(1,k)}\right)\succeq \left[\la_{2j+1}^{(2i-1)}\right]^{(k,1)}\,\cdot
\end{equation}
One the other hand, for $i\geq 2$, by Lemma \ref{lem:prechoice}, facts $(1.a),(1.b)$, in $_{+(i-1)}\C^{(1,k)}$,
\begin{equation}\label{eq:k1proofinf3}
 \left[\mu_{2j}^{(2i)}\right]^{(1,k)}\succeq 0\cdot\left[\la_{2j+2}^{(2i-1)}\right]^{(1,k)} \Longleftrightarrow \left[\mu_{2j}^{(2i-1)}\right]^{(1,k)}\succeq 0\cdot\left[\mu_{2j+2}^{(2i-2)}\right]^{(1,k)}\,\cdot
\end{equation}
Moreover, by \eqref{eq:ordcd} and \eqref{eq:k1belongsji}, in $_{+(i-1)}\C^{(1,k)}$,  
\begin{align*}
\left[\mu_{2j}^{(2i-1)}\right]^{(1,k)}\succeq 0\cdot\left[\mu_{2j+2}^{(2i-2)}\right]^{(1,k)} &\Longleftrightarrow 
p_u(\left[\mu_{2j}^{(2i-1)}\right]^{(1,k)})\succeq 0\cdot p_u(\left[\mu_{2j+2}^{(2i-2)}\right]^{(1,k)})\quad \text{for}\quad u\in \{1,2\}\,,\\
&\Longleftrightarrow 
\begin{cases}
 p_1(\left[\la_{2j}^{(2i-1)}\right]^{(1,k)})\succeq 0\cdot p_1(\left[\mu_{2j+2}^{(2i-2)}\right]^{(1,k)})\\
 _{i}p_2(\left[\la_{2j}^{(2i-1)}\right]^{(1,k)})\succeq 0\cdot p_2(\left[\mu_{2j+2}^{(2i-2)}\right]^{(1,k)})
\end{cases}\\
&\Longleftrightarrow 
\begin{cases}
 0\cdot p_2(\left[\la_{2j}^{(2i-1)}\right]^{(1,k)})\succeq 0\cdot p_1(\left[\mu_{2j+2}^{(2i-2)}\right]^{(1,k)})\\
 p_2(\left[\la_{2j}^{(2i-1)}\right]^{(1,k)})\succeq 0\cdot p_2(\left[\mu_{2j+2}^{(2i-2)}\right]^{(1,k)})
\end{cases}
\end{align*}
As  
$$p_2(\left[\la_{2j}^{(2i-1)}\right]^{(1,k)})\succeq  p_1(\left[\mu_{2j+2}^{(2i-2)}\right]^{(1,k)})\Longrightarrow p_2(\left[\la_{2j}^{(2i-1)}\right]^{(1,k)})\succeq 0\cdot p_2(\left[\mu_{2j+2}^{(2i-2)}\right]^{(1,k)})\,,$$
we then have 
\begin{equation}\label{eq:k1proofinfs1}
\left[\mu_{2j}^{(2i-1)}\right]^{(1,k)}\succeq 0\cdot\left[\mu_{2j+2}^{(2i-2)}\right]^{(1,k)} \Longleftrightarrow p_2\left(\left[\la_{2j}^{(2i-1)}\right]^{(1,k)}\right)\succeq  p_1\left(\left[\mu_{2j+2}^{(2i-2)}\right]^{(1,k)}\right)\,\cdot
\end{equation}
Furthermore, in $_{+(i-1)}\C^{(1,k)}$,  
\begin{align*}
\left[\mu_{2j}^{(2i)}\right]^{(1,k)}\succeq 0\cdot\left[\la_{2j+2}^{(2i-1)}\right]^{(1,k)} &\Longleftrightarrow 
p_u(\left[\mu_{2j}^{(2i)}\right]^{(1,k)})\succeq 0\cdot p_u(\left[\la_{2j+2}^{(2i-1)}\right]^{(1,k)})\quad \text{for}\quad u\in \{1,2\}\,,\\
&\Longleftrightarrow 
\begin{cases}
 p_1(\left[\mu_{2j}^{(2i)}\right]^{(1,k)})\succeq 0\cdot \left[\la_{2j+1}^{(2i-1)}\right]^{(k,1)}\\
 p_2(\left[\mu_{2j}^{(2i)}\right]^{(1,k)})\succeq  \left[\la_{2j+1}^{(2i-1)}\right]^{(k,1)}
\end{cases}
\end{align*}
As  
$$p_2(\left[\mu_{2j}^{(2i)}\right]^{(1,k)})\succeq  \left[\la_{2j+1}^{(2i-1)}\right]^{(k,1)}\Longrightarrow p_1(\left[\mu_{2j}^{(2i)}\right]^{(1,k)})\succeq 0\cdot \left[\la_{2j+1}^{(2i-1)}\right]^{(k,1)}\,,$$
we then have 
\begin{equation}\label{eq:k1proofinfs2}
\left[\mu_{2j}^{(2i)}\right]^{(1,k)}\succeq 0\cdot\left[\la_{2j+2}^{(2i-1)}\right]^{(1,k)} \Longleftrightarrow p_2(\left[\mu_{2j}^{(2i)}\right]^{(1,k)})\succeq  \left[\la_{2j+1}^{(2i-1)}\right]^{(k,1)}\,\cdot
\end{equation}
Therefore, by \eqref{eq:k1proofinfs1} and \eqref{eq:k1proofinfs2}, \eqref{eq:k1proofinf3} is equivalent to
\begin{equation}\label{eq:k1proofinf2}
 p_2(\left[\mu_{2j}^{(2i)}\right]^{(1,k)})\succeq  \left[\la_{2j+1}^{(2i-1)}\right]^{(k,1)}\Longleftrightarrow p_2\left(\left[\la_{2j}^{(2i-1)}\right]^{(1,k)}\right)\succeq  p_1\left(\left[\mu_{2j+2}^{(2i-2)}\right]^{(1,k)}\right)\,\cdot
\end{equation}
\subsubsection{Proof that $\Phi^{(k,1)}(\Bk)\subset \Lk$}
By facts $(2)$ and $(3)$, for $j\geq 1$,
$$\left[\la_{2j-1}^{(1)}\right]^{(k,1)}\succeq \left[\la_{2j-1}^{(2)}\right]^{(k,1)}=p_1\left(\left[\la_{2j}^{(2)}\right]^{(1,k)}\right)=p_1\left(\left[\la_{2j}^{(2)}\right]^{(1,k)}\right)\,\cdot$$
By \eqref{eq:inlk1}, to prove that $\la \in \Lk$, it suffices to show that to show that
$$ p_2\left(\left[\la_{2j}^{(1)}\right]^{(1,k)}\right)=p_2\left(\left[\la_{2j}^{(2)}\right]^{(1,k)}\right)=p_2\left(\left[\mu_{2j}^{(2)}\right]^{(1,k)}\right)\succeq \left[\la_{2j+1}^{(1)}\right]^{(k,1)}\,\cdot$$ 
For $n\geq 1$  such that $\bk{i}$ does not occur in $\nu$ for $i\geq 2n+1$, we have for $1\leq j\leq 2n$ that
$\la_{4n+1}^{(1)}=\la_{4n+2}^{(1)}=\la_{2j-1}^{(2n+2-j)}=\la_{2j}^{(2n+2-j)}=\mu_{2j}^{(2n+2-j)}=0$.
In particular, for $0\leq j\leq n-1$,
\begin{align*}
 p_2\left(\left[\mu_{4j+2}^{(2n+1-2j)}\right]^{(1,k)}\right)&= p_1\left(\left[\mu_{4j+4}^{(2n-2j)}\right]^{(k,1)}\right)\,,\\
 p_2\left(\left[\mu_{4j+4}^{(2n-2j)}\right]^{(1,k)}\right)&=\left[\la_{4j+5}^{(2n-1-2j)}\right]^{(k,1)}
\end{align*}
By \eqref{eq:k1proofinf1}, 
$$p_2\left(\left[\mu_{4j+2}^{(2n+1-2j)}\right]^{(1,k)}\right)= p_1\left(\left[\mu_{4j+4}^{(2n-2j)}\right]^{(k,1)}\right) \Longrightarrow p_2\left(\left[\mu_{4j+2}^{(2n-2j)}\right]^{(1,k)}\right)\succeq \left[\la_{4j+3}^{(2n-1-2j)}\right]^{(k,1)}\,\cdot$$
By iterating \eqref{eq:k1proofinf2} plus \eqref{eq:k1proofinf1}, we obtain by induction  on $i$ for $n\geq i>j\geq 1$,
\begin{align*}
 p_2\left(\left[\mu_{4j+2}^{(2i-2j)}\right]^{(1,k)}\right)&\succeq\left[\la_{4j+3}^{(2i-1-2j)}\right]^{(k,1)}\,,
\\
 p_2\left(\left[\mu_{4j+4}^{(2i-2j)}\right]^{(1,k)}\right)&\succeq\left[\la_{4j+5}^{(2i-1-2j)}\right]^{(k,1)}\,\cdot
\end{align*}
In particular, the result follows from the case $i=j+1$. Note that $\la_{4n}^{(1)}=0$, and by the choice of $t$, we have $t$ is at most equal to $2n$.
\subsubsection{Existence and the uniqueness of the reverse image}
Let $\la$ be a sequence in $\Lk$ with $2t$ parts, and suppose that the first part $\la_{1}<\ak{2n+2}=a_{n+1}^{(k-2)}$. 
Since by  \eqref{eq:inlk1} and \eqref{eq:k1rat0},
$$(0)_{i\geq 1}=\underbrace{0\cdots0}_{t} \cdot p_2\left(\left[\la_{2t}\right]^{(1,k)}\right)\preceq\underbrace{0\cdots0}_{t-1}\left[\la_{2t-1}\right]^{(k,1)}\preceq \cdots\preceq 0\cdot p_2\left(\left[\la_{2}\right]^{(1,k)}\right)\preceq\left[\la_{1}\right]^{(k,1)}\prec (\delta_{i,n+1})_{i\geq 1}\,$$
we necessarily have by Lemma \ref{lem:decalage} that $\left[\la_{2j-1}\right]^{(k,1)}\prec (\delta_{i,n+2-j})_{i\geq 1}$.
Let us set $\la_{j}^{(1)}=\la_j$ for all $1\leq j\leq t$.
By considering the pair $(\la_{2t-1}^{(1)},\la_{2t}^{(1)})$, we have $m_{1}^{(t)}=\la_{2t-1}^{(1)}$ and $m_{i}^{(t)}=0$ for $1<i$. Suppose now that, for some $1\leq j\leq t-1$, we retrieve $m_1^{j+1}$ and $\mu_{2j+2}^{(i)}$ for $i\geq 2$ (and then the data $\la_{2j-1}^{(i)},\la_{2j}^{(i)}, m_i^{(j+1)}$ by fact $(1)$ and $(2)$).
\begin{enumerate}
 \item We retrieve $\left[\la_{2j-1}^{(2)}\right]^{(k,1)}$ and  $m_1^{(j)}$ from fact $(3)$, as $\left[\la_{2j-1}^{(2)}\right]^{(k,1)}=p_1\left(\left[\la_{2j}^{(1)}\right]^{(1,k)}\right)$. Moreover, from $(2)$, 
 $$p_2\left(\left[\la_{2j}^{(1)}\right]^{(1,k)}\right)=p_2\left(\left[\la_{2j}^{(2)}\right]^{(1,k)}\right)=p_2\left(\left[\mu_{2j}^{(2)}\right]^{(1,k)}\right)\succeq \left[\la_{2j+1}^{(1)}\right]^{(k,1)}\,\cdot$$
 \item Now assume that we retrieve the datum $\left[\mu_{2j}^{(2i)}\right]^{(1,k)}$ for some $i\geq 1$, and that
 $$p_2\left(\left[\mu_{2j}^{(2i)}\right]^{(1,k)}\right)\succeq \left[\la_{2j+1}^{(2i-1)}\right]^{(k,1)}\,\cdot$$
 \begin{enumerate}
  \item By Lemma \ref{lem:prechoice} and fact $(1.b)$, we can retrieve $p_2\left(\left[\la_{2j}^{(2i+1)}\right]^{(1,k)}\right)$ from $m_{2i-1}^{(j+1)}$. Thus, \eqref{eq:k1proofinf1} allows us to retrieve  $\left[\la_{2j}^{(2i+1)}\right]^{(1,k)}$, and fact $(1.a)$ then allows us to retrieve $m_{2i}^{(j)}$ from $\left[\mu_{2j}^{(2i)}\right]^{(1,k)}$. Finally, we obtain $\left[\mu_{2j}^{(2i+1)}\right]^{(1,k)}$ from fact $(1.b)$, and by \eqref{eq:k1proofinf1} and \eqref{eq:k1proofinfs1}, in $_{+i}\C^{(k,l)}$,
  $$\left[\mu_{2j}^{(2i+1)}\right]^{(1,k)}\succeq \left[\mu_{2j+2}^{(2i)}\right]^{(1,k)}\,\cdot$$
  \item By Lemma \ref{lem:prechoice} and fact $(1.b)$, we can retrieve $\left[\mu_{2j}^{(2i+2)}\right]^{(1,k)}$ from $m_{2i}^{(j+1)}$. Thus, \eqref{eq:k1proofinf3} and \eqref{eq:k1proofinf3} implies that 
  $$p_2\left(\left[\mu_{2j}^{(2i+2)}\right]^{(1,k)}\right)\succeq \left[\la_{2j+1}^{(2i+1)}\right]^{(k,1)}\,\cdot$$
  Moreover, from fact $(1.a)$, $p_1(\left[\mu_{2j}^{(2i+1)}\right]^{(1,k)})$ and $p_1(\left[\mu_{2j+2}^{(2i)}\right]^{(1,k)})$, we retrieve $m_{2i+1}^{(j)}$.
 \end{enumerate}
 \item As $\left(\left[\la_{2j-1}^{(i)}\right]^{(k,1)}\right)_{i\geq 1}$ is non-increasing, $\left[\la_{2j-1}^{(2i-1)}\right]^{(k,1)}\in \,_i\C^{(k,1)}$ and $\left[\la_{2j-1}^{(2i-1)}\right]^{(k,1)}\prec (\delta_{u,n+2-j})_{u\geq 1}$, we then have that $\la_{2j-1}^{(i)}=m_{i}^{(j)}=0$ for all $i\geq 2n+3-2j$.
\end{enumerate}
Recursively on $i$, we retrieve the datum $(\mu_{2j}^{(i)})_{i\geq 2}$. Remark that the datum $(\mu_{2j}^{(i)})_{i\geq 2}$ allows us to retrieve $\la_{2j-1}^{(i)}, \la_{2j}^{(i)}, m_i^j$ for $i\geq 2$.
Finally, by induction on $j$, we recover the datum $(m_i^{(1)})_{i\geq 1}$ and we conclude.
\subsection{The case $(1,k)$}\label{part:interinf1k}
Analogously to Proposition \ref{prop:welldefinf}, we have the following.
\begin{prop}\label{prop:welldefinf2}
For  $j\geq 1$,
\begin{enumerate}
\item for $i\geq 2$, we have 
\begin{equation}\label{eq:1kbelongsji}
0\cdot \left[\la_{2j}^{(i)}\right]^{(k,1)} = 0\cdot  p_2\left(\left[\la_{2j-1}^{(i)}\right]^{(1,k)}\right)= p_1\left(\left[\la_{2j-1}^{(i)}\right]^{(1,k)}\right) \in \,_{\lceil (i+1)/2\rceil}\C^{(k,1)}\,\cdot
\end{equation}
\begin{enumerate}
\item Let $\left[\mu_{2j-1}^{(2i-1)}\right]^{(1,k)}=\F\left(m_{2i-1}^{(j)},\,_{+(i-1)}\C^{(1,k)},\left[\la_{2j-1}^{(2i)}\right]^{(1,k)}\right)$. 
\begin{enumerate}
 \item Then, $p_1\left(\left[\la_{2j-1}^{(2i-1)}\right]^{(1,k)}\right)=p_1\left(\left[\mu_{2j-1}^{(2i-1)}\right]^{(1,k)}\right)$.
 \item Furthermore, $\left[\la_{2j}^{(2i-2)}\right]^{(k,1)}=\F\left(m_{2i-2}^{(j)},\,_{i-1}\C^{(k,1)},\left[\la_{2j}^{(2i-1)}\right]^{(k,1)}\right)$.
\end{enumerate}
\item Reciprocally, set 
$\left[\mu_{2j-1}^{(2i-2)}\right]^{(1,k)}=\left(p_1\left(\left[\la_{2j-1}^{(2i-1)}\right]^{(1,k)}\right), \,_{i}p_2\left(\left[\la_{2j-1}^{(2i-1)}\right]^{(1,k)}\right)\right)\,\cdot$
\begin{enumerate}
 \item 
Then, $\left[\mu_{2j-1}^{(2i-1)}\right]^{(1,k)}=\F\left(m_{2i-3}^{(j+1)},\,_{+(i-1)}\C^{(1,k)},\left[\mu_{2j-1}^{(2i-2)}\right]^{(1,k)}\right)$.
\item
Moreover,
$p_2\left(\left[\la_{2j-1}^{(2i)}\right]^{(1,k)}\right)=\F\left(m_{2i-2}^{(j+1)},\,_{i}\C^{(k,1)},p_2\left[\mu_{2j-1}^{(2i-1)}\right]^{(1,k)}\right)$.
\end{enumerate}
\end{enumerate}
\item Finally, $\la_{2j}^{(1)}=\la_{2j}^{(2)}$, 
\begin{enumerate}
 \item and $\left[\mu_{2j-1}^{(1)}\right]^{(1,k)}=\left[\la_{2j-1}^{(1)}\right]^{(1,k)}=\F\left(m_1^{(j)},\C^{(1,k)},\left[\la_{2j-1}^{(2)}\right]^{(1,k)}\right)$.
 \item Reciprocally, $\left[\la_{2j-1}^{(1)}\right]^{(1,k)}=\F\left(m_1^{(j)},\C^{(1,k)},\left(0\cdot\left[\la_{2j}^{(1)}\right]^{(k,1)},\left[\la_{2j}^{(1)}\right]^{(k,1)}\right)\right)$.
\end{enumerate}
\end{enumerate}
\end{prop}
\begin{ex}
 For $k=6$ with $\nu=(b_{1}^{(1,6)})^2(b_{2}^{(1,6)})^5(b_{3}^{(1,6)})^2(b_{4}^{(1,6)})^3(b_{6}^{(1,6)})^5=(1+0)^2(6+1)^5(5+1)^2(24+5)^3(90+19)^5$, we have
 $$
\begin{array}{|c||c|c||c|c|}
 \hline 
 i&m_i^{(1)}&[\mu_1^{(i)}]^{(1,6)}&m_i^{(2)}&[\mu_3^{(i)}]^{(1,6)}\\
 \hline\hline
  7&0&(0,0,0,\textcolor{red}{0},0,0,\ldots),(0,0,0,\textcolor{red}{0},0,\ldots)&0&(0,0,0,\textcolor{red}{0},0,\ldots),(0,0,0,\textcolor{red}{0},\ldots)\\
  6&5&(0,0,0,\textcolor{red}{1},1,0,\ldots),(0,0,0,\textcolor{red}{1},0,\ldots)&0&(0,0,0,\textcolor{red}{0},0,\ldots),(0,0,0,\textcolor{red}{0},\ldots)\\
  5&0&(0,0,\textcolor{red}{0},1,1,0,\ldots),(0,0,\textcolor{red}{1},1,0,\ldots)&6&(0,0,\textcolor{red}{1},1,0,\ldots),(0,0,\textcolor{red}{1},0,\ldots)\\
  4&3&(0,0,\textcolor{red}{3},1,1,0,\ldots),(0,0,\textcolor{red}{1},1,0,\ldots)&0&(0,0,\textcolor{red}{1},1,0,\ldots),(0,0,\textcolor{red}{1},0,\ldots)\\
  3&2&(0,\textcolor{red}{2},3,1,1,0,\ldots),(0,\textcolor{red}{3},1,1,0,\ldots)&3&(0,\textcolor{red}{3},1,1,0,\ldots),(0,\textcolor{red}{1},1,0,\ldots)\\
  2&5&(0,\textcolor{red}{0},1,2,1,0,\ldots),(0,\textcolor{red}{1},2,1,0,\ldots)&0&(0,\textcolor{red}{3},1,1,0,\ldots),(0,\textcolor{red}{1},1,0,\ldots)\\
  \hline\hline
  &m_1^{(1)}&[\la_1^{(1)}]^{(6,1)}_1&m_1^{(2)}&[\la_3^{(1)}]^{(6,1)}_1\\
  \hline
  &2&(\textcolor{red}{2},0,1,2,1,0,\ldots),(\textcolor{red}{0},1,2,1,0,\ldots)&7&(\textcolor{red}{0},,2,1,0,\ldots),(\textcolor{red}{0},2,1,0,\ldots)\\
  \hline
\end{array}\,,
$$
$$\begin{array}{|c||c|c||c|c|}
 \hline 
 i&m_i^{(3)}&[\mu_5^{(i)}]^{(1,6)}&m_i^{(4)}&[\mu_7^{(i)}]^{(1,6)}\\
 \hline\hline
  7&0&(0,0,0,\textcolor{red}{0},\ldots),(0,0,0,\textcolor{red}{0},\ldots)&0&(0,0,0,\textcolor{red}{0},\ldots),(0,0,0,\textcolor{red}{0},\ldots)\\
  6&0&(0,0,0,\textcolor{red}{0},\ldots),(0,0,0,\textcolor{red}{0},\ldots)&0&(0,0,0,\textcolor{red}{0},\ldots),(0,0,0,\textcolor{red}{0},\ldots)\\
  5&0&(0,0,\textcolor{red}{0},0,\ldots),(0,0,\textcolor{red}{0},0,\ldots)&0&(0,0,\textcolor{red}{0},0,\ldots),(0,0,\textcolor{red}{0},0,\ldots)\\
  4&1&(0,0,\textcolor{red}{1},0,\ldots),(0,0,\textcolor{red}{0},0,\ldots)&0&(0,0,\textcolor{red}{0},0,\ldots),(0,0,\textcolor{red}{0},0,\ldots)\\
  3&0&(0,\textcolor{red}{0},1,0,\ldots),(0,\textcolor{red}{1},0,0,\ldots)&1&(0,\textcolor{red}{1},0,0,\ldots),(0,\textcolor{red}{0},0,0,\ldots)\\
  2&0&(0,\textcolor{red}{0},1,0,\ldots),(0,\textcolor{red}{1},0,0,\ldots)&0&(0,\textcolor{red}{1},0,0,\ldots),(0,\textcolor{red}{0},0,0,\ldots)\\
  \hline\hline
  &m_1^{(3)}&[\la_5^{(1)}]^{(1,6)}_1&m_1^{(4)}&[\la_7^{(1)}]^{(1,6)}_1\\
  \hline
  &0&(\textcolor{red}{0},0,1,0,\ldots),(\textcolor{red}{0},1,0,0,\ldots)&0&(\textcolor{red}{0},1,0,0,\ldots),(\textcolor{red}{1},0,0,0,\ldots)\\
  \hline
\end{array}\,\,\cdot$$
\end{ex}
$$$$
Similarly to the case $(k,1)$, we picture in the following diagram the relations between the insertions into consecutive pairs.
\begin{figure}[H]
\label{fig:proofinf2}
\begin{tikzpicture}[scale=1.2, every node/.style={scale=0.8}]

\draw (-1+2.5+0.8,1.5) node[left] {$\left[\la_{2j}^{(2i-1)}\right]^{(k,1)}\succeq  \,\,p_2\left(\left[\mu_{2j-1}^{(2i-1)}\right]^{(k,1)}\right)$};
\draw (-2+2.5+0.8,3) node[left] {$\left[\la_{2j}^{(2i)}\right]^{(k,1)}$};
\draw (-2+2.5+0.8,0) node[left] {$\left[\la_{2j}^{(2i-2)}\right]^{(k,1)}$};

\draw[dotted,->] (-2.5,3)--(0,3); \draw[dotted,->] (-2.5,1.5)--(-1.5,1.5); \draw[dotted,->] (-2.5,0)--(-0.2,0);
\draw (-1.4,2.85) node {$p_1^-$};\draw (-2,1.35) node {$p_1^-$}; \draw (-1.4,-.15) node {$p_1^-$};

\draw[dotted,<->] (-1.6+2.5,3)--(5.4-2,3); \draw[dotted,<->] (2.2,1.5)--(5.4-2,1.5);\draw[dotted,<->] (-1.6+2.5,0)--(5.4-2,0);

\draw[dashed] (0.7,2.75)--(0.9,2.5); \draw[dashed,->] (1.3,2)--(1.5,1.75);
\draw (1.5,2.25) node {$\F(m_{2i-2}^{(j+1)},\,_{i}\C^{(k,1)})$};
\draw (0.4,2.5)--(2.65,2.5)--(2.65,2)--(0.4,2)--cycle;

\draw[dashed] (0.7-1.5,2.75-1.5)--(0.9-1.5,2.5-1.5); \draw[dashed,->] (1.3-1.5,2-1.5)--(1.5-1.5,1.75-1.5);
\draw (1.5-1.5,2.25-1.5) node {$\F(m_{2i-2}^{(j)},\,_{i-1}\C^{(k,1)})$};
\draw (0.4-1.5,2.5-1.5)--(2.65-1.5,2.5-1.5)--(2.65-1.5,2-1.5)--(0.4-1.5,2-1.5)--cycle;

\draw (6+0.8-2,3) node[left] {$\left[\la_{2j+2}^{(2i-1)}\right]^{(k,1)}$};
\draw (6+0.8-2,1.5) node[left] {$\left[\la_{2j+2}^{(2i-2)}\right]^{(k,1)}$};
\draw (6+0.8-2,0) node[left] {$\left[\la_{2j+2}^{(2i-3)}\right]^{(k,1)}$};
\draw (1.5+4.5-2,2.25) node {$\F(m_{2i-2}^{(j+1)},\,_{i-1}\C^{(k,1)})$};
\draw (0.4+4.5-2,2.5)--(2.65+4.5-2,2.5)--(2.65+4.5-2,2)--(0.4+4.5-2,2)--cycle;
\draw(1.5+4.5-2,2.65)--(1.5+4.5-2,2.5); \draw[->] (1.5+4.5-2,2)--(1.5+4.5-2,1.75);

\draw (6+0.8-9,3) node[left] {$\left[\la_{2j-1}^{(2i)}\right]^{(1,k)}$};
\draw (6+0.8-9,1.5) node[left] {$\left[\mu_{2j-1}^{(2i-1)}\right]^{(1,k)}$};
\draw (1.5+4.5-9,2.25) node {$\F(m_{2i-1}^{(j)},\,_{+(i-1)}\C^{(1,k)})$};
\draw (0.3+4.5-9,2.5)--(2.75+4.5-9,2.5)--(2.75+4.5-9,2)--(0.3+4.5-9,2)--cycle;
\draw(1.5+4.5-9,2.65)--(1.5+4.5-9,2.5); \draw[->] (1.5+4.5-9,2)--(1.5+4.5-9,1.75);
\draw (6+0.8-9,1.5-1.5) node[left] {$\left[\mu_{2j-1}^{(2i-2)}\right]^{(1,k)}$};
\draw (1.5+4.5-9,2.25-1.5) node {$\F(m_{2i-3}^{(j+1)},\,_{+(i-1)}\C^{(1,k)})$};
\draw (0.3+4.5-9,2.5-1.5)--(2.75+4.5-9,2.5-1.5)--(2.75+4.5-9,2-1.5)--(0.3+4.5-9,2-1.5)--cycle;
\draw[dashed](1.5+4.5-9,2.65-1.5)--(1.5+4.5-9,2.5-1.5); \draw[dashed,->] (1.5+4.5-9,2-1.5)--(1.5+4.5-9,1.75-1.5);
\draw (6+0.8-9,-0.7) node[left] {$\left[\la_{2j-1}^{(2i-2)}\right]^{(1,k)}$};
\draw[dotted,->] (1.5+4.5-9,1.3-1.5)--(1.5+4.5-9,1.1-1.5);

\draw (6+0.8,3) node[left] {$\left[\mu_{2j+1}^{(2i-1)}\right]^{(1,k)}$};
\draw (6+0.8,1.5) node[left] {$\left[\la_{2j+1}^{(2i-2)}\right]^{(1,k)}$};
\draw (6+0.8,1.5-1.5) node[left] {$\left[\mu_{2j+1}^{(2i-3)}\right]^{(1,k)}$};
\draw (1.5+4.5,2.25-1.5) node {$\F(m_{2i-3}^{(j+1)},\,_{+(i-2)}\C^{(1,k)})$};
\draw (0.3+4.5,2.5-1.5)--(2.75+4.5,2.5-1.5)--(2.75+4.5,2-1.5)--(0.3+4.5,2-1.5)--cycle;
\draw(1.5+4.5,2.65-1.5)--(1.5+4.5,2.5-1.5); \draw[->] (1.5+4.5,2-1.5)--(1.5+4.5,1.75-1.5);

\draw[dotted,->] (5.3,3)--(4.4,3);\draw[dotted,->] (5.3,1.5)--(4.4,1.5); \draw[dotted,->] (5.3,0)--(4.4,0);
\draw (4.9,2.85) node {$p_1^-$};\draw (4.9,1.35) node {$p_1^-$}; \draw (4.9,-.15) node {$p_1^-$};
\end{tikzpicture}
\caption{Relations between the insertions into the $j^{th}$ and $(j+1)^{th}$ pairs}
\end{figure}
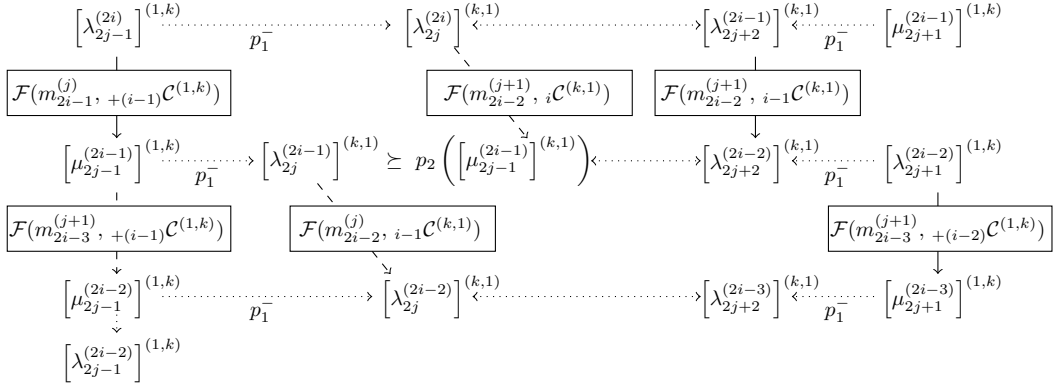
Here $p_1^-$ is such that $p_1=0\cdot p_1^-$.
Analogously to the case $(k,1)$, we have for $i\geq 2$ in $_i\C^{(k,1)}$
\begin{equation}\label{eq:1kproofinf1}
p_2\left(\left[\la_{2j-1}^{(2i)}\right]^{(1,k)}\right)\succeq p_1\left(\left[\mu_{2j+1}^{(2i-1)}\right]^{(1,k)}\right)\Longleftrightarrow  p_2\left(\left[\mu_{2j-1}^{(2i-1)}\right]^{(1,k)}\right)\succeq 0\cdot \left[\la_{2j+2}^{(2i-2)}\right]^{(k,1)}\,,
\end{equation}
and
\begin{align}\label{eq:1kproofinf3}
\left[\mu_{2j-1}^{(2i-1)}\right]^{(1,k)}\succeq 0\cdot \left[\la_{2j+1}^{(2i-2)}\right]^{(1,k)}&\Longleftrightarrow\left[\mu_{2j-1}^{(2i-2)}\right]^{(1,k)}\succeq 0\cdot \left[\mu_{2j+1}^{(2i-3)}\right]^{(1,k)}\,,
\\
 \Updownarrow\qquad\qquad\qquad\qquad&\qquad\qquad\qquad\qquad\qquad\Updownarrow\nonumber\\
\label{eq:1kproofinf2}
 p_2\left(\left[\mu_{2j-1}^{(2i-1)}\right]^{(1,k)}\right)\succeq 0\cdot \left[\la_{2j+2}^{(2i-2)}\right]^{(k,1)}&\Longleftrightarrow p_2\left(\left[\la_{2j-1}^{(2i-2)}\right]^{(1,k)}\right)\succeq p_1\left(\left[\mu_{2j+1}^{(2i-3)}\right]^{(1,k)}\right)\,\cdot
\end{align}
\subsubsection{Proof that $\Phi^{(1,k)}(\Bl)\subset \Ll$}
By \eqref{eq:ordcd}, fact $(2)$, \eqref{eq:1kbelongsji}, for $j\geq 1$,
$$p_2\left(\left[\la_{2j-1}^{(1)}\right]^{(1,k)}\right)\succeq p_2\left(\left[\la_{2j-1}^{(2)}\right]^{(1,k)}\right)=\left[\la_{2j}^{(2)}\right]^{(k,1)}=\left[\la_{2j}^{(1)}\right]^{(k,1)}\,\cdot$$
By \eqref{eq:inl1k}, to prove that $\la \in \Ll$, it suffices to show that to show that
$$ \left[\la_{2j}^{(1)}\right]^{(k,1)}=\left[\la_{2j}^{(2)}\right]^{(k,1)}= p_2\left(\left[\la_{2j-1}^{(2)}\right]^{(1,k)}\right)\succeq p_1\left(\left[\mu_{2j+1}^{(1)}\right]^{(1,k)}\right)=p_1\left(\left[\la_{2j+1}^{(1)}\right]^{(1,k)}\right)\,\cdot$$ 
For $n\geq 1$  such that $\bl{i}$ does not occur in $\nu$ for $i\geq 2n$, we have for $1\leq j\leq 2n$ that
$\la_{2j}^{(2n+1-j)}=\la_{2j-1}^{(2n+1-j)}=\mu_{2j-1}^{(2n+1-j)}=0$.
In particular, for $0\leq j\leq n-1$,
$$
 p_2\left(\left[\la_{4j+1}^{(2n-2j)}\right]^{(1,k)}\right) = p_1 \left(\left[\mu_{4j+3}^{(2n-1-2j)}\right]^{(1,k)}\right)
 $$
 and for $1\leq j\leq n-1$, by \eqref{eq:1kproofinf2}
$$
 p_2\left(\left[\mu_{4j-1}^{(2n+1-2j)}\right]^{(1,k)}\right)=0\cdot \left[\la_{4j+2}^{(2n-2j)}\right]^{(k,1)} \Longrightarrow p_2\left(\left[\la_{4j-1}^{(2n-2j)}\right]^{(1,k)}\right) \succeq p_1 \left(\left[\mu_{4j+1}^{(2n-1-2j)}\right]^{(1,k)}\right)\,\cdot
$$
Using \eqref{eq:1kproofinf1} and \eqref{eq:1kproofinf2}, we obtain by induction  on $i$ for $n\geq i>j\geq 1$,
\begin{align*}
  p_2\left(\left[\la_{4j-1}^{(2i-2j)}\right]^{(1,k)}\right) &\succeq p_1 \left(\left[\mu_{4j+1}^{(2i-1-2j)}\right]^{(1,k)}\right)\\
  p_2\left(\left[\la_{4j+1}^{(2i-2j)}\right]^{(1,k)}\right) &\succeq p_1 \left(\left[\mu_{4j+3}^{(2i-1-2j)}\right]^{(1,k)}\right)\,\cdot\\
\end{align*}
In particular, the result follows from the case $i=j+1$. Note that $\la_{4n-2}^{(2)}=\la_{4n-2}^{(1)}=0$, and by the choice of $t$, we have $t$ is at most equal to $2n-1$.
\subsubsection{Existence and the uniqueness of the reverse image}
Let $\la$ be a sequence in $\Ll$ with $2t$ parts, and suppose that the first part $\la_{2}<\ak{2n+2}=a_{n+1}^{(k-2)}$. 
Since by \eqref{eq:inl1k} and \eqref{eq:1krat0},
\begin{small}
$$(0)_{i\geq 1}=\underbrace{0\cdots0}_{t-1} \cdot \left[\la_{2t}\right]^{(k,1)}\preceq\underbrace{0\cdots0}_{t-1} \cdot p_2\left(\left[\la_{2t-1}\right]^{(1,k)}\right)\preceq \cdots\preceq 0\cdot \left[\la_{4}\right]^{(k,1)}\preceq 0\cdot p_2\left(\left[\la_{3}\right]^{(1,k)}\right)\preceq\left[\la_{2}\right]^{(k,1)}\prec (\delta_{i,n+1})_{i\geq 1}\,$$
\end{small}
we necessarily have by Lemma \ref{lem:decalage} that $\left[\la_{2j}\right]^{(k,1)}\prec (\delta_{i,n+2-j})_{i\geq 1}$.
Let us set $\la_{j}^{(1)}=\la_j$ for all $1\leq j\leq 2t$.
By considering the pair $(\la_{2t-1}^{(1)},\la_{2t}^{(1)})$, we have $m_{1}^{(t)}=\la_{2t-1}^{(1)}$ and $m_{i}^{(t)}=0$ for $1<i$. Suppose now that, for some $1\leq j\leq t-1$, we retrieve $m_1^{j+1}$ and $\la_{2j+1}^{(i)}$ for $i\geq 1$ (and then the data $\mu_{2j+1}^{(i)},\la_{2j+2}^{(i)}, m_i^{(j+1)}$).
\begin{enumerate}
 \item We retrieve $\left[\la_{2j-1}^{(2)}\right]^{(1,k)}$ and  $m_1^{(j)}$ from fact $(2)$ and we have 
 $$p_2\left(\left[\la_{2j-1}^{(2)}\right]^{(1,k)}\right) = \left[\la_{2j}^{(2)}\right]^{(1,k)}\succeq p_1\left(\left[\la_{2j+1}^{(1)}\right]^{(1,k)}\right)\cdot$$
 \item Now assume that we retrieve the datum $\la_{2j-1}^{(2i-2)}$ for some $i\geq 2$, and that 
 $$p_2\left(\left[\la_{2j-1}^{(2i-2)}\right]^{(1,k)}\right)\succeq p_1\left(\left[\la_{2j+1}^{(2i-3)}\right]^{(1,k)}\right)\cdot$$
 We consecutively retrieve $\left[\mu_{2j-1}^{(2i-2)}\right]^{(1,k)}$, then $\left[\mu_{2j-1}^{(2i-1)}\right]^{(1,k)}$ from $m_{2i-3}^{(j+1)}$ and fact $(1.b)$, then  $m_{2i-2}^{(j)}$ from fact $(1.a)$, then $p_2\left(\left[\la_{2j-1}^{(2i)}\right]^{(1,k)}\right)$ from $(1.b)$ and finally $\left[\la_{2j-1}^{(2i)}\right]^{(1,k)}$ from \eqref{eq:1kbelongsji}. Moreover, by \eqref{eq:1kproofinf2} and \eqref{eq:1kproofinf1}, we have 
 $$p_2\left(\left[\mu_{2j-1}^{(2i-1)}\right]^{(1,k)}\right)\succeq 0\cdot \left[\la_{2j+2}^{(2i-2)}\right]^{(k,1)}\text{ and then }p_2\left(\left[\la_{2j-1}^{(2i)}\right]^{(1,k)}\right)\succeq p_1\left(\left[\mu_{2j+1}^{(2i-1)}\right]^{(1,k)}\right)\cdot$$
 \item As $\left(\left[\la_{2j}^{(i)}\right]^{(k,1)}\right)_{i\geq 1}$ is non-increasing, $\left[\la_{2j-1}^{(2i)}\right]^{(k,1)}\in \,_i\C^{(k,1)}$ and $\left[\la_{2j}^{(1)}\right]^{(k,1)}=\left[\la_{2j}^{(2)}\right]^{(k,1)}\prec (\delta_{u,n+2-j})_{u\geq 1}$, we then have that $\la_{2j}^{(i)}$ for all $i\geq 2n+4-2j$.
\end{enumerate}
Recursively on $i$, we obtain the datum $(\la_{2j-1}^{(i)})_{i\geq 2}$. Remark that the datum $(\la_{2j-1}^{(i)})_{i\geq 2}$ allows us to retrieve $\la_{2j}^{(i)}, \mu_{2j-1}^{(i)}, m_i^{(j)}$ for $i\geq 1$.
Finally, by induction on $j$, we recover the datum $(m_i^{(1)})_{i\geq 1}$ and we conclude.
\section{Well-definedness of  $\Phi^{(k,l)}_{2n}$}\label{part:intereven}
We first give an interpretation of the insertion of the parts $\bkl{i}$ into the pairs $(\la_{2j},\la_{2j-1})$ in terms of the admissible words. Denote by $(\la_{2j}^{(i)},\la_{2j-1}^{(i)})$ the pairs $(\la_{2j},\la_{2j-1})$ after the insertion of all the parts $\bkl{i}$, and let $m_i^{(j)}$ denote the number of parts $\bkl{i}$ inserted into the pair $(\la_{2j},\la_{2j-1})$. Note that $m_i^{(n)}$ equals the number of occurrences of $\bkl{i}$ in $\nu$, and the image by $\Phi^{(k,l)}_{2n}$ consists of $(\la_j^{(1)})_{j=1}^{2n}$.
\subsection{The case $k,l\geq 2$}\label{part:interevenkl}
\begin{prop}\label{prop:welldefeven}
For $n\geq j \geq 1$,
\begin{enumerate}
\item for $n+j+1\geq i> 2j\geq 1$, $m_i^{(j)}=0$, and $\la_{2j}^{(i)}=\la_{2j-1}^{(i)}=0$.
\item For $2j\geq i\geq 2$,
\begin{equation}\label{eq:klbelongsjieven}
\left[\la_{2j}^{(i)}\right]^{(k,l)}_{2j}= 0\cdot \left[\la_{2j-1}^{(i)}\right]^{(l,k)}_{2j-1}\in \,_i\C^{(k,l)}_{2j}\,\cdot
\end{equation}
\begin{enumerate}
\item Furthermore,
$\left[\la_{2j}^{(i)}\right]^{(k,l)}_{2j}=\F\left(m_i^{(j)},\,_{i}\C^{(k,l)}_{2j},\left[\la_{2j}^{(i+1)}\right]^{(k,l)}_{2j}\right)$. In particular, $\la_{2j}^{(2j)}=m_{2j}^{(j)}\cdot \akl{2n}$.
\item Reciprocally, $_{i+1}\left[\la_{2j}^{(i)}\right]^{(k,l)}_{2j}=\F\left(m_{i-1}^{(j-1)},\,_{i+1}\C^{(l,k)}_{2j},\left[\la_{2j}^{(i+1)}\right]^{(k,l)}_{2j}\right)$.
\end{enumerate}
\item Finally, $\la_{2j-1}^{(1)}=\la_{2j-1}^{(2)}$, 
\begin{enumerate}
 \item and $\left[\la_{2j}^{(1)}\right]^{(k,l)}_{2j}=\F\left(m_1^{(j)},\C^{(k,l)}_{2j},\left[\la_{2j}^{(2)}\right]^{(k,l)}_{2j}\right)$.
 \item Reciprocally, $\left[\la_{2j}^{(1)}\right]^{(k,l)}_{2j}=\F\left(m_1^{(j)},\C^{(k,l)}_{2j},0\cdot\left[\la_{2j-1}^{(1)}\right]^{(l,k)}_{2j-1}\right)$.
\end{enumerate}
\end{enumerate}
\end{prop}
\medskip
\begin{exs}
We illustrate the previous proposition with Examples \ref{exs:allpaper}.
For $(k,l)=(2,3)$ with $n=3$ and \\
$\nu=(b^{(2,3)}_{1})^5(b^{(2,3)}_{2})^4(b^{(2,3)}_{3})^2(b^{(2,3)}_{4})^3(b^{(2,3)}_{5})(b^{(2,3)}_{6})^3=(1+0)^5(3+1)^4(5+2)^2(12+5)^3(19+8)(45+19)^3$, the corresponding table is the following,
$$
\begin{array}{|c||c|c||c|c||c|c|}
 \hline 
 i&m_i^{(1)}&[\la_6^{(i)}]^{(2,3)}_6&m_i^{(2)}&[\la_4^{(i)}]^{(2,3)}_4&m_i^{(3)}&[\la_2^{(i)}]^{(2,3)}_2\\
 \hline\hline
  6&3&(0,0,0,0,0,\textcolor{red}{3})&0&(0,0,0,0)&0&(0,0)\\
  5&1&(0,0,0,0,\textcolor{red}{1},3)&0&(0,0,0,0)&0&(0,0)\\
  4&3&(0,0,0,\textcolor{red}{0},0,4)&0&(0,0,0,\textcolor{red}{0})&0&(0,0)\\
  3&2&(0,0,\textcolor{red}{2},0,0,4)&2&(0,0,\textcolor{red}{2},0)&0&(0,0)\\
  2&4&(0,\textcolor{red}{0},0,0,1,4)&0&(0,\textcolor{red}{0},2,0)&0&(0,\textcolor{red}{0})\\
  1&5&(\textcolor{red}{0},0,1,0,1,4)&3&(\textcolor{red}{1},0,0,1)&1&(\textcolor{red}{0},0)\\
  \hline
\end{array}\,\,,
$$
and for $(k,l)=(3,2)$ with $n=3$ and $\nu=(b^{(3,2)}_{1})^5(b^{(3,2)}_{2})^4(b^{(3,2)}_{3})^2(b^{(3,2)}_{4})^3(b^{(3,2)}_{5})(b^{(3,2)}_{6})^3=(1+0)^5(2+1)^4(5+3)^2(8+5)^3(19+12)(30+19)^3$, 
$$
\begin{array}{|c||c|c||c|c||c|c|}
 \hline 
 i&m_i^{(1)}&[\la_6^{(i)}]^{(3,2)}_6&m_i^{(2)}&[\la_4^{(i)}]^{(3,2)}_4&m_i^{(3)}&[\la_2^{(i)}]^{(3,2)}_2\\
 \hline\hline
   6&3&(0,0,0,0,0,\textcolor{red}{3})&0&(0,0,0,0)&0&(0,0)\\
  5&1&(0,0,0,0,\textcolor{red}{1},3)&0&(0,0,0,0)&0&(0,0)\\
  4&3&(0,0,0,\textcolor{red}{1},0,4)&0&(0,0,0,\textcolor{red}{0})&0&(0,0)\\
  3&2&(0,0,\textcolor{red}{0},2,0,4)&1&(0,0,\textcolor{red}{1},0)&0&(0,0)\\
  2&4&(0,\textcolor{red}{2},0,0,1,4)&1&(0,\textcolor{red}{1},1,0)&0&(0,\textcolor{red}{0})\\
  1&5&(\textcolor{red}{1},0,0,1,1,4)&1&(\textcolor{red}{0},0,0,1)&0&(\textcolor{red}{0},0)\\
  \hline
\end{array}\,\,\cdot
$$
\end{exs}
We here proceed similarly to Section \ref{part:interinfkl}. The diagram corresponding to the insertion into consecutive pair is the following.
\begin{figure}[H]
\label{fig:proofeven}
\begin{tikzpicture}[scale=1.2, every node/.style={scale=0.8}]
\draw (-1+2.5+0.8,1.5) node[left] {$_{i+1}\left[\la_{2j}^{(i)}\right]^{(k,l)}_{2j}$};
\draw (-2+2.5+0.8,3) node[left] {$\left[\la_{2j}^{(i+1)}\right]^{(k,l)}_{2j}$};

\draw (-1+0.8,1.5) node[left] {$\left[\la_{2j}^{(i)}\right]^{(k,l)}_{2j}$};
\draw (6+0.8,3) node[left] {$\left[\la_{2j-2}^{(i)}\right]^{(k,l)}_{2j-2}$};

\draw (6+0.8,1.5) node[left] {$\left[\la_{2j-2}^{(i-1)}\right]^{(k,l)}_{2j-2}$};

\draw[dotted,<->] (-1.5+2.5,3)--(5.4,3); \draw[dotted,<->] (-0.5+2.5,1.5)--(5.4,1.5);
\draw(0.2,2.75)--(-0.05,2.5); \draw[->](-0.55,2)--(-0.8,1.75);
\draw[dashed] (0.7,2.75)--(0.9,2.5); \draw[dashed,->] (1.3,2)--(1.5,1.75);
\draw[dotted,->] (0.8,1.5)--(-0.5,1.5);
\draw(1.5+4.5,2.65)--(1.5+4.5,2.5); \draw[->] (1.5+4.5,2)--(1.5+4.5,1.75);
 
\draw (-0.6,2.25) node {$\F(m_i^{(j)},\,_{i}\C^{(k,l)}_{2j})$};
\draw (-1.45,2.5)--(0.25,2.5)--(0.25,2)--(-1.45,2)--cycle;

\draw (1.5,2.25) node {$\F(m_{i-1}^{(j-1)},\,_{i+1}\C^{(k,l)}_{2j})$};
\draw (0.4,2.5)--(2.65,2.5)--(2.65,2)--(0.4,2)--cycle;

\draw (1.5+4.5,2.25) node {$\F(m_{i-1}^{(j-1)},\,_{i-1}\C^{(k,l)}_{2j-2})$};
\draw (0.4+4.5,2.5)--(2.65+4.5,2.5)--(2.65+4.5,2)--(0.4+4.5,2)--cycle;

\end{tikzpicture}
\caption{Relation between the insertions in $(\la_{2j},\la_{2j-1})$ and in $(\la_{2j-2},\la_{2j-3})$}
\end{figure}
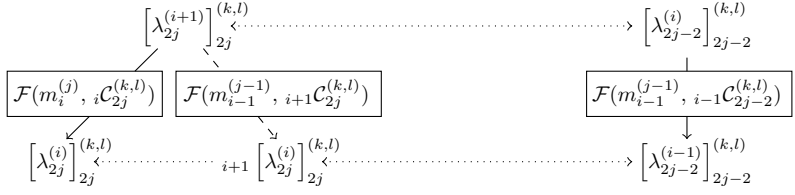
Analogously to \eqref{eq:proofinfkl}, for $1<j\leq n$ and $1<i\leq 2j-1$, 
\begin{equation}\label{eq:klproofeven}
00\cdot\left[\la_{2j-2}^{(i)}\right]^{(k,l)}_{2j-2}\preceq  \left[\la_{2j}^{(i+1)}\right]^{(k,l)}_{2j} \Longleftrightarrow 00\cdot\left[\la_{2j-2}^{(i-1)}\right]^{(k,l)}_{2j-2}\preceq \, _{i+1}\left[\la_{2j}^{(i)}\right]^{(k,l)}_{2j}\,\cdot
\end{equation} 
To prove that $(\la_j^{(1)})_{j=1}^{2n}$ belongs to $ \Lkl_{2n}$, as by \eqref{eq:klbelongsjieven} and fact $(3)$,
$0\cdot\left[\la_{2j-1}^{(1)}\right]^{(l,k)}_{2j-1}\preceq  \left[\la_{2j}^{(1)}\right]^{(k,l)}_{2j}\,,$ 
it suffices to prove that for $n\geq j>1$,
$00\cdot\left[\la_{2j-2}^{(1)}\right]^{(k,l)}_{2j-2}\preceq  \left[\la_{2j}^{(2)}\right]^{(k,l)}_{2j}$.
This holds by recursively using \eqref{eq:klproofeven} and fact $(1)$, as 
$(0)_{i=1}^{2j}=00\cdot\left[\la_{2j-2}^{(2j-1)}\right]^{(k,l)}_{2j-2}\preceq\left[\la_{2j}^{(2j)}\right]^{(k,l)}_{2j}$.\\\\
Let us now construct the inverse image. Let $(\la_j)_{j=1}^{2n}$ be a sequence in $\Lkl_{2n}$, and set $\la_j^{(1)}=\la_j$.
By considering the pair $(\la_1^{(1)},\la_2^{(1)})$, we retrieve $m_1^{(1)},\la_2^{(2)}$  by fact $(3)$, and since by \eqref{eq:inlk12n}, $0\cdot \left[\la_{1}^{(2)}\right]^{(k,l)}_1=\left[\la_{2}^{(2)}\right]^{(k,l)}_2 = (0,m_2^{(1)})$, we have $m_2^{(1)}=\la_2^{(2)}/\akl{2}= \la_1^{(1)}$. Assume now that, for $2\leq j\leq n$ we retrieve $(\la_{2j-2}^{i})_{i=1}^{2j-2}$ (and then $(m_{i}^{(j-1)})_{i=1}^{2j-2}$ by fact $(1.a)$). By fact $(3)$, we retrieve $\la_{2j}^{(2)}$ and $m_1^{(j)}$, and since $00\cdot\left[\la_{2j-2}^{(1)}\right]^{(k,l)}_{2j-2}\preceq\left[\la_{2j}^{(2)}\right]^{(k,l)}_{2j}$. Using Lemma \ref{lem:prechoice}, fact $(1.b)$ and \eqref{eq:klproofeven}, we recursively retrieve $\left[\la_{2j}^{(i)}\right]^{(k,l)}_{2j}$ from $\left[\la_{2j}^{(i-1)}\right]^{(k,l)}_{2j}$ and $m_{i-2}^{(j-1)}$ for $3\leq i\leq 2j$. By fact $(2.a)$, we recover $m_{i}^{(j)}$ for $1\leq i\leq 2j-1$. Finally, as $\left[\la_{2j}^{(2j)}\right]^{(k,l)}_{2j} \in \, _{2j}\C^{(k,l)}_{2j}$, we set $m_{2j}=\la_{2j}^{(2j)}/\akl{2j}$. Therefore, recursively on $j$, we recover the datum $(m_{i}^{(n)})_{i=1}^{2n}$  and we conclude.  
\subsection{The case $(k,1)$}\label{part:interevenk1}
\begin{prop}\label{prop:welldefeven1}
For $n\geq j \geq 1$,
\begin{enumerate}
\item for $n+j+1\geq i> 2j\geq 1$, $m_i^{(j)}=0$, and $\la_{2j}^{(i)}=\la_{2j-1}^{(i)}=0$.
\item for $2j\geq i\geq 3$, we have 
\begin{equation}\label{eq:k1belongsjieven}
\left[\la_{2j}^{(i)}\right]^{(k,1)}_{2j} =  p_1\left(\left[\la_{2j-1}^{(i)}\right]^{(1,k)}_{2j-1}\right)= 0\cdot p_2\left(\left[\la_{2j}^{(i)}\right]^{(1,k)}_{2j-1}\right) \in \,_{\lceil i-1/2\rceil}\C^{(k,1)}_{2j}\,\cdot
\end{equation}
\begin{enumerate}
\item For $j\geq i\geq 2$, let $\left[\mu_{2j-1}^{(2i)}\right]^{(1,k)}_{2j-1}=\F\left(m_{2i}^{(j)},\,_{+(i-1)}\C^{(1,k)}_{2j+1-2i},\left[\la_{2j-1}^{(2i+1)}\right]^{(1,k)}_{2j-1}\right)$.
\begin{enumerate}
 \item Then, $\left[\la_{2j}^{(2i)}\right]^{(k,1)}_{2j}=p_1\left(\left[\mu_{2j-1}^{(2i)}\right]^{(1,k)}_{2j-1}\right)$. In particular, $\la_{2j}^{(2j)}=m_{2j}^{(j)}\cdot \ak{2j}$ and  $$\left[\mu_{2j-1}^{(2j)}\right]^{(1,k)}_{2j-1}=((\underbrace{0,\ldots,0}_{j-1 \text{ times}},m_{2j}^{(j)}),(0)_{u=1}^{j-1})\,\cdot$$
 \item Furthermore, $\left[\la_{2j}^{(2i-1)}\right]^{(k,1)}_{2j}=\F\left(m_{2i-1}^{(j)},\,_{i}\C^{(k,1)}_{2j},\left[\la_{2j}^{(2i)}\right]^{(k,1)}_{2j}\right)$.
\end{enumerate}
\item Reciprocally, for $j\geq i\geq 2$, set
$\left[\mu_{2j-1}^{(2i-1)}\right]^{(1,k)}_{2j-1}=\left(p_1\left(\left[\la_{2j-1}^{(2i-1)}\right]^{(1,k)}_{2j-1}\right), \,_{i}p_2\left(\left[\la_{2j-1}^{(2i-1)}\right]^{(1,k)}_{2j-1}\right)\right)\,\cdot$
\begin{enumerate}
\item Then, $\left[\mu_{2j-1}^{(2i-1)}\right]^{(1,k)}_{2j-1}=\F\left(m_{2i-2}^{(j-1)},\,_{+(i-1)}\C^{(1,k)}_{2j+1-2i},\left[\mu_{2j-1}^{(2i-1)}\right]^{(1,k)}_{2j-1}\right)$,
\item and 
$p_2\left(\left[\mu_{2j-1}^{(2i)}\right]^{(1,k)}_{2j-1}\right)= \F\left(m_{2i-1}^{(j-1)},\,_{i}\C^{(k,1)}_{2j-2},p_2\left(\left[\la_{2j-1}^{(2i+1)}\right]^{(1,k)}_{2j-1}\right)\right)$.
\end{enumerate}
\end{enumerate}
\item Let $\left[\mu_{2j-1}^{(2)}\right]^{(1,k)}_{2j-1}=\left[\la_{2j-1}^{(2)}\right]^{(1,k)}_{2j-1}=\F\left(m_{2}^{(j)},\C^{(1,k)}_{2j-1},\left[\la_{2j-1}^{(3)}\right]^{(1,k)}_{2j-1}\right)$.
\begin{enumerate}
 \item Then, $\left[\la_{2j}^{(2)}\right]^{(k,1)}_{2j}=p_1\left(\left[\mu_{2j-1}^{(2)}\right]^{(1,k)}_{2j-1}\right)$.
 \item Reciprocally, $p_2\left(\left[\mu_{2j-1}^{(2)}\right]^{(1,k)}_{2j-1}\right)= \F\left(m_{1}^{(j-1)},\C^{(k,1)}_{2j-2},p_2\left(\left[\la_{2j-1}^{(3)}\right]^{(1,k)}_{2j-1}\right)\right)$.
\end{enumerate}
\item Finally, $\la_{2j-1}^{(1)}=\la_{2j-1}^{(2)}$ 
\begin{enumerate}
 \item and, $\left[\la_{2j}^{(1)}\right]^{(k,1)}_{2j}=\F\left(m_{1}^{(j)},\C^{(k,1)}_{2j},\left[\la_{2j}^{(2)}\right]^{(k,1)}_{2j}\right)$.
 \item Reciprocally, $\left[\la_{2j}^{(1)}\right]^{(k,1)}_{2j}=\F\left(m_{1}^{(j)},\C^{(k,1)}_{2j},p_1\left(\left[\la_{2j-1}^{(1)}\right]^{(1,k)}_{2j-1}\right)\right)$.
\end{enumerate}
\end{enumerate}
\end{prop}
\begin{ex}
 For $k=6$ with $n=3$ and $\nu=(b_{1}^{(6,1)})^2(b_{2}^{(6,1)})^5(b_{3}^{(6,1)})^2(b_{4}^{(6,1)})^3(b_{6}^{(6,1)})^5=(1+0)^2(1+1)^5(5+6)^2(4+5)^3(15+19)^5$, we have
 $$
\begin{array}{|c||c|c||c|c||c|c|}
 \hline 
 i&m_i^{(3)}&[\mu_5^{(i)}]^{(1,6)}_5&m_i^{(2)}&[\mu_3^{(i)}]^{(1,6)}_3&m_i^{(1)}&[\mu_1^{(i)}]^{(1,6)}_1\\
 \hline\hline
  6&5&(0,0,\textcolor{red}{5}),(0,0)&0&\emptyset&0&\emptyset\\
  5&0&(0,0,\textcolor{red}{5}),(0,0)&0&\emptyset&0&\emptyset\\
  4&3&(0,\textcolor{red}{3},5),(0,\textcolor{red}{5})&0&(0,\textcolor{red}{0}),(0)&0&\emptyset\\
  3&2&(0,\textcolor{red}{1},6),(0,\textcolor{red}{6})&0&(0,\textcolor{red}{0}),(0)&0&\emptyset\\
  2&5&(\textcolor{red}{0},2,6),(\textcolor{red}{2},6)&3&(\textcolor{red}{3},0),(\textcolor{red}{0})&0&(\textcolor{red}{0}),\emptyset\\
  \hline\hline
  &m_1^{(3)}&[\la_6^{(1)}]^{(6,1)}_6&m_1^{(2)}&[\la_4^{(1)}]^{(6,1)}_4&m_1^{(1)}&[\la_2^{(1)}]^{(6,1)}_2\\
  \hline
  &2&(\textcolor{red}{2},2,6)&1&(\textcolor{red}{0},1)&0&(\textcolor{red}{0})\\
  \hline
\end{array}\,\,\cdot
$$
\end{ex}$$$$
The corresponding diagram is the following.
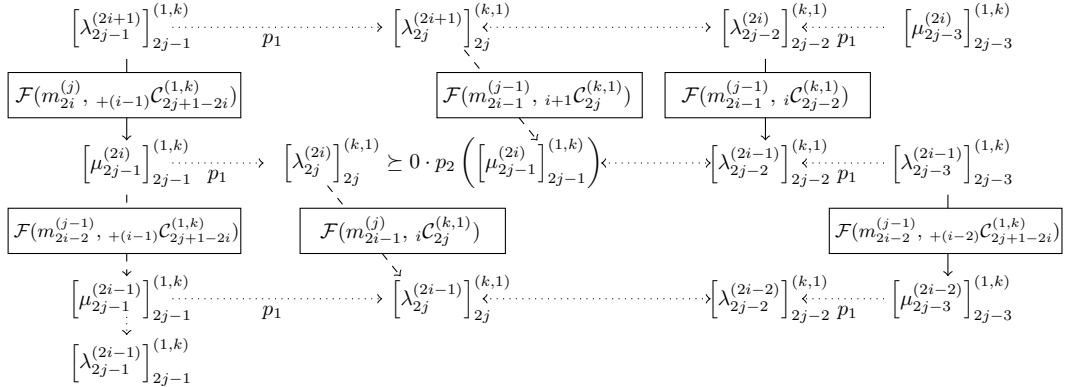
\begin{figure}[H]
\label{fig:proofeven1}
\begin{tikzpicture}[scale=1.2, every node/.style={scale=0.8}]

\draw (-1+2.5+0.8,1.5) node[left] {$\left[\la_{2j}^{(2i)}\right]^{(k,1)}_{2j}\succeq  0\cdot p_2\left(\left[\mu_{2j-1}^{(2i)}\right]^{(1,k)}_{2j-1}\right)$};
\draw (-2+2.5+0.8,3) node[left] {$\left[\la_{2j}^{(2i+1)}\right]^{(k,1)}_{2j}$};
\draw (-2+2.5+0.8,0) node[left] {$\left[\la_{2j}^{(2i-1)}\right]^{(k,1)}_{2j}$};

\draw[dotted,->] (-2.5,3)--(-0.2,3); \draw[dotted,->] (-2.5,1.5)--(-1.5,1.5); \draw[dotted,->] (-2.5,0)--(-0.2,0);
\draw (-1.4,2.85) node {$p_1$};\draw (-2,1.35) node {$p_1$}; \draw (-1.4,-.15) node {$p_1$};

\draw[dotted,<->] (-1.6+2.5,3)--(5.4-2,3); \draw[dotted,<->] (2.2,1.5)--(5.4-2,1.5);\draw[dotted,<->] (-1.6+2.5,0)--(5.4-2,0);

\draw[dashed] (0.7,2.75)--(0.9,2.5); \draw[dashed,->] (1.3,2)--(1.5,1.75);
\draw (1.5,2.25) node {$\F(m_{2i-1}^{(j-1)},\,_{i+1}\C^{(k,1)}_{2j})$};
\draw (0.4,2.5)--(2.65,2.5)--(2.65,2)--(0.4,2)--cycle;

\draw[dashed] (0.7-1.5,2.75-1.5)--(0.9-1.5,2.5-1.5); \draw[dashed,->] (1.3-1.5,2-1.5)--(1.5-1.5,1.75-1.5);
\draw (1.5-1.5,2.25-1.5) node {$\F(m_{2i-1}^{(j)},\,_{i}\C^{(k,1)}_{2j})$};
\draw (0.4-1.5,2.5-1.5)--(2.65-1.5,2.5-1.5)--(2.65-1.5,2-1.5)--(0.4-1.5,2-1.5)--cycle;

\draw (6+0.8-2,3) node[left] {$\left[\la_{2j-2}^{(2i)}\right]^{(k,1)}_{2j-2}$};
\draw (6+0.8-2,1.5) node[left] {$\left[\la_{2j-2}^{(2i-1)}\right]^{(k,1)}_{2j-2}$};
\draw (6+0.8-2,0) node[left] {$\left[\la_{2j-2}^{(2i-2)}\right]^{(k,1)}_{2j-2}$};
\draw (1.5+4.5-2,2.25) node {$\F(m_{2i-1}^{(j-1)},\,_{i}\C^{(k,1)}_{2j-2})$};
\draw (0.4+4.5-2,2.5)--(2.65+4.5-2,2.5)--(2.65+4.5-2,2)--(0.4+4.5-2,2)--cycle;
\draw(1.5+4.5-2,2.65)--(1.5+4.5-2,2.5); \draw[->] (1.5+4.5-2,2)--(1.5+4.5-2,1.75);

\draw (6+0.8-9,3) node[left] {$\left[\la_{2j-1}^{(2i+1)}\right]^{(1,k)}_{2j-1}$};
\draw (6+0.8-9,1.5) node[left] {$\left[\mu_{2j-1}^{(2i)}\right]^{(1,k)}_{2j-1}$};
\draw (1.5+4.5-9,2.25) node {$\F(m_{2i}^{(j)},\,_{+(i-1)}\C^{(1,k)}_{2j+1-2i})$};
\draw (0.2+4.5-9,2.5)--(2.75+4.5-9,2.5)--(2.75+4.5-9,2)--(0.2+4.5-9,2)--cycle;
\draw(1.5+4.5-9,2.65)--(1.5+4.5-9,2.5); \draw[->] (1.5+4.5-9,2)--(1.5+4.5-9,1.75);
\draw (6+0.8-9,1.5-1.5) node[left] {$\left[\mu_{2j-1}^{(2i-1)}\right]^{(1,k)}_{2j-1}$};
\draw (1.5+4.5-9,2.25-1.5) node {\begin{small}$\F(m_{2i-2}^{(j-1)},\,_{+(i-1)}\C^{(1,k)}_{2j+1-2i})$\end{small}};
\draw (0.2+4.5-9,2.5-1.5)--(2.75+4.5-9,2.5-1.5)--(2.75+4.5-9,2-1.5)--(0.2+4.5-9,2-1.5)--cycle;
\draw[dashed](1.5+4.5-9,2.65-1.5)--(1.5+4.5-9,2.5-1.5); \draw[dashed,->] (1.5+4.5-9,2-1.5)--(1.5+4.5-9,1.75-1.5);
\draw (6+0.8-9,-0.7) node[left] {$\left[\la_{2j-1}^{(2i-1)}\right]^{(1,k)}_{2j-1}$};
\draw[dotted,->] (1.5+4.5-9,1.3-1.5)--(1.5+4.5-9,1.1-1.5);

\draw (6+0.8,3) node[left] {$\left[\mu_{2j-3}^{(2i)}\right]^{(1,k)}_{2j-3}$};
\draw (6+0.8,1.5) node[left] {$\left[\la_{2j-3}^{(2i-1)}\right]^{(1,k)}_{2j-3}$};
\draw (6+0.8,1.5-1.5) node[left] {$\left[\mu_{2j-3}^{(2i-2)}\right]^{(1,k)}_{2j-3}$};
\draw (1.5+4.5,2.25-1.5) node {\begin{small}$\F(m_{2i-2}^{(j-1)},\,_{+(i-2)}\C^{(1,k)}_{2j+1-2i})$\end{small}};
\draw (0.2+4.5,2.5-1.5)--(2.75+4.5,2.5-1.5)--(2.75+4.5,2-1.5)--(0.2+4.5,2-1.5)--cycle;
\draw(1.5+4.5,2.65-1.5)--(1.5+4.5,2.5-1.5); \draw[->] (1.5+4.5,2-1.5)--(1.5+4.5,1.75-1.5);

\draw[dotted,->] (5.3,3)--(4.4,3);
\draw[dotted,->] (5.3,1.5)--(4.4,1.5); \draw[dotted,->] (5.3,0)--(4.4,0);
\draw (4.9,2.85) node {$p_1$};\draw (4.9,1.35) node {$p_1$}; \draw (4.9,-.15) node {$p_1$};
\end{tikzpicture}
\caption{Relations between the insertions at step $j$ and step $j-1$}
\end{figure}
Similarly to Section \ref{part:interinfk1}, we have for $1\leq i\leq j-1$,
\begin{equation}\label{eq:k1proofeven1}
 p_2\left(\left[\la_{2j-1}^{(2i+1)}\right]^{(1,k)}_{2j-1}\right)\succeq p_1\left(\left[\mu_{2j-3}^{(2i)}\right]^{(1,k)}_{2j-3}\right)\Longleftrightarrow p_2\left(\left[\mu_{2j-1}^{(2i)}\right]^{(1,k)}_{2j-1}\right)\succeq \left[\la_{2j-2}^{(2i-1)}\right]^{(k,1)}_{2j-2}\,,
\end{equation}
for $2\leq i\leq j$
\begin{align}
\label{eq:k1proofeven3}
 \left[\mu_{2j-1}^{(2i)}\right]^{(1,k)}_{2j-1}\succeq 0\cdot\left[\la_{2j-3}^{(2i-1)}\right]^{(1,k)}_{2j-3} &\Longleftrightarrow \left[\mu_{2j-1}^{(2i-1)}\right]^{(1,k)}_{2j-1}\succeq 0\cdot\left[\mu_{2j-3}^{(2i-2)}\right]^{(1,k)}_{2j-3}\,,\\
 \Updownarrow\qquad\qquad\qquad\qquad&\qquad\qquad\qquad\qquad\qquad\Updownarrow\nonumber\\
 \label{eq:k1proofeven2}
 p_2\left(\left[\mu_{2j-1}^{(2i)}\right]^{(1,k)}_{2j-1}\right)\succeq \left[\la_{2j-2}^{(2i-1)}\right]^{(k,1)}_{2j-2}&\Longleftrightarrow p_2\left(\left[\la_{2j-1}^{(2i-1)}\right]^{(1,k)}_{2j-1}\right)\succeq p_1\left(\left[\mu_{2j-3}^{(2i-2)}\right]^{(1,k)}_{2j-3}\right)\,\cdot
\end{align}
To prove that $(\la_j^{(1)})_{j=1}^{2n}$ belongs to $ \Lk_{2n}$, as by \eqref{eq:k1belongsjieven} and fact $(4)$,
$ p_1\left(\left[\la_{2j-1}^{(1)}\right]^{(1,k)}_{2j-1}\right)\preceq \left[\la_{2j}^{(1)}\right]^{(k,1)}_{2j}\,,$ 
it suffices to prove that for $n\geq j>1$,
$p_2\left(\left[\la_{2j-1}^{(2)}\right]^{(1,k)}_{2j-1}\right)\succeq \left[\la_{2j-2}^{(1)}\right]^{(k,1)}_{2j-2}$. This holds by using recursively \eqref{eq:k1proofeven2} and \eqref{eq:k1proofeven1}, since by fact $(1)$, $(0)_{u=1}^{j-1}=\left[\la_{2j-2}^{(2j-1)}\right]^{(k,1)}_{2j-2}\preceq p_2\left(\left[\la_{2j-1}^{(2j)}\right]^{(1,k)}_{2j-1}\right)$.\\\\To construct the inverse map, first build $m_1^{(1)}=\la_2^{(1)}-\la_{1}^{(1)}$ and $m_2^{(1)}=\la_2^{(2)}=\la_{1}^{(2)}$. Then, recursively retrieve, $\mu_{2j-1}^{(2)},m_1^{(j)}$ from fact $(4)$ and $(3)$, then $(\mu_{2j-1}^{(i)})_{i=3}^{(2j)}$ from $(m_i^{(j-1)})_{j=1}^{(2j-2)}$ using fact $(3)$ and $(2.b)$,  subsequently $(m_i^{(j)})_{i=2}^{(2j-1)}$ by fact $(2.a)$, and finally set $m_{2j}^{(j)}= \la_{2j}^{2j}/\ak{2j}$.
Therefore, we recover the datum $(m_i^{(n)})_{i=1}^{(2j)}$ and we conclude.
\subsection{The case $(1,k)$}\label{part:intereven1k}
\begin{prop}\label{prop:welldefeven2}
For $n\geq j \geq 1$,
\begin{enumerate}
\item for $n+j+1\geq i> 2j\geq 1$, $m_i^{(j)}=0$, and $\la_{2j}^{(i)}=\la_{2j-1}^{(i)}=0$.
\item for $2j\geq i\geq 2$, we have 
\begin{equation}\label{eq:1kbelongsjieven}
0\cdot \left[\la_{2j-1}^{(i)}\right]^{(k,1)}_{2j-1} = 0\cdot  p_2\left(\left[\la_{2j}^{(i)}\right]^{(1,k)}_{2j}\right)= p_1\left(\left[\la_{2j}^{(i)}\right]^{(1,k)}_{2j}\right) \in \,_{\lceil (i+1)/2\rceil}\C^{(k,1)}_{2j+1}\,\cdot
\end{equation}
\begin{enumerate}
\item We have $\la_{2j-1}^{(2j)}=m_{2j}^{(j)}\cdot \ak{2j-1}$. For $j\geq i \geq 2$, \\
let $\left[\mu_{2j}^{(2i-1)}\right]^{(1,k)}_{2j}=\F\left(m_{2i-1}^{(j)},\,_{+(i-1)}\C^{(1,k)}_{2j+2-2i},\left[\la_{2j}^{(2i)}\right]^{(1,k)}_{2j}\right)$. 
\begin{enumerate}
 \item Then, $p_1\left(\left[\la_{2j}^{(2i-1)}\right]^{(1,k)}_{2j}\right)=p_1\left(\left[\mu_{2j}^{(2i-1)}\right]^{(1,k)}_{2j}\right)$.
 \item  Furthermore, $\left[\la_{2j-1}^{(2i-2)}\right]^{(k,1)}_{2j-1}=\F\left(m_{2i-2}^{(j)},\,_{i-1}\C^{(k,1)},\left[\la_{2j}^{(2i-1)}\right]^{(k,1)}_{2j-1}\right)$. 
\end{enumerate}
\item Reciprocally, for $j\geq i\geq 2$, set
$\left[\mu_{2j}^{(2i-2)}\right]^{(1,k)}_{2j}=\left(p_1\left(\left[\la_{2j}^{(2i-1)}\right]^{(1,k)}_{2j}\right), \,_{i}p_2\left(\left[\la_{2j-1}^{(2i-1)}\right]^{(1,k)}_{2j}\right)\right)\,\cdot$
\begin{enumerate}
 \item Then, $\left[\mu_{2j}^{(2i-2)}\right]^{(1,k)}_{2j}=\F\left(m_{2i-3}^{(j-1)},\,_{+(i-1)}\C^{(1,k)}_{2j+2-2i},\left[\mu_{2j}^{(2i-1)}\right]^{(1,k)}_{2j}\right)$,
\item and $p_2\left(\left[\mu_{2j}^{(2i-1)}\right]^{(1,k)}_{2j}\right)=\F\left(m_{2i-2}^{(j-1)},\,_{i}\C^{(k,1)}_{2j-1},p_2\left(\left[\la_{2j}^{(2i)}\right]^{(1,k)}_{2j}\right)\right)$.
\end{enumerate}
\end{enumerate}
\item Finally, $\la_{2j-1}^{(1)}=\la_{2j-1}^{(2)}$, 
\begin{enumerate}
 \item and $\left[\mu_{2j}^{(1)}\right]^{(1,k)}_{2j}=\left[\la_{2j}^{(1)}\right]^{(1,k)}_{2j}=\F\left(m_1^{(j)},\C^{(1,k)}_{2j},\left[\la_{2j}^{(2)}\right]^{(1,k)}_{2j}\right)$.
 \item Reciprocally, $\left[\la_{2j}^{(1)}\right]^{(1,k)}_{2j}=\F\left(m_1^{(j)},\C^{(1,k)}_{2j},\left(0\cdot\left[\la_{2j-1}^{(1)}\right]^{(k,1)}_{2j-1},\left[\la_{2j-1}^{(1)}\right]^{(k,1)}_{2j-1}\right)\right)$.
\end{enumerate}
\end{enumerate}
\end{prop}
\begin{ex}
 For $k=6$ with $n=3$ and $\nu=(b_{1}^{(1,6)})^2(b_{2}^{(1,6)})^5(b_{3}^{(1,6)})^2(b_{4}^{(1,6)})^3(b_{6}^{(1,6)})^5=(1+0)^2(6+1)^5(5+1)^2(24+5)^3(90+19)^5$, we have
 $$
\begin{array}{|c||c|c||c|c||c|c|}
\hline
&m_6^{(3)}&[\la_6^{(6)}]^{(1,6)}_6&m_4^{(2)}&[\la_4^{(4)}]^{(1,6)}_4&m_2^{(1)}&[\la_2^{(2)}]^{(1,6)}_2\\
  \hline
  &2&(0,0,0,\textcolor{red}{5}),(0,0,0)&0&(0,0,0),(0,0)&0&(0,0),(0)\\
  \hline
 \hline 
 i&m_i^{(3)}&[\mu_6^{(i)}]^{(1,6)}_6&m_i^{(2)}&[\mu_4^{(i)}]^{(1,6)}_4&m_i^{(1)}&[\mu_2^{(i)}]^{(1,6)}_2\\
 \hline\hline
  5&0&(0,0,\textcolor{red}{0},5),(0,0,\textcolor{red}{5})&0&\emptyset&0&\emptyset\\
  4&3&(0,0,\textcolor{red}{3},5),(0,0,\textcolor{red}{5})&0&\emptyset&0&\emptyset\\
  3&2&(0,\textcolor{red}{2},3,5),(0,\textcolor{red}{3},5)&3&(0,\textcolor{red}{3},0),(0,\textcolor{red}{0})&0&\emptyset\\
  2&5&(0,\textcolor{red}{0},0,6),(0,\textcolor{red}{0},6)&0&(0,\textcolor{red}{3},0),(0,\textcolor{red}{0})&0&\emptyset\\
  \hline\hline
  &m_1^{(3)}&[\la_6^{(1)}]^{(1,6)}_6&m_1^{(2)}&[\la_4^{(1)}]^{(1,6)}_4&m_1^{(1)}&[\la_2^{(1)}]^{(1,6)}_2\\
  \hline
  &2&(\textcolor{red}{2},0,0,6),(\textcolor{red}{0},0,6)&7&(\textcolor{red}{2},4,0),(\textcolor{red}{4},0)&0&(\textcolor{red}{0},0),(\textcolor{red}{0})\\
  \hline
\end{array}\,\,\cdot
$$
\end{ex}
The corresponding diagram is.
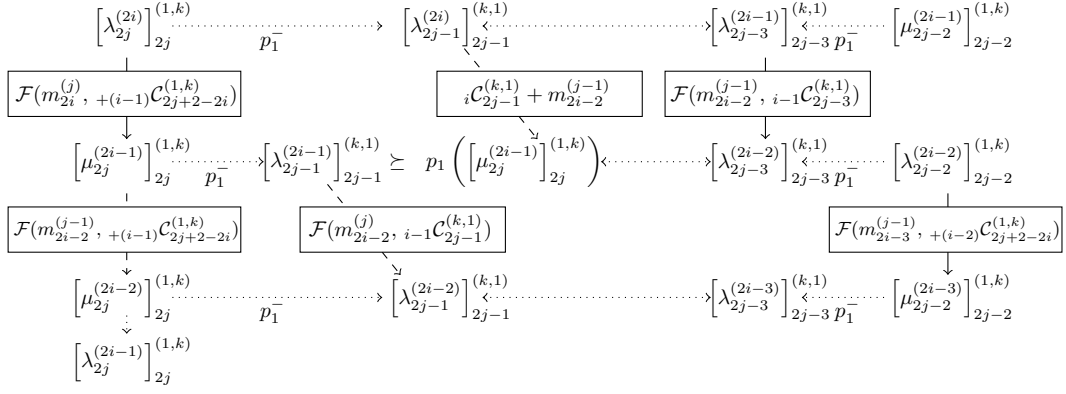
\begin{figure}[H]
\label{fig:proofeven2}
\begin{tikzpicture}[scale=1.2, every node/.style={scale=0.8}]

\draw (-1+2.5+0.8,1.5) node[left] {$\left[\la_{2j-1}^{(2i-1)}\right]^{(k,1)}_{2j-1}\succeq \,\,\,\, p_1\left(\left[\mu_{2j}^{(2i-1)}\right]^{(1,k)}_{2j}\right)$};
\draw (-2+2.5+0.8,3) node[left] {$\left[\la_{2j-1}^{(2i)}\right]^{(k,1)}_{2j-1}$};
\draw (-2+2.5+0.8,0) node[left] {$\left[\la_{2j-1}^{(2i-2)}\right]^{(k,1)}_{2j-1}$};

\draw[dotted,->] (-2.5,3)--(-0.2,3); \draw[dotted,->] (-2.5,1.5)--(-1.5,1.5); \draw[dotted,->] (-2.5,0)--(-0.2,0);
\draw (-1.4,2.85) node {$p_1^-$};\draw (-2,1.35) node {$p_1^-$}; \draw (-1.4,-.15) node {$p_1^-$};

\draw[dotted,<->] (-1.6+2.5,3)--(5.4-2,3); \draw[dotted,<->] (2.2,1.5)--(5.4-2,1.5);\draw[dotted,<->] (-1.6+2.5,0)--(5.4-2,0);

\draw[dashed] (0.7,2.75)--(0.9,2.5); \draw[dashed,->] (1.3,2)--(1.5,1.75);
\draw (1.5,2.25) node {$_{i}\C^{(k,1)}_{2j-1}+m_{2i-2}^{(j-1)}$};
\draw (0.4,2.5)--(2.65,2.5)--(2.65,2)--(0.4,2)--cycle;

\draw[dashed] (0.7-1.5,2.75-1.5)--(0.9-1.5,2.5-1.5); \draw[dashed,->] (1.3-1.5,2-1.5)--(1.5-1.5,1.75-1.5);
\draw (1.5-1.5,2.25-1.5) node {$\F(m_{2i-2}^{(j)},\,_{i-1}\C^{(k,1)}_{2j-1})$};
\draw (0.4-1.5,2.5-1.5)--(2.65-1.5,2.5-1.5)--(2.65-1.5,2-1.5)--(0.4-1.5,2-1.5)--cycle;

\draw (6+0.8-2,3) node[left] {$\left[\la_{2j-3}^{(2i-1)}\right]^{(k,1)}_{2j-3}$};
\draw (6+0.8-2,1.5) node[left] {$\left[\la_{2j-3}^{(2i-2)}\right]^{(k,1)}_{2j-3}$};
\draw (6+0.8-2,0) node[left] {$\left[\la_{2j-3}^{(2i-3)}\right]^{(k,1)}_{2j-3}$};
\draw (1.5+4.5-2,2.25) node {$\F(m_{2i-2}^{(j-1)},\,_{i-1}\C^{(k,1)}_{2j-3})$};
\draw (0.4+4.5-2,2.5)--(2.65+4.5-2,2.5)--(2.65+4.5-2,2)--(0.4+4.5-2,2)--cycle;
\draw(1.5+4.5-2,2.65)--(1.5+4.5-2,2.5); \draw[->] (1.5+4.5-2,2)--(1.5+4.5-2,1.75);

\draw (6+0.8-9,3) node[left] {$\left[\la_{2j}^{(2i)}\right]^{(1,k)}_{2j}$};
\draw (6+0.8-9,1.5) node[left] {$\left[\mu_{2j}^{(2i-1)}\right]^{(1,k)}_{2j}$};
\draw (1.5+4.5-9,2.25) node {$\F(m_{2i}^{(j)},\,_{+(i-1)}\C^{(1,k)}_{2j+2-2i})$};
\draw (0.2+4.5-9,2.5)--(2.75+4.5-9,2.5)--(2.75+4.5-9,2)--(0.2+4.5-9,2)--cycle;
\draw(1.5+4.5-9,2.65)--(1.5+4.5-9,2.5); \draw[->] (1.5+4.5-9,2)--(1.5+4.5-9,1.75);
\draw (6+0.8-9,1.5-1.5) node[left] {$\left[\mu_{2j}^{(2i-2)}\right]^{(1,k)}_{2j}$};
\draw (1.5+4.5-9,2.25-1.5) node {\begin{small}$\F(m_{2i-2}^{(j-1)},\,_{+(i-1)}\C^{(1,k)}_{2j+2-2i})$\end{small}};
\draw (0.2+4.5-9,2.5-1.5)--(2.75+4.5-9,2.5-1.5)--(2.75+4.5-9,2-1.5)--(0.2+4.5-9,2-1.5)--cycle;
\draw[dashed](1.5+4.5-9,2.65-1.5)--(1.5+4.5-9,2.5-1.5); \draw[dashed,->] (1.5+4.5-9,2-1.5)--(1.5+4.5-9,1.75-1.5);
\draw (6+0.8-9,-0.7) node[left] {$\left[\la_{2j}^{(2i-1)}\right]^{(1,k)}_{2j}$};
\draw[dotted,->] (1.5+4.5-9,1.3-1.5)--(1.5+4.5-9,1.1-1.5);

\draw (6+0.8,3) node[left] {$\left[\mu_{2j-2}^{(2i-1)}\right]^{(1,k)}_{2j-2}$};
\draw (6+0.8,1.5) node[left] {$\left[\la_{2j-2}^{(2i-2)}\right]^{(1,k)}_{2j-2}$};
\draw (6+0.8,1.5-1.5) node[left] {$\left[\mu_{2j-2}^{(2i-3)}\right]^{(1,k)}_{2j-2}$};
\draw (1.5+4.5,2.25-1.5) node {\begin{small}$\F(m_{2i-3}^{(j-1)},\,_{+(i-2)}\C^{(1,k)}_{2j+2-2i})$\end{small}};
\draw (0.2+4.5,2.5-1.5)--(2.75+4.5,2.5-1.5)--(2.75+4.5,2-1.5)--(0.2+4.5,2-1.5)--cycle;
\draw(1.5+4.5,2.65-1.5)--(1.5+4.5,2.5-1.5); \draw[->] (1.5+4.5,2-1.5)--(1.5+4.5,1.75-1.5);

\draw[dotted,->] (5.3,3)--(4.4,3);\draw[dotted,->] (5.3,1.5)--(4.4,1.5); \draw[dotted,->] (5.3,0)--(4.4,0);
\draw (4.9,2.85) node {$p_1^-$};\draw (4.9,1.35) node {$p_1^-$}; \draw (4.9,-.15) node {$p_1^-$};
\end{tikzpicture}
\caption{Relations between the insertions at step $j$ and step $j-1$}
\end{figure}
Similarly to Section \ref{part:interinf1k}, we have for $2\leq i\leq j$
\begin{equation}\label{eq:1kproofeven1}
p_2\left(\left[\la_{2j}^{(2i)}\right]^{(1,k)}_{2j}\right)\succeq p_1\left(\left[\mu_{2j-2}^{(2i-1)}\right]^{(1,k)}\right)\Longleftrightarrow  p_2\left(\left[\mu_{2j}^{(2i-1)}\right]^{(1,k)}_{2j}\right)\succeq 0\cdot \left[\la_{2j-3}^{(2i-2)}\right]^{(k,1)}_{2j-3}\,,
\end{equation}
and
\begin{align}\label{eq:1kproofeven3}
\left[\mu_{2j}^{(2i-1)}\right]^{(1,k)}_{2j}\succeq 0\cdot \left[\la_{2j-2}^{(2i-2)}\right]^{(1,k)}_{2j-2}&\Longleftrightarrow\left[\mu_{2j}^{(2i-2)}\right]^{(1,k)}_{2j}\succeq 0\cdot \left[\mu_{2j-2}^{(2i-3)}\right]^{(1,k)}_{2j-2}\,,\\
\Updownarrow\qquad\qquad\qquad\qquad&\qquad\qquad\qquad\qquad\qquad\Updownarrow\nonumber\\
\label{eq:1kproofeven2}
 p_2\left(\left[\mu_{2j}^{(2i-1)}\right]^{(1,k)}_{2j}\right)\succeq 0\cdot \left[\la_{2j-3}^{(2i-2)}\right]^{(k,1)}_{2j-3}&\Longleftrightarrow p_2\left(\left[\la_{2j}^{(2i-2)}\right]^{(1,k)}_{2j}\right)\succeq p_1\left(\left[\mu_{2j-2}^{(2i-3)}\right]^{(1,k)}_{2j-2}\right)\,\cdot
\end{align}
To prove that $(\la_j^{(1)})_{j=1}^{2n}$ belongs to $\Ll_{2n}$, as by \eqref{eq:1kbelongsjieven} and fact $(3)$,
$ p_2\left(\left[\la_{2j}^{(1)}\right]^{(1,k)}_{2j}\right)\succeq \left[\la_{2j-1}^{(1)}\right]^{(k,1)}_{2j-1}\,,$ 
it suffices to prove that for $n\geq j>1$,
$p_2\left(\left[\la_{2j}^{(2)}\right]^{(1,k)}_{2j}\right)\succeq p_1\left(\left[\la_{2j-2}^{(1)}\right]^{(1,k)}_{2j-2}\right)$. This holds by using recursively \eqref{eq:1kproofeven2} and \eqref{eq:1kproofeven1}, since by fact $(1)$, $(0)_{u=1}^{j}=p_1\left(\left[\la_{2j-2}^{(2j-1)}\right]^{(1,k)}_{2j-2}\right)\succeq p_2\left(\left[\la_{2j}^{(2j)}\right]^{(1,k)}_{2j}\right)$.\\\\To construct the inverse map, first build $m_1^{(1)}$ from fact $(3)$ and $m_2^{(1)}=\la_2^{(2)}/\al{2}=\la_{1}^{(2)}$. Then, recursively retrieve, $\la_{2j}^{(2)},m_1^{(j)}$ from fact $(3)$, then $(\la_{2j}^{(i)})_{i=3}^{(2j)}$ from $(m_i^{(j-1)})_{j=1}^{(2j-2)}$ using fact $(2.b)$,  subsequently $(m_i^{(j)})_{i=2}^{(2j-1)}$ by fact $(2.a)$, and finally set $m_{2j}^{(j)}= \la_{2j}^{2j}/\al{2j}$.
Therefore, we recover the datum $(m_i^{(n)})_{i=1}^{(2j)}$ and we conclude.
\section{Well-definedness of  $\Phi^{(k,l)}_{2n-1}$}
\label{part:interodd}
We here consider $n\geq 2$, as $\Llk_{1}=\Bkl_{1}$ and $\Phi^{(k,l)}_1 = Id$.
Let us give an interpretation of the insertion of the parts $\blk{i}$ into the pairs $(\la_{2j+1},\la_{2j})$ in terms of the admissible words. Denote by $(\la_{2j+1}^{(i)},\la_{2j}^{(i)})$ the pairs $(\la_{2j+1},\la_{2j})$ after the insertion of all the parts $\blk{i}$, and let $m_i^{(j)}$ denote the number of parts $\blk{i}$ inserted into the pair $(\la_{2j+1},\la_{2j})$. Note that $m_i^{(n-1)}$ equals the number of occurrences of $\blk{i}$ in $\nu\in \Blk_{2n-1}$, and the image by $\Phi^{(k,l)}_{2n-1}$ consists of $(\la_j^{(1)})_{j=1}^{2n-1}$.
\subsection{The case $k,l\geq 2$}\label{part:interoddkl}
The proof is similar to Section \ref{part:interevenkl} by using the following proposition.
\begin{prop}\label{prop:welldefodd}
For $n-1\geq j \geq 0$,
\begin{enumerate}
\item for $n+j+1\geq i> 2j+1\geq 1$, $m_i^{(j)}=0$, and $\la_{2j+1}^{(i)}=\la_{2j}^{(i)}=0$.
\item For $2j+1\geq i\geq 2$,
\begin{equation}\label{eq:klbelongsjiodd}
\left[\la_{2j+1}^{(i)}\right]^{(l,k)}_{2j+1}= 0\cdot \left[\la_{2j}^{(i)}\right]^{(k,l)}_{2j}\in \,_i\C^{(l,k)}_{2j+1}\,\cdot
\end{equation}
\begin{enumerate}
\item Furthermore,
$\left[\la_{2j+1}^{(i)}\right]^{(l,k)}_{2j+1}=\F\left(m_i^{(j)},\,_{i}\C^{(l,k)}_{2j+1},\left[\la_{2j+1}^{(i+1)}\right]^{(l,k)}_{2j+1}\right)$.
\item Reciprocally, $_{i+1}\left[\la_{2j+1}^{(i)}\right]^{(l,k)}_{2j+1}=\F\left(m_{i-1}^{(j-1)},\,_{i+1}\C^{(l,k)}_{2j+1},\left[\la_{2j+1}^{(i+1)}\right]^{(l,k)}_{2j+1}\right)$.
\end{enumerate}
\item Finally, $\la_{2j}^{(1)}=\la_{2j}^{(2)}$, 
\begin{enumerate}
 \item and $\left[\la_{2j+1}^{(1)}\right]^{(l,k)}_{2j+1}=\F\left(m_1^{(j)},\C^{(l,k)}_{2j+1},\left[\la_{2j+1}^{(2)}\right]^{(l,k)}_{2j+1}\right)$.
 \item Reciprocally, $\left[\la_{2j+1}^{(1)}\right]^{(l,k)}_{2j+1}=\F\left(m_1^{(j)},\C^{(l,k)}_{2j+1},0\cdot\left[\la_{2j}^{(1)}\right]^{(k,l)}_{2j}\right)$.
\end{enumerate}
\end{enumerate}
\end{prop}
\begin{exs}
For $(k,l)=(2,3)$ with $n=4$ and \\ $\nu=(b^{(3,2)}_{1})^5(b^{(3,2)}_{2})^4(b^{(3,2)}_{3})^2(b^{(3,2)}_{4})^3(b^{(3,2)}_{5})(b^{(3,2)}_{6})^3=(1+0)^5(2+1)^4(5+3)^2(8+5)^3(19+12)(30+19)^3$
, the corresponding table is the following,
$$
\begin{array}{|c||c|c||c|c||c|c||c|c|}
 \hline 
 i&m_i^{(3)}&[\la_7^{(i)}]^{(3,2)}_7&m_i^{(2)}&[\la_5^{(i)}]^{(3,2)}_5&m_i^{(1)}&[\la_3^{(i)}]^{(3,2)}_3&m_i^{(0)}&[\la_1^{(i)}]^{(3,2)}_1\\
 \hline\hline
 7&0&(0,0,0,0,0,0,\textcolor{red}{0})&0&(0,0,0,0,0)&0&(0,0,0)&0&(0)\\
  6&3&(0,0,0,0,0,\textcolor{red}{0},1)&0&(0,0,0,0,0)&0&(0,0,0)&0&(0)\\
  5&1&(0,0,0,0,\textcolor{red}{1},0,1)&1&(0,0,0,0,\textcolor{red}{1})&0&(0,0,0)&0&(0)\\
  4&3&(0,0,0,\textcolor{red}{1},0,1,1)&0&(0,0,0,\textcolor{red}{0},1)&0&(0,0,0)&0&(0)\\
  3&2&(0,0,\textcolor{red}{0},2,0,1,1)&1&(0,0,\textcolor{red}{1},0,1)&0&(0,0,\textcolor{red}{0})&0&(0)\\
  2&4&(0,\textcolor{red}{2},0,0,1,1,1)&1&(0,\textcolor{red}{1},1,0,1)&0&(0,\textcolor{red}{0},0)&0&(0)\\
  1&5&(\textcolor{red}{1},0,0,1,1,1,1)&1&(\textcolor{red}{0},0,0,1,1)&0&(\textcolor{red}{0},0,0)&0&(\textcolor{red}{0})\\
  \hline
\end{array}\,\,,
$$
and for $(k,l)=(3,2)$ with $n=4$ and $\nu=(b^{(2,3)}_{1})^5(b^{(2,3)}_{2})^4(b^{(2,3)}_{3})^2(b^{(2,3)}_{4})^3(b^{(2,3)}_{5})(b^{(2,3)}_{6})^3=(1+0)^5(3+1)^4(5+2)^2(12+5)^3(19+8)(45+19)^3$, 
$$
\begin{array}{|c||c|c||c|c||c|c||c|c|}
 \hline 
 i&m_i^{(3)}&[\la_7^{(i)}]^{(2,3)}_7&m_i^{(2)}&[\la_5^{(i)}]^{(2,3)}_5&m_i^{(1)}&[\la_3^{(i)}]^{(2,3)}_3&m_i^{(0)}&[\la_1^{(i)}]^{(2,3)}_1\\
 \hline\hline
 7&0&(0,0,0,0,0,0,\textcolor{red}{0})&0&(0,0,0,0,0)&0&(0,0,0)&0&(0)\\
  6&3&(0,0,0,0,0,\textcolor{red}{1},1)&0&(0,0,0,0,0)&0&(0,0,0)&0&(0)\\
  5&1&(0,0,0,0,\textcolor{red}{1},1,1)&1&(0,0,0,0,\textcolor{red}{1})&0&(0,0,0)&0&(0)\\
  4&3&(0,0,0,\textcolor{red}{0},1,0,2)&0&(0,0,0,\textcolor{red}{0},0)&0&(0,0,0)&0&(0)\\
  3&2&(0,0,\textcolor{red}{2},0,1,0,2)&2&(0,0,\textcolor{red}{2},0,1)&0&(0,0,\textcolor{red}{0})&0&(0)\\
  2&4&(0,\textcolor{red}{0},0,0,2,0,2)&0&(0,\textcolor{red}{0},2,0,1)&0&(0,\textcolor{red}{0},0)&0&(0)\\
  1&5&(\textcolor{red}{0},0,1,0,2,0,2)&3&(\textcolor{red}{1},0,0,1,1)&1&(\textcolor{red}{0},0,0)&0&(\textcolor{red}{0})\\
  \hline
\end{array}\,\,,
$$
\end{exs}
$$$$
The corresponding diagram is the following.
\begin{figure}[H]
\label{fig:proofodd}
\begin{tikzpicture}[scale=1.2, every node/.style={scale=0.8}]

\draw (-1+2.5+0.8,1.5) node[left] {$_{i+1}\left[\la_{2j+1}^{(i)}\right]^{(l,k)}_{2j+1}$};
\draw (-2+2.5+0.8,3) node[left] {$\left[\la_{2j+1}^{(i+1)}\right]^{(l,k)}_{2j+1}$};

\draw (-1+0.8,1.5) node[left] {$\left[\la_{2j+1}^{(i)}\right]^{(l,k)}_{2j+1}$};
\draw (6+0.8,3) node[left] {$\left[\la_{2j-1}^{(i)}\right]^{(l,k)}_{2j-1}$};

\draw (6+0.8,1.5) node[left] {$\left[\la_{2j-1}^{(i-1)}\right]^{(l,k)}_{2j-1}$};

\draw[dotted,<->] (-1.5+2.5,3)--(5.4,3); \draw[dotted,<->] (-0.5+2.5,1.5)--(5.4,1.5);
\draw(0,2.75)--(-0.25,2.5); \draw[->](-0.75,2)--(-1,1.75);
\draw[dashed] (0.7,2.75)--(0.9,2.5); \draw[dashed,->] (1.3,2)--(1.5,1.75);
\draw[dashed,->] (0.6,1.5)--(-0.6,1.5);
\draw(1.5+4.5,2.65)--(1.5+4.5,2.5); \draw[->] (1.5+4.5,2)--(1.5+4.5,1.75);
 
\draw (-0.6,2.25) node {$\F(m_i^{(j)},\,_{i}\C^{(l,k)}_{2j+1})$};
\draw (-1.45,2.5)--(0.25,2.5)--(0.25,2)--(-1.45,2)--cycle;

\draw (1.5,2.25) node {$\F(m_{i-1}^{(j-1)},\,_{i+1}\C^{(l,k)}_{2j+1})$};
\draw (0.4,2.5)--(2.65,2.5)--(2.65,2)--(0.4,2)--cycle;

\draw (1.5+4.5,2.25) node {$\F(m_{i-1}^{(j-1)},\,_{i-1}\C^{(l,k)}_{2j-1})$};
\draw (0.4+4.5,2.5)--(2.65+4.5,2.5)--(2.65+4.5,2)--(0.4+4.5,2)--cycle;

\end{tikzpicture}
\caption{Relation between the insertions in $(\la_{2j+1},\la_{2j})$ and in $(\la_{2j-1},\la_{2j-2})$}
\end{figure}
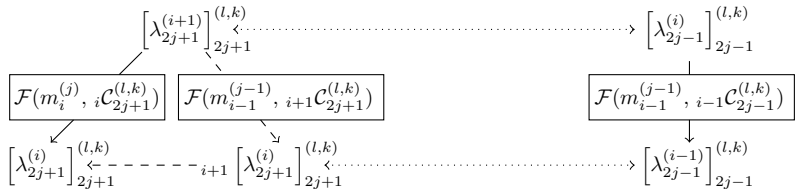
An analogue of \eqref{eq:proofinfkl} consists of the following for  $1<i\leq 2j$: 
\begin{equation}\label{eq:proofodd23}
00\cdot\left[\la_{2j-1}^{(i)}\right]^{(l,k)}_{2j-1}\preceq  \left[\la_{2j+1}^{(i+1)}\right]^{(l,k)}_{2j+1} \Longleftrightarrow 00\cdot\left[\la_{2j-1}^{(i-1)}\right]^{(l,k)}_{2j-1}\preceq\,_{i+1}\left[\la_{2j+1}^{(i)}\right]^{(l,k)}_{2j+1}\,\cdot
\end{equation}
\subsection{The case $(k,1)$}\label{part:interoddk1}
The proof is similar to Section \ref{part:intereven1k} by using the following proposition.
\begin{prop}\label{prop:welldefodd1}
For $n-1\geq j \geq 0$,
\begin{enumerate}
\item for $n+j+1\geq i> 2j+1\geq 1$, $m_i^{(j)}=0$, and $\la_{2j}^{(i)}=\la_{2j-1}^{(i)}=0$.
\item for $2j+1\geq i\geq 2$, we have 
\begin{equation}\label{eq:k1belongsjiodd}
0\cdot \left[\la_{2j}^{(i)}\right]^{(k,1)}_{2j} = 0\cdot  p_2\left(\left[\la_{2j+1}^{(i)}\right]^{(1,k)}_{2j+1}\right)= p_1\left(\left[\la_{2j+1}^{(i)}\right]^{(1,k)}_{2j+1}\right) \in \,_{\lceil (i+1)/2\rceil}\C^{(k,1)}_{2j+2}\,\cdot
\end{equation}
\begin{enumerate}
\item For $j\geq i \geq 1$, let $\left[\mu_{2j+1}^{(2i+1)}\right]^{(1,k)}_{2j+1}=\F\left(m_{2i+1}^{(j)},\,_{+(i)}\C^{(1,k)}_{2j+1-2i},\left[\la_{2j+1}^{(2i+2)}\right]^{(1,k)}_{2j+1}\right)$. 
\begin{enumerate}
 \item Then, $p_1\left(\left[\la_{2j+1}^{(2i+1)}\right]^{(1,k)}_{2j+1}\right)=p_1\left(\left[\mu_{2j+1}^{(2i+1)}\right]^{(1,k)}_{2j+1}\right)$.
 In particular, $\la_{2j+1}^{(2j+1)}=m_{2j+1}^{(j)}\cdot \al{2j+1}$, and 
$$\left[\mu_{2j+1}^{(2j+1)}\right]^{(1,k)}_{2j+1}=((\underbrace{0,\ldots,0}_{j \text{ times}},m_{2j+1}^{(j)}),(0)_{u=1}^{j})\,\cdot$$
 \item Furthermore, $\left[\la_{2j}^{(2i)}\right]^{(k,1)}_{2j}=\F\left(m_{2i}^{(j)},\,_{i}\C^{(k,1)},\left[\la_{2j}^{(2i+1)}\right]^{(k,1)}_{2j}\right)$.
\end{enumerate}
\item Reciprocally, for $j\geq i\geq 1$, set
$\left[\mu_{2j+1}^{(2i)}\right]^{(1,k)}_{2j+1}=\left(p_1\left(\left[\la_{2j+1}^{(2i)}\right]^{(1,k)}_{2j+1}\right), \,_{i}p_2\left(\left[\la_{2j+1}^{(2i)}\right]^{(1,k)}_{2j+1}\right)\right)\,\cdot$
\begin{enumerate}
 \item Then, $\left[\mu_{2j}^{(2i)}\right]^{(1,k)}_{2j+1}=\F\left(m_{2i-1}^{(j-1)},\,_{+(i)}\C^{(1,k)}_{2j+1-2i},\left[\mu_{2j+1}^{(2i+1)}\right]^{(1,k)}_{2j+1}\right)$,
\item and $p_2\left(\left[\mu_{2j+1}^{(2i+1)}\right]^{(1,k)}_{2j+1}\right)=\F\left(m_{2i}^{(j-1)},\,_{i+1}\C^{(k,1)}_{2j},p_2\left(\left[\la_{2j+1}^{(2i+2)}\right]^{(1,k)}_{2j+1}\right)\right)$.
\end{enumerate}
\end{enumerate}
\item Finally, $\la_{2j}^{(1)}=\la_{2j}^{(2)}$, 
\begin{enumerate}
 \item  and $\left[\mu_{2j+1}^{(1)}\right]^{(1,k)}_{2j+1}=\left[\la_{2j+1}^{(1)}\right]^{(1,k)}_{2j+1}=\F\left(m_{1}^{(j)},\C^{(1,k)}_{2j+1},\left[\la_{2j+1}^{(2)}\right]^{(1,k)}_{2j+1}\right)$. 
 \item Reciprocally, $\left[\la_{2j+1}^{(1)}\right]^{(1,k)}_{2j+1}=\F\left(m_{1}^{(j)},\C^{(1,k)}_{2j+1},\left(0\cdot\left[\la_{2j}^{(1)}\right]^{(k,1)}_{2j},\left[\la_{2j}^{(1)}\right]^{(k,1)}_{2j}\right)\right)$.
\end{enumerate}
\end{enumerate}
\end{prop}
\begin{ex}
 For $k=6$ with $n=4$ and $\nu=(b_{1}^{(1,6)})^2(b_{2}^{(1,6)})^5(b_{3}^{(1,6)})^2(b_{4}^{(1,6)})^3(b_{6}^{(1,6)})^5=(1+0)^2(6+1)^5(5+1)^2(24+5)^3(90+19)^5$, we have
 $$
\begin{array}{|c||c|c||c|c||c|c||c|c|}
 \hline 
 i&m_i^{(3)}&[\mu_7^{(i)}]^{(1,6)}_7&m_i^{(2)}&[\mu_5^{(i)}]^{(1,6)}_5&m_i^{(1)}&[\mu_3^{(i)}]^{(1,6)}_3&m_i^{(0)}&[\mu_1^{(i)}]^{(1,6)}_1\\
 \hline\hline
 7&0&(0,0,0,\textcolor{red}{0}),(0,0,0)&0&\emptyset&0&\emptyset&0&\emptyset\\
 6&0&(0,0,0,\textcolor{red}{5}),(0,0,0)&0&\emptyset&0&\emptyset&0&\emptyset\\
  5&0&(0,0,\textcolor{red}{0},5),(0,0,\textcolor{red}{5})&5&(0,0,\textcolor{red}{5}),(0,0)&0&\emptyset&0&\emptyset\\
  4&3&(0,0,\textcolor{red}{3},5),(0,0,\textcolor{red}{5})&0&(0,0,\textcolor{red}{5}),(0,0)&0&\emptyset&0&\emptyset\\
  3&2&(0,\textcolor{red}{2},3,5),(0,\textcolor{red}{3},5)&3&(0,\textcolor{red}{3},5),(0,\textcolor{red}{5})&0&(0,\textcolor{red}{0}),(0)&0&\emptyset\\
  2&5&(0,\textcolor{red}{0},1,6),(0,\textcolor{red}{1},6)&0&(0,\textcolor{red}{3},5),(0,\textcolor{red}{5})&0&(0,\textcolor{red}{0}),(0)&0&\emptyset\\
  \hline\hline
  &m_1^{(3)}&[\la_7^{(1)}]^{(1,6)}_7&m_1^{(2)}&[\la_5^{(1)}]^{(1,6)}_5&m_1^{(1)}&[\la_3^{(1)}]^{(1,6)}_3&m_1^{(0)}&[\la_1^{(1)}]^{(1,6)}_1\\
  \hline
  &2&(\textcolor{red}{2},0,1,6),(\textcolor{red}{0},1,6)&7&(\textcolor{red}{3},0,6),(\textcolor{red}{0},6)&0&(\textcolor{red}{0},0),(\textcolor{red}{0})&0&(0),\emptyset\\
  \hline
\end{array}\,\,\cdot
$$
\end{ex}
$$$$
The corresponding diagram is the following.
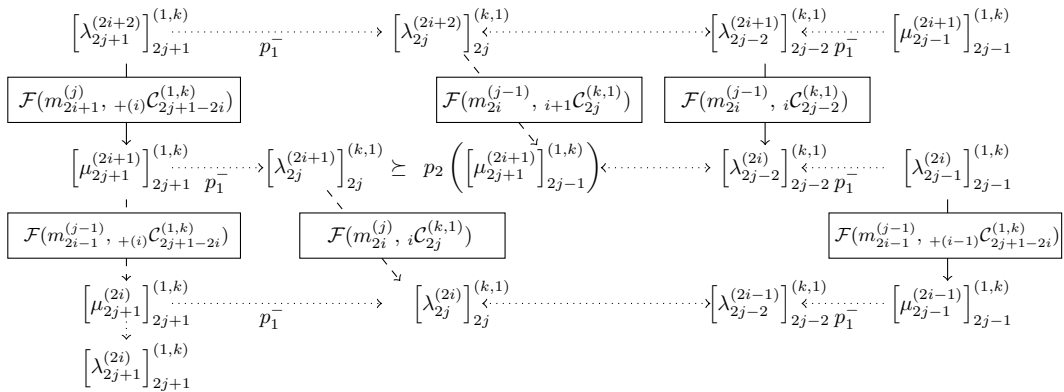
\begin{figure}[H]
\label{fig:proofodd1}
\begin{tikzpicture}[scale=1.2, every node/.style={scale=0.8}]

\draw (-1+2.5+0.8,1.5) node[left] {$\left[\la_{2j}^{(2i+1)}\right]^{(k,1)}_{2j}\succeq  \,\,\, p_2\left(\left[\mu_{2j+1}^{(2i+1)}\right]^{(1,k)}_{2j-1}\right)$};
\draw (-2+2.5+0.8,3) node[left] {$\left[\la_{2j}^{(2i+2)}\right]^{(k,1)}_{2j}$};
\draw (-2+2.5+0.8,0) node[left] {$\left[\la_{2j}^{(2i)}\right]^{(k,1)}_{2j}$};

\draw[dotted,->] (-2.5,3)--(-0.2,3); \draw[dotted,->] (-2.5,1.5)--(-1.5,1.5); \draw[dotted,->] (-2.5,0)--(-0.2,0);
\draw (-1.4,2.85) node {$p_1^-$};\draw (-2,1.35) node {$p_1^-$}; \draw (-1.4,-.15) node {$p_1^-$};

\draw[dotted,<->] (-1.6+2.5,3)--(5.4-2,3); \draw[dotted,<->] (2.2,1.5)--(5.4-2,1.5);\draw[dotted,<->] (-1.6+2.5,0)--(5.4-2,0);

\draw[dashed] (0.7,2.75)--(0.9,2.5); \draw[dashed,->] (1.3,2)--(1.5,1.75);
\draw (1.5,2.25) node {$\F(m_{2i}^{(j-1)},\,_{i+1}\C^{(k,1)}_{2j})$};
\draw (0.4,2.5)--(2.65,2.5)--(2.65,2)--(0.4,2)--cycle;

\draw[dashed] (0.7-1.5,2.75-1.5)--(0.9-1.5,2.5-1.5); \draw[dashed,->] (1.3-1.5,2-1.5)--(1.5-1.5,1.75-1.5);
\draw (1.5-1.5,2.25-1.5) node {$\F(m_{2i}^{(j)},\,_{i}\C^{(k,1)}_{2j})$};
\draw (0.4-1.5,2.5-1.5)--(2.65-1.5,2.5-1.5)--(2.65-1.5,2-1.5)--(0.4-1.5,2-1.5)--cycle;

\draw (6+0.8-2,3) node[left] {$\left[\la_{2j-2}^{(2i+1)}\right]^{(k,1)}_{2j-2}$};
\draw (6+0.8-2,1.5) node[left] {$\left[\la_{2j-2}^{(2i)}\right]^{(k,1)}_{2j-2}$};
\draw (6+0.8-2,0) node[left] {$\left[\la_{2j-2}^{(2i-1)}\right]^{(k,1)}_{2j-2}$};
\draw (1.5+4.5-2,2.25) node {$\F(m_{2i}^{(j-1)},\,_{i}\C^{(k,1)}_{2j-2})$};
\draw (0.4+4.5-2,2.5)--(2.65+4.5-2,2.5)--(2.65+4.5-2,2)--(0.4+4.5-2,2)--cycle;
\draw(1.5+4.5-2,2.65)--(1.5+4.5-2,2.5); \draw[->] (1.5+4.5-2,2)--(1.5+4.5-2,1.75);

\draw (6+0.8-9,3) node[left] {$\left[\la_{2j+1}^{(2i+2)}\right]^{(1,k)}_{2j+1}$};
\draw (6+0.8-9,1.5) node[left] {$\left[\mu_{2j+1}^{(2i+1)}\right]^{(1,k)}_{2j+1}$};
\draw (1.5+4.5-9,2.25) node {$\F(m_{2i+1}^{(j)},\,_{+(i)}\C^{(1,k)}_{2j+1-2i})$};
\draw (0.2+4.5-9,2.5)--(2.75+4.5-9,2.5)--(2.75+4.5-9,2)--(0.2+4.5-9,2)--cycle;
\draw(1.5+4.5-9,2.65)--(1.5+4.5-9,2.5); \draw[->] (1.5+4.5-9,2)--(1.5+4.5-9,1.75);
\draw (6+0.8-9,1.5-1.5) node[left] {$\left[\mu_{2j+1}^{(2i)}\right]^{(1,k)}_{2j+1}$};
\draw (1.5+4.5-9,2.25-1.5) node {\begin{small}$\F(m_{2i-1}^{(j-1)},\,_{+(i)}\C^{(1,k)}_{2j+1-2i})$\end{small}};
\draw (0.2+4.5-9,2.5-1.5)--(2.75+4.5-9,2.5-1.5)--(2.75+4.5-9,2-1.5)--(0.2+4.5-9,2-1.5)--cycle;
\draw[dashed](1.5+4.5-9,2.65-1.5)--(1.5+4.5-9,2.5-1.5); \draw[dashed,->] (1.5+4.5-9,2-1.5)--(1.5+4.5-9,1.75-1.5);
\draw (6+0.8-9,-0.7) node[left] {$\left[\la_{2j+1}^{(2i)}\right]^{(1,k)}_{2j+1}$};
\draw[dotted,->] (1.5+4.5-9,1.3-1.5)--(1.5+4.5-9,1.1-1.5);

\draw (6+0.8,3) node[left] {$\left[\mu_{2j-1}^{(2i+1)}\right]^{(1,k)}_{2j-1}$};
\draw (6+0.8,1.5) node[left] {$\left[\la_{2j-1}^{(2i)}\right]^{(1,k)}_{2j-1}$};
\draw (6+0.8,1.5-1.5) node[left] {$\left[\mu_{2j-1}^{(2i-1)}\right]^{(1,k)}_{2j-1}$};
\draw (1.5+4.5,2.25-1.5) node {\begin{small}$\F(m_{2i-1}^{(j-1)},\,_{+(i-1)}\C^{(1,k)}_{2j+1-2i})$\end{small}};
\draw (0.2+4.5,2.5-1.5)--(2.75+4.5,2.5-1.5)--(2.75+4.5,2-1.5)--(0.2+4.5,2-1.5)--cycle;
\draw(1.5+4.5,2.65-1.5)--(1.5+4.5,2.5-1.5); \draw[->] (1.5+4.5,2-1.5)--(1.5+4.5,1.75-1.5);

\draw[dotted,->] (5.3,3)--(4.4,3);\draw[dotted,->] (5.3,1.5)--(4.4,1.5); \draw[dotted,->] (5.3,0)--(4.4,0);
\draw (4.9,2.85) node {$p_1^-$};\draw (4.9,1.35) node {$p_1^-$}; \draw (4.9,-.15) node {$p_1^-$};
\end{tikzpicture}
\caption{Relations between the insertions at step $j$ and step $j-1$}
\end{figure}
Similarly to Section \ref{part:intereven1k}, we have for $1\leq i\leq j-1$,
\begin{equation}\label{eq:k1proofodd1}
 p_2\left(\left[\la_{2j+1}^{(2i+2)}\right]^{(1,k)}_{2j+1}\right)\succeq p_1\left(\left[\mu_{2j-1}^{(2i+1)}\right]^{(1,k)}_{2j-1}\right)\Longleftrightarrow p_2\left(\left[\mu_{2j+1}^{(2i+1)}\right]^{(1,k)}_{2j+1}\right)\succeq 0\cdot\left[\la_{2j-2}^{(2i)}\right]^{(k,1)}_{2j-2}\,,
\end{equation}
for $1\leq i\leq j$
\begin{align}
\label{eq:k1proofodd3}
 \left[\mu_{2j+1}^{(2i+1)}\right]^{(1,k)}_{2j-1}\succeq 0\cdot\left[\la_{2j-1}^{(2i)}\right]^{(1,k)}_{2j-1} &\Longleftrightarrow \left[\mu_{2j+1}^{(2i)}\right]^{(1,k)}_{2j+1}\succeq 0\cdot\left[\mu_{2j-1}^{(2i-1)}\right]^{(1,k)}_{2j-1}\,,\\
 \Updownarrow\qquad\qquad\qquad\qquad&\qquad\qquad\qquad\qquad\qquad\Updownarrow\nonumber\\
 \label{eq:k1proofodd2}
 p_2\left(\left[\mu_{2j+1}^{(2i+1)}\right]^{(1,k)}_{2j+1}\right)\succeq 0\cdot \left[\la_{2j-2}^{(2i)}\right]^{(k,1)}_{2j-2}&\Longleftrightarrow p_2\left(\left[\la_{2j+1}^{(2i)}\right]^{(1,k)}_{2j+1}\right)\succeq p_1\left(\left[\mu_{2j-1}^{(2i-1)}\right]^{(1,k)}_{2j-1}\right)\,\cdot
\end{align}
\subsection{The case $(1,k)$}\label{part:interodd1k}
The proof is similar to Section \ref{part:interevenk1} by using the following proposition.
\begin{prop}\label{prop:welldefodd2}
For $n-1\geq j \geq 0$,
\begin{enumerate}
\item for $n+j+1\geq i> 2j+1\geq 1$, $m_i^{(j)}=0$, and $\la_{2j}^{(i)}=\la_{2j-1}^{(i)}=0$.
\item for $2j+1\geq i\geq 3$, we have 
\begin{equation}\label{eq:1kbelongsjiodd}
\left[\la_{2j+1}^{(i)}\right]^{(k,1)}_{2j+1} =  p_1\left(\left[\la_{2j}^{(i)}\right]^{(1,k)}_{2j}\right)= 0\cdot p_2\left(\left[\la_{2j}^{(i)}\right]^{(1,k)}_{2j}\right) \in \,_{\lceil i/2\rceil}\C^{(k,1)}_{2j+1}\,\cdot
\end{equation}
\begin{enumerate}
\item 
We have $\la_{2j+1}=m_{2j+1}^{(j)}\cdot \ak{2j+1}$. For $j\geq i\geq 2$, let $\left[\mu_{2j}^{(2i)}\right]^{(1,k)}_{2j}=\F\left(m_{2i}^{(j)},\,_{+(i-1)}\C^{(1,k)}_{2j+2-2i},\left[\la_{2j}^{(2i+1)}\right]^{(1,k)}_{2j}\right)$.
\begin{enumerate}
 \item Then, $\left[\la_{2j+1}^{(2i)}\right]^{(k,1)}_{2j+1}=p_1\left(\left[\mu_{2j}^{(2i)}\right]^{(1,k)}_{2j}\right)$. 
\item Furthermore, $\left[\la_{2j+1}^{(2i-1)}\right]^{(k,1)}_{2j+1}=\F\left(m_{2i-1}^{(j)},\,_{i}\C^{(k,1)}_{2j},\left[\la_{2j+1}^{(2i)}\right]^{(k,1)}_{2j+1}\right)$.
\end{enumerate}
\item Reciprocally, for $j\geq i\geq 2$, set
$\left[\mu_{2j}^{(2i-1)}\right]^{(1,k)}_{2j}=\left(p_1\left(\left[\la_{2j}^{(2i-1)}\right]^{(1,k)}_{2j}\right), \,_{i}p_2\left(\left[\la_{2j}^{(2i-1)}\right]^{(1,k)}_{2j}\right)\right)\,\cdot$
\begin{enumerate}
 \item 
Then, $\left[\mu_{2j}^{(2i-1)}\right]^{(1,k)}_{2j}=\F\left(m_{2i-2}^{(j-1)},\,_{+(i-1)}\C^{(1,k)}_{2j+2-2i},\left[\mu_{2j}^{(2i)}\right]^{(1,k)}_{2j}\right)$,
\item and $p_2\left(\left[\mu_{2j}^{(2i)}\right]^{(1,k)}_{2j}\right)=\F\left(m_{2i-1}^{(j-1)},\,_{i}\C^{(k,1)}_{2j-1},p_2\left(\left[\la_{2j}^{(2i+1)}\right]^{(1,k)}_{2j}\right)\right)$.
\end{enumerate}
\end{enumerate}
\item 
Let $\left[\mu_{2j}^{(2)}\right]^{(1,k)}_{2j}=\left[\la_{2j}^{(2)}\right]^{(1,k)}_{2j}=\F\left(m_{2}^{(j)},\C^{(1,k)}_{2j},\left[\la_{2j}^{(3)}\right]^{(1,k)}_{2j}\right)$.
\begin{enumerate}
\item Then, $\left[\la_{2j+1}^{(2)}\right]^{(k,1)}_{2j+1}=p_1\left(\left[\mu_{2j}^{(2)}\right]^{(1,k)}_{2j}\right)$.
\item Reciprocally, $p_2\left(\left[\mu_{2j}^{(2)}\right]^{(1,k)}_{2j}\right)=\F\left(m_{1}^{(j-1)},\C^{(k,1)}_{2j-1},p_2\left(\left[\la_{2j}^{(3)}\right]^{(1,k)}_{2j}\right)\right)$.
\end{enumerate}
\item Finally, $\la_{2j}^{(1)}=\la_{2j}^{(2)}$, 
\begin{enumerate}
 \item and $\left[\la_{2j+1}^{(1)}\right]^{(k,1)}_{2j+1}=\F\left(m_{1}^{(j)},\C^{(k,1)}_{2j},\left[\la_{2j+1}^{(2)}\right]^{(k,1)}_{2j+1}\right)$.
 \item Reciprocally, $\left[\la_{2j+1}^{(1)}\right]^{(k,1)}_{2j+1}=\F\left(m_{1}^{(j)},\C^{(k,1)}_{2j},p_1\left(\left[\la_{2j}^{(1)}\right]^{(1,k)}_{2j}\right)\right)$.
\end{enumerate}
\end{enumerate}
\end{prop}
\begin{ex}
 For $k=6$ with $n=4$ and $\nu=(b_{1}^{(6,1)})^2(b_{2}^{(6,1)})^5(b_{3}^{(6,1)})^2(b_{4}^{(6,1)})^3(b_{6}^{(6,1)})^5=(1+0)^2(1+1)^5(5+6)^2(4+5)^3(15+19)^5$, we have
 $$
\begin{array}{|c||c|c||c|c||c|c||c|c|}
\hline
  &m_7^{(3)}&[\la_7^{(1)}]^{(6,1)}_7&m_5^{(2)}&[\la_5^{(1)}]^{(6,1)}_5&m_3^{(1)}&[\la_3^{(1)}]^{(6,1)}_3&m_1^{(0)}&[\la_1^{(1)}]^{(6,1)}_1\\
  \hline
  &0&(0,0,0,\textcolor{red}{0})&0&(0,0,\textcolor{red}{0})&0&(0,\textcolor{red}{0})&0&(\textcolor{red}{0})\\
  \hline
 \hline 
 i&m_i^{(3)}&[\mu_6^{(i)}]^{(1,6)}_6&m_i^{(2)}&[\mu_4^{(i)}]^{(1,6)}_4&m_i^{(1)}&[\mu_2^{(i)}]^{(1,6)}_2&m_i^{(0)}&[\mu_0^{(i)}]^{(1,6)}_0\\
 \hline\hline
  6&5&(0,0,\textcolor{red}{0},1),(0,0,\textcolor{red}{1})&0&\emptyset&0&\emptyset&0&\emptyset\\
  5&0&(0,0,\textcolor{red}{0},1),(0,0,\textcolor{red}{1})&0&\emptyset&0&\emptyset&0&\emptyset\\
  4&3&(0,\textcolor{red}{3},0,1),(0,\textcolor{red}{0},1)&1&(0,\textcolor{red}{1},0),(0,\textcolor{red}{0})&0&\emptyset&0&\emptyset\\
  3&2&(0,\textcolor{red}{1},1,1),(0,\textcolor{red}{1},1)&0&(0,\textcolor{red}{1},0),(0,\textcolor{red}{0})&0&\emptyset&0&\emptyset\\
  2&5&(\textcolor{red}{0},2,1,1),(\textcolor{red}{2},1,1)&3&(\textcolor{red}{3},1,0),(\textcolor{red}{1},0)&0&(\textcolor{red}{0},0),(\textcolor{red}{0})&0&\emptyset\\
  \hline\hline
  &m_1^{(3)}&[\la_7^{(1)}]^{(6,1)}_7&m_1^{(2)}&[\la_5^{(1)}]^{(6,1)}_5&m_1^{(1)}&[\la_3^{(1)}]^{(6,1)}_3&m_1^{(0)}&[\la_1^{(1)}]^{(6,1)}_1\\
  \hline
  &2&(\textcolor{red}{2},2,1,1)&1&(\textcolor{red}{0},2,0)&0&(0,\textcolor{red}{0})&0&(\textcolor{red}{0})\\
  \hline
\end{array}\,\,\cdot
$$
\end{ex}$$$$
The corresponding diagram is the following.
\begin{figure}[H]
\label{fig:proofodd2}
\begin{tikzpicture}[scale=1.2, every node/.style={scale=0.8}]

\draw (-1+2.5+0.8,1.5) node[left] {$\left[\la_{2j+1}^{(2i)}\right]^{(k,1)}_{2j+1}\succeq \,\, 0\cdot p_1\left(\left[\mu_{2j}^{(2i)}\right]^{(1,k)}_{2j}\right)$};
\draw (-2+2.5+0.8,3) node[left] {$\left[\la_{2j+1}^{(2i+1)}\right]^{(k,1)}_{2j+1}$};
\draw (-2+2.5+0.8,0) node[left] {$\left[\la_{2j+1}^{(2i-1)}\right]^{(k,1)}_{2j+1}$};

\draw[dotted,->] (-2.5,3)--(-0.2,3); \draw[dotted,->] (-2.5,1.5)--(-1.5,1.5); \draw[dotted,->] (-2.5,0)--(-0.2,0);
\draw (-1.4,2.85) node {$p_1$};\draw (-2,1.35) node {$p_1$}; \draw (-1.4,-.15) node {$p_1$};

\draw[dotted,<->] (-1.6+2.5,3)--(5.4-2,3); \draw[dotted,<->] (2.2,1.5)--(5.4-2,1.5);\draw[dotted,<->] (-1.6+2.5,0)--(5.4-2,0);

\draw[dashed] (0.7,2.75)--(0.9,2.5); \draw[dashed,->] (1.3,2)--(1.5,1.75);
\draw (1.5,2.25) node {$\F(m_{2i-1}^{(j-1)},\,_{i+1}\C^{(k,1)}_{2j+1})$};
\draw (0.4,2.5)--(2.65,2.5)--(2.65,2)--(0.4,2)--cycle;

\draw[dashed] (0.7-1.5,2.75-1.5)--(0.9-1.5,2.5-1.5); \draw[dashed,->] (1.3-1.5,2-1.5)--(1.5-1.5,1.75-1.5);
\draw (1.5-1.5,2.25-1.5) node {$\F(m_{2i-1}^{(j)},\,_{i}\C^{(k,1)}_{2j+1})$};
\draw (0.4-1.5,2.5-1.5)--(2.65-1.5,2.5-1.5)--(2.65-1.5,2-1.5)--(0.4-1.5,2-1.5)--cycle;

\draw (6+0.8-2,3) node[left] {$\left[\la_{2j-1}^{(2i)}\right]^{(k,1)}_{2j-2}$};
\draw (6+0.8-2,1.5) node[left] {$\left[\la_{2j-1}^{(2i-1)}\right]^{(k,1)}_{2j-2}$};
\draw (6+0.8-2,0) node[left] {$\left[\la_{2j-1}^{(2i-2)}\right]^{(k,1)}_{2j-2}$};
\draw (1.5+4.5-2,2.25) node {$\F(m_{2i-1}^{(j-1)},\,_{i}\C^{(k,1)}_{2j-1})$};
\draw (0.4+4.5-2,2.5)--(2.65+4.5-2,2.5)--(2.65+4.5-2,2)--(0.4+4.5-2,2)--cycle;
\draw(1.5+4.5-2,2.65)--(1.5+4.5-2,2.5); \draw[->] (1.5+4.5-2,2)--(1.5+4.5-2,1.75);

\draw (6+0.8-9,3) node[left] {$\left[\la_{2j}^{(2i+1)}\right]^{(1,k)}_{2j}$};
\draw (6+0.8-9,1.5) node[left] {$\left[\mu_{2j}^{(2i)}\right]^{(1,k)}_{2j}$};
\draw (1.5+4.5-9,2.25) node {$\F(m_{2i}^{(j)},\,_{+(i-1)}\C^{(1,k)}_{2j+2-2i})$};
\draw (0.2+4.5-9,2.5)--(2.75+4.5-9,2.5)--(2.75+4.5-9,2)--(0.2+4.5-9,2)--cycle;
\draw(1.5+4.5-9,2.65)--(1.5+4.5-9,2.5); \draw[->] (1.5+4.5-9,2)--(1.5+4.5-9,1.75);
\draw (6+0.8-9,1.5-1.5) node[left] {$\left[\mu_{2j}^{(2i-1)}\right]^{(1,k)}_{2j}$};
\draw (1.5+4.5-9,2.25-1.5) node {\begin{small}$\F(m_{2i-2}^{(j-1)},\,_{+(i-1)}\C^{(1,k)}_{2j+2-2i})$\end{small}};
\draw (0.2+4.5-9,2.5-1.5)--(2.75+4.5-9,2.5-1.5)--(2.75+4.5-9,2-1.5)--(0.2+4.5-9,2-1.5)--cycle;
\draw[dashed](1.5+4.5-9,2.65-1.5)--(1.5+4.5-9,2.5-1.5); \draw[dashed,->] (1.5+4.5-9,2-1.5)--(1.5+4.5-9,1.75-1.5);
\draw (6+0.8-9,-0.7) node[left] {$\left[\la_{2j}^{(2i-1)}\right]^{(1,k)}_{2j}$};
\draw[dotted,->] (1.5+4.5-9,1.3-1.5)--(1.5+4.5-9,1.1-1.5);

\draw (6+0.8,3) node[left] {$\left[\mu_{2j-2}^{(2i)}\right]^{(1,k)}_{2j}$};
\draw (6+0.8,1.5) node[left] {$\left[\la_{2j-2}^{(2i-1)}\right]^{(1,k)}_{2j}$};
\draw (6+0.8,1.5-1.5) node[left] {$\left[\mu_{2j-2}^{(2i-2)}\right]^{(1,k)}_{2j}$};
\draw (1.5+4.5,2.25-1.5) node {\begin{small}$\F(m_{2i-2}^{(j-1)},\,_{+(i-2)}\C^{(1,k)}_{2j+2-2i})$\end{small}};
\draw (0.2+4.5,2.5-1.5)--(2.75+4.5,2.5-1.5)--(2.75+4.5,2-1.5)--(0.2+4.5,2-1.5)--cycle;
\draw(1.5+4.5,2.65-1.5)--(1.5+4.5,2.5-1.5); \draw[->] (1.5+4.5,2-1.5)--(1.5+4.5,1.75-1.5);

\draw[dotted,->] (5.3,3)--(4.4,3); \draw[dotted,->] (5.3,1.5)--(4.4,1.5); \draw[dotted,->] (5.3,0)--(4.4,0);
\draw (4.9,2.85) node {$p_1$};\draw (4.9,1.35) node {$p_1$}; \draw (4.9,-.15) node {$p_1$};
\end{tikzpicture}
\caption{Relations between the insertions at step $j$ and step $j-1$}
\end{figure}
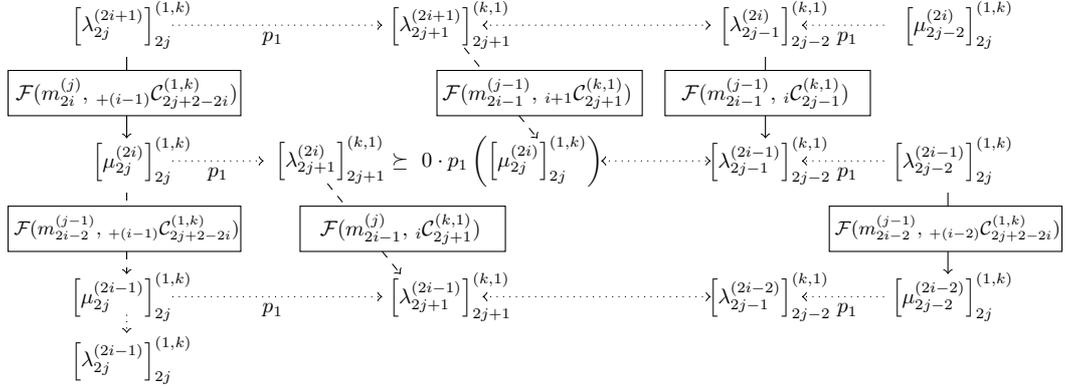
Similarly to Section \ref{part:interinfk1}, we have for $1\leq i\leq j$
\begin{equation}\label{eq:1kproofodd1}
p_2\left(\left[\la_{2j}^{(2i+1)}\right]^{(1,k)}_{2j}\right)\succeq p_1\left(\left[\mu_{2j-2}^{(2i)}\right]^{(1,k)}\right)\Longleftrightarrow  p_2\left(\left[\mu_{2j}^{(2i)}\right]^{(1,k)}_{2j}\right)\succeq \left[\la_{2j-1}^{(2i-1)}\right]^{(k,1)}_{2j-1}\,,
\end{equation}
and for $2\leq i\leq j$
\begin{align}\label{eq:1kproofodd3}
\left[\mu_{2j}^{(2i)}\right]^{(1,k)}_{2j}\succeq 0\cdot \left[\la_{2j-2}^{(2i-1)}\right]^{(1,k)}_{2j-2}&\Longleftrightarrow\left[\mu_{2j}^{(2i-1)}\right]^{(1,k)}_{2j}\succeq 0\cdot \left[\mu_{2j-2}^{(2i-2)}\right]^{(1,k)}_{2j-2}\,,\\
\Updownarrow\qquad\qquad\qquad\qquad&\qquad\qquad\qquad\qquad\qquad\Updownarrow\nonumber\\
\label{eq:1kproofodd2}
 p_2\left(\left[\mu_{2j}^{(2i)}\right]^{(1,k)}_{2j}\right)\succeq  \left[\la_{2j-1}^{(2i-1)}\right]^{(k,1)}_{2j-1}&\Longleftrightarrow p_2\left(\left[\la_{2j}^{(2i-1)}\right]^{(1,k)}_{2j}\right)\succeq p_1\left(\left[\mu_{2j-2}^{(2i-2)}\right]^{(1,k)}_{2j-2}\right)\,\cdot
\end{align}
\section{Equivalence to the Bousquet-M\'elou--Eriksson recursive bijection}
\label{sec:equiv}
\subsection{Equivalence for $\Phi^{(k,l)}_{n}$}
Define the function $\Psi^{(k,l)}_{n}$ from $\Lkl_{n}$ to $\Llk_{n-1}\times \Zz$ as follows.
\begin{enumerate}
\item For $\la=(\la_j)_{j=1}^{2n+1}\in \Lkl_{2n+1}$, set 
$$
\begin{cases}
 m= \la_{2n+1}-\left\lceil \frac{\akl{2n+1}}{\akl{2n}}\la_{2n}\right\rceil\,,\\
 \nu_{2j}=\la_{2j} \ \ \text{for} \ \ 1\leq j\leq n\\
 \nu_{2j-1}= \left\lfloor \frac{\akl{2j-1}}{\akl{2j}}\la_{2j}\right\rfloor + \left\lceil \frac{\akl{2j-1}}{\akl{2j-2}} \la_{2j-2}\right\rceil - \la_{2j-1}\ \ \text{for} \ \ 1<j\leq n \\
 \nu_{1}=\left\lfloor \frac{\akl{1}}{\akl{2}}\la_{2}\right\rfloor-\la_{1}.
\end{cases}
$$
Finally set $\nu=(\nu_j)_{j=1}^{2n}$ and $\Psi^{(k,l)}_{2n+1}(\la)=(\nu,m)$.
\item For $\la=(\la_j)_{j=1}^{2n}\in \Lkl_{2n}$, set 
$$
\begin{cases}
 m= \la_{2n}-\left\lceil \frac{\akl{2n}}{\akl{2n-1}}\la_{2n-1}\right\rceil\,,\\
 \nu_{2j-1}=\la_{2j-1} \ \ \text{for} \ \ 1\leq j\leq n\\
 \nu_{2j}= \left\lfloor \frac{\akl{2j}}{\akl{2j+1}}\la_{2j+1}\right\rfloor + \left\lceil \frac{\akl{2j}}{\akl{2j-1}} \la_{2j-1}\right\rceil - \la_{2j}\ \ \text{for} \ \ 1\leq j< n \,.
\end{cases}
$$
Finally set $\nu=(\nu_j)_{j=1}^{2n-1}$ and $\Psi^{(k,l)}_{2n}(\la)=(\nu,m)$.
\end{enumerate}
\begin{lem}\label{lem:psifinsize}
The map $\Psi^{(k,l)}_{n}$ describe a bijection from $\Lkl_{n}$ to $\Llk_{n-1}\times \Zz$. Furthermore, for $\Psi^{(k,l)}_{2n+1}(\la)=(\nu,m)$, we have $\la_e=\nu_e$ and $\la_o = m+ k\cdot\nu_e-\nu_o$, 
and for $\Psi^{(k,l)}_{2n}(\la)=(\nu,m)$, we have $\la_o=\nu_o$ and $\la_e = m+ l\cdot\nu_o-\nu_e$. 
\end{lem}
By setting 
$$\mathcal{F}_{\Lkl_{n}}(x,y)=\sum_{\la \in \Lkl_{n}} x^{|\la|_o}y^{|\la|_e}\,,$$
Lemma \ref{lem:psifinsize} yields
\begin{align}
 \label{eq:finoddtoeven} \mathcal{F}_{\Lkl_{2n}}(x,y)&=\frac{1}{1-y}\cdot \mathcal{F}_{\Lkl_{2n-1}}(xy^{l},y^{-1})\,,\\
 \label{eq:fineventoodd}\mathcal{F}_{\Lkl_{2n+1}}(x,y)&= \frac{1}{1-x}\cdot \mathcal{F}_{\Lkl_{2n}}(x^{-1},yx^{k})\,\cdot
\end{align} 
In the recursive process, as $\mathcal{F}_{\Lkl_{1}}(x,y)=1/(1-x)$, $\akl{i}=i$ for $i\in \{0,1\}$, and $l\alk{i}-\akl{i-1}=\akl{i+1}$ for $i\geq 1$, the transformations occuring in the generating functions are
\begin{align*}
 x^{\alk{i}}y^{\akl{i-1}}\mapsto x^{\alk{i}}y^{\akl{i+1}}& \Longleftrightarrow \blk{i}\mapsto\bkl{i+1} &\text{for}\quad 1\leq i\leq 2n-1\,,\\
 y = x^{\alk{0}}y^{\akl{1}}& \Longleftrightarrow \text{addition of parts }\bkl{1}  &\quad\text{in \eqref{eq:finoddtoeven}}\\
 x^{\alk{i-1}}y^{\akl{i}}\mapsto x^{\alk{i+1}}y^{\akl{i+1}}&\Longleftrightarrow \bkl{i}\mapsto\blk{i+1} &\text{for}\quad 1\leq i\leq 2n\,,\\
 x = x^{\alk{1}}y^{\akl{0}}& \Longleftrightarrow \text{addition of parts }\blk{1}&\quad\text{in \eqref{eq:fineventoodd}}\,\cdot
 \end{align*}
 Therefore, we recursively obtain Theorem \ref{theo:klseqfin} for $n\in\Zu$. Moreover, the bijections $\Psi^{(k,l)}_{n}$ induce unique recursive bijections $\Lambda^{(k,l)}_{2n}:\Bkl_{2n}\mapsto\Lkl_{2n}$ and $\Lambda^{(k,l)}_{2n-1}:\Blk_{2n}\mapsto\Lkl_{2n-1}$, such that,
for any sequence $(m_j)_{j=1}^{2n+1}$ of non-negative integers, we have  
\begin{align}
\label{eq:psiphiodd}
\Psi^{(k,l)}_{2n+1}\left(\Lambda^{(k,l)}_{2n+1}\left(\prod_{i=1}^{2n+1} (\blk{i})^{m_i}\right)\right)&= \left(\Lambda^{(k,l)}_{2n}\left(\prod_{i=2}^{2n+1} (\bkl{i-1})^{m_i}\right),m_1\right)\\
\label{eq:psiphieven}
\Psi^{(k,l)}_{2n}\left(\Lambda^{(k,l)}_{2n}\left(\prod_{i=1}^{2n} (\bkl{i})^{m_i}\right)\right)&= \left(\Lambda^{(k,l)}_{2n-1}\left(\prod_{i=2}^{2n} (\blk{i-1})^{m_i}\right),m_1\right)
\end{align}
The equivalence is then the following.
\begin{theo}\label{theo:psiphifin}
For $n\in\Zu$, $\Phi^{(k,l)}_{n}=\Lambda^{(k,l)}_{n}$.
\end{theo}
Let $(m_i)_{i=1}^{2n+1} \in \Zz^{2n+1}$. To prove that, it suffices to show \eqref{eq:psiphiodd} and \eqref{eq:psiphieven} with $\Phi^{(k,l)}_{n}$ instead of $\Lambda^{(k,l)}_{n}$.
\begin{enumerate}
 \item Set $\la=(\la_j^{(1)})_{j=1}^{2n+1}=\Phi^{(k,l)}_{2n+1}(\prod_{i=1}^{2n+1}(\blk{i})^{m_i})$, and $\Psi^{(k,l)}_{2n+1}(\la)=(\nu,m)$. We here refer to the notations of Section \ref{part:interodd} for $\la$ with $\la_{j}^{(i)}$ and $m_i^{(j)}$ (and $\mu_{2j+1}^{(i)}$ for the cases $(k,1)$ and $(1,k)$), while using for $\nu$ the notations $\nu_{j}^{(i)}$ and $n_i^{(j)}$ instead of $\la_{j}^{(i)}$ and $m_i^{(j)}$ (and $\eta_{2j+1}^{(i)}$ instead of $\mu_{2j+1}^{(i)}$ for the cases $(k,1)$ and $(1,k)$) in Section \ref{part:intereven}. For the proof of \eqref{eq:psiphiodd}, we show that, for $1\leq j\leq n$ and $1\leq i\leq 2j$, $\nu_{2j}^{(i)}=\la_{2j}^{(i+1)}$ and $n_i^{(j)}=m_{i+1}^{(j)}$.
 \item Set $\la=(\la_j^{(1)})_{j=1}^{2n}=\Phi^{(k,l)}_{2n}(\prod_{i=1}^{2n}(\bkl{i})^{m_i})$, and $\Psi^{(k,l)}_{2n}(\la)=(\nu,m)$. We here refer to the notations of Section \ref{part:intereven} for $\la$ with $\la_{j}^{(i)}$ and $m_i^{(j)}$ (and $\mu_{2j+1}^{(i)}$ for the cases $(k,1)$ and $(1,k)$), while using for $\nu$ the notations $\nu_{j}^{(i)}$ and $n_i^{(j)}$ instead of $\la_{j}^{(i)}$ and $m_i^{(j)}$ (and $\eta_{2j+1}^{(i)}$ instead of $\mu_{2j+1}^{(i)}$ for the cases $(k,1)$ and $(1,k)$) in Section \ref{part:interodd}. For the proof of \eqref{eq:psiphieven}, we show that, for $1\leq j\leq n$ and $1\leq i\leq 2j$, $\nu_{2j-1}^{(i)}=\la_{2j-1}^{(i+1)}$ and $n_i^{(j-1)}=m_{i+1}^{(j)}$.
\end{enumerate}

\subsubsection{The case $k,l\geq 2$}
\begin{enumerate}
\item For $1\leq j\leq n$, we have $\nu_{2j}^{(1)}=\la_{2j}^{(2)}$ by fact $(3)$ of Proposition \ref{prop:welldefodd}. Also, by \eqref{eq:klrateven0} and \eqref{eq:klbelongsjiodd}, 
$$
\left[\la_{2j+1}^{(2)}\right]^{(l,k)}_{2j+1} = 0\cdot\left[\la_{2j}^{(2)}\right]^{(k,l)}_{2j}= \left[\left\lceil \frac{\alk{2j+1}}{\akl{2j}}\la_{2j}^{(2)}\right\rceil\right]^{(l,k)}_{2j+1}
\Longleftrightarrow \quad \la_{2j+1}^{(2)}= \left\lceil \frac{\alk{2j+1}}{\akl{2j}}\la_{2j}^{(2)}\right\rceil\,\cdot
$$
Hence, using fact $(3.a)$ of Proposition \ref{prop:welldefodd}, 
$$\la_{2j+1}^{(1)}-\left\lceil \frac{\alk{2j+1}}{\akl{2j}}\la_{2j}^{(1)}\right\rceil = \la_{2j+1}^{(1)}-\la_{2j+1}^{(2)} = m_1^{(j)}\,\cdot$$
Thus, $m=m_1^{(n)}=m_1$, and since $\la_{1}^{(1)}=m_1^{(0)}$, we have for all $1\leq j\leq n$, 
$$\nu_{2j-1}^{(2)}=\nu_{2j-1}^{(1)}=\left\lfloor \frac{\alk{2j-1}}{\akl{2j}}\la_{2j}^{(2)}\right\rfloor-m_1^{(j-1)}\,\cdot$$
By \eqref{eq:klrateven0}, this is equivalent to saying that $0\cdot\left[\nu_{2j-1}^{(2)}\right]^{(l,k)}_{2j-1}$ is the $(m_1^{(j-1)})^{th}$ sequence that precedes $_2\left[\la_{2j}^{(2)}\right]^{(k,l)}_{2j}$ in $_2\C^{(k,l)}_{2j}$. Using \eqref{eq:klbelongsjieven} and fact $(2.b)$ of Proposition \ref{prop:welldefodd} with $i=2$, we then have to $\left[\nu_{2j}^{(2)}\right]^{(k,l)}_{2j}=\left[\la_{2j}^{(3)}\right]^{(k,l)}_{2j}\in \, _2\C^{(k,l)}_{2j}$. Therefore, since $\nu_{2j}^{(1)}=\la_{2j}^{(2)}$, by fact $(2.a)$ of Proposition \ref{prop:welldefodd} and fact $(3,a)$ Proposition \ref{prop:welldefeven}, $n_1^{(j)}=m_2^{(j)}$. 
Moreover, we have for $j=1$ 
$$n_{2}^{(1)}=\frac{\nu_{2}^{(2)}}{\akl{2}}=\frac{\la_{2}^{(3)}}{\akl{2}}=m_{3}^{(1)}\,\cdot$$
Assuming that for some $n-1\geq j>1$, $n_i^{(j-1)}=m_{i+1}^{(j-1)}$ and $\left[\nu_{2j-2}^{(i+1)}\right]_{2j-2}^{(k,l)}= \left[\la_{2j-2}^{(i+2)}\right]^{(k,l)}_{2j-2}\in \, _{i+1}\C^{(k,l)}_{2j-2}$ for $1\leq i \leq 2j-2$. Then, using fact $(2)$ of Propositions \ref{prop:welldefeven} and \ref{prop:welldefodd}, we recursively retrieve that $n_i^{(j)}=m_{i+1}^{(j)}$  and $\left[\nu_{2j}^{(i+1)}\right]_{2j}^{(k,l)}= \left[\la_{2j}^{(i+2)}\right]^{(k,l)}_{2j}\in \, _{i+1}\C^{(k,l)}_{2j}$ for $2\leq i \leq 2j-1$. Finally, 
$$n_{2j}^{(j)}=\frac{\nu_{2j}^{(2j)}}{\akl{2j}}=\frac{\la_{2j}^{(2j+1)}}{\akl{2j}}=m_{2j+1}^{(j)}\,,$$
and $(0)_{u=1}^{2j}=\left[\nu_{2j}^{(2j+1)}\right]_{2j}^{(k,l)}= \left[\la_{2j}^{(2j+2)}\right]^{(k,l)}_{2j}\in \, _{2j+1}\C^{(k,l)}_{2j}$, and the heredity on $j$ holds. 
\item By definition and fact $(3)$ of Propositions \ref{prop:welldefeven}, $\nu_{2j-1}^{(1)}=\la_{2j-1}^{(2)}$ for $n\geq j\geq 1$. In particular for $j=1$, 
$$n_1^{(0)}=\nu_{1}^{(1)}=\la_{1}^{(2)}= m_2^{(1)}\,\cdot$$
Similarly to $(1)$, by using \eqref{eq:klrateven0}, \eqref{eq:klbelongsjieven} and fact $(3)$ of Propositions \ref{prop:welldefeven} and \ref{prop:welldefodd}, $m=m_1^{(n)}=m_1$, and for all $2\leq j\leq n$, 
$$\nu_{2j-2}^{(2)}=\nu_{2j-2}^{(1)}=\left\lfloor \frac{\akl{2j-2}}{\alk{2j-1}}\la_{2j-1}^{(2)}\right\rfloor-m_1^{(j-1)}\,\cdot$$
By \eqref{eq:klrateven0}, \eqref{eq:klbelongsjiodd} and fact $(2.b)$ of Proposition \ref{prop:welldefeven}, $\left[\nu_{2j-1}^{(2)}\right]^{(l,k)}_{2j-1}=\left[\la_{2j-1}^{(3)}\right]^{(l,k)}_{2j-1}\in \, _2\C^{(k,l)}_{2j-1}$, and since $\left[\nu_{2j-1}^{(1)}\right]^{(l,k)}_{2j-1}=\left[\la_{2j-1}^{(2)}\right]^{(l,k)}_{2j-1}\in \, _2\C^{(k,l)}_{2j-1}$, fact $(2.a)$ of Proposition \ref{prop:welldefeven} and fact $(3,a)$ Proposition \ref{prop:welldefodd} yield $n_1^{(j-1)}=m_2^{(j)}$.\\
Assume that for some $n-1\geq j\geq 2$, $n_i^{(j-2)}=m_{i+1}^{(j-1)}$ and $\left[\nu_{2j-3}^{(i+1)}\right]_{2j-3}^{(l,k)}= \left[\la_{2j-3}^{(i+2)}\right]^{(l,k)}_{2j-3}\in \, _{i+1}\C^{(l,k)}_{2j-3}$ for $1\leq i \leq 2j-3$. Then, using fact $(2)$ of Propositions \ref{prop:welldefeven} and \ref{prop:welldefodd}, we recursively retrieve that $n_i^{(j-1)}=m_{i+1}^{(j)}$  and $\left[\nu_{2j-1}^{(i+1)}\right]_{2j-1}^{(l,k)}= \left[\la_{2j-1}^{(i+2)}\right]^{(l,k)}_{2j-1}\in \, _{i+1}\C^{(l,k)}_{2j-1}$ for $2\leq i \leq 2j-2$. Finally, 
$$n_{2j-1}^{(j-1)}=\frac{\nu_{2j-1}^{(2j-1)}}{\alk{2j-1}}=\frac{\la_{2j-1}^{(2j)}}{\alk{2j-1}}=m_{2j}^{(j)}\,,$$
and $(0)_{u=1}^{2j-1}=\left[\nu_{2j-1}^{(2j)}\right]_{2j-1}^{(l,k)}= \left[\la_{2j-1}^{(2j+1)}\right]^{(l,k)}_{2j-1}\in \, _{2j}\C^{(l,k)}_{2j-1}$, and the heredity on $j$ holds. 
\end{enumerate}
\subsubsection{The case $(k,1)$}
\begin{enumerate}
\item By \eqref{eq:k1belongsjiodd}, and \eqref{eq:k1rat0}, 
for $1\leq j\leq n$,
\begin{align*}
p_1\left(\left[\la_{2j+1}^{(2)}\right]^{(1,k)}_{2j+1} \right)=0\cdot p_2\left(\left[\la_{2j+1}^{(2)}\right]^{(1,k)}_{2j+1} \right)&= 0\cdot\left[\la_{2j}^{(2)}\right]^{(k,1)}_{2j}= \left[\left\lceil \frac{\ak{2j+2}}{\ak{2j}}\la_{2j}^{(2)}\right\rceil\right]^{(k,1)}_{2j+2}\\
\Longleftrightarrow\quad \la_{2j+1}^{(2)}= \left\lceil \frac{\ak{2j+2}}{\ak{2j}}\la_{2j}^{(2)}\right\rceil+ \la_{2j}^{(2)}&=\left\lceil \frac{\al{2j+1}}{\ak{2j}}\la_{2j}^{(2)}\right\rceil\,\cdot
\end{align*}
Hence, using fact $(3.a)$ of Proposition \ref{prop:welldefodd1}, 
$$\la_{2j+1}^{(1)}-\left\lceil \frac{\al{2j+1}}{\ak{2j}}\la_{2j}^{(1)}\right\rceil = \la_{2j+1}^{(1)}-\la_{2j+1}^{(2)} = m_1^{(j)}\,\cdot$$
Thus, $m=m_1^{(n)}=m_1$, and since $\la_{1}^{(1)}=m_1^{(0)}$, we have for all $1\leq j\leq n$, 
$$\nu_{2j-1}^{(2)}=\left\lfloor \frac{\al{2j-1}}{\ak{2j}}\la_{2j}^{(2)}\right\rfloor-m_1^{(j-1)}=\la_{2j}^{(2)}+\left\lfloor \frac{\ak{2j-2}}{\ak{2j}}\la_{2j}^{(2)}\right\rfloor-m_1^{(j-1)}\,\cdot$$
By \eqref{eq:k1rat0}, this is equivalent to saying that $0\cdot\left[\eta_{2j-1}^{(2)}\right]^{(1,k)}_{2j-1}=0\cdot\left[\nu_{2j-1}^{(2)}\right]^{(1,k)}_{2j-1}$ is the $(m_1^{(j-1)})^{th}$ sequence that precedes $\left[\mu_{2j+1}^{(2)}\right]^{(1,k)}_{2j+1}$ in $_{+1}\C^{(1,k)}_{2j-1}$, and using fact $(2.b.i)$ of Proposition \ref{prop:welldefodd1}, that sequence is equal to $\left[\mu_{2j+1}^{(3)}\right]^{(1,k)}_{2j+1}$. Therefore,  by applying $p_1^-$,  \eqref{eq:k1belongsjieven}, fact $(2.a.i)$ of Proposition \ref{prop:welldefodd1} and \eqref{eq:k1belongsjiodd}, we obtain that $\nu_{2j}^{(2)}=\la_{2j}^{(3)}$. Since $\nu_{2j}^{(1)}=\la_{2j}^{(2)}$, facts $(3.a)$ of Proposition \ref{prop:welldefeven1} and $(2.a.ii)$ of Proposition \ref{prop:welldefodd1} yield $n_1^{(j)}=m_2^{(j)}$. Moreover, we have for $j=1$ 
$$n_{2}^{(1)}=\frac{\nu_{2}^{(2)}}{\akl{2}}=\frac{\la_{2}^{(3)}}{\akl{2}}=m_{3}^{(1)}\,\cdot$$
Assuming that for some $n-1\geq j>1$, $n_i^{(j-1)}=m_{i+1}^{(j-1)}$ and $\left[\nu_{2j-2}^{(i+1)}\right]_{2j-2}^{(k,1)}= \left[\la_{2j-2}^{(i+2)}\right]^{(k,1)}_{2j-2}\in \, _{\lceil (i+1)/2\rceil }\C^{(k,1)}_{2j-2}$ for $1\leq i \leq 2j-2$.
As $0\cdot\left[\eta_{2j-1}^{(2)}\right]^{(1,k)}_{2j-1}=\left[\mu_{2j+1}^{(3)}\right]^{(1,k)}_{2j+1}$, $\nu_{2j}^{(1)}=\la_{2j}^{(2)}$  and $n_2^{(j-1)}=m_3^{(j-1)}$, by using facts $(2.b.ii)$ of Propositions \ref{prop:welldefeven1} and \ref{prop:welldefodd1}, we obtain $\nu_{2j}^{(3)}=\la_{2j}^{(4)}$, and then $0\cdot\left[\nu_{2j-1}^{(3)}\right]^{(1,k)}_{2j-1}=\left[\la_{2j+1}^{(4)}\right]^{(1,k)}_{2j+1}$. We thus retrieve by facts $(2)$ that $n_2^{(j)}=m_3^{(j)}$. Finally, by definitions in facts $(2.b)$, $0\cdot\left[\eta_{2j-1}^{(3)}\right]^{(1,k)}_{2j-1}=\left[\mu_{2j+1}^{(4)}\right]^{(1,k)}_{2j+1}$, and using facts $(2.b.i)$, we have $0\cdot\left[\eta_{2j-1}^{(4)}\right]^{(1,k)}_{2j-1}=\left[\mu_{2j+1}^{(5)}\right]^{(1,k)}_{2j+1}$. Then $\nu_{2j}^{(4)}=\la_{2j}^{(5)}$ and facts $(2.a.ii)$ imply that $n_3^{(j)}=m_4^{(j)}$. Iteration of this process leads to $n_i^{(j)}=m_{i+1}^{(j)}$ and $\left[\nu_{2j}^{(i+1)}\right]_{2j}^{(k,1)}= \left[\la_{2j}^{(i+2)}\right]^{(k,1)}_{2j}\in \, _{\lceil (i+1)/2\rceil }\C^{(k,1)}_{2j}$ for $2\leq i \leq 2j-1$. Finally, 
$$n_{2j}^{(j)}= \frac{\nu_{2j}^{(2j)}}{\ak{2j}} =\frac{\la_{2j}^{(2j+1)}}{\ak{2j}}=m_{2j+1}^{(j)}\,,$$
and $(0)_{u=1}^{2j}=\left[\nu_{2j}^{(2j+1)}\right]_{2j}^{(k,l)}= \left[\la_{2j}^{(2j+2)}\right]^{(k,l)}_{2j}\in \, _{2j+1}\C^{(k,l)}_{2j}$, and the heredity on $j$ holds. 
\item By definition and fact $(3)$ of Proposition \ref{prop:welldefeven1}, $\nu_{2j-1}^{(1)}=\la_{2j-1}^{(2)}$ for $n\geq j\geq 1$. In particular for $j=1$, 
$$n_1^{(0)}=\nu_{1}^{(1)}=\la_{1}^{(2)}= m_2^{(1)}\,\cdot$$
Moreover, by fact $(4)$ of Proposition \ref{prop:welldefeven1} and definition of $p_1$, for $1\leq j\leq n$,
$$\la_{2j}^{(1)}-\left\lceil\frac{\ak{2j}}{\al{2j-1}}\la_{2j-1}^{(2)}\right\rceil=m_1^{(j)}\,\cdot$$
Hence, $m=m_1^{(n)}=m_1$, and for all $2\leq j\leq n$, 
$$\nu_{2j-2}^{(2)}=\nu_{2j-2}^{(1)}=\left\lfloor \frac{\ak{2j-2}}{\al{2j-1}}\la_{2j-1}^{(2)}\right\rfloor-m_1^{(j-1)}\,\cdot$$
Therefore, $p_2(\left[\nu_{2j-1}^{(2)}\right]^{(1,k)}_{2j-1})=\left[\nu_{2j-2}^{(2)}\right]^{(k,1)}_{2j-2}$ is the $(m_1^{(j-1)})$ sequence that precedes $p_2(\left[\mu_{2j-1}^{(2)}\right]^{(1,k)}_{2j-1})$ in $\C^{(k,1)}_{2j-2}$, so that by fact $(3.b)$ of Proposition \ref{prop:welldefeven1}, \eqref{eq:k1belongsjieven} and \eqref{eq:k1belongsjiodd}, $\nu_{2j-1}^{(2)}=\la_{2j-1}^{(3)}$. Then, as $\nu_{2j-1}^{(2)}=\la_{2j-1}^{(3)}$, by facts $(3.a)$ of Propositions \ref{prop:welldefeven1} and \ref{prop:welldefodd1}, $n_{1}^{(j-1)}=m_2^{(j)}$. \\
Assume that for some $n-1\geq j\geq 2$, $n_i^{(j-2)}=m_{i+1}^{(j-1)}$ and $\left[\nu_{2j-3}^{(i+1)}\right]_{2j-3}^{(1,k)}= \left[\la_{2j-3}^{(i+2)}\right]^{(1,k)}_{2j-3} \in \, _{+(\lceil(i-1)/2\rceil)}\C^{(1,k)}_{2j-3-2\lceil(i-1)/2\rceil}$ for $1\leq i \leq 2j-3$. As, $\nu_{2j-1}^{(2)}=\la_{2j-1}^{(3)}$, we have by definitions in fact $(2.b)$ that $\left[\eta_{2j-1}^{(2)}\right]_{2j-1}^{(1,k)}=\left[\mu_{2j-1}^{(3)}\right]_{2j-1}^{(1,k)} \in \, _{+(1)}\C^{(1,k)}_{2j-3}$, and since $n_1^{(j-2)}=m_{2}^{(j-1)}$, by fact $(2.b.i)$, $\left[\eta_{2j-1}^{(3)}\right]_{2j-1}^{(1,k)}=\left[\mu_{2j-1}^{(4)}\right]_{2j-1}^{(1,k)}\in \, _{+(1)}\C^{(1,k)}_{2j-3}$. Hence, $\nu_{2j-1}^{(3)}=\la_{2j-1}^{(4)}$, and we recover by facts $(2.a.ii)$ that $n_2^{(j-1)}=m_{3}^{(j)}$. 
Moreover, since $n_2^{(j-2)}=m_{3}^{(j-1)}$, by using facts $(2.b.ii)$, \eqref{eq:k1belongsjieven} and \eqref{eq:k1belongsjieven}, we obtain $\nu_{2j-1}^{(4)}=\la_{2j-1}^{(5)}$, and for $j>2$, using facts $(2)$, we have $n_3^{(j-1)}=m_{4}^{(j)}$. Iterating this process, we have 
$n_{i}^{(j-1)}=m_{i+1}^{(j)}$ and $\nu_{2j-1}^{(i+1)}=\la_{2j-1}^{(i+2)}$ for $2\leq i\leq 2j-2$. Finally, 
$$n_{2j-1}^{(j-1)}=\frac{\nu_{2j-1}^{(2j-1)}}{\al{2j-1}}=\frac{\la_{2j-1}^{(2j)}}{\al{2j-1}}=m_{2j}^{(j)}\,,$$
and $(0)_{u=1}^{2j-1}=\left[\nu_{2j-1}^{(2j)}\right]_{2j-1}^{(l,k)}= \left[\la_{2j-1}^{(2j+1)}\right]^{(l,k)}_{2j-1}\in \, _{2j}\C^{(l,k)}_{2j-1}$, and the heredity on $j$ holds. 
\end{enumerate}
\subsubsection{The case $(1,k)$}
The proofs of $(1)$ and $(2)$ are respectively similar to the proofs of $(2)$ and $(1)$ in the case $(k,1)$ using Propositions \ref{prop:welldefeven2} and \ref{prop:welldefodd2}.
\subsection{Equivalence for $\Phi^{(k,l)}$}
Define the function $\Psi^{(k,l)}$ from $\Lkl$ to $\Llk\times \Zz$ as follows. For $\la=(\la_j)_{j=1}^{2t}\in \Lkl$, if $t=1$, set $\Psi^{(k,l)}(\la)=((0,0),\la_1)$. Otherwise, set
$$
\begin{cases}
 m= \la_1-\left\lfloor s_0^{(k,l)} \la_2\right\rfloor-1\,,\\
 \nu_{2j-1}=\la_{2j} \ \ \text{for} \ \ 1\leq j\leq t-1\\
 \nu_{2j}= \left\lceil  s_2^{(k,l)}\la_{2j}\right\rceil + \left\lfloor s_0^{(k,l)} \la_{2j+2}\right\rfloor - \la_{2j+1}\ \ \text{for} \ \ 1\leq j<t-1 \\
 \nu_{2t-2}=\left\lceil  s_2^{(k,l)}\la_{2t-2}\right\rceil-1-\la_{2t-1}.
\end{cases}
$$
If $\nu_{2t-2}=0$, define $\nu=(\nu_j)_{j=1}^{2t-2}$, otherwise, define $\nu_{2t-1}=\nu_{2t}=0$, and $\nu=(\nu_j)_{j=1}^{2t}$. Finally set 
$\Psi^{(k,l)}(\la)=(\nu,m)$. Observe that in the case where $\nu\neq (0,0)$, it has the same number of parts as $\la$ if and only if its two last part equals $0$, otherwise $\nu$ has two less parts than $\la$.\\
We now defined the inverse map by the following. Let $m\in \Zz$ and $\nu=(\nu_j)_{j=1}^{2t}$. If $\nu=(0,0)$, set 
$\la=(m,0)$. Otherwise, set $t'=t$ if $\nu_{2t-1}>0$, and $t'=t-1$ if $\nu_{2t-1}=0$. Note that $t'\geq 1$. Then set 
$$
\begin{cases}
\la_{2j}=\nu_{2j-1} \ \ \text{for} \ \ 1\leq j\leq t'\\
 \la_1= \left\lfloor s_0^{(k,l)} \nu_1\right\rfloor+1+m\,,\\
 \la_{2j+1}= \left\lceil  s_2^{(k,l)}\nu_{2j-1}\right\rceil + \left\lfloor s_0^{(k,l)} \nu_{2j+1}\right\rfloor - \nu_{2j}\ \ \text{for} \ \ 1\leq j<t'\\
 \la_{2t'+1}=\left\lceil  s_2^{(k,l)}\nu_{2t'-1}\right\rceil-1-\nu_{2t'}.
\end{cases}
$$
We finally set $\la_{2t'+2}=0$ and $\la=(\la_j)_{j=1}^{2t'+2}$, and one can check that $\la$ is the reverse image of $(\nu,m)$ by $\Psi^{(k,l)}$.
\begin{lem}\label{lem:psiinfsize}
The map $\Psi^{(k,l)}$ describes a bijection from $\Lkl$ to $\Llk\times \Zz$, and for $\Psi^{(k,l)}(\la)=(\nu,m)$, we have $\la_e=\nu_o$ and $\la_o = m+ l\cdot\nu_o-\nu_e$.
\end{lem}
By setting 
$$\mathcal{F}_{\Lkl}(x,y)=\sum_{\la \in \Lkl_{n}} x^{|\la|_o}y^{|\la|_e}\,,$$
Lemma \ref{lem:psiinfsize} yields
\begin{align}
 \label{eq:infkltolk} \mathcal{F}_{\Lkl}(x,y)&=\frac{1}{1-x}\cdot \mathcal{F}_{\Llk}(x^{l}y,x^{-1})\,\cdot
\end{align} 
In the recursive process, as $\akl{i}=i$ for $i\in \{0,1\}$, and $l\alk{i}-\akl{i-1}=\akl{i+1}$ for $i\geq 1$, the transformations occurring in the generating functions are
\begin{align*}
 x^{\alk{i}}y^{\akl{i-1}}\mapsto x^{\akl{i+1}}y^{\alk{i}}& \Longleftrightarrow \blk{i}\mapsto\bkl{i+1} &\text{for}\quad 1\leq i\,,\\
 x = x^{\akl{1}}y^{\alk{0}}& \Longleftrightarrow \text{addition of parts }\bkl{1}  &\quad\text{in \eqref{eq:infkltolk}}\,\cdot
 \end{align*}
 Then, by iterating \eqref{eq:infkltolk}, for $(k,l)$ then $(l,k)$, we have for $n\in \Zu$
 \begin{align*} \mathcal{F}_{\Lkl}(x,y)&=\left(\prod_{i=1}^{2n-1}\frac{1}{1-x^{\akl{i}}y^{\alk{i-1}}}\right)\cdot \mathcal{F}_{\Llk}\left(x^{\akl{2n}}y^{\alk{2n-1}},x^{-\akl{2n-1}}y^{-\alk{2n-2}}\right)\\
 &=\left(\prod_{i=1}^{2n}\frac{1}{1-x^{\akl{i}}y^{\alk{i-1}}}\right)\cdot \mathcal{F}_{\Lkl}\left(x^{\akl{2n+1}}y^{\alk{2n}},x^{-\akl{2n}}y^{-\alk{2n-1}}\right)
 \end{align*}
We then obtain Theorem \ref{theo:klseqinf} when $n$ tends to $\infty$.
 The analogous result for Theorem \ref{theo:psiphifin} is the following.
\begin{theo}\label{theo:psiphiinf}
For sequence $(m_j)_{j\geq 1}$ of non-negative integers with finitely many positive, we have
\begin{equation}\label{eq:psiphiinf}
\Psi^{(k,l)}\left(\Phi^{(k,l)}\left(\prod_{i\geq 1} (\bkl{i})^{m_i}\right)\right)= \left(\Phi^{(l,k)}\left(\prod_{i\geq 2} (\blk{i-1})^{m_i}\right),m_1\right)
\end{equation} 
\end{theo}
\begin{proof}
Let $(m_j)_{j\geq 1}$ be a sequence of non-negative integers with finitely many positive integers, and set $$\la = \Phi^{(k,l)}\left(\prod_{i\geq 1} (\bkl{i})^{m_i}\right)=(\la_j^{(1)})_{j=1}^{2t}$$ as defined in Section \ref{part:interinf}. Note that $t=1$ if and only if $m_i=0$ for $i\geq 2$, in which case $\la = (m_1,0)$. In that case, as $\nu=(0,0)$, \eqref{eq:psiphiinf} follows.
Suppose now that $t\geq 2$, i.e there exists $i\geq 2$  such that $m_i>0$. Let $\Psi^{(k,l)}=(\nu,m)$ and write $\nu=(\nu_j)_{j=1}^{2t}$ with $\nu_1>0=\nu_{2t-1}=\nu_{2t}$ (with eventually $\nu_{2t-2}=0$). We here refer to the notations of Section \ref{part:interinf} for $\la$ with $\la_{j}^{(i)}$ and $m_i^{(j)}$ (and $\mu_{j}^{(i)}$ for the cases $(k,1)$ and $(1,k)$), while using for $\nu$ the notations $\nu_{j}^{(i)}$ and $n_i^{(j)}$ instead of $\la_{j}^{(i)}$ and $m_i^{(j)}$ (and $\eta_{j}^{(i)}$ instead of $\mu_{j}^{(i)}$ for the cases $(k,1)$ and $(1,k)$). For the proof of \eqref{eq:psiphiinf}, we show that, for $1\leq j\leq t$ and $1\leq i$, $\nu_{2j-1}^{(i)}=\la_{2j}^{(i+1)}$ and $n_i^{(j)}=m_{i+1}^{(j)}$. This already holds for $j=t$, as $0=n_i^{(t)}=m_{i+1}^{(t)}$ for $i\geq 1$.
The proof is the same as the proof of \eqref{eq:psiphieven} for all the cases. We here illustrate the case $k,l\geq 2$.\\\\
By fact $(2)$ of Proposition \ref{prop:welldefinf}, we then have, for $1\leq j\leq t$, that $\la_{(2j)}^{(2)}=\nu_{2j-1}^{(1)}$.
Also, by \eqref{eq:klratinf0}, $$\left[\lfloor s_0^{(k,l)} \la_{2j}^{(2)}\rfloor+1\right]^{(k,l)}=0\cdot \left[\la_{2j}^{(2)}\right]^{(l,k)}=\left[\la_{2j-1}^{(2)}\right]^{(k,l)}\,\cdot$$ 
Hence, by fact $(2.a)$ of Proposition \ref{prop:welldefinf},
$$\la_{2j-1}^{(1)}-\lfloor s_0^{(k,l)} \la_{2j}^{(2)}\rfloor-1=\la_{2j-1}^{(1)}-\la_{2j-1}^{(2)}=m_1^{(j)}\,\cdot$$
Thus, $m=m_1^{(n)}=m_1$, and since $\la_{2t-1}^{(1)}=m_1^{(t)}$, we have for all $1\leq j\leq t-1$, 
$$\nu_{2j}^{(2)}=\nu_{2j}^{(1)}=\left\lceil \skl{2}\la_{2j}^{(2)}\right\rceil-1-m_1^{(j-1)}\,\cdot$$
By \eqref{eq:klratinf0}, this is equivalent to saying that $0\cdot\left[\nu_{2j}^{(2)}\right]^{(k,l)}$ is the $(m_1^{(j-1)})^{th}$ sequence that precedes $_2\left[\la_{2j}^{(2)}\right]^{(l,k)}$ in $_2\C^{(l,k)}$. Using \eqref{eq:klbelongsji} and fact $(1.b)$ of Proposition \ref{prop:welldefinf}, we then have to $\left[\nu_{2j}^{(2)}\right]^{(l,k)}=\left[\la_{2j}^{(3)}\right]^{(l,k)}\in \, _2\C^{(l,k)}$. Therefore, since $\nu_{2j-1}^{(1)}=\la_{2j}^{(2)}$, by fact $(1.a)$ and fact $(2.a)$ Proposition \ref{prop:welldefinf}, $n_1^{(j)}=m_2^{(j)}$.\\
Assume now that, for some $1\leq j<t$, $\nu_{2j+1}^{(i)}=\la_{2j+2}^{(i+1)}$ and $n_i^{(j+1)}=m_{i+1}^{(j+1)}$ for $i\geq 1$. 
We already have $\nu_{2j-1}^{(1)}=\la_{2j}^{(2)}$, $\nu_{2j-1}^{(2)}=\la_{2j}^{(3)}$ and $n_1^{(j)}=m_{2}^{(j)}$. As $n_1^{(j-1)}=m_{2}^{(j-1)}$, we thus have by fact $(1.b)$ that $\nu_{2j-1}^{(3)}=\la_{2j}^{(4)}$, and by fact $(1.a)$, $n_2^{(j)}=m_{3}^{(j)}$. By iterating the process, the heredity on $j$ holds. In particular, for $j=1$, we retrieve \eqref{eq:psiphiinf}.
\end{proof}
\[\]
\appendix
\section{Proof of technical lemmas}
\subsection{Proof of Lemma \ref{lem:divide}}
Let $\frac{1}{2}\geq x<1$
\begin{enumerate}
 \item For $a\in \Z$,
 \begin{align*}
0<(1-x)a-xb\leq 1 &\Longleftrightarrow \frac{xb}{1-x}<a\leq 1+\frac{x(b+1)}{1-x} \\
&\Longleftrightarrow 1+\left\lfloor \frac{xb}{1-x}\right\rfloor\leq a\leq 1+\left\lfloor \frac{x(b+1)}{1-x} \right\rfloor\,,
 \end{align*}
 and \eqref{eq:divide+} holds. The proof of \eqref{eq:divide++} comes from the equivalence
 $$x<(1-x)a-xb\leq 1\Longleftrightarrow 0<a-\frac{x(b+1)}{1-x}\leq 1$$
 Now suppose that $(a,b)\in S_x^+$. This is equivalent to $1+\left\lfloor \frac{xb}{1-x}\right\rfloor\leq a\leq 1+\left\lfloor \frac{x(b+1)}{1-x} \right\rfloor$. Then, $(a+1,b)\in S_x^+$ if and only if $1+\left\lfloor \frac{xb}{1-x}\right\rfloor< a+1\leq 1+\left\lfloor \frac{x(b+1)}{1-x} \right\rfloor$, so that $1+\left\lfloor \frac{xb}{1-x}\right\rfloor\leq  a < 1+\left\lfloor \frac{x(b+1)}{1-x} \right\rfloor$. Moreover, $(a,b+1)\in S_x^+$ if and only if $1+\left\lfloor \frac{x(b+1)}{1-x} \right\rfloor\leq a \leq 1+\left\lfloor \frac{x(b+1)}{1-x} \right\rfloor$, so that $a = 1+\left\lfloor \frac{x(b+1)}{1-x} \right\rfloor$. 
 The fact that for $(a,b)\in S_x^+$, either $(a+1,b)$ or $(a,b+1)$ belongs  to $S_x^+$ and not both at the same time imply that $(a,b)\mapsto a+b$ is a bijection from $S_x^+$ to $\Z$. Furthermore, 
  \begin{align*}
0<(1-x)a-xb\leq 1 &\Longleftrightarrow b<(1-x)(a+b)\leq 1+b \\
&\Longleftrightarrow b=\left\lceil (1-x)(a+b)\right\rceil-1\,,\\\\
0<(1-x)a-xb\leq 1&\Longleftrightarrow a-1\leq x(a+b)< a \,\\
&\Longleftrightarrow a=\left\lfloor x(a+b)\right\rfloor+1\cdot
\end{align*}
 \item For $a\in \Z$,
 \begin{align*}
0\leq (1-x)a-xb< 1 &\Longleftrightarrow \frac{xb}{1-x}\leq a< 1+\frac{x(b+1)}{1-x} \\
&\Longleftrightarrow \left\lceil \frac{xb}{1-x}\right\rceil\leq a\leq \left\lceil \frac{x(b+1)}{1-x} \right\rceil\,,
 \end{align*}
 and \eqref{eq:divide-} holds. The proof of \eqref{eq:divide++} comes from the equivalence
 $$x\leq (1-x)a-xb\leq 1\Longleftrightarrow 0\leq a-\frac{x(b+1)}{1-x}<1$$  
  Furthermore, 
  \begin{align*}
0\leq (1-x)a-xb< 1 &\Longleftrightarrow b\leq (1-x)(a+b)< 1+b \\
&\Longleftrightarrow b=\left\lfloor (1-x)(a+b)\right\rfloor\,,\\\\
0\leq (1-x)a-xb<1&\Longleftrightarrow a-1< x(a+b)\leq a \,\\
&\Longleftrightarrow a=\left\lceil x(a+b)\right\rceil\cdot
\end{align*}
The remaining of the reasoning is the same as the case $S_x^+$. 
\end{enumerate}
\subsection{Proof of Lemma \ref{lem:prechoice}}
By definition, there exists a unique bijection $\alpha :S \to \Zz$ such that $c_1\preceq c_1$ if an only if $f(r)\leq f(s)$. Then, $d_i=\F(m,S,c_i)$ if and only if $f(d_i)=f(c_i)+m$, for $i\in\{1,2\}$. The first equivalence immediately follows. Moreover, if $f(d_2)\geq f(d_1)=f(c_1)+m$, then $f_{d_2}-m \geq f(c_1) \geq 0$, and we set $c_2=f^{-1}\left(f(d_1)-m\right)$.   
\subsection{Proof of Lemma \ref{lem:klsommax}}
By \eqref{eq:klseqdefbis}, for $i\leq j$
\begin{align*}
\sum_{h=i}^j o_h^{(k,l)}w_h^{(k,l)} &= \sum_{h=i}^{j}w_{h-1}^{(k,l)}-2w_h^{(k,l)} + w_{h+1}^{(k,l)} \\
&=\sum_{h=i-1}^{j-1}w_h^{(k,l)}-2\sum_{h=i}^{j}w_h^{(k,l)}+\sum_{h=i+1}^{j+1}w_h^{(k,l)}\\
&=w_{i-1}^{(k,l)}-w_{j+1}^{(k,l)}-w_{i}^{(k,l)}+w_{j+1}^{(k,l)}\,\cdot
\end{align*}
Therefore, for $1\leq i<j$, as $c_h=o_h^{(k,l)}+\chi(ih\in \{i,j\})$, \eqref{eq:klsommax} follows.
\subsection{Proof of Lemma \ref{lem:decalage}}
By definition, a forbidden pattern always begins and ends by a positive integer. Therefore, for $(c_i)_{i\geq 1}$ which that not contain any forbidden pattern, $0\cdot (c_i)_{i\geq 1}$ also does not contain any forbidden pattern. Hence, for $n, i \in\Zu$,  $(c_i)_{i\geq 1}\in \,_i^n\C^{(k,l)}$ if and only if $0\cdot (c_i)_{i\geq 1}\in \,_{i+1}^{n+1}\C^{(k,l)}$. Moreover, in  $_i^n\C^{(k,l)}$, $(c_u)_{u\geq}\prec (d_u)_{u\geq 1}$ if and only if there exists $i\leq j <n$ such that $c_j<d_j$ and $c_u=d_u$ for all $j<u<n$. By definition, by setting $(c'_u)_{u\geq1}=0\cdot(c_u)_{u\geq 1}$ and $(d'_u)_{u\geq1}=0\cdot(d_u)_{u\geq1}$, this is equivalent to saying that $c'_{j+1}<d_{j'+1}$ and $c'_u=d'_u$ for all $n+1>i>j+1$, which is equivalent to $(c'_i)_{i\geq}\prec (d'_i)_{i\geq 1}$ in $_{i+1}^{n+1}\C^{(k,l)}$.\\
To prove the second part of the Lemma, first note that $_{n+1}\C^{(k,l)}_n = \{(0)_{i=1}^n\}$ and  $_{n+2}\C^{(k,l)}_{n+1} = \{(0)_{i=1}^{n+1}\}$, so that $0\cdot\,_{n+1}\C^{(k,l)}_{n}=\,_{n+2}\C^{(k,l)}_{n+1}$. For $1\leq j\leq n$, as $0\cdot \,_{j}^n\C^{(k,l)} = \,_{j+1}^{n+1}\C^{(k,l)}$ and $\prec$ is preserved, we naturally have 
$$0\cdot \,_{j}\C^{(k,l)}_n \equiv 0\cdot \,_{j}^n\C^{(k,l)}\times \Zz = \,_{j+1}^{n+1}\C^{(k,l)}\times \Zz \equiv \,_{j+1}\C^{(k,l)}_{n+1}\,,$$
so that $0\cdot \,_{j}\C^{(k,l)}_n = \,_{j+1}\C^{(k,l)}_{n+1}$ with $\prec$ preserved.\\\\
For $k\geq 4$, we have $0\cdot\,_i^n\C^{(k,1)}=0\cdot\,_i^n\C^{(k-2)}=\,_{i+1}^{n+1}\C^{(k-2)}=\,_{i+1}^{n+1}\C^{(k,1)}$, and $\prec$ is preserved. The same holds for $_j\C^{(k,1)}_{2n}=_j\C^{(k-2)}_{n}$, as
$0\cdot\,_j\C^{(k,1)}_{2n}=0\cdot\,_j\C^{(k-2)}_n=\,_{j+1}\C^{(k-2)}_{n+1}=\,_{j+1}\C^{(k,1)}_{2n+2}$. To prove the result on $_j\C^{(k,1)}_{2n-1}$, it suffices to observe that $0\cdot\,^u_j\C^{(k,1)}_{2n-1}=\,^u_{j+1}\C^{(k,1)}_{2n+1}$ for $u\in \{0,1\}$, and that the order $\prec$ is preserved on each set. 
\subsection{Proof of Lemma \ref{lem:follow}}
Let $n$ be a positive integer and $(c_i)_{i\geq 1}\in \, _n\C^{(k,l)}$.
\begin{enumerate}
\item Suppose that there exists $j\geq n$ such that $c_i= o^{(k,l)}_i+\delta_{i,j}$ for $ n\leq i\leq j$, and set $(d_i)_{i\geq 1}$
$$
0=d_1=\cdots=d_j\quad, \quad d_{j+1}=c_{j+1}+1\quad\text{and}\quad d_{i}=c_{i} \text{ for }i\geq j+2\,\cdot
$$
We first show that $(d_i)_{i\geq 1}\in \, _n\C^{(k,l)}$. Since $d_1=\cdots=d_{n-1}=0$, it suffices to prove that it does not contain neither $o^{(k,l)}_i+2$ or a forbidden sub-sequence. 
\begin{enumerate}
\item Assume that there is some $i$ such that $d_i=o^{(k,l)}_i+2$. This is only possible when $i=j+1$ and $c_{j+1}=o^{(k,l)}_{j+1}+1$. Thus, the sequence $(c_i)$ contains the sub-sequence $(c_j,c_{j+1})= (o^{(k,l)}_j+1,o^{(k,l)}_{j+1}+1)$, which is a forbidden sub-sequence by definition.
\item Now assume that there is a forbidden sub-sequence in $(d_i)$. Since $(d_i)_{i\geq j+2}=(c_i)_{i\geq j+2}$, this forbidden sub-sequence necessarily starts from $j+1$, and by \eqref{eq:forbpat}, there exists $i>j+1$ such that $d_h=o^{(k,l)}_h+\chi(h\in \{j+1,i\})$ for $j+1\leq h\leq j$. Therefore, $c_h=o^{(k,l)}_h+\chi(h\in \{j,i\})$ for $j\leq h\leq j$ and $(c_j,\ldots,c_i)$ is a forbidden sub-sequence.
\end{enumerate}
The assumption of the two possible cases of forbidden sub-sequence leads to a contradiction. Thus, $(d_i)_{i\geq 1}\in \, _n\C^{(k,l)}$. In addition, $c_{j+1}<d_{j+1}$ and $c_i=d_i$ for $i>j+1$ so that $(c_i)\prec (d_i)$. We finally prove that $(d_i)$ follows $(c_i)$ in terms of $\prec$ in $ \, _n\C^{(k,l)}$. Suppose the contrary and assume that there exists $(e_i)\in  \, _n\C^{(k,l)}$ such that $(c_i)\prec (e_i)\prec (d_i)$. 
\begin{enumerate}
\item Since $c_{j+1}+1=d_{j+1}$ and $c_i=d_i$ for $i>j+1$, we necessarily have that $c_{j+1}\leq e_{j+1} \leq c_{j+1}+1=d_{j+1}$ and $c_i=e_i=d_i$ for $i>j+1$. The fact that $0=d_i\leq e_i$ for $i\leq j$ implies that $e_{j+1}<d_{j+1}$, and then $e_{j+1}=c_{j+1}$.
\item Note that $c_j=o^{(k,l)}_j+1 \leq e_j \leq o^{(k,l)}_j+1$ and then $e_j=o^{(k,l)}_j+1=c_j$. Also, since $(c_i)\prec (e_i)$ and $e_i=c_i$ for $i\geq j$ and $i<n$, there exists $n\leq g<j$ such that $o^{(k,l)}_g=c_g<e_g\leq o^{(k,l)}_g+1$ and $o^{(k,l)}_h=c_h=e_h$ for $g<h<n$. Thus, $e_h=o^{(k,l)}_h+\chi(h\in \{g,j\})$ for $g\leq h\leq j$, and $(e_g,\ldots,e_j)$ is a forbidden sub-sequence, which contradicts the fact that $(e_i)$ is a $(k,l)$-admissible word.
\end{enumerate}
From what precedes, $(d_i)$ follows $(c_i)$ in terms of $\prec$ in $ \, _n\C^{(k,l)}$.
\item Suppose that there does not exist $j\geq n$ such that $c_i= o^{(k,l)}_i+\delta_{i,j}$ for $ n\leq i\leq j$, and set $(d_i)\in \, _n\C^{(k,l)}$ such that
$
d_{n}=c_{n}+1\quad\text{and}\quad d_{i}=c_{i} \text{ for }i\geq n+1
$. In the set of integer sequences with the $(n-1)^{th}$ first terms equals to $0$, one can check that $(d_i)$ follows $(c_i)$ in terms of $\prec$. It now suffices to show that $(d_i)\in \, _n\C^{(k,l)}$. Assume the contrary. As in the previous reasoning, we either have $d_n=o^{(k,l)}_n+2$ or $(d_i)$ contains a forbidden sub-sequence.
\begin{enumerate}
\item If $d_n=o^{(k,l)}_n+2$, then $c_n=o^{(k,l)}_n+1$ and we have the fact in $(1)$ with $j=n$.
\item Suppose now that $(d_i)$ contains a forbidden sub-sequence, and this sub-sequence necessarily starts from $n$ and there exist $j>n$ such that $d_i = o^{(k,l)}_i+\chi(h\in \{n,j\})$ for $n\leq i\leq j$. Thus, $c_i = o^{(k,l)}_i+\delta_{h,j}$ for $n\leq i\leq j$ and we have the fact in $(1)$ for $j>n$.
\end{enumerate}
In both cases, the reasoning leads to the fact $(1)$, and we have a contradiction.
\end{enumerate}
\subsection{Proof of Lemma \ref{lem:followfin}}
For $j=n$, $0=c_1=\ldots=c_{n-1}=d_1=\cdots=d_{n-1}$, and $d_n=c_{n}+1$, so that only $(2)$ holds. Now assume that $1\leq j<n$. 
Let us write $(c_1,\cdots,c_n) \equiv ((c_1,\ldots,c_{n-1},0,\ldots), c_n)$. If $(c_1,\ldots,c_{n-1},0,\ldots)\neq (\underbrace{0,\ldots,0}_{j-1\text{ times}},o_j^{(k,l)},\ldots,o_{n-2}^{(k,l)},o_{n-1}^{(k,l)}+1,0,\ldots)$, then the sequence $(d_1,\cdots,d_n) \equiv ((d_1,\ldots,d_{n-1},0,\ldots), d_n)$ is such that $c_n=d_n$ and $(d_1,\ldots,d_{n-1},0,\ldots)$ follows $(c_1,\ldots,c_{n-1},0,\ldots)$ in $_j^n\C^{(k,l)}$. By Lemma \ref{lem:follow}, $(1)$ holds for $1\leq h<n-1$ or $(2)$ holds. If $(c_1,\ldots,c_{n-1},0,\ldots)=(\underbrace{0,\ldots,0}_{j-1\text{ times}},o_j^{(k,l)},\ldots,o_{n-2}^{(k,l)},o_{n-1}^{(k,l)}+1,0,\ldots)$, then  $d_n=c_n+1$, and $(d_1,\ldots,d_{n-1},0,\ldots)=(0,\ldots)$, and $(1)$ holds for $h=n-1$.
\subsection{Proof of Lemma \ref{lem:followfinbis}}
For $j=n$, $0=c_1=\ldots=c_{n-1}=d_1=\cdots=d_{n-1}$, and $d_n=c_{n}+1$, so that only $(2)$ holds. Now assume that $1\leq j<n$. 
Let us write $(c_1,\cdots,c_n) \equiv ((c_1,\ldots,c_{n-1},0,\ldots), c_n)$.
\begin{enumerate}
 \item If $(c_1,\ldots,c_{n-1},0,\ldots)\prec (\underbrace{0,\ldots,0}_{j-1\text{ times}},o_j^{(k,l)},\ldots,o_{n-3}^{(k,l)},o_{n-2}^{(k,l)}+1,0,\ldots)$, then the sequence $(d_1,\cdots,d_n) \equiv ((d_1,\ldots,d_{n-1},0,\ldots), d_n)$ is such that $c_n=d_n$, $0=c_{n-1}=d_{n-1}$, and $(d_1,\ldots,d_{n-2},0,\ldots)$ follows $(c_1,\ldots,c_{n-2},0,\ldots)$ in $_j^{n-1}\C^{(k,l)}$. By Lemma \ref{lem:follow}, $(1)$ holds for $1\leq h<n-2$ or $(2)$ holds.
 \item If $(c_1,\ldots,c_{n-1},0,\ldots)=(\underbrace{0,\ldots,0}_{j-1\text{ times}},o_j^{(k,l)},\ldots,o_{n-3}^{(k,l)},o_{n-2}^{(k,l)}+1,0,\ldots)$, then the sequence $(d_1,\cdots,d_n) \equiv ((d_1,\ldots,d_{n-1},0,\ldots), d_n)$ is such that $c_n=d_n$, $1=c_{n-1}+1=d_{n-1}$, and $(d_1,\ldots,d_{n-2},0,\ldots)=(0,\ldots)$, and $(1)$ holds for $h=n-2$.
 \item If 
 $$(\underbrace{0,\ldots,0}_{n-2\text{ times}},1,0,\ldots)\preceq (c_1,\ldots,c_{n-1},0,\ldots)\prec (\underbrace{0,\ldots,0}_{j-1\text{ times}},o_j^{(k,l)},\ldots,o_{n-2}^{(k,l)},o_{n-2}^{(k,l)}+2,0,\ldots)\,,$$
 we necessarily have $c_{n-1}\geq 1$, and the sequence $(d_1,\cdots,d_n) \equiv ((d_1,\ldots,d_{n-1},0,\ldots), d_n)$ is such that $c_n=d_n$, $0<c_{n-1}=d_{n-1}$, and $(d_1,\ldots,d_{n-2},d_{n-1}-1,0,\ldots)$ follows $(c_1,\ldots,c_{n-2},c_{n-1},0,\ldots)$ in $_j^{n}\C^{(k,l)}$. By Lemma \ref{lem:follow}, $(1)$ holds for $1\leq h\leq n-2$ or $(2)$ holds.
 \item If  $(c_1,\ldots,c_{n-1},0,\ldots)=(\underbrace{0,\ldots,0}_{j-1\text{ times}},o_j^{(k,l)},\ldots,o_{n-2}^{(k,l)},o_{n-2}^{(k,l)}+2,0,\ldots)$, then $d_n=c_n+1$, and $(d_1,\ldots,d_{n-1},0,\ldots)=(0,\ldots)$. Then, $(1)$ holds for $h=n-1$.
\end{enumerate}
\subsection{Proof of Lemma \ref{lem:follow1k}}
Let $d\in\Zz$ and $n\in \Zu$.
\begin{enumerate}
 \item As $\sk{0}+\sk{2}=\sk{1}=1$ and $\sk{0}>\sk{2}$, by \eqref{eq:divide+} with $x= \sk{0}$ and $b\in \Zz$, Definition \ref{def:1k} and \eqref{eq:klratinf0} in $(k-2,k-2)$, 
 $$\C^{(1,k)}\setminus\{((0)_{i\geq 1},(0)_{i\geq 1})\}=\{(s_1,s_2)\in \C^{(k,1)})\setminus\{(0)_{i\geq 1}\}\times \C^{(k,1)}: 0\cdot s_2 \preceq s_1 \preceq 0\cdot \F(1,\C^{(k,1)},s_2)\}\,\cdot$$
 Hence 
 $$\C^{(1,k)}=\{(s_1,s_2)\in \C^{(k,1)})\times \C^{(k,1)}: 0\cdot s_2 \preceq s_1 \preceq 0\cdot \F(1,\C^{(k,1)},s_2)\}\,\cdot$$
 Moreover,  by fact $(1.a)$ of Lemma \ref{lem:divide}, if $s$ is not special, then $t_2=s_2$ and $t_1$ is the sequence that follows $s_1$ in $\C^{(k,1)}$. By fact $(1.b)$ of Lemma \ref{lem:divide}, if $s$ is special, then $t_2=s'_2$ and $t_1=s_1$.\\
Finally, if $(t_1,t_2)= \F(m,\C^{(1,k)},(s_1,s_2))$, then $(a)$ and $(b)$ implies that  $m=m'_1+m'_2$ where $t_i= \F(m'_i,\C^{(k,1)},s_i)$ for $i\in \{1,2\}$. We then have 
$$0\cdot\F(0,\C^{(k,1)},s_2)\preceq s_1\preceq 0\cdot\F(1,\C^{(k,1)},s_2)\preceq \cdots \preceq 0\cdot\F(m'_2,\C^{(k,1)},s_2)\preceq t_1\preceq 0\cdot\F(m'_2+1,\C^{(k,1)},s_2)\,,$$
so that
\begin{align*}
 m_2&=\sharp \{c \in \,_2\C^{(k,1)}: s_1\prec c\preceq t_1\}\\
 &= m'_2-\chi(s_1= 0\cdot \F(1,\C^{(k,1)},s_2)) + \chi(t_1= 0\cdot \F(1,\C^{(k,1)},t_2)\,\cdot
\end{align*}
Since $m'_1=m_1+m'_2$, we obtain \eqref{eq:1ksominf}, and fact $(1)$ holds for $d=0$.
 Using Lemma \ref{lem:decalage}, we deduce from $d=0$ the result for any $d\geq 1$.
 \item The same goes for fact $(2)$, as $\ak{n+1}>\ak{n-1}$, and $\ak{n+1}+\ak{n-1}=\al{n}$. Using \eqref{eq:divide-} with $x= \ak{n+1}/\al{n}$ and $b\geq 0$, Definition \ref{def:1k} and \eqref{eq:k1rat0}, fact $(2)$ is satisfied for $d=0$. Using Lemma \ref{lem:decalage}, we deduce from $d=0$ the result for any $d\geq 1$.
\end{enumerate}

\subsection{Proof of Lemma \ref{lem:psiinfsize}}
For $t-2\geq j\geq 1$
$$\nu_{2j}-1-\left\lfloor s_0^{(k,l)} \nu_{2j+1}\right\rfloor=\left\lceil  s_2^{(k,l)}\la_{2j}\right\rceil-1-\la_{2j+1}\,,$$
for $t-1\geq j\geq 1$,
$$\left\lceil  s_2^{(k,l)}\nu_{2j-1}\right\rceil-1-\nu_{2j} = \la_{2j+1}-1-\left\lfloor s_0^{(k,l)} \la_{2j+2}\right\rfloor\,,$$ and $\la_{1}-1-\left\lfloor s_0^{(k,l)} \la_{2}\right\rfloor=m$.   
Using Remark \ref{rem:rat}, $\la \in \Lkl$ if and only if $(\nu,m)\in \Llk\times \Zz$.  Finally, we have that 
$$|\la|_e=\sum_{i=1}^{t-1} \la_{2j} = \sum_{i=1}^{t-1} \nu_{2j-1}=|\nu|_o $$
and 
\begin{align*}
 |\la|_o&= m+ 1+\left\lfloor s_0^{(k,l)} \nu_{1}\right\rfloor + \left(\sum_{j=1}^{t-2} \left\lceil  s_2^{(k,l)}\nu_{2j-1}\right\rceil+ \left\lfloor s_0^{(k,l)} \nu_{2j+1}\right\rfloor - \nu_{2j}\right)
 + \left\lceil  s_2^{(k,l)}\nu_{2t-3}\right\rceil-1 - \nu_{2t-2}\\
 &=  m- \sum_{i=1}^{t-1}\nu_{2t-2} + \sum_{i=1}^{t-1} \left\lceil  s_2^{(k,l)}\nu_{2j-1}\right\rceil+ \left\lfloor s_0^{(k,l)} \nu_{2j-1}\right\rfloor\\
 &=m- \sum_{i=1}^{t-1}\nu_{2t-2} + l\sum_{i=1}^{t-1} \nu_{2j-1}\\
 &= m+l\cdot|\nu|_o-|\nu|_e\,\cdot
\end{align*}
\[\]
\section{Proof of Propositions}
\subsection{Proof of Proposition \ref{prop:propklgen}}
\begin{enumerate}
\item By \eqref{eq:formuleklu}, for $n,m\in \Z$,
\begin{align*}
\alk{2n-1}\akl{2m}-\akl{2n}\alk{2m-1} &= \sqrt{\frac{l}{k}}\cdot\frac{(u^{2n-1}-u^{-2n+1})(u^{2m}-u^{-2m})-(u^{2n}-u^{-2n})(u^{2m-1}-u^{-2m+1})}{(u-u^{-1})^2}\,\\
&= \sqrt{\frac{l}{k}}\cdot\frac{-u^{2n-2m-1}-u^{-2n+2m+1}+u^{2n-2m+1}+u^{-2n+2m-1}}{(u-u^{-1})^2}\,\\
&= \sqrt{\frac{l}{k}}\cdot\frac{(u^{2n-2m}-u^{-2n+2m})(u-u^{-1})}{(u-u^{-1})^2}\,,\\
&= \akl{2n-2m}\,,
\end{align*}
and
\begin{align*}
\alk{2n-1}\akl{2m+1}-\akl{2n}\alk{2m} &= \frac{(u^{2n-1}-u^{-2n+1})(u^{2m+1}-u^{-2m-1})-(u^{2n}-u^{-2n})(u^{2m}-u^{-2m})}{(u-u^{-1})^2}\,\\
&= \frac{-u^{2n-2m-2}-u^{-2n+2m+2}+u^{2n-2m}+u^{-2n+2m}}{(u-u^{-1})^2}\,\\
&= \frac{(u^{2n-2m-1}-u^{-2n+2m+1})(u-u^{-1})}{(u-u^{-1})^2}\,,\\
&= \akl{2n-2m-1}\,\cdot
\end{align*}
and \eqref{eq:crosseven} follows. Analogously, we prove \eqref{eq:crossodd}.
\item 
By \eqref{eq:formuleklu}, for $n \in \Z$,
\begin{align*}
\akl{2n}-\sqrt{\frac{l}{k}}u\cdot\alk{2n-1} &= \sqrt{\frac{l}{k}}\cdot\frac{(u^{2n}-u^{-2n})-u(u^{2n-1}-u^{-2n+1})}{u-u^{-1}}\,\\
&= \sqrt{\frac{l}{k}}\cdot\frac{u^{-2n+2}-u^{-2n}}{u-u^{-1}}\,\\
&= \sqrt{\frac{l}{k}}u^{-2n+1}
\end{align*}
and
\begin{align*}
\akl{2n+1}-\sqrt{\frac{l}{k}}u\cdot\alk{2n} &= \frac{(u^{2n+1}-u^{-2n-1})-u(u^{2n}-u^{-2n})}{u-u^{-1}}\,\\
&= \frac{u^{-2n+1}-u^{-2n-1}}{u-u^{-1}}\,\\
&=u^{-2n}\,\cdot
\end{align*}
Since $(\alk{2n})_{n\geq 1}$ and $(\alk{2n-1})_{n\geq 1}$ are both increasing sequence of positive integers, they tend to $\infty$, so that, with the fact that $u^{-1}\leq 1$, both sequences $(\akl{2n}/\alk{2n-1})_{n\geq 1}$ and $(\akl{2n+1}/\alk{2n})_{n\geq 1}$ tend to $\sqrt{\frac{l}{k}}u$. Moreover, by \eqref{eq:crosseven} and \eqref{eq:crossodd}, for $n\geq 1$
\begin{align*}
\frac{\akl{2n}}{\alk{2n-1}}-\frac{\akl{2n+1}}{\alk{2n}} = \frac{\akl{1}}{\alk{2n-1}\alk{2n}}\\
\frac{\akl{2n+1}}{\alk{2n}}-\frac{\akl{2n+2}}{\alk{2n+1}} = \frac{\akl{1}}{\alk{2n}\alk{2n+1}}\\
\end{align*}
so that $(\akl{n}/\alk{n-1})_{n\geq 2}$ is decreasing.
\end{enumerate}
\subsection{Proof of Proposition \ref{prop:kladmseq}}
\begin{enumerate}
\item
To prove \eqref{eq:preckladm}, it suffices to prove that
$$(c_i)\prec (d_i)\Longrightarrow \Gamma_{(k,l)}((c_i))<\Gamma_{(k,l)}((d_i))\,\cdot$$
In fact, $\prec$ is a total strict order, i.e for $(c_i)\neq(d_i)$, we either have $(c_i)\prec (d_i)$ or $(c_i)\succ (d_i)$, and the previous implication yields
$$(c_i)\succ(d_i)\Longrightarrow \Gamma_{(k,l)}((c_i))>\Gamma_{(k,l)}((d_i))\,\cdot$$
\begin{enumerate}
\item Let $(c_i)_{i\geq 1}$ and $(d_i)_{i\geq 1}$ such that $(c_i)\prec (d_i)$, and let $n$ be the unique positive integer such that $c_n<d_n$ and $c_i=d_i$ for $i>n$. Thus, 
\begin{align*}
\Gamma_{(k,l)}((d_i))-\Gamma_{(k,l)}((c_i))&= \sum_{i=1}^{n} (d_i-c_i)w^{(k,l)}_{i}\geq w^{(k,l)}_{n} - \sum_{i=1}^{n-1}c_iw^{(k,l)}_{i}\,\cdot
\end{align*}  
By \eqref{eq:klsommaxbis}, we have that
$$w^{(k,l)}_{n}+w^{(k,l)}_0-w^{(k,l)}_{1}-w^{(k,l)}_{n-1}= \sum_{i=1}^{n-1} o^{(k,l)}_i w^{(k,l)}_{i}\,\cdot $$
 Therefore, if $c_i\leq o^{(k,l)}_i$ for $1\leq i\leq n-1$, then 
\begin{align*}
w^{(k,l)}_{n} - \sum_{i=1}^{n-1}c_iw^{(k,l)}_{i} &= (w^{(k,l)}_{1}-w^{(k,l)}_{0})+w^{(k,l)}_{n-1} + \sum_{i=1}^{n-1}(o^{(k,l)}_i-c_i)w^{(k,l)}_{i}\\
&\geq (w^{(k,l)}_{1}-w^{(k,l)}_{0})+w^{(k,l)}_{n-1} \\
&>0\,\cdot
\end{align*}
Now suppose that there exists $t\geq 1$ indices $i$, namely $1\leq i_1<\cdots<i_t\leq n-1$, such that $c_i>o^{(k,l)}_i$. They then satisfy $c_i=o^{(k,l)}_i+1$.
As $c_{i_j},\cdots, c_{i_{j+1}}$ is not a forbidden sub-sequence, we necessarily have that the set $\{i_j+1,i_{j+1}-1\}$ is not empty and contains a integer $i'_j$ such that $c_{i'_j}<o^{(k,l)}_{i'_j}$. Then, 
\begin{align*}
w^{(k,l)}_{n} - \sum_{i=1}^{n-1}c_iw^{(k,l)}_{i} &= w^{(k,l)}_{1}-w^{(k,l)}_{0}+(w^{(k,l)}_{n-1} -w^{(k,l)}_{i_t}) + \left(\sum_{i=i_t+1}^{n-1}(o^{(k,l)}_i-c_i)w^{(k,l)}_{i}\right)\\
&+ \left(\sum_{i=1}^{i_1-1} (o^{(k,l)}_i-c_i)w^{(k,l)}_{i}\right)+\sum_{j=1}^{t-1} \left(-w^{(k,l)}_{i_j} + \sum_{i=i_j+1}^{i_{j+1}-1}(o^{(k,l)}_i-c_i)w^{(k,l)}_{i}\right)\\
&\geq (w^{(k,l)}_{1}-w^{(k,l)}_{0}) + (w^{(k,l)}_{n-1} -w^{(k,l)}_{i_t}) + \sum_{j=1}^{t-1}w^{(k,l)}_{i'_j}-w^{(k,l)}_{i_j}\\
&\geq (w^{(k,l)}_{1}-w^{(k,l)}_{0}) + t-1\\
&>0\,\cdot
\end{align*} 
and we conclude. The fact that $\prec$ is a total strict order on $\C^{(k,l)}$ implies that $\Gamma_{(k,l)}$ is indeed injective.
\item We now prove that, for $(s_i^{(k,l)})=(\akl{i})$, $\Gamma_{(k,l)}$ is surjective on $\Zz$.  Consider the function described as the inverse. For any $m\in \Zz$, since $(\akl{i})_{i\geq 0}$ is increasing, tends to $\infty$ and $\akl{0}=0$, $n$ is simply equal to $\min\{i: \akl{i}>m\}$. Moreover, the recursive construction of $c_i$ for $1\leq i<n$ implies that
\begin{align*}
m-\sum_{j=i}^{n-1}c_j\akl{j} &= m-\sum_{j=i+1}^{n-1}c_j\akl{j} - \akl{i}\left\lfloor \frac{m-\sum_{j=i+1}^{n-1}c_j\akl{j}}{\akl{i}}\right\rfloor\\
&\in \{0,\ldots, \akl{i}-1\}\,\cdot
\end{align*}
Thus, the fact that $\akl{1}=1$ yields $m=\sum_{j=1}^{n-1}c_j\akl{j} = \sum_{j=1}^{\infty}c_j\akl{j}$. Furthermore, we recursively have that 
\begin{align*} 0\leq c_i &\leq \left\lfloor \frac{\akl{i+1}-1}{\akl{i}}\right\rfloor \\
&=\left\lfloor \frac{\akl{i+1}+\akl{i-1}-(\akl{i-1}+1)}{\akl{i}}\right\rfloor\\
&=o^{(k,l)}_i+2+\left\lfloor \frac{-(\akl{i-1}+1)}{\akl{i}}\right\rfloor\\
&=o^{(k,l)}_i+1\\
\end{align*} 
since,  by the fact that $(\akl{j})_{j\geq 0}$ is a increasing sequence of non-negative integers, $-1\leq \frac{-(\akl{i-1}+1)}{\akl{i}}<0$. Then $c_i \in \{0,\ldots,o^{(k,l)}_i+1\}$, so that the sequence $(c_i)_{i\geq 1}$ satisfied the two first fact for the definition of $(k,l)$-sequences. Finally, if we suppose by that a forbidden sub-sequence  $c_i,\ldots,c_j$ occurs in $(c_1,\ldots,c_{n-1})$, by \eqref{eq:klsommaxbis}, 
\begin{align*}
\akl{j+1}&> m-\sum_{h=j+1}^{n-1} c_h \akl{h} =\sum_{h=1}^{j} c_h \akl{h}\\
&\geq \sum_{h=i}^{j} c_h \akl{h}=\akl{j+1}+\akl{i-1}\\
&\geq \akl{j+1}\,,
\end{align*}
and we have a contradiction. Hence, $(c_i)_{i\geq 1}$ belongs to $\C_{(k,l)}$, and $\Gamma_{(k,l)}$ is surjective as $m = \sum_{i\geq 1}c_i\akl{i}$. Thus, $\Gamma_{(k,l)}$ is a bijection. We finally conclude by observing that, by \eqref{eq:preckladm},
$$\left[\{0,\ldots,\akl{n}-1\}\right]^{(k,l)} = \{(c_i)_{i\geq 1}\prec (\delta_{i,n})_{i\geq 1}\}= \, ^n\C^{(k,l)}$$
\item As $_n\left[m\right]^{(k,l)}\preceq \left[m\right]^{(k,l)}$ and the relation $\preceq$ is transitive, to prove \eqref{eq:leq}, it suffices to show that for $(c_i)_{i\geq 1}\in _n\C^{(k,l)}$ such that $(c_i)\succ \,_n\left[m\right]^{(k,l)}$, $(c_i)\succ \,_n\left[m\right]^{(k,l)}$. Set $\left[m\right]^{(k,l)}=(m_i)_{i\geq 1}$ and assume that $(c_i)_{i\geq 1}\in _n\C^{(k,l)}$ and $(c_i)\succ \,_n\left[m\right]^{(k,l)}$. As both sequences have their $(n-1)^{th}$ first parts equal to $0$, this means that there exists $i\geq n$ such $c_i>m_i$ and $c_j=m_j$ for $j>i$, and this implies that  $(c_i)\succ \left[m\right]^{(k,l)}$.
\end{enumerate}
\item Let $(c_i)_{i\geq 1}$ and $(d_i)_{i\geq 1}$ such that $(c_i)\ll (d_i)$, and let $n$ be the unique positive integer such that $c_n<d_n$ and $c_i=d_i$ for $1\leq i<n$. Thus, for an $m>n$ such that $c_i=0$ for $i>m$
\begin{align*}
\Gamma_{(k,l)}((d_i))-\Gamma_{(k,l)}((c_i))&= \sum_{i=n}^\infty (d_i-c_i)w^{(k,l)}_{i}\geq w^{(k,l)}_{n}- \sum_{i=n+1}^m c_iw^{(k,l)}_{i}\,\cdot
\end{align*}  
By \eqref{eq:klsommaxbis}, we have that
$$w^{(k,l)}_{n}+w^{(k,l)}_{m+1}-w^{(k,l)}_{n+1}-w^{(k,l)}_{m}= \sum_{i=n+1}^{m} o^{(k,l)}_i w^{(k,l)}_{i}\,\cdot $$
 Therefore, if $c_i\leq o^{(k,l)}_i$ for $1\leq i\leq n-1$, then 
\begin{align*}
w^{(k,l)}_{n} - \sum_{i=n+1}^{m} c_iw^{(k,l)}_{i} &= (w^{(k,l)}_{m}-w^{(k,l)}_{m+1})+w^{(k,l)}_{n+1} + \sum_{i=n+1}^{m}(o^{(k,l)}_i-c_i)w^{(k,l)}_{i}\\
&\geq (w^{(k,l)}_{m}-w^{(k,l)}_{m+1})+w^{(k,l)}_{n+1} \\
&\geq 0\,\cdot
\end{align*}
Now suppose that there exists $t\geq 1$ indices $i$, namely $n+1\leq i_1<\cdots<i_t\leq m$, such that $c_i>o^{(k,l)}_i$. As in the previous case, we necessarily have that the set $\{i_{j'-1}+1,i_{j}-1\}$ is not empty and contains a integer $i'_j$ such that $c_{i'_j}<o^{(k,l)}_{i'_j}$. Then, 
\begin{align*}
w^{(k,l)}_{n} - \sum_{i=n+1}^{m}c_iw^{(k,l)}_{i} &= (w^{(k,l)}_{m}-w^{(k,l)}_{m+1})+(w^{(k,l)}_{n+1}-w^{(k,l)}_{i_1}) + \left(\sum_{i=n+1}^{i_1-1}(o^{(k,l)}_i-c_i)w^{(k,l)}_{i}\right)\\
&+ \left(\sum_{i=i_t+1}^{m} (o^{(k,l)}_i-c_i)w^{(k,l)}_{i}\right)+\sum_{j=2}^{t} \left(-w^{(k,l)}_{i_j} + \sum_{i=i_{j-1}+1}^{i_{j}-1}(o^{(k,l)}_i-c_i)w^{(k,l)}_{i}\right)\\
&\geq (w^{(k,l)}_{m}-w^{(k,l)}_{m+1})+(w^{(k,l)}_{n+1}-w^{(k,l)}_{i_1}) + \sum_{j=2}^{t}w^{(k,l)}_{i'_j}-w^{(k,l)}_{i_j}\\
&\geq (w^{(k,l)}_{m}-w^{(k,l)}_{m+1})\\
&\geq 0\,\cdot
\end{align*} 
and this yields \eqref{eq:ggkladm}.
We also observe that, 
by \eqref{eq:preckladm}, for $(c_i) \in \,_{n+1}\C^{(k,l)}$, $(c_i)\ll (\delta_{i,n})$ so that $\Gamma_{(k,l)}(\,_{n+1}\C^{(k,l)}) \subset (0,s_n^{(k,l)})$.
\end{enumerate}
\subsection{Proof of Proposition \ref{prop:klinserted}}
\begin{enumerate}
 \item 
 \begin{enumerate}
 \item
 For $n\geq 1$, and $(d_i)_{i\geq 1}$ that follows $(c_i)_{i\geq 1}$ in $_n\C^{(l,k)}$, we equivalently have that $0\cdot (d_i)_{i\geq 1}$ follows $0\cdot(c_i)_{i\geq 1}$ in $_{n+1}\C^{(k,l)}$. Hence
 $0\cdot (d_i)_{i\geq 1}$ belongs to 
$_{n+2}\C^{(k,l)}$ if and only if the first fact of Lemma \ref{lem:follow} occurs, and by using the corresponding notations, \eqref{eq:follow1}  and \eqref{eq:klsommaxbis} give
\begin{align*}
\sum_{i= n}^\infty (d_i-c_i)\akl{i+1} & = \akl{j+2}-\akl{j+1} \sum_{i=n}^{j} o_{i+1}^{(k,l)}\cdot\akl{i+1}\\ 
& = \akl{n+1}-\akl{n}\,\cdot 
\end{align*}
On the other hand, $0\cdot (d_i)_{i\geq 1}$ does not belong to 
$_{n+2}\C^{(k,l)}$ if and only if the second fact of Lemma \ref{lem:follow} occurs, and \eqref{eq:follow2} 
gives $\sum_{i= n}^\infty (d_i-c_i)\akl{i+1} = \akl{n+1}$. 
As $\akl{n}\neq 0$, by Lemma \ref{lem:decalage}, we have the equivalences
\begin{align*}
 (d_i)_{i\geq 1}\in \,_{n+1}\C^{(l,k)}&\Longleftrightarrow \sum_{i= n}^\infty (d_i-c_i)\akl{i+1} = \akl{n+1}-\akl{n}\\
 (d_i)_{i\geq 1}\notin \,_{n+1}\C^{(l,k)}&\Longleftrightarrow \sum_{i= n}^\infty (d_i-c_i)\akl{i+1} = \akl{n+1}\,\cdot
\end{align*}
\item
Let us now set the function 
\begin{align*}
 \C^{(l,k)}&\to \mathbb{R}_{\geq 0}\\\
(c_i)_{i\geq 1}&\mapsto \sum_{i\geq 1} c_i(\akl{i+1}-\skl{0}\alk{i}) \,\cdot
\end{align*}
By \eqref{eq:ratio} and Proposition \ref{prop:kladmseq}, this is exactly $\Gamma_{(l,k)}$ for $w_i^{(l,k)}=\skl{i+1}$. In particular, $(w_i^{(l,k)})_{n\geq 1}$ is non-increasing, and fact $(2)$ of Proposition \ref{prop:kladmseq} implies that 
$\Gamma_{(l,k)}(_{n}\C^{(l,k)})\subset (0;w_{n-1}^{(l,k)})=(0;\skl{n})$.\\
Hence, $\Gamma_{(l,k)}( (d_i)_{i\geq 1}),\Gamma_{(l,k)}( (c_i)_{i\geq 1})\in (0;\skl{n})$ and  $\Gamma_{(l,k)}( (d_i)_{i\geq 1})\neq 0$(as $(d_i)_{i\geq 1}\neq 0$ and $\skl{i}>0$).
In addition,
$$\Gamma_{(l,k)}( (d_i)_{i\geq 1})-\Gamma_{(l,k)}( (c_i)_{i\geq 1})= \sum_{i= 1}^\infty (d_i-c_i)(\akl{i+1}-\skl{0}\alk{i})=\sum_{i= 1}^\infty (d_i-c_i)\skl{i+1}\,\cdot$$
By Lemma \ref{lem:follow} and \eqref{eq:klsommaxbis},
\begin{align*}
 (d_i)_{i\geq 1}\in \,_{n+1}\C^{(l,k)}&\Longrightarrow \Gamma_{(l,k)}( (d_i)_{i\geq 1})-\Gamma_{(l,k)}( (c_i)_{i\geq 1})= \skl{n+1}-\skl{n}\,,
 \\
 &\Longrightarrow \Gamma_{(l,k)}( (c_i)_{i\geq 1})= \skl{n} -\skl{n+1}+ \Gamma_{(l,k)}( (d_i)_{i\geq 1})\,,\\
  &\Longrightarrow \Gamma_{(l,k)}( (c_i)_{i\geq 1})>\skl{n} -\skl{n+1}\,,\\\\
 (d_i)_{i\geq 1}\notin \,_{n+1}\C^{(l,k)}&\Longrightarrow \Gamma_{(l,k)}( (d_i)_{i\geq 1})-\Gamma_{(l,k)}( (c_i)_{i\geq 1})= \skl{n+1}\,
 \\
 &\Longrightarrow \Gamma_{(l,k)}( (c_i)_{i\geq 1})=  \Gamma_{(l,k)}( (d_i)_{i\geq 1})-\skl{n+1}\leq \skl{n} -\skl{n+1}\,,\\
 &\Longrightarrow \Gamma_{(l,k)}( (c_i)_{i\geq 1}) \leq \skl{n} -\skl{n+1}\,,
\end{align*}
and \eqref{eq:klratinfdif1} and \eqref{eq:klratinfdif2} follow.
\end{enumerate}
\item 
\begin{enumerate}
\item Let $n\geq 1$, $2n-1\geq j\geq 1$, and  assume that $(d_i)_{i=1}^{2n-1}$ follows $(c_i)_{i=1}^{2n-1}$ in $_{j}\C^{(l,k)}_{2n-1}$. Analogously to the previous reasoning, by Lemma $\ref{lem:followfin}$ and \eqref{eq:klsommaxbis}, we obtain
\begin{align*}
 (d_i)_{i=1}^{2n-1}\in \,_{j+1}\C^{(l,k)}_{2n-1}&\Longleftrightarrow \sum_{i= j}^{2n-1} (d_i-c_i)\akl{i+1} = \akl{j+1}-\akl{j}\\
 (d_i)_{i=1}^{2n-1}\notin \,_{j+1}\C^{(l,k)}_{2n-1}&\Longleftrightarrow \sum_{i=j}^{2n-1} (d_i-c_i)\akl{i+1} = \akl{j+1}\,\cdot
\end{align*}
\item Set the function 
\begin{align*}
 \Gamma_{(l,k)}^{2n-1}  \colon\C^{(l,k)}_{2n-1}&\to \Zz\\\
(c_i)_{i=1}^{2n-1}&\mapsto \sum_{i=1}^{2n-1} c_i(\alk{2n-1}\akl{i+1}-\akl{2n}\alk{i}) \,\cdot
\end{align*}
By \eqref{eq:crosseven}, for 
\begin{align*}
\Gamma_{(l,k)}^{2n-1}((c_i)_{i=1}^{2n-1})&=\sum_{i=1}^{2n-1}c_i\akl{2n-1-i}=\sum_{i=0}^{2n-2}c_{2n-1-i}\akl{i}\\
&=\sum_{i=1}^{2n-2}c_{2n-1-i}\akl{i}\,\cdot
\end{align*}
Note that, as $i$ and $2n-1-i$ have different parities, $o^{(l,k)}_{2n-1-i}=o_i^{(k,l)}$, so that if $(c_1,\ldots,c_{2n-2},0,\ldots)\in \,_j\C^{(l,k)}$, then $(c_{2n-2},\ldots, c_1,0,\cdots)\in \,^{2n-j}\C^{(k,l)}$. Therefore, by fact $(1)$ of Proposition \ref{prop:kladmseq}, 
$$\Gamma_{(l,k)}^{2n-1}((c_i)_{i=1}^{2n-1}), \Gamma_{(l,k)}^{2n-1}((d_i)_{i=1}^{2n-1})\in \{0,\akl{2n-j}-1\}\,\cdot$$
Hence, by \eqref{eq:klsommaxbis} and Lemma \ref{lem:followfin},  
\begin{align*}
 (d_i)_{i=1}^{2n-1}\in \,_{j+1}\C^{(l,k)}_{2n-1}&\Longrightarrow \Gamma_{(l,k)}^{2n-1}((d_i)_{i=1}^{2n-1})- \Gamma_{(l,k)}^{2n-1}((c_i)_{i=1}^{2n-1})= \akl{2n-1-j}-\akl{2n-j}\,,
 \\
 &\Longrightarrow \Gamma_{(l,k)}^{2n-1}( (c_i)_{i= 1}^{2n-1})= \akl{2n-1-j}-\akl{2n-j}+ \Gamma_{(l,k)}^{2n-1}( (d_i)_{i=1}^{2n-1})\,,
 \\
 &\Longrightarrow \Gamma_{(l,k)}^{2n-1}( (c_i)_{i= 1}^{2n-1})\geq \akl{2n-1-j}-\akl{2n-j}\,,\\\\
 (d_i)_{i\geq 1}\notin \,_{j+1}\C^{(l,k)}_{2n-1}&\Longrightarrow \Gamma_{(l,k)}^{2n-1}((d_i)_{i=1}^{2n-1})- \Gamma_{(l,k)}^{2n-1}((c_i)_{i=1}^{2n-1})= \akl{2n-1-j}\,
 \\
 &\Longrightarrow \Gamma_{(l,k)}( (c_i)_{i=1}^{2n-1})=  \Gamma_{(l,k)}( (d_i)_{i=1}^{2n-1})-\akl{2n-1-j}\,,\\
 &\Longrightarrow \Gamma_{(l,k)}( (c_i)_{i=1}^{2n-1})< \akl{2n-j} -\akl{2n-1-j}\,,
\end{align*}
and \eqref{eq:klratevendif1} and \eqref{eq:klratevendif1} follow.
\end{enumerate}
 \item The reasoning is the same, by using \eqref{eq:crossodd} instead of \eqref{eq:crosseven}, the function 
 \begin{align*}
 \Gamma_{(k,l)}^{2n}  \colon\C^{(k,l)}_{2n}&\to \Zz\\\
(c_i)_{i=1}^{2n}&\mapsto \sum_{i=1}^{2n} c_i(\akl{2n}\alk{i+1}-\alk{2n+1}\alk{i})= \sum_{i=1}^{2n} c_i\akl{2n-i}\,,
\end{align*}
and observing that
$(c_1,\ldots,c_{2n-1},0,\ldots)\in \,_j\C^{(k,l)}$ implies $(c_{2n-1},\ldots, c_1,0,\cdots)\in \,^{2n+1-j}\C^{(k,l)}$.
\end{enumerate}
\subsection{Proof of Proposition \ref{prop:k1inserted}}
The case $n$ even only consists of facts $(2)$ and $(3)$ for $(k-2,k-2)$ instead of $(k,l)$. We now consider the case $n$ odd, written $2n-1$. 
\begin{enumerate}
 \item For positive integer $n$, $\lfloor n-1\rfloor\geq j\geq 1$, and $(d_i)_{i=1}^{n}$ follows $(c_i)_{i=1}^{n}$ in $_j\C^{(k,1)}_{2n-1}$, we have by fact $(2)$ of Lemma \ref{lem:followfinbis} that
\begin{align*}
 (d_i)_{i=1}^{2n-1}\notin \,_{j+1}\C^{(k,1)}_{2n-1}&\Longrightarrow (d_n-c_n)\ak{2n+1}+\sum_{i=j}^{n-1} (d_i-c_i)\ak{2i+2} = \ak{2j+2}\\
\end{align*}
Now suppose that fact $(1)$ occurs in Lemma \ref{lem:followfinbis}, which is equivalent to saying that $(d_i)_{i=1}^{2n-1}\notin \,_{j+1}\C^{(k,1)}_{2n-1}$.
In the case $h=n-1$, we have by \eqref{eq:follow1finbis} that
$c_i=o_i^{(k-2)}+2\delta_{i,n-1}$ and $d_i=0$ for $j\leq i\leq n-1$, and $d_{n}=c_n+1$. Therefore, by \eqref{eq:klsommaxbis}, 
\begin{align*}
(d_n-c_n)\ak{2n+1}+\sum_{i=j}^{n-1} (d_i-c_i)\ak{2i+2}&= \ak{2n+1} - 2\ak{2n} - \sum_{i=j}^{n-1} o_i^{(k-2)}\ak{2i+2}\\
&= a_{n+1}^{(k-2)} - a_{n}^{(k-2)} - \sum_{i=j+1}^{n} o_{i}^{(k-2)}a_{i}^{(k-2)}\\
&=a_{j+1}^{(k-2)}-a_{j}^{(k-2)}\\
&=\ak{2j+2}-\ak{2j}\,\cdot
\end{align*}
In the remaining case  $j\leq h\leq n-2$, by \eqref{eq:follow1finbis} and \eqref{eq:klsommaxbis} 
\begin{align*}
(d_n-c_n)\ak{2n+1}+\sum_{i=j}^{n-1} (d_i-c_i)\ak{2i+2}&= \ak{2h+4} - \ak{2h+2} - \sum_{i=j}^{h} o_i^{(k-2)}\ak{2i+2}\\
&= a_{h+2}^{(k-2)} - a_{h+1}^{(k-2)} - \sum_{i=j+1}^{h+1} o_{i}^{(k-2)}a_{i}^{(k-2)}\\
&=a_{j+1}^{(k-2)}-a_{j}^{(k-2)}\\
&=\ak{2j+2}-\ak{2j}\,\cdot
\end{align*}
Thus, 
\begin{align*}
 (d_i)_{i=1}^{2n-1}\in \,_{j+1}\C^{(k,1)}_{2n-1}&\Longrightarrow (d_n-c_n)\ak{2n+1}+\sum_{i=j}^{n-1} (d_i-c_i)\ak{2i+2} = \ak{2j+2}-\ak{2j}\,,
\end{align*}
and since $\ak{2j}\neq 0$,
\begin{align*}
 (d_i)_{i=1}^{2n-1}\in \,_{j+1}\C^{(k,1)}_{2n-1}&\Longleftrightarrow (d_n-c_n)\ak{2n+1}+\sum_{i=j}^{n-1} (d_i-c_i)\ak{2i+2} = \ak{2j+2}-\ak{2j}\,,\\
  (d_i)_{i=1}^{2n-1}\notin \,_{j+1}\C^{(k,1)}_{2n-1}&\Longleftrightarrow (d_n-c_n)\ak{2n+1}+\sum_{i=j}^{n-1} (d_i-c_i)\ak{2i+2} = \ak{2j+2}\,\cdot
\end{align*}
\item 
Let us now set the function 
\begin{align*}
 \Gamma_{(k,1)}^{2n-1}  \colon \C^{(k,1)}_{2n-1}&\to \Zz\\\
(c_i)_{i=1}^{n}&\mapsto \sum_{i=1}^{n-1} c_i(\al{2n-1}\ak{2i+2}-\al{2i+1}\alk{2i}) \,\cdot
\end{align*}
We then have by \eqref{eq:crosseven} and \eqref{eq:crossodd} that
\begin{align*}
 \Gamma_{(k,1)}^{2n-1}((c_i)_{i=1}^{n}) &= \sum_{i=1}^{n-1}c_i(a_{n}^{(k-2)}\cdot a_{i+1}^{(k-2)}-a_{n+1}^{(k-2)}\cdot a_{i}^{(k-2)}) + \sum_{i=1}^{n-1}c_i(a_{n-1}^{(k-2)}\cdot a_{i+1}^{(k-2)}-a_{n}^{(k-2)}\cdot a_{i}^{(k-2)})\\
 &=\sum_{i=1}^{n-1}c_ia_{n-i}^{(k-2)} + \sum_{i=1}^{n-1}c_ia_{n-1-i}^{(k-2)}\\
 &=\sum_{i=1}^{n-1}c_{n-i}a_{i}^{(k-2)} + \sum_{i=1}^{n-2}c_{n-1-i}a_{i}^{(k-2)}
\end{align*}
Observe that $(c_1,\ldots,c_{n-2},0,\ldots)\in \,_j\C^{(k-2)}$ implies $(c_{n-2},\ldots,c_{1},0,\ldots)\in \,^{n-j}\C^{(k-2)}$, and by fact $(1)$ of Proposition \ref{prop:kladmseq}, 
$$0\leq \sum_{i=1}^{n-2}c_{n-1-i}a_{i}^{(k-2)}<a_{n-j}^{(k-2)}\,\cdot$$
On one hand, if $c_{n-1}=0$, then $(c_1,\ldots,c_{n-2},0,\ldots)\in \,_j\C^{(k-2)}$ implies $(0,c_{n-2},\ldots,c_{1},0,\ldots)\in\, ^{n+1-j}\C^{(k-2)}$, so that
$$0\leq \sum_{i=1}^{n-1}c_{n-i}a_{i}^{(k-2)}<a_{n+1-j}^{(k-2)}\,\cdot$$
One the other hand, if $c_{n-1}>0$, then $(c_1,\ldots,c_{n-2},c_{n-1}-1,0,\ldots)\in \,_j\C^{(k-2)}$ implies that \\
$(c_{n-1}-1,c_{n-2},\ldots,c_{1},0,\ldots)\in\, ^{n+1-j}\C^{(k-2)}$, so that 
$$0<\sum_{i=1}^{n-1}c_{n-i}a_{i}^{(k-2)}= a_{1}^{(k-2)} + (c_{n-1}-1)a_{1}^{(k-2)}+\sum_{i=2}^{n-1}c_{n-i}a_{i}^{(k-2)}\leq a_{n+1-j}^{(k-2)}\,\cdot$$
Therefore, $\Gamma_{(k,1)}^{2n-1}((c_i)_{i=1}^{n})\in (0;a_{n+1-j}^{(k-2)}+a_{n+1-j}^{(k-2)}(=(0;\ak{2n+1-2j}($.
Moreover,
$$\Gamma_{(k,1)}^{2n-1}((d_i)_{i=1}^{n})-\Gamma_{(k,1)}^{2n-1}((c_i)_{i=1}^{n})= \sum_{i=1}^{n-1}(d_{n-i}-c_{n-i})(a_{i}^{(k-2)}+a_{i-1}^{(k-2)})$$
Hence, 
\begin{align*}
  (d_i)_{i=1}^{2n-1}\notin \,_{j+1}\C^{(k,1)}_{2n-1}&\Longrightarrow \Gamma_{(k,1)}^{2n-1}((d_i)_{i=1}^{n})-\Gamma_{(k,1)}^{2n-1}((c_i)_{i=1}^{n}) = (a_{n-j}^{(k-2)}+a_{n-j-1}^{(k-2)})\,,\\
  &\Longrightarrow \Gamma_{(k,1)}^{2n-1}((c_i)_{i=1}^{n}) =\Gamma_{(k,1)}^{2n-1}((d_i)_{i=1}^{n})-\ak{2n-1-2j}\,,\\
  &\Longrightarrow \Gamma_{(k,1)}^{2n-1}((c_i)_{i=1}^{n})< \ak{2n+1-2j}-\ak{2n-1-2j}\,,
\end{align*}
We now study the cases of occurrence of fact $(1)$ in Lemma \ref{lem:followfinbis}. When $h=n-1$, by \eqref{eq:klsommaxbis},
 \begin{align*}
 \Gamma_{(k,1)}^{2n-1}((d_i)_{i=1}^{n})-\Gamma_{(k,1)}^{2n-1}((c_i)_{i=1}^{n})&= -2(a_{1}^{(k-2)}+a_{0}^{(k-2)})-\sum_{i=1}^{n-j}o_{n-i}^{(k-2)}(a_{i}^{(k-2)}+a_{i-1}^{(k-2)})\\
 &= -(a_{0}^{(k-2)}+a_{-1}^{(k-2)})-(a_{1}^{(k-2)}+a_{0}^{(k-2)})\\
 &\qquad-(a_{n+1-j}^{(k-2)}+a_{n-j}^{(k-2)})+(a_{n-j}^{(k-2)}+a_{n-j-1}^{(k-2)})\\
 &=\ak{2n-1-2j}-\ak{2n+1-2j}\,\cdot
 \end{align*}
  When $h\leq n-2$, 
 \begin{align*}
 \Gamma_{(k,1)}^{2n-1}((d_i)_{i=1}^{n})-\Gamma_{(k,1)}^{2n-1}((c_i)_{i=1}^{n})&= (a_{n-1-h}^{(k-2)}+a_{n-2-h}^{(k-2)})-(a_{n-h}^{(k-2)}+a_{n-1-h}^{(k-2)})\\
 &\qquad-\sum_{i=n-h}^{n-j}o_{n-i}^{(k-2)}(a_{i}^{(k-2)}+a_{i-1}^{(k-2)})\\
 &= -(a_{n+1-j}^{(k-2)}+a_{n-j}^{(k-2)})+(a_{n-j}^{(k-2)}+a_{n-j-1}^{(k-2)})\\
 &=\ak{2n-1-2j}-\ak{2n+1-2j}\,\cdot
 \end{align*}
Therefore, 
\begin{align*}
  (d_i)_{i=1}^{2n-1}\notin \,_{j+1}\C^{(k,1)}_{2n-1}&\Longrightarrow \Gamma_{(k,1)}^{2n-1}((d_i)_{i=1}^{n})-\Gamma_{(k,1)}^{2n-1}((c_i)_{i=1}^{n}) = \ak{2n-1-2j}-\ak{2n+1-2j}\,,\\
  &\Longrightarrow \Gamma_{(k,1)}^{2n-1}((c_i)_{i=1}^{n})=\Gamma_{(k,1)}^{2n-1}((d_i)_{i=1}^{n})+\ak{2n+1-2j}-\ak{2n-1-2j}\,,\\
  &\Longrightarrow \Gamma_{(k,1)}^{2n-1}((c_i)_{i=1}^{n})\geq \ak{2n+1-2j}-\ak{2n-1-2j}\,\cdot
\end{align*}
and \eqref{eq:k1ratdif1} and \eqref{eq:k1ratdif2} follow.
\end{enumerate}
\subsection{Proof of Proposition \ref{prop:klratioinfini}}$\,$
\begin{enumerate}
 \item Applying fact $(1)$ of Proposition \ref{prop:klinserted} for $n=1$,
 $$\sum_{i\geq 1}c_i(\akl{i+1}-\skl{0}\alk{i})=\sum_{i\geq 1}c_i\skl{i+1}\in (0;\skl{1})=(0;1)\,$$
 Moreover, there exists $i\geq 1$ such that $c_i\neq 0$. Thus, 
 $$0<\sum_{i\geq 1}c_i\akl{i+1} - \skl{0}\sum_{i\geq 1}c_i\alk{i}\leq 1 \Longleftrightarrow \sum_{i\geq 1}c_i\akl{i+1} = \left\lfloor\skl{0}\sum_{i\geq 1}c_i\alk{i}\right\rfloor+1\,\cdot$$
 Therefore, for $(r,s)\in \Zz\times \Zu$, by Remark \ref{rem:rat},
 \begin{align*}
  r>\skl{0}s \Longleftrightarrow r\geq \lfloor \skl{0}s \rfloor +1 &\Longleftrightarrow [r]^{(k,l)}\succeq\left[\lfloor \skl{0}s \rfloor +1\right]^{(k,l)}= 0\cdot \left[s\right]^{(l,k)}\,,\\
  &\Longleftrightarrow\, _2[r]^{(k,l)}\succeq 0\cdot \left[s\right]^{(l,k)}\,\cdot
 \end{align*}
 For $r>\skl{0}$, this implies that
 $$\skl{0}\lceil r/\skl{0}\rceil \geq r>\skl{0}(\lceil r/\skl{0}\rceil-1) \Longleftrightarrow 0\cdot\left[\skl{0}\lceil r/\skl{0}\rceil \right]^{(l,k)}\succ \, _2[r]^{(k,l)}\succeq\left[\skl{0}(\lceil r/\skl{0}\rceil-1)\right]^{(l,k)}\,,$$
so that $ _2[r]^{(k,l)}\succeq\left[\skl{0}(\lceil r/\skl{0}\rceil-1)\right]^{(l,k)}$, and \eqref{eq:klratinf0} follows.
Note that $r\geq 0=s=\skl{0}s$, we still have  $_2[r]^{(k,l)}\succeq 0\cdot \left[s\right]^{(l,k)}$. As by definition,
$$_2[r]^{(k,l)}\succeq 0\cdot \left[s\right]^{(l,k)}\Longleftrightarrow [r]^{(k,l)}\succeq 0\cdot \left[s\right]^{(l,k)}\,,$$
\eqref{eq:inlkl} is immediate.
 \item By fact $(2)$ of Proposition \ref{prop:klinserted} for $j=1$, 
 $$0\leq \sum_{i=1}^{2n-1}c_i\akl{i+1} - \frac{\akl{2n}}{\alk{2n-1}}\sum_{i\geq 1}c_i\alk{i}< \frac{\akl{2n-1}}{\alk{2n-1}}=1 \Longleftrightarrow \sum_{i=1}^{2n-1}c_i\akl{i+1} = \left\lceil \frac{\akl{2n}}{\alk{2n-1}}\sum_{i\geq 1}c_i\alk{i}\right\rceil\,\cdot$$
 As the previous reasoning, this leads to \eqref{eq:klrateven0} for $n$ odd. For the case $n$ even, it suffices to consider in fact $(3)$ of Proposition \ref{prop:klinserted}, $j=1$ and $(l,k)$ instead of $(k,l)$, and we obtain 
 $$0\leq \sum_{i=1}^{2n}c_i\akl{i+1} - \frac{\akl{2n+1}}{\alk{2n}}\sum_{i\geq 1}c_i\alk{i} < \frac{\alk{2n}}{\alk{2n}}=1 \Longleftrightarrow \sum_{i=1}^{2n}c_i\akl{i+1} = \left\lceil \frac{\akl{2n+1}}{\alk{2n}}\sum_{i\geq 1}c_i\alk{i}\right\rceil\,\cdot$$
 From \eqref{eq:klrateven0}  and the fact that $\akl{2j-1}=\alk{2j-1}$, \eqref{eq:inlklodd} and \eqref{eq:inlkleven} are immediate.
\end{enumerate}
\subsection{Proof of Proposition \ref{prop:k1ratiofin}}
The case $n$ even follows from \eqref{eq:klrateven0} with $(k-2,k-2)$ instead of $(k,l)$. For the odd case, written $2n-1$, for $(c_i)_{i=1}^{n}\in \C_{2n-1}^{(k,1)}$,  by Proposition \ref{prop:k1inserted} for $j=1$, 
\begin{align*}
 c_{n} \ak{2n+1}+\sum_{i=1}^{n-1} c_i\ak{2i+2} - \frac{\ak{2n+1}}{\ak{2n-1}}\left(c_{n} \ak{2n-1}+\sum_{i=1}^{n-1} c_i\ak{2i}\right) = \sum_{i=1}^{n-1} c_i\ak{2i+2} - \frac{\ak{2n+1}}{\ak{2n-1}}\left(\sum_{i=1}^{n-1} c_i\ak{2i}\right) \in (0, 1(\,\cdot
\end{align*}
The remaining of the proof is the same as the proof of \eqref{eq:klrateven0}. 
\subsection{Proof of Proposition \ref{prop:1kratiofin}}$\,$
\begin{enumerate}
 \item For $r,s\in \Zu$ 
 $$r>\sk{0}s \Longleftrightarrow r\geq \lfloor \sk{0}s\rfloor+1 $$
 and \eqref{eq:1kratinf0} follows from definition \ref{def:1k}. Note that $[r]^{(k,1)}\succeq p_1([s]^{(1,k)})$ still holds for $r\geq s=0$.
 Analogously, 
 $$r>\sll{0}s \Longleftrightarrow \sk{2}r>s \Longleftrightarrow  \lceil \sk{2}r\rceil-1\leq s$$
 and \eqref{eq:1kratinf1} follows from definition \ref{def:1k}.
 Here again, $p_2([r]^{(1,k)})\succeq [s]^{(k,1)}$ holds for $r\geq s = 0$. Hence, \eqref{eq:inlk1} and \eqref{eq:inl1k} are immediate.
 \item For $r,s\in \Zz$,
 $$r\geq \frac{\ak{n+1}}{\al{n}}s \Longleftrightarrow r\geq \left\lceil \frac{\ak{n+1}}{\al{n}}s\right\rceil $$
 and \eqref{eq:1krat0} follows from definition \ref{def:1k}.
 Analogously,
 $$r\geq \frac{\al{n+1}}{\ak{n}}s \Longleftrightarrow \frac{\ak{n}}{\ak{n+1}}r\geq s \Longleftrightarrow \left\lfloor \frac{\ak{n}}{\al{n+1}}s\right\rfloor$$
 and \eqref{eq:1krat1} follows from definition \ref{def:1k}.
 Hence, \eqref{eq:inlk12n} and \eqref{eq:inlk12n1} follows from \eqref{eq:1krat0} for the odd cases  and \eqref{eq:1krat1} for the even cases, and  \eqref{eq:inl1k2n} and \eqref{eq:inl1k2n1} follows from \eqref{eq:1krat0} for the even cases  and \eqref{eq:1krat1} for the odd cases. 
\end{enumerate}

\subsection{Proof of Proposition \ref{prop:welldefinf}}
Note that for $i\geq 1$, $m_i$ recursively derives the numbers $m_{i+1-j}^{(j)}$ and the pairs $(\la_{2j-1}^{(i+1-j)},\la_{2j}^{(i+1-j)})$ for $1\leq j\leq i$, as the $\bkl{i-j}$ inserted into the pair $(\la_{2j+1},\la_{2j+2})$ can only be provided by the $\bkl{i+1-j}$ inserted into the pair $(\la_{2j-1},\la_{2j})$.
For $n$ the greatest index $i$ such that $\bkl{i}$ occurs in $\nu$, we have for $i>n$ and $1\leq j\leq i$ that
$\la_{2j-1}^{(i+1-j)}=\la_{2j}^{(i+1-j)}=0$, and
$$\left[\la_{2j-1}^{(i+1-j)}\right]^{(k,l)}=0\cdot\left[\la_{2j}^{(i+1-j)}\right]^{(l,k)}\in \,_{i+1-j}\C^{(k,l)}\,\cdot$$
We now fix $j\geq 1$.
\begin{enumerate}
\item Suppose now that for some $2\leq i\leq n-j$, \eqref{eq:klbelongsji} holds for $i+1$.
\begin{enumerate}
\item Assume that $\left[\la_{2j-1}\right]^{(k,l)}=0\cdot\left[\la_{2j}\right]^{(l,k)}\in \,_{i}\C^{(k,l)}$.
We first observe that, \eqref{eq:condinsinf1} is equivalent to 
$$\skl{i-1}\geq \la_{2j-1}-s_0^{(k,l)}\cdot \la_{2j}> \skl{i-1}-\skl{i}\,,$$
and by \eqref{eq:klratinfdif2}, doing \eqref{eq:insert1type1} consists of replacing $\left[\la_{2j-1}\right]^{(k,l)}$ by $\F(1,\, _{i}\C^{(k,l)},\left[\la_{2j-1}\right]^{(k,l)})$, and $\left[\la_{2j}\right]^{(l,k)}$ by $\F(1,\, _{i-1}\C^{(l,k)},\left[\la_{2j}\right]^{(l,k)})$.
Otherwise,
$$\skl{i-1}-\skl{i}\geq \la_{2j-1}-s_0^{(k,l)}\cdot \la_{2j}\geq 0\,,$$
and by \eqref{eq:klratinfdif1}, doing \eqref{eq:insert1type2} consists of replacing $\left[\la_{2j-1}\right]^{(k,l)}$ by $\F(1,\, _{i}\C^{(k,l)},\left[\la_{2j-1}\right]^{(k,l)})$, and $\left[\la_{2j}\right]^{(l,k)}$ by $\F(1,\, _{i-1}\C^{(l,k)},\left[\la_{2j}\right]^{(l,k)})$. Therefore, the pair $(\left[\la_{2j-1}\right]^{(k,l)},\left[\la_{2j}\right]^{(l,k)})$ is replaced by the pair\\ $(\F(1,\, _{i}\C^{(k,l)},\left[\la_{2j-1}\right]^{(k,l)}),\F(1,\, _{i-1}\C^{(l,k)},\left[\la_{2j}\right]^{(l,k)}))$, and by Lemma \ref{lem:decalage},
$$0\cdot\F(1,\, _{i-1}\C^{(l,k)},\left[\la_{2j}\right]^{(l,k)})= \F(1,\, _{i}\C^{(k,l)},0\cdot\left[\la_{2j}\right]^{(l,k)})=\F(1,\, _{i}\C^{(k,l)},\left[\la_{2j-1}\right]^{(k,l)})\,\cdot$$
As $\left[\la_{2j-1}^{(i+1)}\right]^{(k,l)}=0\cdot\left[\la_{2j}^{(i+1)}\right]^{(l,k)}\in \,_{i+1}\C^{(k,l)}\subset \, _{i}\C^{(k,l)}$, we indeed retrieve fact $(1.a)$ and \eqref{eq:klbelongsji} holds for $i$.
\item Note that, by \eqref{eq:klratinfdif2}, there are as many stored $\bkl{i-1}$ for the insertion into the pair $(\la_{2j+1},\la_{2j+2})$ as following sequences of $\la_{2j-1}^{(i+1)}$ that belong to $_{i+1}\C^{(k,l)}$, i.e.
$$m_{i-1}^{(j+1)}=\sharp\left\{(c_i)\in \, _{i+1}\C^{(k,l)}: \left[\la_{2j-1}^{(i+1)}\right]^{(k,l)}\prec (c_i)\preceq \left[\la_{2j-1}^{(i)}\right]^{(k,l)} \right\}\,\cdot$$
Thus,  $_{i+1}\left[\la_{2j-1}^{(i)}\right]^{(k,l)}=\F(m_{i-1}^{(j+1)},\,_{i+1}\C^{(k,l)},\left[\la_{2j-1}^{(i+1)}\right]^{(k,l)})$, and we retrieve fact $(1.b)$.
\end{enumerate}
\item Finally, as $\alk{0}=0$, the insertion of $\bkl{1}$ does not affect the size of $\la_{2j}^{(2)}$, and then $\la_{2j}^{(1)}=\la_{2j}^{(2)}$.
\begin{enumerate}
 \item Since $\akl{1}=1$, $\la_{2j}^{(1)}=\la_{2j}^{(2)}+m_1^{(j)}$, and fact $(2.a)$ holds.
 \item Reciprocally, fact $(2.b)$ follows from \eqref{eq:klbelongsji} for $i=2$ and the fact  $\la_{2j}^{(1)}=\la_{2j}^{(2)}$.
\end{enumerate}
\end{enumerate}
\subsection{Proof of Proposition \ref{prop:welldefinf1}}
Note that for $i\geq 1$, $m_i$ recursively derives the numbers $m_{i'}^{(j)}$ for $i'$ at most equal to $i+1-j$, as the $\bkl{i-j}$ inserted into the pair $(\la_{2j+1},\la_{2j+2})$ can only be provided by the $\bkl{i+1-j}$ or $\bkl{i+2-j}$ inserted into the pair $(\la_{2j-1},\la_{2j})$.
Let $n$ be such that $\bkl{i}$ does not occur in $\nu$ for $i\geq 2n+1$. For $i\geq 2n+1$ and $1\leq j\leq i$, we have $2i-j+1\geq 2n+1$, and then 
$\la_{2j-1}^{(2i+1-2j)}=\la_{2j}^{(2i+1-2j)}=0$, and
$$\left[\la_{2j-1}^{(2i+1-2j)}\right]^{(k,l)}=0\cdot\left[\la_{2j}^{(2i+1-2j)}\right]^{(l,k)}\in \,_{i+1-j}\C^{(k,l)}\,\cdot$$
We now fix $j\geq 1$.
\begin{enumerate}
 \item Suppose that for some $1\leq i\leq 2n+1-j$, \eqref{eq:k1belongsji} is satisfied for $2i+1$.
 \begin{enumerate}
 \item First, observe that, for $i\geq 2$,
 \begin{align*}
  \la_{2j-1}-\sk{0}\cdot \la_{2j}> \sk{2i-2}-\sk{2i} &\Longleftrightarrow \sk{2}\la_{2j-1}-\sk{0}(\la_{2j}-\la_{2j-1})> \sk{2i-2}-\sk{2i}\\
  &\Longleftrightarrow \la_{2j-1}-s_0^{(k-2)}(\la_{2j}-\la_{2j-1})> s_{i-1}^{(k-2)}-s_{i}^{(k-2)}\,\cdot
 \end{align*}
Hence, if $\left[\la_{2j-1}\right]^{(k,1)} =  p_1\left(\left[\la_{2j}\right]^{(1,k)}\right)= 0\cdot p_2\left(\left[\la_{2j}\right]^{(1,k)}\right) \in \,_{i}\C^{(k,1)}$, then\\ 
$\left[\la_{2j-1}\right]^{(k-2)} =  0\cdot\left[\la_{2j}-\la_{2j-1}\right]^{(k-2)} \in \,_{i}\C^{(k-2)}\,,$
By \eqref{eq:klratinfdif1} and \eqref{eq:klratinfdif2}, doing \eqref{eq:insert1type121} or \eqref{eq:insert1type221} (according to whether \eqref{eq:condinsinf121} holds or not) is equivalent to replacing the pair
$(\left[\la_{2j-1}\right]^{(k-2)},\left[\la_{2j}-\la_{2j-1}\right]^{(k-2)})$
 by $\F(1,\,_{i}\C^{(k-2)},\left[\la_{2j-1}\right]^{(k-2)}),\F(1,\,_{i-1}\C^{(k-2)},\left[\la_{2j}-\la_{2j-1}\right]^{(k-2)})$. Hence, by Lemma \ref{lem:decalage}, the new pair also satisfies 
 $$\left[\la_{2j-1}\right]^{(k,1)} =  p_1\left(\left[\la_{2j}\right]^{(1,k)}\right)= 0\cdot p_2\left(\left[\la_{2j}\right]^{(1,k)}\right) \in \,_{i}\C^{(k,1)}\,,$$
 so that \eqref{eq:k1belongsji} recursively holds.\\
 Let us now set $\left[\mu_{2j}^{(2i)}\right]^{(1,k)}= \F(m_{2i}^{(j)},\,_{+(i-1)}\C^{(1,k)},\left[\la_{2j}^{(2i+1)}\right]^{(1,k)})$. By \eqref{eq:1ksominf}, 
 $$m_{2i}^{(j)}=a+2b-\chi\left(p_1\left(\left[\mu_{2j}^{(2i)}\right]^{(1,k)}\right)=0\cdot \F(1,\,_i\C^{(k,1)}, p_2\left(\left[\mu_{2j}^{(2i)}\right]^{(1,k)}\right)\right)\,,$$
 where
 \begin{align*}
  a&= \sharp \left\{r\in\,_{i}\C^{(k,1)}\setminus\, _{i+1}\C^{(k,1)}: \left[\la_{2j-1}^{(2i+1)}\right]^{(k,1)}\prec r\preceq  p_1\left(\left[\mu_{2j}^{(2i)}\right]^{(1,k)}\right)\right\}\,,\\
  b&= \sharp \left\{r\in\, _{i+1}\C^{(k,1)}: \left[\la_{2j-1}^{(2i+1)}\right]^{(k,1)}\prec r\preceq  p_1\left(\left[\mu_{2j}^{(2i)}\right]^{(1,k)}\right)\right\}\,\cdot
 \end{align*}
 \begin{enumerate}
  \item On the other hand, using \eqref{eq:klratinfdif1} and \eqref{eq:klratinfdif2}, at the end of the insertion of $\bk{2i}$, we have 
  $$m_{2i}^{(j)}=a'+2b'-\chi\left(\bk{2i-2} \text{ is stored}\right)\,,$$
  where
  \begin{align*}
  a'&= \sharp \left\{r\in\,_{i}\C^{(k,1)}\setminus\, _{i+1}\C^{(k,1)}: \left[\la_{2j-1}^{(2i+1)}\right]^{(k,1)}\prec r\preceq \left[\la_{2j-1}^{(2i)}\right]^{(k,1)}\right\}\,,\\
  b'&= \sharp \left\{r\in\, _{i+1}\C^{(k,1)}: \left[\la_{2j-1}^{(2i+1)}\right]^{(k,1)}\prec r\preceq \left[\la_{2j-1}^{(2i)}\right]^{(k,1)}\right\}\,\cdot
 \end{align*}
 Assume that $\left[\la_{2j-1}^{(2i)}\right]^{(k,1)}\succ p_1\left(\left[\mu_{2j}^{(2i)}\right]^{(1,k)}\right)$. We then have 
 \begin{align*}
  a'-a&= \sharp \left\{r\in\,_{i}\C^{(k,1)}\setminus\, _{i+1}\C^{(k,1)}: p_1\left(\left[\mu_{2j}^{(2i)}\right]^{(1,k)}\right)\prec r\preceq \left[\la_{2j-1}^{(2i)}\right]^{(k,1)}\right\}\,,\\
  b'-b&= \sharp \left\{r\in\, _{i+1}\C^{(k,1)}: p_1\left(\left[\mu_{2j}^{(2i)}\right]^{(1,k)}\right)\prec r\preceq \left[\la_{2j-1}^{(2i)}\right]^{(k,1)}\right\}\,\\\\
  a'+b'-a-b&= \sharp \left\{r\in\, _{i}\C^{(k,1)}: p_1\left(\left[\mu_{2j}^{(2i)}\right]^{(1,k)}\right)\prec r\preceq \left[\la_{2j-1}^{(2i)}\right]^{(k,1)}\right\}\\
  &\geq 1\,\cdot
 \end{align*}
 Hence, 
 \begin{align*}
  D&=\chi\left(\bk{2i-2} \text{ is stored}\right)-\chi\left(p_1\left(\left[\mu_{2j}^{(2i)}\right]^{(1,k)}\right)=0\cdot \F(1,\,_i\C^{(k,1)}, p_2\left(\left[\mu_{2j}^{(2i)}\right]^{(1,k)}\right)\right)\\
  &= a'-a+2(b'-b)\\
  &=a'+b'-a-b + (b'-b)\geq 1
 \end{align*}
 and this is only possible when the equality holds, as $\chi(prop)\in \{0,1\}$. Then, $\bk{2i-2}$ is stored and $p_1\left(\left[\mu_{2j}^{(2i)}\right]^{(1,k)}\right)\prec 0\cdot \F(1,\,_i\C^{(k,1)}, p_2\left(\left[\mu_{2j}^{(2i)}\right]^{(1,k)}\right)$. However, $\bk{2i-2}$ can only be stored at the end of the insertion of the $\bk{2i}$ with an occurrence of \eqref{eq:insert1type221}, and by \eqref{eq:klratinfdif2}, this means that $\left[\la_{2j-1}^{(2i)}\right]^{(k,1)}\in \,_{i+1}\C^{(k,1)}$. Thus, $b'-b\geq 1$, and $D\geq 2$, and a contradiction occurs. With the same reasonning, we show that we cannot have $\left[\la_{2j-1}^{(2i)}\right]^{(k,1)}\prec p_1\left(\left[\mu_{2j}^{(2i)}\right]^{(1,k)}\right)$. Therefore, $\left[\la_{2j-1}^{(2i)}\right]^{(k,1)}= p_1\left(\left[\mu_{2j}^{(2i)}\right]^{(1,k)}\right)$ and we obtain fact $(1.a.ii)$. Remark that $a=a'$, $b=b'$ and
 $\bk{2i-2}$ is stored if and only if $p_1\left(\left[\mu_{2j}^{(2i)}\right]^{(1,k)}\right)= 0\cdot \F(1,\,_i\C^{(k,1)}, p_2\left(\left[\mu_{2j}^{(2i)}\right]^{(1,k)}\right)$.
 \item Fact $(1.a.ii)$ is straightforward, since we insert once $\bk{2i-1}$ whenever \eqref{eq:condinsinf121} occurs or not.
 \end{enumerate}
 \item Reciprocally, let us consider $\left[\mu_{2j}^{(2i-1)}\right]^{(1,k)}=\left(p_1\left(\left[\la_{2j}^{(2i-1)}\right]^{(1,k)}\right), \,_{i}p_2\left(\left[\la_{2j}^{(2i-1)}\right]^{(1,k)}\right)\right)\,\cdot$
 By \eqref{eq:k1belongsji} in $2i-1$, 
 $$0\cdot p_2\left(\left[\mu_{2j}^{(2i)}\right]^{(1,k)}\right) = \,_{i+1}p_1\left(\left[\mu_{2j}^{(2i)}\right]^{(1,k)}\right)\preceq \left[\la_{2j-1}^{(2i-1)}\right]^{(k,1)} \prec 0\cdot \F(1,\,_i\C^{(k,1)}, p_2\left(\left[\mu_{2j}^{(2i)}\right]^{(1,k)}\right)\,\cdot$$
 Thus, by \eqref{eq:1ksominf}, if $\left[\mu_{2j}^{(2i-1)}\right]^{(1,k)}= \F(m,\,_{+(i-1)}\C^{(1,k)},\left[\mu_{2j}^{(2i)}\right]^{(1,k)})$, then 
 $$m=a+2b+\chi\left(p_1\left(\left[\mu_{2j}^{(2i)}\right]^{(1,k)}\right)= 0\cdot \F(1,\,_i\C^{(k,1)}, p_2\left(\left[\mu_{2j}^{(2i)}\right]^{(1,k)}\right)\right)$$
 where
 \begin{align*}
  a&= \sharp \left\{r\in\,_{i}\C^{(k,1)}\setminus\, _{i+1}\C^{(k,1)}: \left[\la_{2j-1}^{(2i)}\right]^{(k,1)}\prec r\preceq  \left[\la_{2j-1}^{(2i-1)}\right]^{(k,1)}\right\}\,,\\
  b&= \sharp \left\{r\in\, _{i+1}\C^{(k,1)}: \left[\la_{2j-1}^{(2i)}\right]^{(k,1)}\prec r\preceq  \left[\la_{2j-1}^{(2i-1)}\right]^{(k,1)}\right\}\,\cdot
 \end{align*}
 \begin{enumerate}
  \item Since for each inserted $\bk{2i-1}$, we store one $\bk{2i-2}$ when we do \eqref{eq:insert1type221} and two $\bk{2i-2}$ when we do \eqref{eq:insert1type121}, the quantity $a+2b$ then represents the numbers of stored  $\bk{2i-2}$ during the insertions of  $\bk{2i-1}$. In addition, $p_1\left(\left[\mu_{2j}^{(2i)}\right]^{(1,k)}\right)= 0\cdot \F\left(1,\,_i\C^{(k,1)}, p_2\left(\left[\mu_{2j}^{(2i)}\right]^{(1,k)}\right)\right)$ if and only if one $\bk{2i-2}$ is stored at the end of the insertions of  $\bk{2i-1}$. Therefore, $m=m_{2j-2}^{(j+1)}$ and we retrieve fact $(2.b.i)$.
 \item Recall that, in $(1.a.i)$,
 \begin{align*}
 Q&=\sharp \left\{r\in\, _{i}\C^{(k,1)}: p_2\left(\left[\la_{2j}^{(2i+1)}\right]^{(1,k)}\right)\prec r\preceq  p_2\left(\left[\mu_{2j}^{(2i)}\right]^{(1,k)}\right)\right\}\\
 &=\sharp \left\{r\in\, _{i+1}\C^{(k,1)}: \left[\la_{2j-1}^{(2i+1)}\right]^{(k,1)}\prec r\preceq  \left[\la_{2j-1}^{(2i)}\right]^{(k,1)}\right\}-\chi\left(\bk{2i-2} \text{ is stored}\right)\,\cdot
 \end{align*}
But this quantity is exactly the number of $\bk{2i-1}$ stored during the insertions of $\bk{2i}$. Hence $Q=m_{2i-1}^{(j+1)}$.
 \end{enumerate}
\end{enumerate}
\item Remark that applying \eqref{eq:insert1type123} is equivalent to replacing $(\la_{2j-1},\la_{2j}-\la_{2j-1})$ by\\
$(\la_{2j-1},\la_{2j}-\la_{2j-1}+1)$, and applying \eqref{eq:insert1type221} for $i=2$ is equivalent to replacing\\ 
$(\la_{2j-1},\la_{2j}-\la_{2j-1})$ by $(\la_{2j-1}+1,\la_{2j}-\la_{2j-1})$. Moreover,
 \begin{align*}
  \la_{2j-1}-\sk{0}\cdot \la_{2j}> \sk{0} &\Longleftrightarrow \sk{2}\la_{2j-1}-\sk{0}(\la_{2j}-\la_{2j-1})> \sk{0}
 \end{align*}
 Hence, if $(\la_{2j-1},\la_{2j}-\la_{2j-1})=(0,0)$, we replace by $(1,0)$, so that by fact $(1.a)$ of Lemma \ref{lem:follow1k}, $[\la_{2j}]^{(1,k)}=(\left[\la_{2j-1}\right]^{(k,1)},[\la_{2j}-\la_{2j-1})]^{(k,1)})$ is replaced by $\F(1,\C^{(1,k)},[\la_{2j}]^{(1,k)})$, and $\left[\la_{2j-1}\right]^{(k,1)}$ is replaced by $p_1(\F(1,\C^{(1,k)},[\la_{2j}]^{(1,k)}))$. Otherwise, assume that $\left[\la_{2j-1}\right]^{(k,1)}= p_1([\la_{2j}]^{(1,k)})$. Then, \eqref{eq:condinsinf123} is equivalent to
 $$1>\sk{2}\la_{2j-1}-\sk{0}(\la_{2j}-\la_{2j-1})> \sk{0}$$
so that by \eqref{eq:divide++} and \eqref{eq:klratinf0}, $p_1([\la_{2j}]^{(1,k)}) = 0\cdot \F(1,\C^{(k,1)},p_2([\la_{2j}]^{(1,k)}))$. Therefore, by fact $(1.b)$ of Lemma \ref{lem:follow1k} in $d=0$, 
$$\F(1,\C^{(1,k)},[\la_{2j}]^{(1,k)})= ([\la_{2j-1}]^{(k,1)}, [\la_{2j}-\la_{2j-1}+1]^{(1,k)})\,\cdot$$
Similarly, when \eqref{eq:condinsinf123} does not occur, by fact $(1.a)$ of Lemma \ref{lem:follow1k} in $d=0$,  
$$\F(1,\C^{(1,k)},[\la_{2j}]^{(1,k)})= ([\la_{2j-1}+1]^{(k,1)}, [\la_{2j}-\la_{2j-1}]^{(1,k)})\,\cdot$$
We thus retrieve the fact that $[\la_{2j}]^{(1,k)}=(\left[\la_{2j-1}\right]^{(k,1)},[\la_{2j}-\la_{2j-1})]^{(k,1)})$ is replaced by \\
$\F(1,\C^{(1,k)},[\la_{2j}]^{(1,k)})$, and $\left[\la_{2j-1}\right]^{(k,1)}$ is replaced by $p_1(\F(1,\C^{(1,k)},[\la_{2j}]^{(1,k)}))$.
Therefore, fact $(2.a)$ holds. Moreover, we store as many $\bk{1}$ as occurrences of fact $(1.b)$ of Lemma \ref{lem:follow1k} in $d=0$, and fact $(2.b)$ follows. 
\item Fact $(3)$ is straightforward as its proof follows the same reasoning as the proof of fact $(2)$ of Proposition \ref{prop:welldefinf}.
\end{enumerate}
\subsection{Proof of Proposition \ref{prop:welldefinf2}}
The proof of fact $(1)$ is the same as the proof of fact $(1)$ of Proposition \eqref{prop:welldefinf1}, given the following equivalences:
\begin{align*}
  \la_{2j-1}-\sll{0}\cdot \la_{2j}> \sll{2i-1}-\sll{2i+1} &\Longleftrightarrow \sll{1}(\la_{2j-1}-\la_{2j})-\sll{-1}\la_{2j}> \sll{2i-1}-\sll{2i+1}\\
  &\Longleftrightarrow (\la_{2j-1}-\la_{2j})-s_0^{(k-2)}\la_{2j}> s_{i-1}^{(k-2)}-s_{i}^{(k-2)}\,\cdot
 \end{align*}
 The proof of $(2)$ follows the proof of fact $(2)$ of Proposition \ref{prop:welldefinf}.
\subsection{Proof of Proposition \ref{prop:welldefeven}}
Set $m_{2n+1}=0$ the number of inserted $\bkl{2n+1}$ in $\la_{2n},\la_{2n-1}$. Note that for $2n+1\geq i\geq 1$, $m_i$ recursively derives the numbers $m_{i-j}^{(n-j)}$ and the pairs $(\la_{2n-2j}^{(i-j)},\la_{2n-2j-1}^{(i-j)})$ for $1\leq j\leq \max\{n-1,i-1\}$, as the $\bkl{i-j}$ inserted into the pair $(\la_{2n-2j},\la_{2n-2j-1})$ can only be provided by the $\bkl{i+1-j}$ inserted into the pair $(\la_{2n-2j+2},\la_{2n-2j+1})$.
We thus have for $1\leq j\leq n$ that
$\la_{2j}^{(n+1+j)}=\la_{2j-1}^{(n+1+j)}=0$.
\begin{enumerate}
\item Suppose that for some $n\geq j>1$, $m_i^{(j)}=0$ and $\la_{2j}^{(i)}=\la_{2j-1}^{(i)}=0$, for $n+1+j\geq i>2j$. This implies that for $m_{i}^{(j-1)}=0$ and $\la_{2j-2}^{(i)}=\la_{2j-3}^{(i)}=0$, for $n+j\geq i>2j-1$. Moreover, by \eqref{eq:klratevendif2}, the insertion of $\bkl{2j}$ into the pair $(\la_{2j},\la_{2j-1})$ can only be done using \eqref{eq:insert2type2}, so that there is no part $\bkl{2j-1}$ stored for the insertion into the pair $(\la_{2j-2},\la_{2j-3})$. Therefore, $m_{2j-1}^{j-1}=0$, $\la_{2j-2}^{(2j-1)}=\la_{2j-3}^{(2j-1)}=0$, and we conclude that $m_i^{(j-1)}=0$ and $\la_{2j-2}^{(i)}=\la_{2j-3}^{(i)}=0$, for $n+j\geq i>2j-2$. Since the assumption holds for $j=n$, fact $(1)$ recursively holds for $n\geq j\geq 1$.
\item As the insertion of $\bkl{2j}$ into the pair $(\la_{2j},\la_{2j-1})$ only consists of iteration of \eqref{eq:insert2type2}, and $\la_{2j}^{(2j+1)}=\la_{2j-1}^{(2j+1)}=0$, we then obtain that 
$(\la_{2j}^{(2j)},\la_{2j-1}^{(2j)})= (m_{2j}^{(j)}\akl{2j},m_{2j}^{(j)}\alk{2j-1})$. Hence, the relation \eqref{eq:klbelongsjieven} holds for $i=2j$, as well as fact $(2.a)$.\\\\
Assume that $\left[\la_{2j}\right]^{(k,l)}_{2j}=0\cdot\left[\la_{2j-1}\right]^{(l,k)}_{2j-1}\in \,_{i}\C^{(k,l)}_{2j}$.
We first observe that, \eqref{eq:condinsinf1} is equivalent to 
$$\akl{2j+1-i}> \alk{2j-1}\la_{2j}-\akl{2j}\la_{2j-1}\geq \akl{2j+1-i}-\akl{2j-i}\,,$$
and by \eqref{eq:klratevendif2}, doing \eqref{eq:insert2type1} consists of replacing $\left[\la_{2j}\right]^{(k,l)}_{2j}$ by $\F(1,\, _{i}\C^{(k,l)}_{2j},\left[\la_{2j-1}\right]^{(k,l)}_{2j})$, and $\left[\la_{2j-1}\right]^{(l,k)}_{2j-1}$ by $\F(1,\, _{i-1}\C^{(l,k)}_{2j-1},\left[\la_{2j-1}\right]^{(l,k)}_{2j-1})$.
Otherwise,
$$\akl{2j+1-i}-\akl{2j-i}> \alk{2j-1}\la_{2j}-\akl{2j}\la_{2j-1}\geq 0\,,$$
and by \eqref{eq:klratevendif1}, doing \eqref{eq:insert2type2} consists of replacing $\left[\la_{2j}\right]^{(k,l)}_{2j}$ by $\F(1,\, _{i}\C^{(k,l)}_{2j},\left[\la_{2j-1}\right]^{(k,l)}_{2j})$, and $\left[\la_{2j-1}\right]^{(l,k)}_{2j-1}$ by $\F(1,\, _{i-1}\C^{(l,k)}_{2j-1},\left[\la_{2j-1}\right]^{(l,k)}_{2j-1})$. Therefore, the pair $(\left[\la_{2j}\right]^{(k,l)}_{2j},\left[\la_{2j-1}\right]^{(l,k)}_{2j-1})$ is replaced by the pair\\
$(\F(1,\, _{i}\C^{(k,l)}_{2j},\left[\la_{2j}\right]^{(k,l)}_{2j}),\F(1,\,_{i-1}\C^{(l,k)}_{2j-1},\left[\la_{2j-1}\right]^{(l,k)}_{2j-1}))$, and by Lemma \ref{lem:decalage},
$$0\cdot\F(1,\, _{i-1}\C^{(l,k)}_{2j-1},\left[\la_{2j-1}\right]^{(l,k)}_{2j-1})= \F(1,\, _{i}\C^{(k,l)}_{2j},0\cdot\left[\la_{2j}\right]^{(l,k)}_{2j-1})=\F(1,\, _{i}\C^{(k,l)}_{2j},\left[\la_{2j}\right]^{(k,l)}_{2j})\,\cdot$$
As $\left[\la_{2j}^{(i+1)}\right]^{(k,l)}_{2j}=0\cdot\left[\la_{2j-1}^{(i+1)}\right]^{(l,k)}_{2j-1}\in \,_{i+1}\C^{(k,l)}_{2j}\subset \, _{i}\C^{(k,l)}_{2j}$, we indeed retrieve fact $(2.a)$ and \eqref{eq:klbelongsjieven} holds for $i$.
\item Note that, by \eqref{eq:klratevendif2}, there are as many stored $\bkl{i-1}$ for the insertion into the pair $(\la_{2j-2},\la_{2j-3})$ as following sequences of $[\la_{2j}^{(i+1)}]^{(k,l)}_{2j}$ that belong to $_{i+1}\C^{(k,l)}_{2j}$, i.e.
$$m_{i-1}^{(j-1)}=\sharp\left\{r\in \, _{i+1}\C^{(k,l)}_{2j}: \left[\la_{2j}^{(i+1)}\right]^{(k,l)}_{2j}\prec r\preceq \left[\la_{2j}^{(i)}\right]^{(k,l)}_{2j} \right\}\,\cdot$$
Thus,  $_{i+1}\left[\la_{2j}^{(i)}\right]^{(k,l)}_{2j}=\F(m_{i-1}^{(j-1)},\,_{i+1}\C^{(k,l)}_{2j},\left[\la_{2j}^{(i+1)}\right]^{(k,l)}_{2j})$, and we retrieve fact $(2.b)$.
\item Finally, as $\alk{0}=0$, and $\akl{1}=1$, fact $(3)$ follows. 
\end{enumerate}
\subsection{Proof of Proposition \ref{prop:welldefeven1}}
The proof of fact $(1)$ is the same as in the previous proof. The reasoning for the proof of fact $(2),(3)$ and $(4)$ follows the proof of Proposition \ref{prop:welldefinf1}, given the equivalence
\begin{align*}
 \ak{2j-1}\la_{2j}-\ak{2j}\la_{2j-1}\geq \ak{2j+2-2i}-\ak{2j-2i} \Longleftrightarrow \ak{2j-2}\la_{2j}-\ak{2j}(\la_{2j-1}-\la_{2j})\geq \ak{2j+2-2i}-\ak{2j-2i}
\end{align*}
and by using respectively facts $(2)$ of Lemmas \ref{lem:divide} and \ref{lem:follow1k}, \eqref{eq:k1ratdif1}, \eqref{eq:k1ratdif2}, \eqref{eq:k1rat0} instead of facts $(1)$ of Lemmas \ref{lem:divide} and \ref{lem:follow1k}, \eqref{eq:klratinfdif1}, \eqref{eq:klratinfdif2}, \eqref{eq:klratinf0}.
\subsection{Proof of Proposition \ref{prop:welldefeven2}}
The proof of fact $(1)$ is the same as in the previous proof. The reasoning for the proof of fact $(2),(3)$ follows the proof of Proposition \ref{prop:welldefinf2}, given the equivalence
\begin{align*}
 \al{2j-1}\la_{2j}-\al{2j}\la_{2j-1}\geq \al{2j+1-2i}-\al{2j-1-2i} \Longleftrightarrow \ak{2j-1}(\la_{2j}-\la_{2j-1})-\ak{2j+1}\la_{2j-1}\geq \ak{2j+1-2i}-\ak{2j-1-2i}\,,
\end{align*}
and by using respectively facts $(2)$ of Lemmas \ref{lem:divide} and \ref{lem:follow1k}, \eqref{eq:k1ratdif1}, \eqref{eq:k1ratdif2}, \eqref{eq:k1rat0} instead of facts $(1)$ of Lemmas \ref{lem:divide} and \ref{lem:follow1k}, \eqref{eq:klratinfdif1}, \eqref{eq:klratinfdif2}, \eqref{eq:klratinf0}.
\subsection{Proof of Proposition \ref{prop:welldefodd}}
It uses the same reasoning as the proof of Proposition \ref{prop:welldefeven}. We simply intertwine $(k,l)$ and $(k,l)$, replace $2n$ by $2n-1$, $2j$ by $2j+1$, and the intervals $1 \text{ or } 2\leq j\leq n \text{ or } n-1$ by $1 \text{ or } 2\leq j\leq n-1 \text{ or } n-2$.
\subsection{Proof of Proposition \ref{prop:welldefodd1}}
The proof of fact $(1)$ is the same as in the previous proof. The reasoning for the proof of fact $(2),(3)$ follows the proof of Proposition \ref{prop:welldefeven2}, given the equivalence
\begin{align*}
 \ak{2j}\la_{2j+1}-\ak{2j+1}\la_{2j}\geq \ak{2j+2-2i}-\ak{2j-2i} \Longleftrightarrow \ak{2j}(\la_{2j+1}-\la_{2j})-\ak{2j+2}\la_{2j}\geq \ak{2j+2-2i}-\ak{2j-2i}\,\cdot
\end{align*}
\subsection{Proof of Proposition \ref{prop:welldefodd2}}
The proof of fact $(1)$ is the same as in the previous proof. The reasoning for the proof of fact $(2),(3)$ and $(4)$ follows the proof of Proposition \ref{prop:welldefeven}, given the equivalence
\begin{align*}
 \al{2j}\la_{2j+1}-\al{2j+1}\la_{2j}\geq \al{2j+3-2i}-\al{2j+1-2i} \Longleftrightarrow \ak{2j-1}\la_{2j+1}-\ak{2j+1}(\la_{2j}-\la_{2j+1})\geq \ak{2j+1-2i}-\ak{2j-1-2i}\,\cdot
\end{align*}

\end{document}